\documentclass[letterpaper, english, 10pt]{smfart}

\usepackage[all]{xy}

\usepackage{setspace}
\usepackage[british]{babel}

\usepackage[OT2,T1]{fontenc}

\usepackage{amssymb} 
\usepackage{mathrsfs} 

\usepackage{stmaryrd}
\usepackage{amsthm}

\usepackage{amsmath}
\usepackage[latin1]{inputenc}

\allowdisplaybreaks[1]

\usepackage{graphicx}
\usepackage{manfnt}

\usepackage{smfthm}

\usepackage{math tools} 

\usepackage{enumitem}

\usepackage[margin=3cm]{geometry}

\usepackage[pagebackref=true, colorlinks=true, linkcolor=black, citecolor=blue, urlcolor=black]{hyperref}

\usepackage[msc-links, alphabetic]{amsrefs}

\DeclareSymbolFont{cyrletters}{OT2}{wncyr}{m}{n}
\DeclareMathSymbol{\Sha}{\mathalpha}{cyrletters}{"58}

\newtheorem{theoA}{Theorem}

\newtheorem*{coro*}{Corollary}
\newtheorem*{conj*}{Conjecture}
\newtheorem*{lemm*}{Lemma}

\newcommand{\smallmat}[4]{\bigl(\begin{smallmatrix}#1&#2\\#3&#4\end{smallmatrix}\bigr)}
\providecommand{\twomat}[4]{\left(\begin{array}{cc}#1&#2\\#3&#4\end{array}\right)}
\providecommand{\smalltwomat}[4]{\left(\begin{smallmatrix}#1&#2\\#3&#4\end{smallmatrix}\right)}

\theoremstyle{definition}

\theoremstyle{remark}
\newtheorem{remark*}{Remark}

\NumberTheoremsIn{subsection}

\numberwithin{equation}{subsection}

\newcommand{\Div}{\mathrm{Div}}

\newcommand{\bmu}{\boldsymbol{\mu}} 

\newcommand{\one}{\mathbf{1}}
\newcommand{\Q}{\mathbf{Q}}
\newcommand{\qqq}{\mathbf{q}}

\newcommand{\GL}{\mathbf{GL}}
\newcommand{\SO}{\mathbf{SO}}
\newcommand{\GO}{\mathbf{GO}}

\newcommand{\X}{\mathscr{X}}
\newcommand{\calE}{\mathscr{E}}

\newcommand{\calP}{\mathscr{P}}

\renewcommand{\mod}{\ \mathrm{mod}\,}
\newcommand{\R}{\mathbf{R}}
\newcommand{\Z}{\mathbf{Z}}

\newcommand{\frakp}{\mathfrak{p}}

\newcommand{\ad}{\mathrm{ad}}

\newcommand{\into}{\hookrightarrow}

\newcommand{\Y}{\mathscr{Y}}

\newcommand{\calQ}{\mathscr{Q}}
\newcommand{\calN}{\mathscr{N}}

\newcommand{\calI}{\mathscr{I}}
\newcommand{{\calG}}{\mathscr{G}}

\newcommand{\calR}{\mathscr{R}}
\newcommand{\calS}{\mathscr{S}}

\newcommand{\bcalS}{\baar{\mathscr{S}}}

\newcommand{\lf}{\ell_{\vphi^{p},\alpha}}

\newcommand{\C}{\mathbf{C}}

\newcommand{\N}{\mathbf{N}}
\newcommand{\OO}{\mathscr{O}}
\newcommand{\A}{\mathbf{A}}

\newcommand{\bks}{\backslash}
\newcommand{\baar}{\overline}

\newcommand{\eps}{\varepsilon}
\newcommand{\vphi}{\varphi}
\newcommand{\vpi}{\varpi}

\newcommand{\Up}{\mathrm{U}}

\newcommand{\Alb}{\mathrm{Alb}\,}

\newcommand{\wtil}{\widetilde}

\newcommand{\B}{\mathbf{B}}

\newcommand{\W}{\mathscr{W}}

\newcommand{\vol}{\mathrm{vol}}
\newcommand{\Tr}{\mathrm{Tr}}

\newcommand{\Gal}{\mathrm{Gal}}

\newcommand{\Pic}{\mathrm{Pic}\,}

\newcommand{\Ker}{\mathrm{Ker}\,}

\newcommand{\Hom}{\mathrm{Hom}\,}
\newcommand{\End}{\mathrm{End}\,}

\newcommand{\llb}{\llbracket}
\newcommand{\rrb}{\rrbracket}

\newcommand{\Spec}{\mathrm{Spec}\,}

\newcommand{\tZ}{\wtil{Z}}

\def\Xint#1{\mathchoice
      {\XXint\displaystyle\textstyle{#1}}%
      {\XXint\textstyle\scriptstyle{#1}}%
      {\XXint\scriptstyle\scriptscriptstyle{#1}}%
      {\XXint\scriptscriptstyle\scriptscriptstyle{#1}}%
      \!\int}
   \def\XXint#1#2#3{{\setbox0=\hbox{$#1{#2#3}{\int}$}
        \vcenter{\hbox{$#2#3$}}\kern-.5\wd0}}
   
   \def\dashint{\Xint-}

\title[The $p$-adic Gross--Zagier formula on Shimura curves]{The $p$-adic Gross--Zagier formula\\on Shimura curves}
\author{Daniel Disegni}
\address{Departement de Math\'{e}matiques,
Universit\'{e} Paris-Sud,
91405 Orsay Cedex,
France 
}
\email{daniel.disegni@math.u-psud.fr}

\begin{document}

\begin{abstract}
We prove a  general formula for the $p$-adic heights of Heegner points on modular abelian varieties with potentially ordinary (good or semistable) reduction at the primes above $p$. The formula is in terms of  the cyclotomic derivative of a Rankin--Selberg $p$-adic $L$-function, which we construct. 
It generalises previous work of Perrin-Riou, Howard, and the author, to the context of the work of Yuan--Zhang--Zhang on the archimedean Gross--Zagier formula and of Waldspurger on toric periods. 
We further construct analytic functions interpolating Heegner points in the anticyclotomic variables, and obtain a version of our formula for them. It is complemented, when the relevant root number is $+1$ rather than $-1$, by an anticyclotomic version of the Waldspurger formula.

When combined with work of Fouquet, the anticyclotomic Gross--Zagier formula implies one divisibility in a $p$-adic Birch and Swinnerton-Dyer conjecture in anticyclotomic families. 
Other applications described in the text will appear separately. 
\end{abstract}

\maketitle

\tableofcontents

\section{Introduction}

The main results  of this paper\footnote{\emph{2010 Mathematics Subject Classification}: 11G40 (primary), 11F33, 11F41 (secondary). 

\emph{Keywords}: Gross--Zagier formula, $p$-adic $L$-function, Heegner points, $p$-adic heights, Birch and Swinnerton-Dyer conjecture.

\emph{Note added after publication}: the main theorem is off by a factor of $2$. This and other errata are noted in the author's \emph{The $p$-adic Gross--Zagier formula on Shimura curves, II: nonsplit primes}, Appendix B.
}
 are the  general formula for the $p$-adic heights of Heegner points of Theorem \ref{B} below, and its version in anticyclotomic families (contained in Theorem \ref{C}).  They are preceded by a flexible construction of the relevant $p$-adic $L$-function (Theorem \ref{A}), and complemented by a version of the Waldspurger formula in anticyclotomic families (presented in Theorem \ref{C} as well). In Theorem \ref{iwbsd}, we give an application to a version of the $p$-adic Birch and Swinnerton-Dyer conjecture in anticyclotomic families. In Theorem \ref{nvh}, we state a result on the generic non-vanishing of $p$-adic heights on CM abelian varieties, as a special case of a theorem to 
  appear in joint work with A. Burungale.

Our theorems are key  ingredients of a new  Gross--Zagier formula for exceptional zeros \cite{exc}, and of a universal $p$-adic Gross--Zagier formula specialising to analogues of Theorem \ref{B} in all weights. These  will be given in separate works. Here we would just like to mention that all of them, as well as Theorem \ref{nvh}, make essential use of the new generality of the present work.

\medskip

The rest of this introductory section contains  the statements of our results, followed by an outline of their proofs. To avoid interrupting the flow of exposition, the discussion of previous and related works (notably by Perrin-Riou and Howard) has mostly been concentrated in \S\ref{history}.

\subsection{Heegner points and multiplicity one}\label{hm1}
  Let $A$ be a simple abelian variety of $\GL_{2}$-type over a totally real field $F$; recall that this means that $M:=\End^{0}(A)$ is a field of dimension equal to the dimension of $A$. One knows how to systematically construct points on $A$ when $A$ admits parametrisations by Shimura curves in the following sense. Let $\B$ be a quaternion algebra over the ad\`ele ring $\A=\A_{F}$ of $F$, and assume that $\B$ is  \emph{incoherent}, i.e.  that its ramification set $\Sigma_{\B}$ has odd cardinality. We further assume that $\Sigma_{\B}$ contains all the archimedean places of $F$. Under these conditions there is a tower of Shimura curves $\{X_{U}\}$ over $F$ indexed by the open compact subgroups $U\subset\B^{\infty\times}$; let $X=X(\B):=\varprojlim_{U} X_{U}$. For each $U$, there is   a canonical Hodge class $\xi_{U}\in \Pic(X_{U})_{\Q}$ having degree $1$ in each connected component, inducing a compatible family $\iota_{\xi}=(\iota_{\xi, U})_{U}$ of quasi-embeddings\footnote{By `quasi-embedding', we mean an element of $\Hom(X_{U},J_{U})\otimes {\Q}$, a multiple of which is an embedding.}  $\iota_{\xi,U}\colon X_{U}\hookrightarrow J_{U}:=\Alb X_{U}$. We write $J:=\varprojlim J_{U}$. The $M$-vector space 
$$\pi=\pi_{A}=\pi_{A}(\B):={\varinjlim}_{U} \Hom^{0}(J_{U}, A)$$
is either zero or a smooth irreducible admissible representation of $\B^{\infty\times}$. It comes with  a natural stable lattice $\pi_{\Z}\subset \pi$, and its  central character
$$\omega_{A}\colon F^{\times}\bks {\A}^{\times}\to M^{\times}$$
corresponds, up to  twist by the cyclotomic character, to the determinant of the Tate module under the class field theory isomorphism. When $\pi_{A}$ is nonzero, $A$ is said to be  parametrised by $X(\B)$. Under the conditions we are going to impose on $A$, the existence of such a parametrisation, for a suitable choice of $\B$ (see below), is equivalent to the \emph{modularity} conjecture. Recall that the latter asserts the existence of a unique $M$-rational (Definition \ref{M-rat} below) automorphic representation $\sigma_{A}$ of weight~$2$ such that there is an equality of $L$-functions $L(A,s+1/2)=L(s, \sigma_{A})$. The conjecture is known to be true for ``almost all''  elliptic curves $A$ (see \cite{bao}), and when $A_{\baar{F}}$ has complex multiplication.
 
\subsubsection{Heegner points} Let $A$ be parametrised by $X(\B)$ and let $E$ be a CM extension of $F$ admitting an $\A^{\infty}$-embedding $E_{\A^{\infty}}\hookrightarrow \B^{\infty}$, which we fix; we denote by $\eta$ the associated quadratic character and by $D_{E}$ its absolute discriminant. Then $E^{\times}$ acts on $X$ and by the theory of complex multiplication each closed point of the subscheme $X^{E^{\times}}$ is defined over $E^{\rm ab}$, the maximal abelian extension of $E$. We fix one such CM point $P$. Let $L(\chi)$ be a field extension of  $M$ 
and let
$$\chi\colon E^{\times}\bks E_{\A^{\infty}}^{\times}\to L(\chi)^{\times}$$
be a finite order Hecke character such that 
$$\omega_{A}\cdot \chi|_{{\A^{\infty,\times}}}=1;$$  We can view $\chi$ as a character of $\mathscr{G}_{E}:=\Gal(\baar{E}/E)$ via the reciprocity map of class field theory (normalised, in this work, by sending uniformisers to geometric Frobenii). For each $f\in \pi_{A}$, we then have  a \emph{Heegner point}
$$P(f, \chi)=\int_{\Gal(E^{\rm ab}/E)}f(\iota_{\xi}(P)^{\tau})\otimes\chi(\tau)\, d\tau \in A(\chi).$$ 
Here  the integration uses the Haar measure of total volume~$1$, and
$$A(\chi):=(A(E^{\rm ab})\otimes_{M}L(\chi)_{\chi})^{\Gal(E^{\rm ab}/E)},$$
where $L(\chi)_{\chi}$ denotes the one-dimensional Galois module  $L(\chi)$ with action given by~$\chi$. The functional 
$f\mapsto P(f, \chi)$
defines an element of 
$$\Hom_{E_{\A^{\infty}}^{\times}}(\pi\otimes \chi, L(\chi))\otimes_{L(\chi)} A(\chi).$$
A foundational local result of Tunnell and Saito \cite{tunnell, saito} asserts that, for any irreducible representation $\pi$ of $\B^{\times}$, the $L(\chi)$-dimension of  
$${\rm H}(\pi, \chi)=\Hom_{E_{\A^{\infty}}^{\times}}(\pi\otimes \chi, L(\chi))$$
is either zero or one. It is one exactly when, for all places $v$ of $F$, the local condition
\begin{gather}\label{local cond}
\eps(1/2, \pi_{E,v}\otimes \chi_{v})=\chi_{v}(-1)\eta_{v}(-1)\eps(\B_{v})
\end{gather}
holds, where $\pi_{E}$ is the base-change of $\pi$ to $E$, $\eta=\eta_{E/F}$ is the quadratic character of $\A^{\times}$ associated to $E$ and $\eps(\B_{v})=+1$ if $\B_{v}$ is split and $-1$ if $\B_{v}$ is ramified. In this case, denoting by $\pi^{\vee}$ the $M$-contragredient representation, there is an explicit  generator
$$Q=\prod_{v\nmid \infty}Q_{v}\in {\rm H}(\pi, \chi)\otimes_{L(\chi)}{\rm H}(\pi^{\vee}, \chi^{-1})$$
defined by integration of local matrix coefficients  
\begin{align}\label{Qvdef}
Q_{v}(f_{1,v}, f_{2,v}, \chi)=
 {L(1, \eta_{v})L(1,\pi_{v}, \ad) \over \zeta_{F,v}(2) L(1/2, \pi_{E,v}\otimes\chi_{v}) }
 \int_{E_{v}^{\times}/F_{v}^{\times}} \chi_{v}(t_{v}) (\pi(t_{v})f_{1,v}, f_{2,v})_{v}\, dt_{v}, 
\end{align}
for a decomposition $(\cdot,\cdot)=\otimes_{v}(\cdot,\cdot)_{v}$ of the pairing $\pi\otimes_{M}\pi^{\vee}\to M$, and Haar measures $dt_{v}$ assigning to $\OO_{E,v}^{\times}/\OO_{F_{v}}^{\times}$ the volume $1$ if $v$ is unramified in $E$ and $2$ if $v$ ramifies in $E$. The normalisation is such that given $f_{1}$, $f_{2}$, all but finitely many terms in the product are equal to~$1$. The  pairings $Q_{v}$ in fact depend on the choice of decomposition, which in general needs an extension of scalars; the global pairing is defined over $M$ and independent of choices.

Note that the local root numbers are unchanged if one replaces $\pi$ by its Jacquet--Langlands transfer to another quaternion algebra, and that when $\pi=\pi_{A}$ they equal the local root numbers $\eps(A_{E,v,}, \chi_{v})$ of the motive $H_{1}(A\times_{\Spec F}\Spec E)\otimes_{M}\chi$ \cite{gross-mot}. This way one can view the local conditions 
\begin{gather*}
\eps(A_{E,v}, \chi_{v})=\chi_{v}(-1)\eta_{v}(-1)\eps(\B_{v})
\end{gather*}
as determining  a unique totally definite quaternion algebra $\B\supset E_{\A}$ over $\A$, which is incoherent precisely when the global root number $\eps(A_{E}, \chi)=-1$. In this case, $A$ is parametrised by $X(\B)$ in the sense described above if and only if $A$ is modular in the sense that the Galois representation afforded by its Tate module is attached to a cuspidal automorphic representation of $\GL_{2}(\A_{F})$ of parallel weight $2$. We assume this to be the case.

\subsubsection{Gross--Zagier formulas} 
There is a natural identification $\pi^{\vee}=\pi_{A^{\vee}}$, where $A^{\vee}$ is the dual abelian variety (explicitly, this is induced by the perfect $M=\End^{0}(A)$-valued pairing  $f_{1,U}\otimes f_{2, U}\mapsto \vol(X_{U})^{-1}f_{1, U}\circ f_{2}^{\vee}$ using the canonical autoduality of $J_{U}$ for any sufficiently small $U$; the normalising factor $\vol(X_{U})\in \Q^{\times} $ is the hyperbolic volume of $X_{U}(\C_{\tau})$ for any $\tau\colon F\into \C$, see \cite[\S1.2.2]{yzz}). Similarly to the above, 
 we have a Heegner point functional $P^{\vee}(\cdot, \chi^{-1})\in {\rm H} (\pi^{\vee}, \chi^{-1})\otimes_{L } A^{\vee }(\chi^{-1})$. Then the multiplicity one result of Tunnell and Saito implies that for each bilinear pairing 
$$\langle\, , \, \rangle\colon A(\chi)\otimes_{L(\chi)} A^{\vee}(\chi^{-1})\to V$$ 
with values in an $L(\chi)$-vector space $V$, there is an element $\mathscr{L}\in V$ such that 
$$\langle P(f_{1}, \chi), P(f_{2}, \chi^{-1})\rangle = \mathscr{L}\cdot  Q(f_{1}, f_{2}, \chi)$$
for all $f_{1}\in \pi$, $f_{2}\in \pi^{\vee}$. 

In this framework, we may call ``Gross--Zagier formula'' a formula for $\mathscr{L}$ in terms of $L$-functions.
When   $\langle\, , \, \rangle$ is the N\'eron--Tate height pairing valued in $\C\stackrel{\iota}{\hookleftarrow} M$ for an archimedean place $\iota$, the generalisation by Yuan--Zhang--Zhang  \cite{yzz} of the classical Gross--Zagier formula (\cite{GZ, shouwu, asian, shouwu-msri}) yields
\begin{align}\label{gzf}
\mathscr{L}=
 { c_{E}\over 2}\cdot 
   {\pi^{2[F:\Q]} |D_{F}|^{1/2}   L'(1/2, \sigma_{A,E}^{\iota}\otimes\chi^{\iota}) \over 2     L(1, \eta) L(1, \sigma_{A}^{\iota},\ad)}
\end{align}
where
\begin{align}\label{cE}
c_{E}:={   \zeta_{F}(2)\over  (\pi/2)^{[F:\Q]}|D_{E}|^{1/2}  L(1, \eta) } \in\Q^{\times},
\end{align}
and, in the present introduction, $L$-functions are as usual Euler products over all the \emph{finite} places.\footnote{In \cite{yzz}, the formula has a slightly different appearance from \eqref{gzf}, owing to the following conventions adopted there: the $L$- and zeta functions are complete including the archimedean factors;  the functional $Q$ includes  archimedean factors $Q_{v}(f_{1,v}, f_{2,v}, \chi)$, which can be shown to  equal to $1/\pi$; and finally  the product Haar measure on $E_{\A^{\infty}}^{\times}/\A^{\infty,\times}$ equals $|D_{E}|^{-1/2}$ times our measure
(cf. \cite[\S 1.6.1]{yzz}). 

(When  ``$\pi$'' appears as a factor in a numerical formula, it  denotes $\pi=3.14...$; there should be  no risk of confusion with the representation $\pi_{A}$.)}
(However in the main body of the paper  we will embrace the convention of \cite{yzz} of including the archimedean factors.)
The most important factor is the central derivative of the $L$-function $L(s, \sigma_{A,E}^{\iota}\otimes\chi)$.

When $\langle\, , \, \rangle$ is the product of the $v$-adic logarithms on $A(F_{v})$ and $A^{\vee}(F_{v})$, for a prime $v$ of $F$ which splits in $E$, the $v$-adic Waldspurger formula of Liu--Zhang--Zhang \cite{lzz} (generalising \cite{bdp}) identifies $\mathscr{L}$ with the special value of a $v$-adic Rankin--Selberg $L$-function obtained by interpolating the values $L(1/2, \sigma_{A,E}\otimes \chi'')$ at anticyclotomic Hecke characters $\chi''$ of $E$ of higher weight at $v$ (in particular, the central value for the given character $\chi$ lies \emph{outside} the range of interpolation).

The object of this paper is a formula for $\mathscr{L}$ when $\langle\, , \, \rangle$ is a $p$-adic height pairing. In this case $\mathscr{L}$ is  given  by the central derivative of a $p$-adic Rankin--Selberg $L$-function obtaining by interpolation of $L(1/2, \sigma_{A,E}, \chi')$ at finite order Hecke characters of $E$, precisely up to the factor $c_{E}/2$ of \eqref{gzf}. We describe in more detail the objects involved.

\subsection{The $p$-adic $L$-function}\label{intro-plF}
We construct the relevant $p$-adic $L$-function as a function on a space of $p$-adic characters (which can be regarded as an abelian eigenvariety), characterised by an interpolation property at locally constant characters. It further depends on a choice of local models at~$p$ (in the present case, additive characters); this point  is relevant for the study of fields of rationality and does not seem to have received much attention in the literature on $p$-adic $L$-functions.

\begin{defi}\label{M-rat}   An \emph{$M$-rational}\footnote{See \cite[\S 3.2.2]{yzz} for more details on this  notion.}
 cuspidal automorphic representation of $\GL_{2}$ of weight $2$ is  a representation $\sigma^{\infty}$ of $\GL_{2}(\A^{\infty})$ on  a rational vector space $V_{\sigma^{\infty}}$ with $\End_{\GL_{2}(\A^{\infty})}\sigma^{\infty}=M$ (then $V_{\sigma^{\infty}}$ acquires the structure of  an $M$-vector space), such that   $\sigma^{\infty}\otimes_{\Q}\sigma_{\infty}^{(2)}=\oplus_{\iota\colon M\hookrightarrow \C} \sigma^{\iota}$ is a direct sum of  irreducible cuspidal automorphic representations; here $\sigma_{\infty}^{(2)}$, a complex representation of $\GL_{2}(F_{\infty})\cong \GL_{2}(\R)^{[F:\Q]}$, is the product of discrete series of parallel weight~$2$ and trivial central character. 
 \end{defi}

\medskip

We fix from now on a rational prime $p$. 
\begin{defi}\label{def-n-ord}
Let  $F_{v}$ and $L$    be  finite extensions of $\Q_{p}$, let $\sigma_{v}$ be  a smooth irreducible representation of $\GL_{2}(F_{v})$ on an $L$-vector space, and let $\alpha_{v}\colon F_{v}^{\times }\to \OO_{L}^{\times}$ be a smooth character valued in the units of $L$. We say that $\sigma_{v}$ is  \emph{nearly ordinary for weight~$2$}
 with unit character $\alpha_{v}$ if $\sigma_{v}$ is an infinite-dimensional   subrepresentation of the un-normalised principal series ${\rm Ind}(|\cdot |_{v}\alpha_{v},\beta_{v} )$ for some other character $\beta_{v}\colon F_{v}^{\times}\to L^{\times}$. (Concretely,   $\sigma_{v}$ is then either an irreducible principal series or special of the form ${\rm St}(\alpha_{v}):= {\rm St}\otimes (\alpha_{v}\circ\det) $, where ${\rm St}$ is the Steinberg representation.)
 
 If $M$ is a number field, $\frakp$ is a prime of $M$ above $p$, and  $\sigma_{v}$ is a representation of $\GL_{2}(F_{v})$ on an $M$-vector space, we say that $\sigma_{v}$ is \emph{nearly $\frakp$-ordinary}  for weight~$2$ if there is a finite extension $L$ of $M_{\frakp}$ such that $\sigma_{v}\otimes_{M}L$ is nearly $\frakp$-ordinary for weight~$2$.
 
   In the rest of this paper we omit the clause `for weight~$2$'.\footnote{Which we have introduced in order to avoid misleading the reader into thinking of ordinariness  of an automorphic representation as a purely local notion (but see \cite{emerton-int} for how to approach it as such).}

    \end{defi}
    
Fix an $M$-rational cuspidal automorphic representation $\sigma^{\infty}$ of $\GL_{2}(\A^{\infty})$ of weight $2$; if there is no risk of confusion we will  lighten the notation and write  $\sigma$ instead of $\sigma^{\infty}$. Let  $\omega\colon F^{\times}\bks \A^{\times}\to M^{\times}$ be the central character of $\sigma$, which is necessarily of finite order.

 Fix moreover a prime $\frakp$ of $M$ above $p$ and assume that for all $v\vert p$  the local components $\sigma_{v}$ of $\sigma$  are nearly $\frakp$-ordinary with respective characters $\alpha_{v}$; we write $\alpha$ to denote the collection $(\alpha_{v})_{v\vert p }$. We replace $L$ by its subfield $M_{\frakp}(\alpha)$   generated by the values of all the  $\alpha_{v}$, and we similarly let $M(\alpha)\subset L$ be the finite extension of $M$ generated by the values of all the  $\alpha_{v}$.

\subsubsection{Spaces of $p$-adic and locally constant characters}
Fix throughout this work an arbitrary compact open subgroup $V^{p}\subset\widehat{\OO}_{E}^{p, \times}:=\prod_{w\nmid p} \OO_{E,w}^{\times}$.
 Let 
$$
\Gamma =E_{\A^{\infty}}^{\times}/\baar{E^{\times}V^{p}},\quad
\Gamma_{F} = \A^{\infty,\times}/\baar{F^{\times}\widehat{\OO}_{F}^{p, \times}}.
$$
 Then we have rigid spaces $\Y'=\Y'_{\omega}(V^{p})$, $\Y=\Y_{\omega}(V^{p})$,  $\Y_{F}$ of respective dimensions $[F:\Q]+1+\delta$, $[F:\Q]$, $1+\delta$ (where $\delta\geq 0$ is the Leopoldt defect of $F$, conjectured to be zero)  representing the functors on $L$-affinoid algebras
\begin{align*}
\Y'_{\omega}(V^{p})(A)&=\{\chi'\colon \Gamma \to A^{\times }\, :\, \omega\cdot\chi'|_{\widehat{\OO}_{F}^{p, \times}}=1\},\\
\Y_{\omega}(V^{p})(A)&=\{\chi\colon \Gamma \to A^{\times }\, :\, \omega\cdot\chi|_{\A^{\infty,\times}}=1\},\\
\Y_{F}(A)&=\{\chi_{F}\colon \Gamma_{F} \to A^{\times }\},
\end{align*}
where the sets on the right-hand sides are intended to consist of continuous homomorphisms. The inclusion $\Y\subset \Y'$  sits in the Cartesian diagram
\begin{equation}\label{cart}
\xymatrix{
\Y\ar[r]\ar[d] &\Y'\ar[d]\\
\{\one\} \ar[r] &\Y_{F}
,}
\end{equation}
where the vertical maps are given by $\chi'\mapsto \chi_{F}=\omega\cdot \chi'|_{\A^{\infty,\times}}$. When $\omega=\one$, $\Y_{\one}$ is a group object (the ``Cartier dual'' of $\Gamma/\Gamma_{F}$); in general, $\Y_{\omega}$ is a principal homogeneous space for   $\Y_{\one}$ under the action $\chi_{0}\cdot\chi=\chi_{0}\chi$. 

Let ${\bmu}_{\Q}$ denote the ind-scheme over $\Q$ of all roots of unity and ${\bmu}_{M}$ its base-change to $M$. Then there  are ind-schemes  $\Y'^{\,\rm l.c.}$, $\Y^{\rm l.c.}$, $\Y_{F}^{\rm l.c.}$, ind-finite  over $M$, representing the functors on $M$-algebras 
\begin{align*}
\Y'^{\,\rm l.c.}(A)&=\{\chi'\colon \Gamma \to  \bmu_{M}(A)\, :\, \omega\cdot\chi'|_{\widehat{\OO}_{F}^{p, \times}}=1\},\\
\Y^{\rm l.c.}(A)&=\{\chi\colon \Gamma \to \bmu_{M}(A) :\, \omega\cdot\chi|_{\A^{\infty,\times}}=1\},\\
\Y_{F}^{\rm l.c.}(A)&=\{\chi_{F}\colon \Gamma_{F} \to \bmu_{M}(A)\,    \}.
\end{align*}
where the sets on the right-hand sides are intended to consist of locally constant (equivalently, finite order) characters.
\begin{defi}\label{alg section}  Let  $\Y^{?}$  be  one of the above rigid spaces and $\Y^{?, {\rm l.c., \, an}}\subset \Y^{?}$ be the (ind-)rigid space which is the analytification of $\Y^{?, {\rm l.c.}}_{L}:=\Y^{?, {\rm l.c.}}\times_{\Spec M}\Spec L$. For any finite extension $M'$  of $M$ contained in $L$, there is a natural map
of locally $M'$-ringed spaces   $j_{M'}\colon \Y^{?,{\rm l.c., \, an}}\to \Y^{?,{\rm l.c.}}_{M'}$.
 Let $M'$ be a finite extension of $M$ contained in $L$. We say that a section $G$ of the structure sheaf of  ${\Y^{?}}$ is \emph{algebraic on $\Y^{?,{\rm l.c.}}_{M'}$}  if     its restriction   to ${\Y^{?,{\rm l.c., \, an}}}$  equals  $j_{M'}^{\sharp}G'$  for a (necessarily unique) section $G'$ of the structure sheaf of   ${\Y^{?,{\rm l.c.}}_{M'}}$.\footnote{To avoid all confusions due to the  clash of notation, $\Y_{F}^{\rm l.c.}$ will always denote the  $M$-scheme of locally constant characters of $\Gamma_{F}$ introduced above, and \emph{not} the `base-change of $\Y^{\rm l.c.}$ to $F$' (which is not defined as $M$ is not a subfield of $F$ in the generality adopted here).}
\end{defi}
In the situation of the definition, we will abusively still denote by $G$ the function $G'$ on ${\Y^{?,{\rm l.c.}}_{M'}}$.

\subsubsection{Local additive character}
Let $v$ be a non-archimedean place of $F$, $p_{v}\subset \OO_{F,v}$ the maximal ideal, $d_{v}\subset \OO_{F,v}$ the different. We define the space of additive characters of $F_{v}$ of level ${0}$ to be  
$$\Psi_{v}:=\Hom(F_{v}/d_{v}^{-1}\OO_{F,{v}}, \bmu_{\Q}) -  \Hom(F_{v}/p_{v}^{-1}d_{v}^{-1}\OO_{F,{v}}, \bmu_{\Q}),$$
where we regard $\Hom(F_{v}/p_{v}^{n}\OO_{F,v}, \bmu_{\Q})$ as a profinite group scheme over $\Q$.\footnote{If $F_{v}=\Q_{p}$, then  $\Hom(F_{v}/\OO_{F,v}, \bmu_{\Q})=T_{p}\bmu_{\Q}$, the $p$-adic Tate module of roots of unity. One could also construct and use a scheme parametrising all nontrivial characters of $F_{v}$.} The scheme $\Psi_{v}$ is a torsor for the action of $\OO_{F,v}^{\times}$ (viewed as a constant profinite group scheme over $\Q$) by $a.\psi(x):=\psi(ax)$.

If $\omega_{v}'\colon \OO_{F,v}^{\times}\to \OO(\X)^{\times}$ is a continuous character for a scheme or rigid space $\X$, we denote by $\OO_{\X\times\Psi_{v}}(\omega_{v}')\subset  \OO_{\X\times\Psi_{v}}$ the subsheaf of functions $G$ satisfying $G(x,a.\psi)=\omega_{v}'(a)(x)G(x,\psi)$ for  $a\in  \OO_{F,v}^{\times}$. By the defining property we can identify $\OO_{\X\times\Psi_{v}}(\omega_{v}')$ with ${\rm p}_{\X*}\OO_{\X\times\Psi_{v}}(\omega_{v}')$ (where ${\rm p}_{\X}\colon \X\times \Psi_{v}\to \X$ is the projection),    a locally free rank one $\OO_{\X}$-module with action by $\calG_{\Q}:=\Gal(\baar{\Q}/ \Q)$.
Finally we denote $\Psi_{p}:=\prod_{v\vert p}\Psi_{v}$ and, if $\omega_{p}'=\prod_{v\vert p}\omega_{v}'\colon \OO_{F, p}^{\times}\to \OO(\X)^{\times}$,
$$\OO_{\X\times\Psi_{p}}(\omega_{p}')=\bigotimes_{v\vert p} \OO_{\X\times\Psi_{v}}(\omega'_{v}),$$
 where the tensor product is in the category of $\OO_{\X}$-modules. Its space of global sections over $\X$ will be denoted by $\OO_{\X\times \Psi_{p}}(\X\times \Psi_{p}, \omega'_{p})$ or simply $\OO_{\X\times \Psi_{p}}(\X, \omega'_{p})$.

These sheaves will appear in the next theorem with $\omega_{v}'=\omega_{v}\chi_{F,{\rm univ},v}^{-1}\colon \OO_{F, v}^{\times}\to \OO({\Y'})^{\times}$ where $\omega_{v}$ is the central character of $\sigma_{v}$ and $\chi_{F,{\rm univ},v}\colon \OO_{F,v}^{\times}\to \OO({\Y_{F}})^{\times}\to \OO({\Y'})^{\times}$ comes from  the restriction of the universal $\OO({\Y_{F}})^{\times}$-valued character of $\Gamma_{F}$. As $\Psi_{p}$ is a scheme over $\Q$ and $\chi_{F, {\rm univ}}$ is obviously algebraic on $\Y_{F}^{\rm l.c.}$,  the notion of Definition \ref{alg section} extends 
 to define $\Y'^{\rm l.c.}_{M'}$-algebraicity of sections of $\OO_{\Y'\times \Psi_{p}}(\omega_{v}')$ (and we use the terminology ``algebraic on $\Y'^{\rm l.c.}_{M'}\times \Psi_{p}(\omega_{v}')$'').

As a last preliminary, we introduce  notation for bounded functions: if $\X$ is a rigid space then $\OO_{\X}(\X)^{\rm b}\subset \OO_{\X}(\X)$ is the space of global sections $G$ such that that $\sup_{x\in \X}|G(x)|$ is finite; similarly, in the above situation, we let
$$\OO_{\X\times \Psi_{p}}(\X, \omega_{p}')^{\rm b}:=\{G\in \OO_{\X \times \Psi_{p}}(\X, \omega_{p}')\, : \sup_{x\in \X} |G(x, \psi)| \text{\ is finite for some }\psi\in \Psi_{p}\}.$$
As $\omega_{p}'$ is continuous, $\omega_{p}'(a) $ is  bounded  in $a\in \OO_{F,p}^{\times}$: we could then equivalently replace ``is finite for some $\psi\in \Psi_{p}$'' with ``is uniformly bounded for all  $\psi\in \Psi_{p}$''.

\begin{theoA}\label{A} There is a bounded analytic function 
$$L_{p, \alpha}(\sigma_{E})\in \OO_{\Y'\times \Psi_{p}}(\Y', \omega_{p}^{-1}\chi_{F, {\rm univ}, p})^{\rm b},$$
 uniquely  determined by the following property: $L_{p, \alpha}(\sigma_{E})$ is algebraic on $\Y'^{\rm l.c.}_{M(\alpha)}\times \Psi_{p}$ and, for each $\C$-valued geometric point 
$$(\chi', \psi_{p})\in \Y'^{\, \rm l.c.}_{M(\alpha)}\times \Psi_{p}(\C),$$
  letting $\iota\colon M(\alpha)\into \C$ be the embedding induced by the composition $\Spec\C\stackrel{\chi'}{\to} \Y'^{\rm l.c.}_{M(\alpha)}\to \Spec M(\alpha)$, 
   we have
 $$L_{p,\alpha}(\sigma_{E})(\chi', \psi_{p}) =\prod_{v|p}Z_{v}^{\circ}(\chi'_{v}, \psi_{v})
  { \pi^{2[F:\Q]} |D_{F}|^{1/2} L(1/2, \sigma_{E}^{\iota}\otimes \chi'{})  
  \over 2    L(1, \eta)  L(1, \sigma^{\iota}, \ad)}$$
in $\C$. The interpolation factor is explicitly
 $$Z_{v}^{\circ}(\chi_{v}', \psi_{v}):=   
 {\zeta_{F,v}(2)L(1, \eta_{v})^{2}   \over L(1/2, \sigma_{E,v}\otimes \chi_{v}') }
  {\prod_{w\vert v} Z_{w}(\chi_{w}', \psi_{v})} $$ 
with 
\begin{equation*}
Z_{w}( \chi'_{w}, \psi_{v})=
\begin{cases}\displaystyle
\alpha_{v}(\vpi_{v})^{-v(D)} \chi_{w}'(\vpi_{w})^{-v(D)} {1-\alpha_{v}(\vpi_{v})^{-f_{w}}\chi_{w}'(\vpi_{w})^{-1}\over 1-\alpha_{v}(\vpi_{v})^{f_{w}}\chi'_{w}(\vpi_{w})q_{F,v}^{-f_{w}}} &\text{\ if $\chi_{w}'\cdot \alpha_{v}\circ q_{w}$ is unramified,}\\
 \tau(\chi'_{w}\cdot \alpha_{v}\circ q, \psi_{E_{w}}) &\text{\ if $\chi_{w}'\cdot \alpha_{v}\circ q_{w}$ is ramified.}
\end{cases}
\end{equation*}

Here $d$, $D\in \A^{\infty,\times}$ are generators of the different of $F$ and the relative discriminant of $E/F$ respectively, $q_{w}$ is the relative norm of $E/F$, $f_{w}$ is the inertia degree of $w|v$, and $q_{F,v}$ is the cardinality of the residue field at $v$; finally,  
 for any character $\wtil{\chi}_{w}'$ of $E_{w}^{\times}$ of conductor $\mathfrak{f}$,
$$ \tau(\wtil{\chi}_{w}', \psi_{E_{w}}):=\int_{w(t)=-w({\frak f})} \wtil{\chi}'_{w}(t)\psi_{E,w}(t) \, {dt} $$
with $dt$ the  additive Haar measure on $E_{w}$ giving $\vol(\OO_{E_{w}}, dt)=1$, and  $\psi_{E, w}=\psi_{F,v}\circ {\rm Tr}_{E_{w}/F_{v}}$.

\end{theoA}

\begin{rema}\label{Z psi} It follows from the description of Lemma \ref{basic local integral} that the interpolation  factors $Z_{v}^{\circ}$, $Z_{w}$ are  sections of  $\OO_{\Y_{v}'^{\rm l.c.}\times \Psi_{v}}(\omega_{v}\chi^{-1}_{F, {\rm univ}, v}),$ where $\Y_{v}'^{\rm l.c.}$ is the ind-finite reduced ind-scheme over $M(\alpha)$ representing $\bmu_{M(\alpha)}$-valued characters of $E_{v}^{\times}$. (Later, we  will also similarly denote by $\Y_{v}^{\rm l.c.}\subset \Y_{v}'^{\rm l.c.}$ the subscheme of characters satisfying $\chi_{v}|_{F_{v}^{\times}}=\omega_{v}^{-1}$.)
\end{rema}

In fact, we only construct $L_{p,\alpha}(\sigma_{E})$ as a bounded section of $\OO_{\Y'\times \Psi_{p}}(\omega_{p}^{-1} \chi_{F,{\rm univ}, p})(D)$, where $D$ is a divisor on $\Y'$ supported away from $\Y$ (i.e.,  for any polynomial function $G$ on $\Y'$ with divisor of zeroes $\geq D$, the function $G\cdot L_{p, \alpha}(\sigma_{E})$ is a bounded global section of $\OO_{\Y'\times \Psi_{p}}(\omega_{p}^{-1}\chi_{F,{\rm univ}, p})$);\footnote{A similar difficulty is encountered for example by Hida in \cite{Hi}.} see Theorem \ref{theo A text} together with Proposition \ref{interpolation factor} for the precise statement. This is sufficient for our purposes and to determine $L_{p, \alpha}(\sigma_{E})$ uniquely. One can then deduce that it is possible to take $D=0$ by comparing our $p$-adic $L$-function to some other construction where this difficulty does not arise. One such construction has been announced by David Hansen.

\subsection{$p$-adic Gross--Zagier formula}

Let us go back to the situation in which $A$ is a modular abelian variety of $\GL_{2}$-type, associated with an automorphic representation $\sigma_{A}$ of ${\rm Res}_{F/\Q}\GL_{2}$ of character $\omega=\omega_{A}$.

\subsubsection{$p$-adic heights} Several  authors (Mazur--Tate, Schneider, Zarhin,  {Nekov{\'a}{\v{r}}}, \ldots) have defined   $p$-adic  height pairings on $A(\baar{F})\times A^{\vee}(\baar{F})$ for an abelian variety $A$.  
These pairings  are analogous to  the classical  N\'eron--Tate height pairings: in particular they admit a decomposition into a sum of local symbols indexed by the (\emph{finite}) places of $F$; for $v\nmid p$ such symbols can be calculated from intersections of zero-cycles and degree-zero divisors on the local integral models of $A$.

In the general context of  {Nekov{\'a}{\v{r}}} \cite{nekheights}, adopted in this paper and recalled in \S\ref{lgh},  height pairings  can be defined for any geometric Galois representation $V$ over a $p$-adic field; we are interested in the case $V=V_{\frakp }A \otimes_{M_{\frakp}}L$  where $M=\End^{0}A$ and $L$ is a finite extension of a $p$-adic  completion $M_{\frakp}$ of $M$. Different from the N\'eron--Tate heights, $p$-adic heights  are associated with 
the auxiliary choice of splittings of the Hodge filtration on  ${\bf D}_{\rm dR}(V|_{\mathscr{G}_{F_{v}}})$ for the primes $v\vert p$; in our case, ${\bf D}_{\rm dR}(V|_{\mathscr{G}_{F_{v}}})= H^{1}_{\rm dR}(A^{\vee}/F_{v})\otimes_{M_{\frakp}} L$.  When $V|_{\calG_{F_{v}}}$ is potentially ordinary,
   meaning  that it is reducible in the category of de Rham representations
   (see more precisely Definition  \ref{pot-ord}),\footnote{This is a $\frakp$-partial version of the notion of $A_{F_{v}}$ acquiring ordinary (good or semistable) reduction over a finite extension of $F_{v}$.}
 there is a canonical such choice.
 If
 $A$ is modular corresponding to an $M$-rational cuspidal automorphic representation $\sigma_{A}^{\infty}$, it follows from \cite[Th\'eor\`eme A]{carayol-hilbert}, together with  \cite[(proof of) Proposition 12.11.5 (iv)]{nek-selmer}, that the restriction of $V=V_{\frakp}A\otimes L$ to $\calG_{F_{v}}$  is potentially ordinary 
 if and only if $\sigma_{A,v}\otimes L$ is nearly $\frakp$-ordinary.

 We assume this to be  the case for all $v\vert p$.
One then has a canonical $p$-adic height pairing
\begin{equation}\label{pht}
\langle \,, \,\rangle
\colon A(\baar{F})_{\Q}\otimes_{M} A^{\vee}(\baar{F})_{\Q}\to \Gamma_{F} \hat{\otimes} L.
\end{equation}
whose precise definition  will be recalled at the end of  \S\ref{lgh}.   Its equivariance properties under the action of $\calG_{F}=\Gal(\baar{F}/F)$ allow to deduce from it pairings
\begin{equation}\label{phtchi}
\langle \,, \,\rangle\colon A(\chi)\otimes_{L(\chi)} A^{\vee}(\chi^{-1})\to \Gamma_{F} \hat{\otimes} L(\chi)
\end{equation}
for any character $\chi\in \Y^{\rm l.c.}_{L}$.  

\begin{rema} Suppose that   $\ell\colon \Gamma_{F}\to L(\chi)$ is any continuous homomorphism  such that, for all $v\vert p$, $\ell_{v}|_{\OO_{F,v}^{\times}}\neq 0$; we then call $\ell$ a \emph{ramified logarithm}. Then it is conjectured, but not known in general,  that the   the pairings deduced from \eqref{phtchi} by composition with $\ell$ are non-degenerate. See Theorem \ref{nvh} for a new result in this direction.
\end{rema}

\begin{rema}\label{no exc rem}
If $\chi$ is not exceptional in the sense of the next definition, then \eqref{phtchi} is known to coincide with the norm-adapted height pairings \`a la Schneider \cite{schneider, nekheights}, by \cite{nekheights}, and with the  Mazur--Tate \cite{MT} height pairings, by \cite{iovita-werner}.\end{rema}
\begin{defi}\label{no exc} A locally constant  character $\chi_{w}$ of $E_{w}^{\times}$ is said to be \emph{not exceptional} if $Z_{w}(\chi_{w})\neq 0$. \footnote{As a function on $\Psi_{v}$; by Remark \ref{Z psi}, this is equivalent to $Z_{w}(\chi_{w},\psi_{v})\neq 0$ for every $\psi_{v}\in \Psi_{v}$.} A character  $\chi\in \Y_{M(\alpha)}^{\rm l.c.}$ is said to be \emph{not exceptional} if for all $w\vert p$, $\chi_{w}$ is not exceptional.
\end{defi}

\subsubsection{The formula}
  Let $\Y=\Y_{\omega}\subset \Y'=\Y_{\omega}'$ be the rigid spaces defined above. Denote by $\mathscr{I}_{\Y}\subset \OO_{\Y'}$ the ideal sheaf of $\Y$ and by
$\mathscr{N}^{*}_{\Y/\Y'}=(\mathscr{I}_{\Y} /\mathscr{I}_{\Y}^{2})|_{\Y}$ the conormal sheaf. By \eqref{cart}, 
it is canonically trivial: 
$$\mathscr{N}^{*}_{\Y/\Y'}\cong \OO_{\Y}\otimes T^{*}_{\one}\Y_{F} \cong \OO_{\Y}\otimes (  \Gamma_{F} \hat{\otimes}L).$$
 For a section $G$ of $\calI_{\Y}$, denote  ${\rm d}_{F}G\in \mathscr{N}^{*}_{\Y/\Y'}$ its image; it can be thought of as the differential in the $1+\delta$ cyclotomic variable(s).

Let $\chi\in \Y^{\rm l.c., \, an}$ be a  character  such that $\eps(A_{E}, \chi)=-1$; denote by $L(\chi)$ its residue field. By the interpolation property,  the complex functional equation and the constancy of local  root numbers, the $p$-adic $L$-function $L_{p, \alpha}(\sigma_{A,E})$ is a section of $\mathscr{I}_{\Y}$ in the connected component of $\chi\in \Y'$ (see Lemma \ref{isinI}).  Let $\B$ be the incoherent quaternion algebra determined by \eqref{local cond} and let $\pi_{A}=\pi_{A}(\B)$, $\pi_{A^{\vee}}=\pi_{A^{\vee}}(\B)$.

\begin{theoA}\label{B} Suppose that
\begin{itemize}
\item   for all $v|p$, $A/F_{v}$ has potentially $\frakp$-ordinary  good or semistable   reduction, 
\item    for all $v|p$, $E_{v}/F_{v}$ is split, 
\item  the sign $\eps(A_{E}, \chi)=-1$ and $\chi$ is not exceptional (Definition \ref{no exc}). 
\end{itemize} 
Then for  all $f_{1}\in \pi_{A}$, $f_{2}\in \pi_{A^{\vee}}$ we have\footnote{The formula is off by a factor of $2$. See footnote on the first page.}
$$\langle P(f_{1}, \chi), P^{\vee}(f_{2}, \chi^{-1})\rangle = 
{c_{E}\over 2}\cdot
\prod_{v|p}Z_{v}^{\circ}(  \chi_{v})^{-1}\cdot {\rm d}_{F} L_{p, \alpha}(\sigma_{A,E})(\chi)\cdot  Q(f_{1}, f_{2}, \chi)$$
in $\mathscr{N}^{*}_{{\Y/\Y'}}{}_{| \chi }\cong \Gamma_{F} \hat{\otimes} L(\chi)$. Here $c_{E}$ is as in \eqref{cE}.
\end{theoA}
In the right-hand side, we have considered Remark \ref{Z psi} and used the canonical isomorphism $\OO_{\Psi_{p}}(\omega_{p}^{-1})\otimes_{M}\OO_{\Psi_{p}}(\omega_{p})=M$.

\subsection{Anticyclotomic theory}\label{intro anti}

Consider the setup of \S\ref{intro-plF}. Recall  that in the  case $\eps(1/2,\sigma_{E}, \chi)=+1$, the definite quaternion algebra $\B$ defined by \eqref{local cond} is coherent, i.e. it arises as $\B=B\otimes_{F}\A_{F}$ for a quaternion algebra $B$ over $F$; we may assume that the embedding $E_{\A}\into \B$ arises from an embedding $i\colon E\into B$. Let $\pi$ be  the  automorphic representation of $\B^{\times}$ attached to $\sigma$ by the Jacquet--Langlands correspondence;  it is realised in the space of locally constant functions  $B^{\times}\bks \B^{\times}\to M$, and this gives a stable lattice $\pi_{\OO_{M}}\subset \pi$. Then, given a character $\chi\in \Y^{\rm l.c.}$,    the formalism of \S\ref{hm1} applies to the period functional $p\in {\rm H}(\pi, \chi)$ defined by
\begin{align}\label{perint}
p(f, \chi):=\int_{E^{\times}\bks E^{\times}_{\A^{\infty}}} f(i(t)) \chi(t)\, dt
\end{align}
and to its dual $p^{\vee}(\cdot, \chi^{-1})\in {\rm H}(\pi^{\vee}, \chi^{-1})$. Here $dt$ is the Haar measure of total volume $1$.
 
The  formula expressing the decomposition of their product was proved by Waldspurger (see \cite{wald} or \cite{yzz}): for all finite order characters $\chi\colon E^{\times}\bks E_{\A}^{\times}\to M(\chi)^{\times}$ valued in some extension $M(\chi)\supset M$, and for all $f_{1}\in \pi$, $f_{2}\in \pi^{\vee}$,  we have
\begin{align}\label{WF}
p(f_{1}, \chi)p^{\vee}(f_{2}, \chi^{-1})=
{c_{E}\over 4}\cdot 
  {  \pi^{2[F:\Q]}   |D_{F}|^{1/2}  L(1/2, {\sigma}_{E}\otimes  \chi) 
  \over 2  L(1, \eta) L(1, \sigma,\ad)}\cdot Q(f_{1}, f_{2}, \chi)
\end{align}
in $M(\chi)$. Notice that here we could trivially modify the right-hand side to replace the complex $L$-function with the $p$-adic $L$-function, thanks to the interpolation property defining the latter. 

The $L$-function terms of both the Waldspurger and the $p$-adic Gross--Zagier formula thus admit an interpolation as analytic functions (or sections of a sheaf) on $\Y_{\omega}$. We can show that the other  terms do as well.

\medskip

Let  $\pi$  be the $M$-rational representation of the (coherent or incoherent) quaternion algebra $\B^{\times}\supset E_{\A}^{\times}$ considered above, with central character $\omega$. It will be convenient to denote $\pi^{+}=\pi$, $\pi^{-}=\pi^{\vee}$, ,$p^{+}=p$, $p^{-}=p^{\vee}$,   $\Y_{\pm}=\Y_{\omega^{\pm 1}}$, and in the incoherent case $A^{+}=A$, $A^{-}=A^{\vee}$, $P^{+}=P$, $P^{-}=P^{\vee}$, $\sigma=\sigma_{A}$.

We have a natural isomorphism 
 $ \Y_{+}\cong \Y_{-}$ given by inversion.  If
 $\mathscr{F}$ is a
  sheaf on $\Y_{-}$, we denote by $\mathscr{F}^{\iota}$ its pullback to a
  sheaf on $\Y_{+}$; the same notation is used to transfer sections of such sheaves.

\subsubsection{Big Selmer groups and heights}

Let $\chi_{\rm univ}^{\pm}\colon \Gamma \to (\OO({\Y_{\pm}})^{{\rm b}})^{\times}$ be the tautological  character such that $\chi_{\rm univ}^{\pm}(t)(\chi)=\chi(t)^{\pm 1}$ for all $\chi\in \Y_{\pm}$, and define an $\OO(\Y_{\pm})^{\rm b}$-module
$${\bf S}_{\frakp}(A_{E}^{\pm},\chi_{\rm univ}^{\pm}, \Y_{\pm})^{\rm b}:=H^{1}_{f}(E, V_{\frakp}A_{E}^{\pm}\otimes \OO({\Y_{\pm}})^{\rm b}(\chi_{\rm univ}^{\pm})),$$
where   $\OO({\Y_{\pm}})^{\rm b}(\chi_{\rm univ}^{\pm})$  denotes  the module of bounded global sections $\OO({\Y_{\pm}})^{\rm b}$ with $\calG_{E}$-action  
by $\chi_{\rm univ}^{\pm}$.  Here, for a topological $\Q_{p}[\mathscr{G}_{E}]$-module $V$ which is potentially ordinary at all $w\vert p$ in the sense of Definition \ref{pot-ord} below, with exact sequences $0\to V_{w}^{+}\to V_{w}\to V_{w}^{-}\to 0$, the (Greenberg) Selmer group $H^{1}_{f}(E, V)\subset H^{1}(E, V):=H^{1}(\calG_{E},V)$ is the group
of those {continuous} cohomology classes $c$ which are unramified away from $p$ and such that, for every $w\vert p$, the restriction of $c$ to a decomposition group at a  $w$ is in the kernel of
$$H^{1}(E_{w}, V)\to H^{1}(E_{w}, V^{-})..$$
(In the case at hand,  $V_{w}^{-}=V_{\frakp}A_{E}^{\pm}|_{\calG_{E,w}}^{-}$ is the maximal potentially unramified quotient of $V_{\frakp}A_{E}^{\pm}|_{\calG_{E,w}}$; cf. \S\ref{lgh}). For every \emph{non-exceptional} $\chi^{\pm}\in \Y^{\rm l.c.}_{\pm}$, the specialisation $ {\bf S}_{\frakp}(A_{E}^{\pm},\chi_{\rm univ}^{\pm}, \Y_{\pm})^{\rm b}\otimes L(\chi)$ is isomorphic to the target of the Kummer map
\begin{align}\label{kappais}
\kappa\colon A_{E}^{\pm}(\chi^{\pm})\to H^{1}_{f}(E, V_{\frakp}A_{E}\otimes L(\chi^{\pm})_{\chi^{\pm}}).
\end{align}

The work of Nekov{\'a}{\v{r}} \cite{nek-selmer} explains the exceptional specialisations and provides a height pairing on the big  Selmer groups. The key underlying object is the  \emph{Selmer complex}
\begin{align}\label{thebigSC}
\wtil{{\rm R}\Gamma}_{f}(E,V_{\frakp}A_{E}^{\pm}\otimes \OO(\Y^{\circ}_{\pm})^{\rm b}(\chi^{\pm}_{\rm univ}) ),
\end{align}
  an object in the derived category of $\OO(\Y_{\pm})^{\rm b }$-modules defined as in \cite[\S 0.8]{nek-selmer} taking $T= V_{\frakp}A_{E}^{\pm}\otimes \OO(\Y^{\circ}_{\pm})^{\rm b}(\chi^{\pm}_{\rm univ})$ and $U_{w}^{+}=V_{\frakp}A_{E}^{\pm}|_{\calG_{E_{w}}}^{+}\otimes \OO(\Y^{\circ}_{\pm})^{\rm b}(\chi^{\pm}_{\rm univ})$ in the notation of \emph{loc. cit.} Its
first cohomology group 
$$\wtil{H}^{1}_{f}(E,V_{\frakp}A_{E}^{\pm}\otimes \OO(\Y^{\circ}_{\pm})^{\rm b}(\chi_{\rm univ}^{\pm}) )$$
satisfies the following property. 
  For every $L$-algebra quotient  $R$ of  $\OO(\Y^{\circ}_{+})^{\rm b}$, letting $\chi_{R}^{\pm}\colon \Gamma \to R^{\times}$  be the character deduced from $\chi_{\rm univ}^{\pm}$,  there is  an exact sequence (\emph{ibid.} (0.8.0.1))
\begin{align}
0\to \bigoplus_{w\vert p} H^{0}(E_{w}, V_{\frakp}A_{E}^{\pm}|_{\calG_{E,w}}^{-}\otimes R(\chi_{R}^{\pm }))\to
 \wtil{H}^{1}_{f}(E,V_{\frakp}A_{E}^{\pm}\otimes \OO(\Y^{\circ}_{+})^{\rm b}(\chi_{\rm univ}^{\pm }) )\otimes R\to 
 {H}^{1}_{f}(E,V_{\frakp}A_{E}^{\pm}\otimes R(\chi_{R}^{\pm }))\to 0. 
\end{align}
When  $R=\OO(\Y_{\pm})^{\rm b}$ itself, each  group $H^{0}(E_{w}, V_{\frakp}A_{E}^{\pm}|_{\calG_{E,w}}^{-}\otimes \OO(\Y_{\pm })^{\rm b}(\chi_{\rm univ}^{\pm}))$ vanishes as $\chi_{{\rm univ}, w}$ is infinitely ramified;   hence  
$$\wtil{H}^{1}_{f}(E,V_{\frakp}A_{E}^{\pm}\otimes \OO(\Y^{\circ}_{\pm})^{\rm b}(\chi_{\rm univ}^{\pm}) )\cong {\bf S}_{\frakp}(A_{E}^{\pm}, \chi_{\rm univ}^{\pm}, \Y_{\pm})^{\rm b}.$$ 
When $R=L(\chi)$ with $\chi\in \Y^{\rm l.c.}$, the group $H^{0}(E_{w}, V_{\frakp}A_{E}^{\pm}|_{\calG_{E,w}}^{-}\otimes L(\chi^{\pm})_{\chi^{\pm}})$ vanishes unless $\chi_{w}\cdot\alpha_{v}\circ q_{w}=\one$ on $E_{w}^{\times}$, that is 
unless $\chi_{w}$ is exceptional.

Finally, by \cite[Chapter 11]{nek-selmer}, there is a  big height pairing
\begin{align}\label{bight}
\Big\langle\, , \,\Big\rangle \colon 
{\bf S}_{\frakp}(A_{E}^{+},\chi_{\rm univ}^{+}, \Y_{+})^{\rm b} \otimes_{\OO_{\Y_{+}}} {\bf S}_{\frakp}(A_{E}^{-},\chi_{\rm univ}^{-}, \Y_{-})^{{\rm b},\iota}\to   \calN^{*}_{\Y_{+}/\Y'_{+}}(\Y_{+})^{\rm b}
\end{align}
interpolating the  height pairings on  ${H}^{1}_{f}(E, V_{\frakp}A\otimes L(\chi^{\pm})_{\chi^{\pm}})$ for  non-exceptional $\chi\in \Y^{\rm l.c.}$ (and more generally  certain `extended' pairings on $\wtil{H}^{1}_{f}(E, V_{\frakp}A\otimes L(\chi^{\pm})_{\chi^{\pm}})$ for all $\chi\in \Y^{\rm l.c.}$; these will play no role here).

\subsubsection{Heegner--theta elements and anticyclotomic formulas} 
Keep the assumptions that for all $v\vert p$, $E_{v}/F_{v}$ is split and   $\pi_{v}\cong \sigma_{v}$ is $\mathfrak{p}$-nearly ordinary with unit character $\alpha_{v}$.
Then,
 after tensoring with $\OO_{\Psi_{p}}(\Psi_{p})$ (in order to use Kirillov models at $p$)
we will have a decomposition $\pi^{\pm}\cong \pi^{\pm, p}\otimes\pi_{p}^{\pm}$, which  is an isometry with respect to   pairings $(\, , \,)^{p}$, $(\, ,\, )_{p}$ on each of the factors.  By \eqref{Qvdef}, for each $\chi=\chi^{p}\chi_{p}\in \Y^{ {\rm l.c.}}_{M}$ we can then define a toric period 
\begin{gather}\label{toric away p}
Q^{p}(f^{+,p}, f^{-,p}, \chi)\in M(\chi)\otimes \OO_{\Psi_{p}}(\omega_{p}^{-1}).
\end{gather}

Given   $f^{\pm, p}\in\pi^{\pm, p}$, we will construct an explicit pair of elements
\begin{align}\label{fA}
f_{\alpha}^{\pm}=(f_{\alpha,V_{p}}^{\pm})=(f^{\pm, p}\otimes f^{\pm}_{\alpha,p,  V_{p}})_{V_{p}}
  \in \pi^{\pm, p}_{M(\alpha)}\otimes \varprojlim_{V_{p}}\pi_{p}^{\pm, V_{p}}
\end{align}
 where the inverse system is indexed by  compact open subgroups $V_{p}\subset E_{p}^{\times}\subset \B_{p}^{\times}$ containing $\Ker(\omega_{p})$, with transition maps being given by averages under their $\pi^{\pm}_{p}$-action, and  $ f^{\pm}_{\alpha,p,  V_{p}}$ are suitable elements of $ \pi_{p}^{\pm, V_{p}}$.
We compute in Lemma \ref{toric at p} that we have
$$Q_{p}(f_{\alpha,p}^{+}, f_{\alpha,p}^{-})= \zeta_{F,p}(2)^{-1} \prod_{v\vert p} Z_{v}^{\circ}$$
as sections of  $\bigotimes_{v\vert p}\OO_{\Y_{v}^{\rm l.c.}\times \Psi_{v}}(\omega_{v})$,  where the left-hand side in the above expression is computed, for each $\chi_{p}\in \prod_{v\vert p} \Y_{v}^{\rm l.c.}$, as the limit of $Q_{p}(f_{\alpha,p, V_{p}}^{+}, f_{\alpha,p,V_{p}}^{-})$ as $V_{p}\to \Ker(\omega_{p})$.

For the following theorem, note that all the local  signs   in  \eqref{local cond} extend to  locally constant functions of $\Y_{+}$ (this is a simple special case of \cite[Proposition 3.3.4]{PX}); the quaternion algebra over $\A$ determined by \eqref{local cond} is then also constant along the connected components of  $\Y_{+}$.   We will  say that a connected component $\Y^{\circ}_{+}\subset \Y_{+}$ is \emph{of type $\eps\in \{\pm1\}$} if $\eps(1/2, \sigma_{E}, \chi)=\eps$ along $\Y^{\circ}$.

\begin{theoA}\label{C}
Let $\Y^{\circ}_{+}\subset \Y_{+}$ be a connected component of type $\eps$,  let $\B$ be the quaternion algebra determined by \eqref{local cond}, and let $\pi^{\pm}$ be the representations of $\B^{\times}$ constructed above. Finally, let  $\Y_{-}^{\circ}\subset \Y_{-}$ be the  image of $\Y^{\circ}_{+}$ under the inversion map.
\begin{enumerate} 
\item\label{C1} (Heegner--theta elements.)\ \ \ 
For each  $f^{\pm, p}\in \pi^{\pm, p}$, there are elements 
\begin{align*}
\Theta_{\alpha}^{\pm}(f^{\pm, p})& \in \OO_{\Y_{+}}(\Y^{\circ}_{\pm})^{\rm b} &\text{if $\eps=+1$,}\\
\calP_{\alpha}^{\pm}(f^{\pm,p}) & \in 
{\bf S}_{\frakp}(A_{E}^{\pm},\chi_{\rm univ}^{\pm},\Y^{\circ}_{\pm})^{\rm b}& \text{if $\eps=-1$,}\\
\end{align*}
uniquely determined  by the property   that, for any  compact open subgroup $V_{p}\subset E_{p}^{\times }$ and any $V_{p}$-invariant character $\chi^{\pm}\in \Y_{\pm}^{\circ}$, 
we have
\begin{align*}
\Theta^{\pm}_{\alpha}(f^{\pm, p})(\chi^{\pm}) &= p(f^{\pm}_{ \alpha, V_{p}}, \chi^{\pm}),\\
\calP^{\pm}_{\alpha}(f^{\pm, p})(\chi^{\pm}) &=\kappa(P(f^{\pm}_{ \alpha, V_{p}}, \chi^{\pm})).
\end{align*}
where $f^{\pm}_{\alpha}$ is the element \eqref{fA}, $p(\cdot) $ is the period integral \eqref{perint}, and $\kappa$ is the Kummer map \eqref{kappais}.
\item\label{C2}   There is an element 
$$\calQ=\zeta_{F,p}(2)^{-1}\prod_{v\nmid p}\calQ_{v}\in \Hom_{ \OO(\Y_{+}^{\circ})^{\rm b}[E^{\times}_{\A^{p\infty}}]}(\pi^{+, p}\otimes \pi^{-, p}\otimes  \OO(\Y_{+}^{\circ})^{\rm b}, \OO(\Y_{+}^{\circ})^{\rm b}\otimes\OO_{\Psi_{p}}(\omega_{p}^{-1}))$$
uniquely determined by the property that, for all $f^{\pm, p}\in \pi^{\pm,p}$ and all $\chi\in \Y^{\circ \, {\rm l.c.}}_{+}$, we have 
$$\calQ(f^{+, p}, f^{-, p})(\chi)=\zeta_{F,p}(2)^{-1} \cdot Q^{p}(f^{+, p}, f^{-, p},\chi^{p}).$$
\item (Anticyclotomic Waldspurger formula.)\label{C3}\ \
If $\eps=+1$, we have\footnote{The formula is off by a factor of $2$. See footnote on the first page.}
\begin{align*} 
\Theta^{+}_{\alpha}(f^{+,p})\cdot \Theta^{-}_{\alpha}(f^{-,p})^{\iota}  = 
{c_{E}\over 4}\cdot
 L_{p ,\alpha}(\sigma_{E})\cdot\mathscr{Q}(f^{+,p},f^{-,p}) 
\end{align*}
in $ \OO(\Y_{+}^{\circ})^{\rm b}$.
\item\label{C4} (Anticyclotomic Gross--Zagier formula.) \ \ If $\eps=-1 $ and $A$ has potentially $\frakp$-ordinary reduction at all $v\vert p$,  we have\footnote{The formula is off by a factor of $2$. See footnote on the first page.}

\begin{align*} 
  \Big\langle \calP^{+}_{\alpha}(f^{+,p}), \calP^{-}_{\alpha}(f^{ -,p})^{\iota} \Big\rangle  = 
{c_{E}\over 2}\cdot
  {\rm d}_{F} L_{p ,\alpha}(\sigma_{E}) \cdot\mathscr{Q}(f^{+,p},f^{-,p})
\end{align*}
 in $ \calN^{*}_{\Y_{+}/\Y'_{+}} (\Y^{\circ}_{+} ) ^{\rm b} $.
\end{enumerate}
\end{theoA}

In parts \ref{C3} and \ref{C4}, we have used the canonical isomorphism $\OO_{\Psi_{p}}(\omega_{p})\otimes\OO_{\Psi_{p}}(\omega_{p}^{-1})\cong M$. The height pairing of part \ref{C4} is \eqref{bight}.

\begin{rema}  Theorem \ref{C}.\ref{C4} specialises to $0=0$ at any {exceptional} character $\chi\in \Y^{\rm l.c.}$, and in fact by the archimedean Gross--Zagier formula of \cite{yzz}  it follows that  the `pair of points' $\calP^{+}_{\alpha}(f^{+,p})\otimes\calP^{-}_{\alpha}(f^{ -,p})^{\iota} $ itself  vanishes there. The leading term of $L_{p, \alpha}$ at exceptional characters is studied in \cite{exc}.
\end{rema}

\subsection{Applications}  Theorem \ref{B} has by now  standard applications to the $p$-adic and the classical Birch and Swinnerton-Dyer conjectures;  the interested reader will have no difficulty in obtaining them as in  \cite{PR, dd}. We obtain in particular one $p$-divisibility in the classical Birch and Swinnerton-Dyer conjecture for a $p$-ordinary CM elliptic  curve $A$ over a totally real field as in \cite[Theorem D]{dd} without the spurious assumptions of \emph{loc. cit.} on the behaviour of $p$ in $F$. In the rest of this subsection we describe two other applications.

\subsubsection{On the $p$-adic Birch and Swinnerton-Dyer conjecture in anticyclotomic families}
The next theorem, which can be thought of as a case of the $p$-adic Birch and Swinnerton-Dyer conjecture in anticyclotomic families, combines Theorem \ref{C}.\ref{C4} with work of Fouquet \cite{fouquet} to generalise a result of Howard \cite{howard} towards a conjecture of  Perrin-Riou \cite{PR2}.    We first introduce some notation: let  $\Lambda:=\OO(\Y^{\circ}_{+})^{\rm b}$,
 and let the \emph{anticyclotomic height regulator} 
 \begin{align}\label{anticyc reg}
 \mathscr{R}\subset \Lambda\hat{\otimes} \,{\rm Sym}^{r}\Gamma_{F} 
 \end{align}
be the discriminant of  \eqref{bight} on the $\Lambda$-module
$${\bf S}_{\frakp}(A^{+}_{E},\chi_{\rm univ}^{+}, \Y_{+}^{\circ})^{\rm b}\otimes_{\Lambda} {\bf S}_{\frakp}(A^{-}_{E},\chi_{\rm univ}^{-}, \Y^{\circ}_{-})^{{\rm b},\iota},$$
where the integer $r$ in \eqref{anticyc reg} is the generic rank of the finite type $\Lambda$-module ${\bf S}_{\frakp}(A_{E}^{+},\chi_{\rm univ}^{+}, \Y^{\circ}_{+})^{\rm b}$.
Recall that this module is the first cohomology of the Selmer complex $\wtil{{\rm R}\Gamma}_{f}(E,V_{\frakp}A^{+}\otimes \OO(\Y^{\circ}_{+})^{\rm b}(\chi^{+}_{\rm univ}) )$ of 
 \eqref{thebigSC}. Let 
   $$\wtil{H}^{2}_{f} (E,V_{\frakp}A^{+}\otimes \OO(\Y)^{\rm b}(\chi_{\rm univ}) )_{\rm tors}$$
   be the torsion part of the  second cohomology group.
    Its characteristic ideal in $\Lambda$
     can roughly be thought of as interpolating the $p$-parts of the rational terms (order of the Tate--Shafarevich group,  Tamagawa numbers) appearing on the algebraic  side of  the Birch and Swinnerton-Dyer conjecture for $A(\chi)$.

\begin{theoA}\label{iwbsd} In the situation of Theorem \ref{C}.\ref{C4}, assume furthermore that
\begin{itemize}
\item  $p\geq 5$;
\item $V_{\frakp}A$ is potentially crystalline as a $\calG_{F_{v}}$-representation for all   $v\vert p$;
\item the character $\omega$ is trivial and $\Y^{\circ}$ is the connected component of $\one \in \Y$;
\item the residual representation $\baar{\rho}\colon \calG_{F}\to {\rm Aut}_{{\bf F}_{\frakp}}(T_{\frakp}A\otimes {\bf F}_{\frakp})$ is  irreducible (where ${\bf F}_{\frakp}$ is the residue field of $M_{\frakp}$), and it remains irreducible when restricted to the Galois group of the Hilbert class field of $E$;
\item for all $v\vert p$, the image of $ \rho|_{\mathscr{G}_{F,v}}$  is not scalar.
\end{itemize}

Then  
 $${\bf S}_{\frakp}(A_{E},\chi_{\rm univ}, \Y^{\circ})^{\rm b}, \qquad  {\bf S}_{\frakp}(A_{E},\chi_{\rm univ}^{-1}, \Y^{\circ})^{{\rm b}, \iota} $$
both  have generic rank~$1$ over 
 $\Lambda$,
a non-torsion element of their   tensor product over $\Lambda$    is given by any $\mathscr{P}_{\alpha}^{+}(f^{+,p})\otimes \mathscr{P}^{-}_{\alpha}(f^{-,p})^{\iota}$ such that $\mathscr{Q}(f^{+,p},f^{-,p})\neq 0$, and 
 \begin{equation}\label{iw-f}
\begin{gathered}
({\rm d}_{F}L_{p, \alpha}(\sigma_{E})|_{\Y^{\circ}})
\quad \subset\quad
  \mathscr{R}\cdot {\rm char}_{\Lambda} \wtil{H}^{2}_{f}(E, V_{\frakp}A\otimes \Lambda(\chi_{\rm univ}) )_{\rm tors}
\end{gathered}
\end{equation}
as $\Lambda$-submodules of $\Lambda\hat{\otimes}\Gamma_{F}$. 
\end{theoA}
The `potentially crystalline' assumption for $V_{\frakp}A$, which is satisfied if $A$ has potentially   good reduction at all $v\vert p$,  is imposed in order for $V_{\frakp}A\otimes \OO(\Y^{\circ})^{\rm b}$ to be `non-exceptional' in the sense of \cite{fouquet} (which is more restrictive than ours); the assumption  on $\omega$  allows to invoke the results of \cite{CV05, AN} on the non-vanishing of anticyclotomic Heegner points, and to write $A=A^{+}=A^{-}$,  $\Y=\Y_{+}=\Y_{-}=\Y_{\one}$.
See \cite[Theorem B (ii)]{fouquet} for the exact assumptions  needed, which are slightly weaker.

The proof of Theorem \ref{iwbsd} will be  given in \S\ref{sec iwbsd}.

\begin{rema} When $F=\Q$, the converse divisibility to \eqref{iw-f} was recently proved by X. Wan \cite{xin-howard} under some assumptions.
\end{rema}

\subsubsection{Generic non-vanishing of $p$-adic heights on CM abelian varieties} The non-vanishing of (cyclotomic) $p$-adic heights is in general, as we have mentioned, a deep conjecture (or  a ``strong suspicion'') of Schneider  \cite{schneider2}. The following result  provides some new evidence towards it. It is a corollary of  Thereom \ref{C}.\ref{C4} together with the non-vanishing results for Katz $p$-adic $L$-functions  of Hida \cite{hidamu}, Hsieh \cite{hsiehmu}, and Burungale \cite{ashay} (via a factorisation of the $p$-adic  $L$-function). The result  is a special case of a finer  one to appear in forthcoming joint work with A. Burungale. For CM elliptic curves over $\Q$, it  was known as a consequence of different non-vanishing results of Bertrand \cite{bertrand} and Rohrlich  \cite{rohrlich} (see \cite[Appendix A, by K. Rubin]{agb}).

\begin{theoA}\label{nvh} In the situation of Theorem \ref{C}.\ref{C4}, suppose that $A_{E}$ has complex multiplication\footnote{In the strict sense that the algebra $\End^{0}(A_{E})$ of endomorphisms \emph{defined over $E$} is a CM field.}
  and  that $p\nmid 2D_{F}h_{E}^{-}$, where $h_{E}^{-}=h_{E}/h_{F}$ is the relative class number. Let $\big\langle\, ,\, \big\rangle_{\rm cyc}$ be the pairing deduced from \eqref{bight} by the map $\mathscr{N}_{\Y/\Y'}(\Y^{\circ})^{\rm b}\cong \OO(\Y^{\circ})^{\rm b} \hat{\otimes}\Gamma_{F}\to \OO(\Y^{\circ})^{\rm b} \hat{\otimes}\Gamma_{\rm cyc}$, where $\Gamma_{\rm cyc}=\Gamma_{\Q}$ viewed as a quotient of $\Gamma_{F}$ via the ad\`elic norm map.

Then for any $f^{\pm, p}$ such that $\mathscr{Q}(f^{+,p}, f^{-,p})\neq 0 $ in $\OO(\Y^{\circ})^{\rm b}$, we have
$$\Big\langle \calP_{\alpha}^{+}(f^{+,p}), \calP_{\alpha}^{-}(f^{-,p})^{\iota} \Big\rangle_{\rm cyc}\neq 0\qquad \text{in }  \OO(\Y^{\circ})^{\rm b} \hat{\otimes}\Gamma_{\rm cyc}.$$
\end{theoA}

\subsection{History and related work}\label{history} We briefly discuss previous work towards our main theorems, and some related works.   We will loosely term ``classical context'' the following specialisation of the setting of our main results: $A$ is an elliptic curve over $\Q$ with conductor $N$ and \emph{good} ordinary reduction at $p$; $p$ is odd; the quadratic imaginary field $E$ has discriminant coprime to $N$ and it satisfies the  Heegner condition: all primes dividing $N$ split in $E$ (this implies that $\B^{\infty}$ is split); the parametrisation $f\colon J\to A$ factors through the Jacobian of the modular curve $X_{0}(N)$; the character $\chi$ is unramified everywhere, or unramified away from $p$. 

\paragraph{Ancestors} In the classical context, Theorems \ref{A} and \ref{B} were proved by Perrin-Riou \cite{PR};  intermediate steps towards the present generality were taken in \cite{dd}, \cite{lima}. When $E/F$ is split above $p$, Theorem \ref{A} can essentially be deduced from a general theorem of Hida \cite{Hi} (cf. \cite[\S 7.3]{xin-mc}), except for the location of the possible poles.  Theorem \ref{C}.\ref{C4} and Theorem \ref{iwbsd} in the classical  context are due to Howard \cite{howard} (in fact Theorems \ref{B} and \ref{C}.\ref{C4} were first envisioned by Mazur \cite{mazur-icm} in that context, whereas Perrin-Riou  \cite{PR2} had conjectured the equality in \eqref{iw-f}).  Theorem \ref{C}.\ref{C3} is hardly new and has many antecedents in the literature: see e.g. \cite{VO} and references therein. 

\paragraph{Relatives} Some analogues of Theorem \ref{B} were proven in situations which differ from the classical context in directions which are orthogonal to those of the present work:  Nekov{\'a}{\v{r}} \cite{nekovar} and Shnidman \cite{ari} dealt with the case of higher weights; Kobayashi \cite{kobayashi} dealt with the case of elliptic curves with supersingular reduction. 

\paragraph{Friends} We have already mentioned two other fully general Gross--Zagier formulas in the sense of \S\ref{hm1}, namely the original archimedean one of \cite{yzz} generalising \cite{GZ}, and a different $p$-adic formula proved in \cite{lzz} generalising \cite{bdp}. The panorama of existing formulas of this type is complemented by a handful of   results, mostly in the classical context, valid in the presence of an exceptional zero (the case excluded in Theorem \ref{B}). We refer the reader to \cite{exc} where we prove a new such formula for $p$-adic heights and review other ones due to   Bertolini--Darmon.
It is to be expected that all of those results should be generalisable    to the framework of \S\ref{hm1}.

\paragraph{Children} Finally, explicit versions of any Gross--Zagier formula in the framework of \S\ref{hm1}   can be obtained by the explicit computation of the local integrals $Q_{v}$. This is carried out in \cite{explicit} where it is applied to the cases of the archimedean Gross--Zagier formula and of the Waldspurger formula; the application to an explicit version of Theorem \ref{B}  can be obtained in exactly  the same manner. An explicit version of the anticyclotomic formulas of Theorem \ref{C} can also be obtained as a consequence: see \cite{exc} for a special case.

\subsection{Outline of proofs and organisation of the paper}
Let us briefly explain the main arguments and at the same time the  organisation of the paper. 

For the sake of simplicity, the notation used in this introductory discussion slightly differs from that of the body text, and we ignore  powers of $\pi$, square roots of discriminants, and choices of additive character.
\subsubsection{Construction of the $p$-adic $L$-function (\S\ref{sec plf3})}
It is crucial for us to have a flexible construction which does not depend on choices of newforms. 
The starting point is Waldspurger's \cite{wald} Rankin--Selberg integral
\begin{align}\label{outline-plf}
{(\vphi, { I( \phi, \chi'))_{\rm Pet}} \over  2L(1, \sigma, {\rm ad}) /\zeta_{F}(2)  } 
={L(1/2, \sigma_{E}\otimes \chi') \over  L(1, \eta)} \prod_{v}  {R}^{\natural}_{v}(\vphi_{v},\phi_{v}, \chi'_{v}),
\end{align}
 where $\vphi\in \sigma$, $I( \phi, \chi')$ is a mixed theta-Eisenstein series depending on a choice of an ad\`elic Schwartz function $\phi$, and $ {R}^{\natural}_{v}(\vphi_{v},\phi_{v}, \chi')$ are normalised local integrals (almost all of which are equal to~$1$). Then, after dividing both sides by the period $2L(1, \sigma, \rm ad)$, we can:
 \begin{itemize}
  \item interpolate the kernel $\chi'\mapsto I(0, \phi, \chi')$ to a $\Y'$-family $\calI(\phi^{p\infty}; \chi')$ of $p$-adic modular forms     for \emph{any} choice of the components $\phi^{p\infty}$, and a  \emph{well-chosen}  $\phi_{p\infty}$ (we will set  $\phi_{v}(x,u)$
to be   ``standard'' at $v\vert \infty$, and close to a delta function in $x$ at $v\vert p$);
 \item interpolate the functional  ``Petersson product with $\vphi$'' to a functional $\lf$ on $p$-adic modular forms, for \emph{any} $\vphi\in \sigma$ which is a `$\Up_{v}$-eigenvector of eigenvalue $\alpha_{v}$' at the places $v\vert p$, and is antiholomorphic at infinity;
 \item interpolate the normalised local integrals   $\chi'\mapsto {R}^{\natural}_{v}(\vphi_{v},\phi_{v}, \chi'_{v})$ to functions $\mathscr{R}^{\natural}_{v}(\vphi_{v},\phi_{v})$ for all $v\nmid p\infty$ and any $\vphi_{v}$, $\phi_{v}$.
  \end{itemize}
To conclude, we cover $\Y'$ by finitely many open subsets $\mathscr{U}_{i}$; for each $i$ we choose appropriate $\vphi^{p}$, $\phi^{p}$ and we define \footnote{This is the point which possibly produces poles.}
$$L_{p, \alpha}(\sigma_{E})|_{\mathscr{U}_{i}}={\lf(\calI(\vphi^{p\infty}))\over \prod_{v\nmid p\infty} \mathscr{R}^{\natural}_{v}(\vphi_{v},\phi_{v})}.$$ The explicit computation of the local integrals at $p$ (in the Appendix) and at infinity yields the interpolation factor.

\subsubsection{Proof of the Gross--Zagier formula and its anticyclotomic version} We outline the main arguments of our proof, with an emphasis on the reduction steps.
\paragraph{Multiplicity one}  We borrow or adapt many ideas (and calculations) from \cite{yzz}, in particular the systematic use of the multiplicity one principle of \S\ref{hm1}. As both sides of the formula are functionals in the same one-dimensional vector space,  it is  enough to prove the result for one pair $f_{1}$, $f_{2}$ with $Q(f_{1}, f_{2}, \chi)\neq0$; finding such $f_{1}\otimes f_{2}$ is a local problem. 
 It is equivalent to  choosing functions  $\vphi_{v}\otimes \phi_{v}'=\theta^{-1}(f_{1}\otimes f_{2})$  as just above, by the Shimizu lift $\theta$ realising the Jacquet--Langlands correspondence (\S\ref{shim}). For $v\vert p$ we thus have an explicit choice of such, corresponding to the one made above.\footnote{The notation $\phi'$ refers to the application to $\phi$ of a local operator at $v\vert p$ appearing in the interpolation of the Petersson product.}   For $v\nmid p$, we can introduce several restrictions on $(\vphi_{v},\phi_{v})$  as in \cite{yzz}, with the effect of simplifying many calculations of local heights (\S\ref{local ass}).

\paragraph{Arithmetic theta lifting and kernel identity (\S\ref{sec GS})} In \cite{yzz}, the authors introduce an arithmetic-geometric analogue of the Shimizu lift, by means of which they are able to write also the Heegner-points side of their formula as a Petersson product with $\vphi$ of a certain geometric kernel. We can adapt without difficulty their results to reduce our formula to the assertion\footnote{Together with a comparison of local terms at $p$ described below.} that 
$${\rm d}_{F}\calI( \phi^{p\infty}; \chi)-2 L_{(p)} (1, \eta)\wtil{Z}(\phi^{\infty}, \chi)$$
is killed by the $p$-adic Petersson product $\lf$. Here $\wtil{Z}(\phi^{\infty}, \chi)$ is a modular form depending on $\phi$ encoding the height pairings of CM points on Shimura curves and their Hecke translates; it generalises the classical generating series $\sum_{m}\langle \iota_{\xi}(P)[{\chi}], T(m) \iota_{\xi}(P)\rangle \qqq^{m}$. 

\paragraph{Decomposition and comparison (\S\S\ref{sec: expand an}-\ref{dec gk})} Both terms  in the kernel identity are sums of local terms indexed by the finite places of $F$. For $v\nmid p$, we compute both sides and  show that the difference essentially coincides with the one computed in \cite{yzz}: it is either zero or, at bad places, a modular form orthogonal to all forms in $\sigma$. In fact we can show this only for a certain restricted set of $q$-expansion coefficients; as the global kernels are $p$-adic modular forms, this will suffice by a simple approximation argument (Lemma \ref{approx lemma}).

\paragraph{$p$-adic Arakelov theory or analytic continuation} 
The argument just sketched relies on calculations of  arithmetic intersections of CM points; this in general does not suffice, as we need to consider the contribution of the Hodge classes in the generating series too.  It will turn out that such contribution vanishes; two approaches can be followed to show this. The first one, in analogy with  \cite{shouwu, yzz} and already used  in a simpler context  in \cite{dd}, is to make use of Besser's $p$-adic Arakelov theory\footnote{Recall that an Arakelov theory is  an arithmetic intersection theory which allows to pair cycles of any degree, recovering the height pairing for cycles of degree zero.}
 \cite{besser} in order to separate such contribution.
 
We will follow an alternative approach (see Proposition \ref{BandC}), which exploits the generality of our context and the existence of extra variables in the $p$-adic world.  Once constructed the  Heegner--theta element  $\mathscr{P}^{\pm}$,  the anticyclotomic formula of Theorem \ref{C}.\ref{C4}  is essentially a corollary of Theorem \ref{B} for all finite order characters $\chi$: we only need to check the compatibility $Q_{v}(f_{\alpha,v}^{+}, f_{\alpha,v}^{-}, \chi)= \zeta_{F,v}(2)^{-1}\cdot Z_{v}( \chi_{v})$ for all $v\vert p$ by explicit computation. 
Conversely, thanks to the multiplicity one result, it is also true that Theorem \ref{B} for any $\chi$ is obtained as a  corollary of Theorem \ref{C}.\ref{C4} by specialisation.  We make  use of both of   these  observations: we first prove Theorem \ref{B} for all but finitely many finite order characters $\chi$; this suffices to deduce Theorem \ref{C}.\ref{C4} by an analytic continuation argument, which finally yields Theorem B for the remaining characters $\chi$ as well. The initially excluded characters are those (such as the trivial character when it is contemplated) for which the  contribution  of the Hodge classes is not already annihilated by $\chi$-averaging; for all other characters the  Arakelov-theoretic arguments just mentioned are then unnecessary. 

\paragraph{Annihilation of $p$-adic heights (\S\ref{lsp})} We are left to deal with the contribution of the places $v\vert p$. We can show quite easily that this is zero for the analytic kernel. As in the original work of Perrin-Riou \cite{PR}, the vanishing of the contribution of the  geometric kernel is the heart of the argument. We establish it via an elaboration of a method of Nekov{\'a}{\v{r}} \cite{nekovar} and Shnidman \cite{ari}. The key new ingredient in adapting it to our semistable case is a simple integrality criterion for local heights in terms of intersections, introduced in \S\ref{intH} after a review of the theory of  heights.

\paragraph{Local toric period} Finally,  in the Appendix we  compute  the local toric period $Q(\theta(\vphi_{v}\otimes \phi_{v}'), \chi_{v})$ for $v\vert p$ and compare it to the interpolation factor of the $p$-adic $L$-function. Both are highly ramified local integrals, and they turn out to differ by the multiplicative constant  $L(1,\eta_{v})$; this  completes the comparison between the kernel identity and Theorem B.

\subsection{Acknowledgements}  I would like to thank S. Zhang for warm encouragement and A. Burungale, M. Chida, O. Fouquet, M.-L. Hsieh, A. Iovita, S. Kudla, V. Pilloni, E. Urban, Y. Tian, X. Yuan,  and W. Zhang for  useful conversations or correspondence. I am especially grateful to S. Shah for his help with the chirality of Haar mesaures and to A. Shnidman for extended discussions on heights. Finally, I would like to thank the referee for a careful reading and  A. Besser for kindly agreeing to study the compatibility  between $p$-adic heights and Arakelov theory  on curves of bad reduction: his work gave me confidence at a time when I had not found the alternative argument used here. 
\smallskip

Parts of the present paper were written while the author was a postdoctoral fellow at MSRI, funded under NSF grant  0932078000, and at the CRM (Montreal). Final revisions were made when the author was a postdoctoral fellow funded by the Fondation Math\'ematique Jacques Hadamard.

\subsection{Notation}

We largely follow the notation and conventions of \cite[\S 1.6]{yzz}.

\subsubsection{$L$-functions} In the rest of the paper (and \emph{unlike} in the Introduction, where we adhere to the more standard convention), \emph{all  complex $L$- and zeta functions are complete including the $\Gamma$-factors at the infinite places}. (This is to facilitate  referring to the results and calculations of \cite{yzz} where this convention is adopted.)

\subsubsection{Fields and ad\`eles} The fields $E$ and $F$ will be as fixed in the Introduction unless otherwise noted. The ad\`ele ring of $F$ will be denoted  $\A_{F}$ or simply $\A$; it contains the ring $\A^{\infty}$ of finite ad\`eles. We let $D_{F}$ and $D_{E}$ be the absolute discriminants of $F$ and $E$ respectively. We also choose an id\`ele $d\in \A^{\infty,\times}$ generating the different of $F/\Q$, and an id\`ele $D\in \A^{\infty,\times}$ generating the relative discriminant of $E/F$. 

We use standard notation to restrict ad\`elic objects (groups, $L$-functions,\ldots) away from a finite set of places $S$, e.g. $\A^{S}:=\prod_{v\notin S}'F_{v}$, whereas $F_{S}:=\prod_{v\in S}F_{v}$. When $S$ is the set of places above $p$ (respectively $\infty$) we use this notation with `$S$' replaced it by `$p$' (respectively `$\infty$').

We denote by $F_{\infty}^{+}\subset F_{\infty}$ the group of $(x_{\tau})_{{\tau\vert \infty}}$ with $x_{\tau}>0$ for all $\tau$, and we let  $\A^{\times}_{+}:=\A^{\infty ,\times}F_{\infty}^{+}$, $F^{\times}_{+}:=F^{\times}\cap F_{\infty}^{+}$. 

For a non-archimedean prime $v$ of a number field $F$, we denote by $q_{F,v}$ the cardinality of the residue field and by $\vpi_{v}$ a uniformiser.

\subsubsection{Subgroups of $\GL_{2}$} We consider  $\GL_{2}$ as an algebraic group over $F$. We denote by $P$, respectively $P^{1}$, the subgroup of $\GL_{2}$, respectively ${\bf SL}_{2}$, consisting of upper-triangular matrices; by $A\subset P\subset \GL_{2}$ the diagonal torus; and by $N\subset P\subset  \GL_{2}$ the unipotent radical of $P$.  We let $n(x):=\smalltwomat 1x{}1$ and 
$$w:=\twomat {}1{-1}{}.$$

\subsubsection{Quadratic torus} We let $T:={\rm Res}_{E/F}{\bf G}_{m}$; the embedding $T(\A^{\infty})\subset \B^{\times}$ is fixed. We let $Z:={\bf G}_{m, F}$, and view it both as a subgroup of $T$ and as the centre of $\GL_{2}$.
\subsubsection{Automorphic quotients} If $G$ is a reductive group over the totally real field $F$, we denote
$$[G]:=G(F)\bks G(\A)/Z(\A).$$

\subsubsection{Measures} We choose local and global Haar measures as in \cite{yzz}. In particular,  we have
$$\vol(\GL_{2}(\OO_{F,v}))=|d|_{v}^{2}\zeta_{F,v}(2)^{-1}$$
for all non-archimedean $v$. 

We denote by $dt$ the local and global measures on $T/Z$ of \cite{yzz},  which give $\vol([T], dt)=2L(1, \eta)$.
The global measure 
$$d^{\circ} t := |D_{F}|^{1/2}|D_{E/F}|^{1/2} dt$$
gives $\vol([T], d^{\circ}t)\in \Q^{\times}$. 

\subsubsection{Regularised averages and integration} We borrow some notation from \cite[\S 1.6.7]{yzz}. If $G$ is a topological group with a left Haar measure $dg$ with finite volume, 	we define $$\dashint_{G}f\, dg:={1\over \vol (G)}\int_{G} f(g)\, dg.$$ (This reduces to the usual average when $G $ is a finite group.)

 If $F$ is a totally real field and $f$ is a function on $ F^{\times}\bks \A^{\times}$ invariant under $F_{\infty}^{\times}$, we denote 
$$\dashint_{\A^{\times}} f(z)\, dz:=\dashint_{F^{\times}\bks \A^{\times}/ F_{\tau}^{\times}}    f(z)\, dz,$$
where $\tau$ is any archimedean place of $F$. If $f$ is further invariant under a compact open subgroup $U$, this reduces to the average over $F^{\times}\bks \A^{\times}/F_{\infty}^{\times}U$. 

Finally, let $G$ be a reductive group over $F$  with an embedding of ${\bf G}_{m/F}$ into the centre $G$, and assume that $\, dg $ is a left Haar measure giving finite  volume to $[G]=G(F)\bks G(\A)/Z(\A)$. Let $f$ be a function on $G(F)\bks G(\A)/Z(F_{\infty})$, then we define
$$\int^{*}_{[G]}f(g)\, dg:=\int_{[G]}\dashint_{Z(\A)}f(zg)\,dz\, dg$$
and
$$\dashint_{[G]}f(g)\, dg:={1\over \vol([G])}\int_{[G]}\dashint_{Z(\A)}f(zg)\,dz\, dg.$$
Note in particular that for functions which factor through a compact quotient of $G(F)\bks G(A)$  and are locally constant there, the regularised integration reduces to a  finite sum and, when using $\Q$-valued  measures such as  the measure $d^{\circ }t$ on $T$, it makes sense for functions taking $p$-adic values as well.

\subsubsection{Multiindices} If $S$ is a set and $r\in \Z^{S}$, $p\in G^{S}$ for some group $G$, we often write $p^{r}:=\prod_{v\in S}p_{v}^{r_{v}}$. This will typically be applied in the following situation: $S=S_{p}$ is the set of places of $F$ above $p$, $G$ is the (semi)group of ideals of $\OO_{F}$, and $p_{v}$ is the ideal corresponding to $v$.

\subsubsection{Functions of $p$-adic characters} When $\Y^{?}$ is one of the rigid spaces introduced above and $G(A)\in \OO(\Y^{?})$ is a function on $\Y^{?}$ depending on other `parameters' $A$ (e.g. a $p$-adic $L$-function), we write $G(A; \chi)$ for the evaluation $G(A)(\chi)$.

\section{$p$-adic modular forms}

\subsection{Modular forms and their $q$-expansions}
Let $K\subset \GL_{2}(\widehat{\OO}_{F})$ be an open compact subgroup. Recall that  a Hilbert automorphic form of level $K$ is a smooth function  of moderate growth
$$\vphi \colon\GL_{2}(F)\bks\GL_{2}(\A)/K\to \C.$$
Let $k\in \Z^{\Hom(F, \R)}$. Then an automorphic form is said to be of weight $k$ if it satisfies 
$$\vphi(gr_{\theta})=\vphi(g)\psi_{\infty}(k\cdot \theta)$$
for all $r_{\theta}=\left(\smalltwomat {\cos 2\pi\theta_{v}}{\sin2\pi\theta_{v}}{-\sin2\pi\theta_{v}}{\cos2\pi\theta_{v}}\right)_{v\vert\infty }\in \SO_{2}(F_{\infty})$. It is said to be holomorphic of weight $k$ if for all $g\in \GL_{2}(\A^{\infty})$, the function of $z_{\infty}=(x_{v}+iy_{v})_{v\vert\infty}\in \mathfrak{h}^{\Hom(F, \R)}$, 
$$z_\infty\mapsto |y_{\infty}|_{\infty}^{-k/2}\vphi\left(g\smalltwomat {y_{\infty}}{x_{\infty}}{}1\right)$$ is holomorphic. Holomorphic Hilbert automorphic forms will be simply called \emph{modular} forms.

Let $\omega\colon F^{\times}\bks \A^{\times}\to \C^{\times}$ be a finite order character. Then $\vphi$ is said to be of character $\omega$ if it satisfies $\vphi(zg)=\omega(z)\vphi(g)$ for all $z\in Z(\A)\cong \A^{\times}$. 
We denote by $M_{k}(K, \C)$ the space of modular forms of level $K$ and weight $k$, and by $S_{k}(K,\C)$ its subspace of cuspforms. We further denote by $M_{k}(K,\omega, \C)$, $S_{k}(K, \omega, \C)$ the subspaces of forms of character $\omega$. We identify a scalar weight $k\in\Z_{\geq 0}$ with the corresponding parallel weight $(k, \ldots, k)\in \Z_{\geq 0}^{\Hom(F, \R)}$. 

For $v$ a finite place of $F$ and  $N$ an ideal of $\OO_{F,v}$, we define subgroups  of $\GL_{2}(\OO_{F,v})$ by\begin{align*}
K_{0}(N)_{v}&=\left\{\twomat abcd\,|\ c\equiv 0 \mod N\right\},\\
K_{1}(N)_{v}&=\left\{\twomat abcd\,|\ c, d-1\equiv 0 \mod N\right\},\\
K^{1}(N)_{v}&=\left\{\twomat abcd\,|\ c, a-1\equiv 0 \mod N\right\},\\
K_{1}^{1}(N)_{v}&=\left\{\twomat abcd\,|\ c, a-1, d-1\equiv 0 \mod N\right\},\\
K(N)_{v}&=\left\{\twomat abcd\,|\ b, c, a-1, d-1\equiv 0 \mod N\right\}.
\end{align*}
If $N$ is an ideal of $\OO_{F}$ and   $*\in\{{}_{0},{}_{1}, {}^{1}, {}^{1}_{1}, \emptyset\}$, we define subgroups $K{*}(N)$ of $\GL_{2}(\widehat{\OO}_{F})$ by $K{*}(N)=\prod_{v}K{*}(N)_{v}$.  
If $p$ is a rational prime and $r\in \Z_{\geq 0}^{\{v\vert p\}}$, we further define $K{*}(p^{r})_{p}=\prod_{v}K{*}(\vpi_{v}^{r_{v}})_{v}\subset\GL_{2}(\OO_{F, p})$.

Fix a nontrivial character $\psi\colon \A/F\to \C^{\times}$. Any automorphic form  $\vphi$ admits a Fourier--Whittaker expansion $\vphi(g)=\sum_{a\in F}W_{a}(g)$, where $W_{a}(g)=W_{\vphi ,\psi,a }(g)$ satisfies $W_{a}(n(x)g)=\psi(ax)W_{a}(g)$ for all $x\in\A$. If $\vphi$ is holomorphic of weight $k$ we can  further write $W_{a}(g)=W_{a^{\infty}}(g^{\infty})W_{a, \infty}(g_{\infty})$ with $W_{a, \infty}(g)=\prod_{v\vert \infty} W_{a,v}(g)$, where $W_{a,v}=W_{a,v}^{(k_{v})}$ is the \emph{standard holomorphic Whittaker  function} of weight $k$ given by (suppressing the subscripts and using the Iwasawa decomposition)
\begin{align}\label{holwhitt}
W_{a}^{(k)}(\smalltwomat z{}{}z\smalltwomat yx{}1 r_{\theta}) = \begin{cases} |y|^{k/2}\psi(a(x+iy))\psi(k\theta)\one_{\R_{+}}(ay) &\ \text{ if} \ a\neq 0\\
 |y|^{k/2}\psi(k\theta)\one_{\R_{+}}(y) &\ \text{ if} \ a=0.
\end{cases}
\end{align}
(Similarly, we have a description in terms of the \emph{standard antiholomorphic Whittaker function}
\begin{align}\label{antiholwhitt}
W_{a}^{(-k)}(\smalltwomat z{}{}z\smalltwomat yx{}1 r_{\theta}) = \begin{cases} |y|^{k/2}\psi(a(x+iy))\psi(-k\theta)\one_{\R_{+}}(-ay) &\ \text{ if} \ a\neq 0\\
 |y|^{k/2}\psi(-k\theta)\one_{\R_{+}}(-y) &\ \text{ if} \ a=0
\end{cases}
\end{align}
for antiholomorphic forms of weight $-k<0$.)

 In this case we have an expansion 
$$\vphi\left(\twomat y x{}1\right)=|y|^{k/2}_{\infty}\sum_{a\in F_{\geq 0}}  W_{a}^{\infty}(\smalltwomat{y^{\infty}}{}{}1)\psi_{\infty}(iay_{\infty})\psi(ax)$$
for all $y\in \A^{\times}_{+}, x\in \A$; here $F_{\geq 0}$ denotes the set of $a\in F $ satisfying $\tau(a)\geq 0$ for all $\tau\colon F\into \R$.

For a field $L$, let the \emph{space of formal $q$-expansions} $ C^{\infty}(\A^{\infty,\times},L)\llbracket \qqq^{F_{\geq 0}}\rrbracket^{\circ}$ be the set of those  formal sums ${ W}=\sum_{a\in F_{\geq 0}}W_{a}\qqq^{a}$ with coefficients   $W_{a}\in C^{\infty}(\A^{\infty,\times},L)$
such that, for some  compact subset $A_{ W}\subset \A^{\infty}$, we have 
 $W_{a}(y)=0$ unless  $ay\in A_{ W}$.
 
Let $\vphi$ be  a holomorphic automorphic form.
The expression 
\begin{align}\label{qexpmap}
{}^{\qqq}\vphi(y):=\sum_{a\in F}W_{\vphi, a}^{\infty}(\smalltwomat y{}{}1)\,\qqq^{a}, \qquad y\in \A^{\infty, \times} 
\end{align}
 belongs to $  C^{\infty}(\A^{\infty,\times},\C)\llbracket \qqq^{F_{\geq 0}}\rrbracket^{\circ}$ and it is called 
 the \emph{formal $q$-expansion} of $\vphi$. 
The space of formal $q$-expansions is an algebra in the obvious way, compatibly with the algebra structure on automorphic forms. 

\medskip

\begin{prop}[$q$-expansion principle]\label{q-exp-princ}  Let $K\subset \GL_{2}(\widehat{\OO}_{F})$ be an open compact subgroup, and let $k\in \Z_{\geq 0}^{\Hom(F, \R)}$. The $q$-expansion map defined by \eqref{qexpmap}
\begin{align*}M_{k}(K, \C)&\to C^{\infty}(\A^{\infty,\times}, \C)\llbracket \qqq^{F_{\geq 0}}\rrbracket^{\circ} \\
\vphi & \mapsto {}^{\qqq}\vphi
\end{align*}
is injective.
\end{prop}

We say that a formal $q$-expansion is \emph{modular} if it belongs to the image of the $q$-expansion map.

\begin{proof} This is (a weak form) of the $q$-expansion principle of \cite[Th\'eor\`eme 6.7 (i)]{rapoport}. In fact our  modular forms $\vphi$ are identified with  tuples $(\vphi_{c})_{c\in {\rm Cl}(F)^{+}}$ of Hilbert modular forms in the sense of \cite{rapoport}. Then the non-vanishing of ${}^{\qqq}\vphi$ for $\vphi\neq 0 $ is obtained by applying the result of \emph{loc. cit.} to each   $\vphi_{c}$. See \cite[Lemme 6.12]{rapoport} for  the comparison between various notions of Hilbert modular forms used there.
\end{proof}

The spaces of formal $q$-expansions  introduced so far will often be  convenient for us in terms of notation, but they are  redundant: if $k\in \Z_{\geq 0}$, $\vphi\in M_{k}(K, \C)$, we have $W_{a}(\smalltwomat y{}{}1)=W_{1}(\smalltwomat {ay}{}{}1)$ for all $a\in F^{\times}$, $y\in\A_{+}^{\times}$.  Moreover if $K\subset K(N)$, then $|y^{\infty}|^{-k/2}W_{0}(\smalltwomat{y^{\infty}}{}{}1)$ and $|y^{\infty}|^{-k/2}W^{\infty}_{1}(\smalltwomat{y^{\infty}}{}{}1)$ are further invariant under the action of  $U_{F}(N)=\{u\in \widehat{\OO}_{F}^{\times}\, |\ u\equiv 1\mod N\}$ by multiplication on $y$ (see \cite[Theorem 1.1]{Hi}).
We term \emph{reducible of weight $k$} those formal $q$-expansions satisfying these  conditions for some~$N$.

Define the space of reduced $q$-expansions (of level $N$) with values in a ring  $A$ to be
 $${ M}' (K(N), L):= C({\A^{\infty, \times}}/F^{\times}_{+}U_{F}(N), L)\times  L^{\A^{\infty,\times}/U_{F}(N)};$$
if $K$ is any compact open subgroup, we define $M'(K, L):= M'(K(N), L)$ for the largest subgroup $K(N)\subset K$. Let $M'(L):=\bigcup_{N} M'(K(N), L)$ and $M'(K^{p}, L):= \bigcup_{r}M'(K^{p}K_{p}(p^{r}), L)$.

Given a reducible  $q$-expansion $W$ of weight $k$, we can then define the associated \emph{reduced $q$-expansion} 
$(W_{0}^{\natural}(y), (W_{a}^{\natural})_{a\in \A^{\infty\times}})\in M'(L)$
by 
 $$W_{0}^{\natural}(y):= |y|^{-k/2}W_{0}^{\infty}(\smalltwomat y{}{}1), \qquad W_{a}^{\natural}:= |a|^{-k/2} W_{1}^{\infty}(\smalltwomat{a}{}{}1).$$
 
If $A\subset\C$ is a subring, we denote by $M_{k}(K, A)\subset M_{k}(K, \C)$, $S_{k}(K, A)\subset S_{k}(K, \C)$ the subspaces of forms with reduced $q$-expansion coefficients in $A$. If $A$ is any $\Q$-algebra, we let $M_{k}(K, A)=M_{k}(K, \Q)\otimes A$, $S_{k}(K, A)=S_{k}(K, \Q)\otimes A$. Then it makes sense to talk about the $q$-expansion of an element of those spaces.

 If $\vphi$ is a modular form, we still denote by ${}^{\qqq}\vphi$ its reduced $q$-expansion; in cases where the distinction is significant, the precise meaning of the expression ${}^{\qqq}\vphi$ will be clear from its context.

\subsubsection{$p$-adic modular forms} 
Let  $N\subset \OO_{F}$ be a nonzero ideal prime to $p$, $U_{F}(Np^{\infty})=\cap_{r\geq 0}U_{F}(Np^{r})$. We endow the quotient ${\A^{\infty, \times}}/F^{\times}_{+}U_{F}(Np^{\infty})$ with the profinite topology.  Let $L$ be a complete Banach ring with norm $|\cdot|$. We define the space of $p$-adic reduced $q$-expansions  with values in $L$ to be 
 $${M}' (K^{p}(N), L):= C({\A^{\infty, \times}}/F^{\times}_{+}U_{F}(Np^{\infty}), L)\times  L^{\A^{\infty,\times}/U_{F}(Np^\infty)}.$$
If $K^{p}\subset \GL_{2}(\A^{p\infty})$ is a compact open subgroup
in general, we   define ${\bf M}' (K^{p}, L):={\bf M}' (K^{p}(N), L)$ for the largest subgroup $K^{p}(N)\subset K^{p}$.

Define  a `norm' (possibly taking the value $\infty$) $||\cdot ||$ on ${ M}' (K^{p}(N), L)$  by
\begin{gather}\label{norm}||(W_{0}^{\natural},(W_{a}^{\natural})_{a\in {\A^{\infty,\times}/U_{F}(Np^\infty)}})||:=\sup_{(y,a)} \{| W_{0}^{\natural }(y)|,| W_{a}^{\natural}|\}\end{gather}
It induces a `norm' on the (isomorphic) space of reducible $q$-expansions with values in $L$. 
Let ${  M}' (K^{p}, L)^{\circ}\subset M'(K^{p}, L)$ be the set of elements on which $||\ ||$ is finite.
We define  the  Banach space of \emph{$p$-adic reduced $q$-expansions}
$${\bf  M}' (K^{p}, L)$$
to be the completion of $M'(K^{p}, L)^{\circ}$ with respect to the norm $||\cdot ||$.
We denote by ${\bf S}'(L)\subset {\bf M}'(L)$ the space of reduced $q$-expansions with vanishing constant coefficients; when  there is no risk of confusion we shall omit $L$ from the notation.

Suppose that $L$ is  a field extension of $\Q_{p}$.  The space of $p$-adic modular forms of tame level $K^{p}\subset \GL_{2}(\widehat{\OO}^{p}_{F})$ with coefficients in $L$, denoted by  ${\bf M}(K^{p},L)$, is defined to be the closure in  ${\bf M}'(K^{p}, L)$, of the subspace generated by the reduced $q$-expansions of elements of $M_{2}(K^{p}K^{1}(p^{\infty})_{p},L) = \cup_{r\geq 0} M_{2}(K^{p}K^{1}(p^{r})_{p}, L)$. Tame levels and coefficients rings will be omitted from the notation when they are understood from context.
 We denote by   ${\bf S}:={\bf M}\cap {\bf S}'$ the space of $p$-adic modular cuspforms.

\subsubsection{Approximation} The $q$-expansion principle of Proposition \ref{q-exp-princ} is complemented by the following (obvious) result to  provide a $p$-adic replacement for the approximation argument in \cite{yzz}.

\begin{lemm}[Approximation]\label{approx lemma}  Let  $S$ be a finite set of finite places of $F$, not containing any place $v$ above $p$. Let $\vphi$ be a $p$-adic modular cuspform all whose reduced $q$-expansion coefficients $W_{a,\vphi}^{\natural}$ are zero for all $a\in F^{\times}{\A^{S\infty,\times}}$. Then $\vphi=0.$
\end{lemm}
\begin{proof} The form $\vphi$ has some tame level $K^{p}$; then its coefficients are invariant under the action of some compact open $U_{F}^{p}\subset \A^{\infty,\times}$ on the indices $a$. Since  $F^{\times}{\A^{S\infty,\times}}U_{F}^{p}=\A^{\infty,\times}$, the lemma follows. 
\end{proof}
Let $\baar{\bf S}{}'$   be the quotient of ${\bf S}'$ by the subspace of reduced $q$-expansions which are zero at all $a\in F^{\times}{\A^{S\infty,\times}}$, and let $\baar{\bf S}$ be the image of ${\bf S}$ in $\baar{\bf S}{}'$ (these notions depend on the set $S$, which in our uses will be clear from the context). Then the lemma says that in 
\begin{align}\label{S to S'}
{\bf S}{\to} \baar{\bf S}\into \baar{\bf S}{}',
\end{align}
the first map is an isomorphism and the composition is an injection. We use the notation $ \baar{\bf S}_{S}(K^{p})$, $ \baar{\bf S}{}'_{S}(K^{p})$ if we want to specify the set of places $S$ and the tame level $K^{p}$.

\subsubsection{Families} Let $\Y^{?}$ be one of the rigid spaces defined in the Introduction, and $K^{p}\subset \GL_{2}(\A^{p\infty})$ be a compact open subgroup.
\begin{defi} A $\Y^{?}$-family of $q$-expansions of modular forms of tame level $K^{p}$ is a reduced  $q$-expansion 
$\underline{\vphi}$  with values in $\OO(\Y^{?})$, whose coefficients are algebraic on $\Y^{?{\rm l.c.}}$, and 
such that for every  point $\chi \in \Y^{?{\rm l.c.}}$, 
$\underline{\vphi}(\chi)$ is the reduced  $q$-expansion of a classical modular form $\vphi(\chi)$ of level $K^{p}K^{1}(p^{\infty})_{p}$ with coefficients in $M(\chi)$. We say that $\underline{\vphi}$ is \emph{bounded} if it is bounded for the norm \eqref{norm}.
\end{defi}

\subsubsection{Twisted modular forms} It  will be convenient  to  consider the following relaxation of the notion of modular forms.
\begin{defi} A \emph{twisted} Hilbert automorphic form of weight $k\in\Z_{\geq 0}^{\Hom(F, \R)}$ and level $K\subset \GL_{2}(\widehat{\OO}_{F})$ is a smooth function
$$\wtil{\vphi}\colon \GL_{2}(\A)/K\times \A^{\times} \to \C$$
satisfying:
\begin{itemize}
\item for all $\gamma \in \GL_{2}(F)$, $r_{\theta}\in \SO_{2}(F_{\infty})$, 
$$\wtil{\vphi}(\gamma gr_{\theta} , u)= \wtil{\vphi}(g, \det (\gamma)^{-1}u)\psi_{\infty}(k\cdot\theta);$$
\item $\wtil{\vphi}$ is of moderate growth in the variable $g\in \GL_{2}(\A)$ and  for all $g\in \GL_{2}(\A)$, $u=u_{\infty}u^{\infty}\mapsto \wtil{\vphi}(g,u)$ is the product of a function of the variable  $u^{\infty}$ and of the function $\one_{F_{\infty}^{+}}(u_{\infty})$ of the variable $u_{\infty}$. 
\item there exists a a compact open subgroup $U_{F}\subset \A^{\infty,\times}$ such that  for all $g$,  $\wtil{\vphi}(g, \cdot)$ is invariant under $U_{F}$;
\item for each  $g\in \GL_{2}(\A)$, there is an open compact subset  $K_{g}\subset \A^{\infty,\times}$ such that  $\wtil{\vphi}(g, \cdot)$ is  supported in $K_{g}F_{\infty}^{+}$.
\end{itemize}
Let $\omega\colon F^{\times}\bks \A^{\times}\to \C^{\times}$ be a finite order character. 
We say that a twisted automorphic form $\wtil{\vphi}$ has central character $\omega$ if it satisfies 
$$\wtil{\vphi}(zg, u)=\omega(z)\wtil{\vphi}(g, z^{-2}u)$$
 for all $z\in  Z(\A)\cong \A^{\times}$.  We say that it is holomorphic (of weight $k$) or simply a \emph{twisted modular form} if
  $z_\infty\mapsto |y_{\infty}|_{\infty}^{-k/2}\wtil{\vphi}\left(g\smalltwomat {y_{\infty}}{x_{\infty}}{}1,u\right)$ 
  is holomorphic in $z_{\infty}=(x_{v}+iy_{v})_{v\vert\infty}\in \mathfrak{h}^{\Hom(F, \R)}$, 
for all $u\in \A^{\times}$.  

We let $M_{k}^{\rm tw}(K, \C)$ denote the space of twisted modular forms of weight $k$ and $M_{k}^{\rm tw}(K, \omega, \C)$ its subspace of forms with central character $\omega$.  We omit the $K$ from the notation if we do not wish to specify the level.
\end{defi}

If $\wtil{\vphi}$ is a twisted modular form, then for each $g, u$, the function $x\mapsto \vphi(n(x)g,u)$ descends to $F\bks \A$ and therefore it admits a Fourier--Whittaker expansion in the usual way.   To the restriction of $\wtil{\vphi}$ to  $\GL_{2}(\A)\times F^{\times}$  we  then attach a \emph{twisted formal $q$-expansion}
$$\sum_{a\in F}|y|^{k/2}_{\infty}W_{a}^{\infty}(\smalltwomat{y}{}{}1,u)\, \qqq^{a}\ \in C^{\infty}(\A^{\infty,\times}  \times F^{\times}, \C)\llbracket \qqq^{F_{\geq 0}}\rrbracket^{\circ}$$
such that 
$$\wtil{\vphi}\left(\twomat y x{}1, u\right)=|y|_{\infty}^{k/2}\sum_{a\in F_{\geq 0}}  W_{a}^{\infty}(\smalltwomat {y^{\infty}}{}{}1,u)\psi_{\infty}(iay_{\infty})\psi(ax)$$
for all $y\in \A^{\times}_{+}, x\in \A, u\in  F^{\times}$.  
Here the space $C^{\infty}(\A^{\infty,\times}  \times F^{\times}, \C)\llbracket \qqq^{F_{\geq 0}}\rrbracket^{\circ}$ consists of  $q$-expansions $W$ whose coefficients $W_{a}(y,u)$ vanish for $ay$ outside of some compact open subset $A_{W}\subset \A^{\infty}$.

Let  $\wtil{\vphi}$ be a twisted modular form, let $U_{F}\subset \A^{\infty,\times}$ be a compact open subgroup satisfying the condition of the previous definition,  and let $\mu_{U_{F}}=F^{\times }\cap U_{F}$. Then 
the sum $$\vphi(g):=\sum_{u\in \mu_{U_{F}}^{2}\bks F^{\times}}\wtil{\vphi}(g,u)$$
is finite for each $g$ (if $K_{g}\subset \A^{\infty,\times}$ is a  compact subset such that  $K_{g}F_{\infty}^{\times}$ contains the  support of $\wtil{\vphi}(g, \cdot)$,   the sum  is supported on $\mu_{U_{F}}^{2}\bks( F^{\times }\cap K_{g})$,   which is commensurable with the finite group $\mu_{U_{F}}^{2}\bks \OO_{F}^{\times }$). It defines a modular form in the usual sense, with formal $q$-expansion 
$${}^{\qqq}\vphi(y)=\sum_{u\in  \mu_{U_{F}}^{2}\bks  F^{\times}}{}^{\qqq}\wtil{\vphi}(y,u).$$

One can, similarly to the above,  
 define a norm on the space of    twisted formal $q$-expansion coefficients of a fixed parallel weight $k$ with values in a Banach ring $L$, namely $||W||:=\sup_{(a,y^{\infty},u)}|y^{\infty}|^{-k/2}|W_{a}^{\infty}(y,u)|$.  The $p$-adic completion ${\bf M}_{k}^{\rm tw}{}(K^{p}, L)$ of the subspace of     $q$-expansions of twisted modular forms (of some tame level $K^{p}$) is called the space of $p$-adic twisted modular  forms (of tame level $K^{p}$). Finally, there is a notion of $\Y^{?}$-family of $q$-expansions of twisted modular  forms.

\subsection{Hecke algebra and operators $\Up_{v}$}\label{up etc} Let $L$ be a field and let 
$$\mathscr{H}(L)=C^{\infty}_{c}(\GL_{2}(\A^{\infty}),L )$$ be the  Hecke algebra of smooth compactly supported functions with the convolution operation (denoted by $*$), and for any finite set of non-archimedean places $S$ let $\mathscr{H}^{S}(L)=C^{\infty}_{c}(\GL_{2}(\A^{S\infty}), L)$, $\mathscr{H}_{S}(L)=C^{\infty}_{c}(\GL_{2}(F_{S}),L) $. When $L=\Q$ it will be omitted from the notation.

The group $\GL_{2}(\A^{\infty})$ has a natural left action on automorphic forms by right multiplication. This action is extended to elements $f\in \mathscr{H}\otimes \C$ by 
$$T(f)\vphi(g)=\int_{ \GL_{2}(\A^{\infty})} f(h) \vphi(gh) \, dh,$$
where $dh=\prod dh_{v}$ with $dh_{v}$ the Haar measure on $\GL_{2}(F_{v})$ assigning volume $1$ to $\GL_{2}(\OO_{F,v})$.
If $K\subset \GL_{2}(\A^{\infty})$ is a compact open subgroup, we define $e_{K}=T(\vol(K)^{-1}\one_{K})\in \mathscr{H}$. It acts as a projector on $K$-invariant forms. If $g\in \GL_{2}(\A^{\infty})$ and $K$, $K'\subset  \GL_{2}(\widehat{\OO}_{F })$ are open compact subgroups, we define the operator $[KgK']:=T(\one_{KgK'})$.

By the strong multiplicity one theorem, for  each level $K$, each   $M$-rational automorphic representation $\sigma$ which is discrete series of weight~$2$ at all infinite places, and  each finite set of non-archimedean places $S$ such that $K$ is maximal away from $S$, there are  spherical (that is, $K(1)^{S}$-biinvariant) elements $T({\sigma})\in \mathscr{H}^{S}(M)$ whose action on $M_{2}(K, M)$ is given by the idempotent  projection $e_{\sigma}$ onto $\sigma^{K}\subset M_{2}(K, M)$. 

On the space ${\bf M} (K^{p}, L)$ of $p$-adic modular forms, with $K^{p}\supset K(N)^{p}$, there is a continuous action 
 of $Z(Np^{\infty}):=\A^{\infty,\times}/\baar{F^{\times}U_{F}(Np^{\infty})}$,
extending  the central action $z. \vphi(g)=\vphi(gz)$  on modular forms. For a continuous  character $\omega\colon Z(Np^{\infty})\to L^{\times}$, we denote by ${\bf M}(K^{p}, \omega, L)$ the set of $p$-adic modular forms $\vphi$ satisfying $z. \vphi=\omega(z)\vphi$, and by ${\bf S}(K^{p}, \omega, L)$ its subspace of cuspidal forms. If $\omega$ is the restriction of a finite order character of $Z(\A^{\infty})/U_{F}(Np^{\infty})$, then
 we have $M_{2}(K^{p}K(p^{\infty})_{p}, \omega, L)\subset {\bf M}(K^{p}, \omega, L)$.

The action of $\mathscr{H}^{Sp}=C^{\infty}_{c}(\GL_{2}(\A^{Sp\infty}), \Q)$ extends continuously to the space ${\bf S}(K^{p}, \omega, L)$ if $K^{p}$ is maximal away from $S$; explicitly, if $\vphi$ is the  $q$-expansion with reduced coefficients  $W_{a, \vphi}^{\natural}$ and $h(x)=
\one_{K(1)^{Np}\smalltwomat{\vpi_{v}}{}{}1 K(1)^{Np}}$, we have 
\begin{align}\label{hecke on q}
W_{a, T(h)\vphi}^{\natural} = W_{a\vpi_{v}, \vphi}^{\natural} +\omega^{-1}(\vpi_{v})W_{a/\vpi_{v}, \vphi}^{\natural}.
\end{align}
Moreover if $S'$ is another set of finite places not containing those above $p$ and $S''=S\cup S'$, the action of $\mathscr{H}^{S''p}$ extends in the same way to the space $\baar{\bf S}{}'=\baar{\bf S}{}'_{S'}(K^{p})$ defined after Lemma \ref{approx lemma}.

\subsubsection{Operators $\Up_{v}$ } Let $v$ be a finite place of $F$, $\vpi_{v}\in F_{v}$ a uniformiser, $K^{v}\subset  \GL_{2}(\widehat{\OO}^{v}_{F })$ a compact open subgroup. For each $r\geq 1$, we define Hecke operators
\begin{align*}
\Up_{v,r}^{*}&=[K^{v}K_{1}^{1}(\vpi_{v}^{r})_{v} \smalltwomat {\vpi_{v}}{}{}1 K^{v}K_{1}^{1}(\vpi_{v}^{r})_{v}],\\
\Up_{v,*, r}&=[K^{v}K^{1}_{1}(\vpi_{v}^{r})_{v} \smalltwomat 1{}{}{\vpi^{-1}_{v}} K^{v}K^{1}_{1}(\vpi_{v}^{r})_{v}].
\end{align*}
They depend on the choice of uniformisers $\vpi_{v}$, although a sufficiently high (depending on $r$) integer power of them does not. They are compatible with changing  $r$ in the sense that $\Up_{v,*,r}e_{K^{1}(\vpi_{v}^{r'})_{v}}=\Up_{v,* r'}$ for $r'\leq r$ and similarly for $\Up_{v,r}^{*}$; we will hence omit the $r$ from the notation. If $\vphi\in S_{2}(K^{p}K^{1}(p^{r})_{p}, \omega)$ has reduced $q$-expansion coefficients $W_{\vphi, a}^{\natural}$ for $a\in \A^{\infty,\times}$, then $\Up_{v, *}\vphi$ has reduced $q$-expansion coefficients $W_{\Up_{v, *}\vphi, a}^{\natural }= \omega^{-1}(\vpi_{v})W_{\vphi,a\omega_{v}}^{\natural}$.
By this formula we can extend $\Up_{v, *}$ to a continuous operator on $p$-adic reduced $q$-expansions, and in particular on $p$-adic modular forms.

\subsubsection{Atkin-Lehner operators} Let $v$ be a finite place and fix the same  uniformiser $\vpi_{v}$ as in the previous paragraph. Then we define elements
$$w_{r,v}:=\twomat {}1{-\vpi_{v}^{r}}{}\in \GL_{2}(F_{v})\subset \GL_{2}(\A)$$
for $r\geq 0$, and denote by the same names the operators they induce on automorphic forms by right multiplication.
We have $w_{r,v}^{-1}K_{1}^{1}(\vpi^{s})_{v}w_{r,v}=K^{1}_{1}(\vpi^{s}_{v})_{v}$.

If $r=(r_{v})_{v\vert p}$, we  define $w_{r}=(w_{r_{v},v})_{v\vert p}\in \GL_{2}(F_{p})=\prod_{v\vert p}\GL_{2}(F_{v})$, and similarly $w_{r}^{-1}$.

\subsection{Universal Kirillov and Whittaker models}\label{univ kir}
Let $F_{v}$ be a non-archimedean local field, and recall the space $\Psi_{v}$ of abstract additive characters of level~$0$ of $F_{v}$ defined in \S\ref{intro-plF}. Let $\psi_{{\rm univ},v}\colon F_{v} \to \OO(\Psi_{v})^{\times}$ be the tautological character, which we identify with  an action of the unipotent subgroup $N=N(F_{v})\cong F_{v}\subset \GL_{2}(F_{v})$ on the  sheaf $\OO_{\Psi_{v}}$.
 Let $\sigma_{v}$ be an infinite-dimensional representation of $\GL_{2}(F_{v})$ on a vector space over a number  field $M$. A \emph{Whittaker model}   over $M\otimes \OO_{\Psi_{v}}$ for  $\sigma_{v}\otimes_{\Q}\OO_{\Psi_{v}} $ is a non-trivial $\GL_{2}(F_{v})$-equivariant map    $\sigma_{v}\otimes \OO_{\Psi_{v}}\to M\otimes {\rm Ind}_{N}^{\GL_{2}(F_{v})}\psi_{{\rm univ},v}$ of free sheaves over $M\otimes \OO_{\Psi_{v}}$. We will often identify this map with its image.

 Let $P_{0}\subset \GL_{2}(F_{v})$ be the mirabolic group of matrices $\smalltwomat a b{}1$. A
  \emph{Kirillov model}   over $M\otimes \OO_{\Psi_{v}}$ for  $\sigma_{v}\otimes_{\Q}\OO_{\Psi_{v}} $ is a non-trivial  $P_{0}$-equivariant map $\sigma_{v}\otimes \OO_{\Psi_{v}}\to  M\otimes {\rm Ind}_{N}^{P_{0}}\psi_{{\rm univ}, v}$. We will often identify this map with its image and the image with a  subsheaf of $C^{\infty}(F_{v}^{\times}, M)\otimes\OO_{\Psi_{v}}$ by restricting functions from $P_{0}$ to $\{\smalltwomat a{}{}1\ |\ a\in F_{v}^{\times}\} \cong F_{v}^{\times}$.

\begin{lemm}\label{lem univ kir} Let $\sigma_{v}$ be an irreducible admissible infinite-dimensional representation of $\GL_{2}(F_{v})$ on a rational vector space, $M=\End(\sigma_{v})$. Then $\sigma_{v}\otimes_{\Q}\OO_{\Psi_{v}} $ admits a  Whittaker model  $\mathscr{W}(\sigma_{v}, \psi_{{\rm univ},v})$ (respectively, a   Kirillov model $\mathscr{K}(\sigma_{v}, \psi_{{\rm univ}, v})$) over $M\otimes\OO_{\Psi_{v}}$, unique up to $(M\otimes \OO_{\Psi_{v}})^{\times}$,  whose specialisation at every closed point  $\psi_{v}\in \Psi_{v}$ is the unique Whittaker model    $\mathscr{W}(\sigma_{v}, \psi_{v})$ (respectively, the unique Kirillov model $\mathscr{K}(\sigma_{v}, \psi_{v})$)  of $\sigma_{v}\otimes \Q(\psi_{v})$.

 If we view $\mathscr{W}(\sigma_{v}, \psi_{{\rm univ}, v})$ (respectively $\mathscr{K}(\sigma_{v}, \psi_{{\rm univ}, v})$) as a subsheaf  of $C^{\infty}(\GL_{2}(F_{v}), M)\otimes \OO_{\Psi_{v}}$ (respectively as a subsheaf of  $C^{\infty}(F_{v}^{\times}, M)\otimes\OO_{\Psi_{v}}$), then the restriction map $W\mapsto f$, $f(y):=W\left(\smalltwomat y{}{}1\right)$ induces an isomorphism $\mathscr{W}(\sigma_{v}, \psi_{{\rm univ}, v})\to \mathscr{K}(\sigma_{v}, \psi_{{\rm univ}, v})$.

We call  $\mathscr{W}(\sigma_{v}, \psi_{{\rm univ},v})$ (respectively $\mathscr{K}(\sigma_{v}, \psi_{{\rm univ}, v})$) the \emph{universal Whittaker model} (respectively,  \emph{the universal Kirillov  model}) for $\sigma_{v}$. The universal Kirillov model admits a natural $M$-structure, that  is an  $M$-vector space\footnote{Which is \emph{not} stable under the action of $\GL_{2}(F_{v})$.} 
$$\mathscr{K}(\sigma_{v}, \psi_{{\rm univ}, v})_{M}\subset C^{\infty}(F_{v}^{\times}, M)$$
 such that $\mathscr{K}(\sigma_{v}, \psi_{{\rm univ}, v})_{M}\otimes \OO_{\Psi_{v}}= \mathscr{K}(\sigma_{v}, \psi_{{\rm univ}, v})$.
\end{lemm}
\begin{proof}  The proof of existence and uniqueness of Whittaker models given e.g. in \cite[\S 4.4]{bump} carries over to our context  after replacing $\C$ by $M\otimes \OO_{\Psi_{v}}$ and the fixed $\C^{\times}$-valued character $\psi_{v}$ of \emph{loc. cit.} with ${\psi_{{\rm univ}, v}}$. The analogous result for Kirillov models, together with the isomorphism  $\mathscr{W}(\sigma_{v}, \psi_{{\rm univ}, v})\to \mathscr{K}(\sigma_{v}, \psi_{{\rm univ}, v})$, follows formally from Frobenius reciprocity as in \cite[Corollary 36.2]{bh}. We prove the assertion on the $M$-structure for  $\mathscr{K}(\sigma_{v}, \psi_{{\rm univ}, v})$, after dropping subscripts $v$.

As in  the classical case,   the  space of Schwartz functions $\calS(F^{\times}, M) \otimes \OO_{\Psi}$ is an irreducible $P_{0}$-representation (see \cite[Corollary 8.2]{bh}), hence contained in ${\mathscr{K}}:=\mathscr{K}(\sigma^{\iota}, \psi_{\rm univ})\subset C^{\infty}(F^{\times}, M)\otimes \OO_{\Psi}$. Moreover  $\baar{\mathscr{K}}:=\mathscr{K}/\calS(F^{\times},M)\otimes \OO_{\Psi}$ is a free sheaf over $M\otimes \OO_{\Psi}$ of  rank $d\leq 2$ depending on the type of $\sigma_{}$ (as can be checked on the points of $\Psi$ by the classical theory). Since the space $\calS(F^{\times}, M) \otimes \OO_{\Psi}$ has the obvious $M$-structure $\calS(F^{\times}, M) $, it suffices to describe $d$ generators for $\baar{\mathscr{K}}$ represented by functions in $C^{\infty}(F^{\times}, M)$.

 If $\sigma$ is supercuspidal then $d=0$ and there is nothing to prove. If $\sigma={\rm St}(\mu|\cdot |^{-1})$ is  special with $M^{\times}$-valued central character $\mu^{2}|\cdot|^{-2}$, then $d=1$ and a generator for $\baar{\mathscr{K}}$  is $f_{\mu}(y):=\mu(y)\one_{\OO_{F}- \{0\}}(y)$. If $\sigma$ is an irreducible principal series ${\rm Ind}(\mu_{1}, \mu_{2}|\cdot|^{-1})$ (plain un-normalised induction) with $M^{\times}$-valued characters $\mu_{1}$, $\mu_{2}$, then  $d=2$; if $\mu_{1}\neq \mu_{2}$, a pair of generators  for $\baar{\mathscr{K}}$ is $\{f_{\mu_{1}}, f_{\mu_{2}}\}$. If $\mu_{1}=\mu_{2}=\mu$, a pair of generators is $\{f_{\mu}, f_{\mu}'\}$ with $f_{\mu}'(y):= v(y)\mu(y)\one_{\OO_{F}- \{0\}}(y).$
\end{proof}

We will often slightly abusively identify Whittaker and Kirillov models by $W\mapsto f$, $f(y)=W\left(\smalltwomat y{}{}1\right)$.

If $\sigma^{\infty}$ is an $M$-rational automorphic representation of weight $2$, then after choosing any embedding $\iota\colon M\into \C$ and any nontrivial character $\psi\colon \A/F\to \C^{\times}$, the $q$-expansion coefficients of any $\vphi\in \sigma^{\infty}$ can be identified with the product of 
the local Kirillov-restrictions  $f_{v}$ of the  Whittaker function $W=W_{v}$ of $\vphi^{\iota}$ (when $W$ is indeed factorisable). Equivalently, the $f_{v}$ belong to the $M$-rational subspaces and are therefore independent of the choice of additive character. 

\begin{lemm}\label{kir pair}
In the situation of the previous lemma, there is a pairing
 $$(\, , \, )_{v}\colon \mathscr{K}(\sigma_{v}, \psi_{{\rm univ},v})\otimes_{M} \mathscr{K}(\sigma_{v}^{\vee}, \psi_{{\rm univ},v}^{-1})\to M\otimes \OO_{\Psi_{v}}$$ 
such that  for any $f_{1}$, $f_{2}$ in the  $M$-rational subspaces  $\mathscr{K}(\sigma_{v}, \psi_{{\rm univ},v})_{M}$, respectively $  \mathscr{K}(\sigma_{v}^{\vee}, \psi^{-1}_{{\rm univ}, v})_{M}$, the paring  $(f_{1}, f_{2})_{v}\in M$, and that for  any $\iota\colon M\into \C$, we have
\begin{align}\label{eq kir pair}
\iota(f_{1},f_{2})_{v}= {   \zeta_{F,v}(2)  \over L(1, \sigma_{v}^{\iota}\times \sigma_{v}^{\iota \vee})} \int_{F_{v}^{\times}}\iota f_{1}(y) \iota f_{2}(y)
 {d^{\times}y\over |d|_{v}^{1/2}}.
\end{align}
\end{lemm}
The  right-hand side is understood in the sense of analytic continuation to $s=0$ for the function of $s$ defined, for $\Re(s)$ sufficiently large, by the normalised  convergent integral $$ {   \zeta_{F,v}(2)\over L(1+s, \sigma_{v}^{\iota}\times \sigma_{v}^{\iota \vee})} \int_{F^{\times}}\iota f_{1}(y) \iota f_{2}(y)|y|^{s} {d^{\times}y\over |d|_{v}^{1/2}}.$$

The normalisation is such that the pairing equals $1$ when $\sigma_{v}$ is an unramified principal series and the $f_{i}$ are normalised new vectors.
\begin{proof} We use the notation of the proof of Lemma \ref{lem univ kir}, dropping all subscripts $v$.  We simply need to show that the given expression belongs to $\iota M$ if $f_{1}$, $f_{2}$ belong to the $M$-rational subspace of $\mathscr{K}$ and that any pole of the integral  $I_{s}(f_{1}, f_{2}):=\int_{F^{\times}}\iota f_{1}(y) \iota f_{2}(y)|y|^{s}
 {d^{\times}y\over |d|^{1/2}}$
is cancelled by a pole of $ L(1+s, \sigma^{\iota}\times \sigma^{\iota \vee})$.
 If either of $f_{i}\in \calS(F^{\times}, M)$, the integral is just a finite sum of elements in $\iota M$. Then we only need to compute the integral  when $f_{1}$, $f_{2}$ are among the $M$-rational generators of $\baar{\mathscr{K}}$, which is a standard calculation. 

In our application there will be no poles  by the Weil conjectures, so we limit ourselves to proving the statement  in   the case where  $\sigma= {\rm Ind}(\mu_{1}, \mu_{2}|\cdot|^{-1})$ is a principal series with $\mu_{1}\neq \mu_{2}$. (The other cases are similar, cf. also the proof of Proposition \ref{interp rnat}.) Then $\sigma^{\vee}={\rm Ind}(\mu_{1}',\mu_{2}'|\cdot|^{-1})$ with $\mu_{1}'=\mu_{1}^{-1}|\cdot|$, $\mu_{2}'=\mu_{2}^{-1}|\cdot|$, and (dropping also the $\iota$ from the notation) $L(1+s, \sigma\times \sigma^{\vee})= (1-q_{F}^{-1-s})^{-2}(1-\mu_{1}\mu_{2}'(v)q_{F}^{-s})^{-1}(1-\mu_{1}'\mu_{2}(v)q_{F}^{-s})^{-1}$, where $\mu(v):=\mu(\vpi_{v})$ if $\mu$ is an unramified character and $\mu(v):=0$ otherwise. 

Assume that $f_{1}=f_{\mu_{1}}$ (the case $f_{1}=f_{\mu_{2}}$ is similar). If $f_{2}=f_{\mu_{1}'}$ then $I_{s}(f_{1}, f_{2})=(1-q_{F}^{-1-s})^{-1}$
has no pole at $s=0$. 
 If $f_{2}=f_{\mu_{2}'}$, then    $I_{s}(f_{1}, f_{2})=(1-\mu_{1}\mu_{2}'(v)q_{F}^{-s})^{-1}$, whose inverse is a factor of $L(1+s, \sigma_{v}^{}\times \sigma_{v}^{ \vee})^{-1} $ in $M[q_{F}^{-s}]$.
\end{proof}

\subsection{$p$-crtical forms and the $p$-adic Petersson product}
As in \cite{dd}, we  introduce the following notion.
\begin{defi}\label{crit} Let $W=(0, (W_{a}))\in {\bf S}'(L)$ be a reduced $q$-expansion without constant term, with values in a $p$-adic field $L$, and let $v\vert p$. We say that $W$ is \emph{$v$-critical} if for some integer $r$, the following condition is satisfied: there is $c\in \Z$ such that, for each $a\in \A^{\infty, \times}$ 
with $v(a)=r$
and $s\in \N$, 
$$W_{a\vpi_{v}^{s}}\in q_{F,v}^{s-c}\OO_{L}.$$
We say that $W$ is \emph{$p$-critical} if it is a sum of $v$-critical $q$-expansions for $v\vert p$.
\end{defi}

For each $v\vert p$,  we define \emph{ordinary projectors} $e_{v}$ and $e$ on ${\bf M}(K^{p}, \omega, L)$ by 
$$e_{v}(\vphi')=\lim_{n\to \infty}\Up_{v, *}^{n!}\vphi', \qquad e:=\prod_{v}e_{v}.$$
They are independent of the choice of uniformisers. The  image of $e_{v}$  is contained in $M_{2}(K^{p}K^{1}_{1}(p^{\infty}),\omega, L).$
It is clear that $v$-critical forms belong to the kernel of $e_{v}$.

\subsubsection{$p$-adic Petersson product}
Let $M$ be a number field, and 
let $\sigma^{\infty}$ be an \emph{$M$-rational}
 cuspidal automorphic representation of $\GL_{2}$ of weight $2$ as in Definition \ref{M-rat},
with central character $\omega\colon F^{\times}\bks {\A}\to M^{\times}$.

Following Hida, we will define a $p$-adic analogue of the Petersson inner product with a form $\varphi$ in $\sigma^{\infty}$ when $\sigma_{p}$ is $\frakp$-ordinary for a prime $\frakp \vert p$ of $M$. 
First we  define an algebraic version of the Petersson product, which requires no ordinariness assumption. If $\iota\colon M\into \C$, let $\vphi^{\iota}:=\iota\vphi\otimes\vphi_{\infty}\in \sigma^{\iota}$ be the automorphic form whose   Whittaker function at infinity is anti-holomorphic of smallest $K_{\infty}$-type.

\begin{lemm}\label{alg-pet} 
 There is a unique pairing 
$$(\, , \,)_{\sigma^{\infty}}\colon \sigma^{\infty} \otimes_{M}M_{2}( \omega^{-1},M)\to M$$
such that for any $\vphi_{1}\in\sigma^{\infty}$, $\vphi_{2}\in M_{2}(\omega^{-1},M)$, and any $\iota\colon M\into \C$, we have 
$$(\vphi_{1}, \vphi_{2})_{\sigma}= {|D_{F}|^{1/2}\zeta_{F}(2)\over L(1, \sigma^{\iota}, \ad)} ( \vphi_{1}^{\iota},\iota\vphi_{2}),$$
where  
$$ ( \vphi_{1}',\vphi_{2}'):=
 \int_{\GL_{2}(F)Z(\A)\bks \GL_{2}(\A)}  \vphi'_{1}(g)\vphi_{2}'(g)\, dg$$
is the usual Petersson product  on complex automorphic forms with respect to the Tamagawa masure $dg$. 
\end{lemm}
\begin{proof}
Note first that if such a pairing exists, it annihilates forms on the right-hand side which are orthogonal (under the complex Petersson product in any embedding) to forms in $\sigma$. Then we just need to use a well-known formula for the adjoint $L$-value in terms of Petersson product; we quote it in the version given in   \cite[p. 55]{tyz}: for an antiholomorphic form $\vphi_{1}'$ in the space of $\sigma^{\iota}$ and a holomorphic form $\vphi_{2}'$ in the space of  $\sigma^{\vee, \iota}$, both  rational over $\iota M$, with factorisable Whittaker functions $W_{i}^{\iota}$, we have
\begin{align}
{  |D_{F}|^{1/2} \zeta_{F}(2)   (\vphi'_{1}, \vphi'_{2})\over  2 L(1, \sigma^{\iota}, {\rm ad})}=\prod_{v}
(W_{1,v}^{\iota}, W_{2,v}^{\iota})_{v}
\end{align}
where for all $v$ the local pairings are given  by the right-hand side of \eqref{eq kir pair} and do not depend on the choice of additive characters.
 Each local factor in the product is rational over $\iota M$ and almost all of them are equal to~$1$. 
\end{proof}
\begin{rema} If $\vphi_{2}\in M_{2}(M)$ does not have central character $\omega^{-1}$, we can still define 
$(\vphi_{1}, \vphi_{2})_{\sigma^{\infty}}:=(\vphi_{1}, \vphi_{2, \omega^{-1}})_{\sigma^{\infty}}$ where
$$\vphi_{2, \omega^{-1}}(g):=\dashint_{Z(F)\bks Z(\A)}\vphi_{2}(zg)\omega(z)\, dz.$$
\end{rema}

 Now fix a prime $\frakp\vert p$ of $M$  and a finite extension $L$ of $M_{\frakp}$, and assume that for all $v\vert p$, $\sigma_{v}\otimes L$ is nearly $\frakp$-ordinary with unit character  $\alpha_{v}\colon F_{v}^{\times }\to \OO_{L}^{\times}$ in the sense of Definition \ref{def-n-ord}.
 Fix a Whittaker  functional $\mathscr{W}_{p}=\prod_{v}\mathscr{W}_{v}$ at $p$ and let 
  $\varphi\in \sigma^{\infty}\otimes_{M}M(\alpha)$ be a form in the space of $\sigma^{\infty}$ 
whose image under $\mathscr{W}_{v}$ is the function (viewed in the $M(\alpha)$-rational part of any Kirillov model)
\begin{align}\label{kir-ord}
W_{v}(y)=\one_{\OO_{F}-\{0\}}(y)|y|\alpha_{v}(y).
\end{align}
Note that $W_{v}$, viewed in a Kirillov model associated to an additive character of level $0$,    satisfies $\Up_{v}^{*}W_{v}=\alpha_{v}(\vpi_{v})W_{v}$. 

In the next proposition, we use the notation $\alpha(\vpi)^{r}:= \prod_{v\vert p}\alpha_{v}(\vpi_{v})^{r_{v}}$. 
\begin{prop}\label{p-pet} There exists a unique bounded linear functional
$$\lf\colon {\bf M}(K^{p}, \omega^{-1}, L)\to L $$ 
satisfying the following:
\begin{enumerate}
\item\label{on sr} Let $r=(r_{v})_{v\vert p }\in \Z_{\geq 1}^{\{v\vert p\}}$. The restriction of $\lf$ to $M_{2}(K^{p}K^{1}(p^{r})_{p}, M(\alpha))$ 
is given by  
\begin{equation}\label{lf-formula}
\lf(\vphi')={  \alpha(\vpi)^{-r}  (w_{r}\vphi, \vphi')_{\sigma^{\infty}} } 
={ \alpha(\vpi)^{-r} (\vphi,w_{r}^{-1} \vphi')_{\sigma^{\infty}}  }\in M(\alpha)
\end{equation}
for any choice of uniformisers $\vpi_{v}$ in the definitions of $\Up_{v, *}$, $\Up_{v}^{*}$, $w_{r}$. 
\item\label{uv left} We have 
 $$\lf(\Up_{v,*}\vphi')=\alpha_{v}(\vpi_{v})\lf(\vphi')$$
 for all $v\vert p$ and all $\vphi'$.  
 \item\label{Lkillscrit}  $\lf$  vanishes on $p$-critical forms. 
 \item\label{idempot}  Let $T(\sigma^{\vee})\in  \mathscr{H}^{S}({M})$ (where $S$ is any sufficiently large set of finite places containing those dividing $p$) be any element whose image $T(\sigma)^{\iota}\in \mathscr{H}^{S}( M)\otimes_{M, \iota}\C$ acts on $S_{2}(K^{p}K^{1}(p^{r})_{p}, \C)$ as the idempotent projector onto $(\sigma^{\vee,\iota})^{K^{p}K^{1}(p^{r})_{p}}$ for any $\iota\colon M\into \C$ and $r\geq 1$. Let  $T_{\iota_{\frakp}}(\sigma^{\vee})$ be the image of $T(\sigma^{\vee})$ in $\mathscr{H}^{S}(M )\otimes_{M, \iota_{\frakp}} L$. Then 
 $$\lf\circ T_{\iota_{\frakp}}(\sigma^{\vee}) = \lf.$$ 
\end{enumerate}
\end{prop}
\begin{proof} 
By property  \ref{uv left}, for each $v$ we must have 
\begin{equation}\label{popop}
\lf(e_{v}\vphi')=\lim_{n\to \infty} \lf (\Up_{v, *}^{n!}\vphi')=\lim_{n\to \infty} \alpha_{v}(\vpi_{v})^{n!}\lf(\vphi')=\lf(\vphi')
\end{equation}
as $\alpha_{v}(\vpi_{v})$ is a $p$-adic unit; note that this expression does not depend on the choice of uniformisers. It follows that $\lf$ must factor through the ordinary projection $$ e\colon {\bf M}(K^{p}, \omega^{-1}, L)\to M_{2}(K^{p}K^{1}(p^{\infty})_{p},\omega^{-1}, L),$$
which implies property \ref{Lkillscrit}. On the image of $e$,   $\lf$ must be defined defined by  \eqref{lf-formula}, which makes uniqueness and property \ref{idempot} clear.

It remains to show  the existence (that is, that \eqref{lf-formula} is compatible with changing $r$) and that the first equality in \eqref{lf-formula} holds for all $r$ for the functional $\lf$ just defined (the second one is trivial). For the latter, 
 we have
\begin{multline}
(w_{r}\vphi, \Up_{v, *}\vphi')=(w_{r}\vphi,  K^{1}(\vpi_{v}^{r_{v}})_{v}\smalltwomat {1}{}{}{\vpi_{v}^{-1}}\vphi')
=(\smalltwomat 1{}{}{\vpi_{v}} K^{1}(\vpi_{v}^{r_{v}})_{v} w_{r}\vphi, \vphi')\\
= (w_{r}K_{1}(\vpi_{v}^{r_{v}})_{v}\smalltwomat {\vpi_{v}}{}{}1 \vphi, \vphi')=(w_{r}\Up_{v}^{*}\vphi, \vphi')=\alpha_{v}(\vpi_{v})(w_{r}\vphi, \vphi').
\end{multline}
The compatibility with change of $r$ can be seen by a similar calculation.
\end{proof}
We still use the notation $\lf$ for the linear form deduced from $\lf$ by extending scalars to some $L$-algebra. The analogous remark will apply to   $(\, , \,)_{\sigma_{\infty}}$.

\section{The $p$-adic $L$-function}\label{sec plf3}

\subsection{Weil representation}
We start by recalling from \cite{wald, yzz}  the definition of the Weil representation for groups of similitudes.  
\subsubsection{Local case} Let $V=(V,q)$ be a quadratic space of even dimension  over a local field $F$ of characteristic not 2. Fix  a nontrivial additive character $\psi$ of $F$. For simplicity we assume $V$ has even dimension. For $u\in F^\times$, we denote by $V_u$ the quadratic space $(V,uq)$. We let $\GL_2(F)\times \GO(V)$ act on the usual  space of Schwartz functions\footnote{The notation is only provisional for the archimedean places, see below.}  $\mathscr{S}(V\times F^\times)$  as follows (here $\nu\colon \GO(V)\rightarrow\mathbf{G}_m$ denotes the similitude character):
\begin{itemize}
\item $r(h)\phi(x,u)=\phi(h^{-1}x,\nu(h)u)$ \quad for $h\in \GO(V)$;
\item$r(n(b))\phi(x,u)=\psi(buq(x))\phi(x,u)$ \quad for $n(b)\in N(F)\subset\GL_2(F)$;
\item $r\left(\begin{pmatrix}a&\\& d\end{pmatrix}\right)\phi(x,u)=\chi_{V_{u}}(a)|{a\over d}|^{\dim V\over 4}\phi(at, d^{-1}a^{-1}u)$;
\item $r(w)\phi(x,u)=\gamma(V_u)\hat{\phi}(x,u)$ for $w=\begin{pmatrix}&1\\-1 &\end{pmatrix}.$
\end{itemize}
Here $\chi_{V}=\chi_{(V,q)}$ is the quadratic character attached to $V$,  $\gamma(V,q)$ is a fourth root of unity, and $\hat{\phi}$ denotes Fourier transform in the first variable with respect to the self-dual measure for the character $\psi_{u}(x)=\psi(ux)$. We will need to note the following facts (see e.g. \cite{JL}): $\chi_{V}$  is trivial if $V$ is a quaternion algebra over $F$ or $V=F\oplus F$, and $\chi_{V}=\eta$ if $V$ is a separable quadratic extension $E$ of $F$ with associated character $\eta$; and $\gamma(V)=+1$ if $V$ is the space of $2\times 2$ matrices or $V=F\oplus F$, $\gamma(V)=-1$ if $V$ is a non-split quaternion algebra.

\medskip

We state here a lemma which will be useful later.
\begin{lemm}\label{uppertriang} Let $F$ be a  non-archimedean local field  and   $\phi\in \calS(V\times F^{\times}) $ a Schwartz function with support  contained in 
$$\{ (x,u)\in V\times F^{\times}\,\colon\,  uq(x)\in \OO_{F}\}.$$
Suppose that the character $\psi$ used to construct the Weil representation has level $0$. Then $\phi$ is invariant under $K_{1}^{1}(\vpi^{r})\subset \GL_{2}(\OO_{F})$ for sufficiently large $r$.
 If moreover $\phi(x,u)$ only depends on $x$ and on the valuation $v(u)$, then $\phi$ is invariant under $K^{1}(\vpi^{r})$.
\end{lemm}
\begin{proof} By continuity of the Weil representation,
 for the first assertion
  it suffices to show the invariance under $N(\OO_{F})$ . This follows from the observation that under our assumption, in the formula
 $$r(n(b))\phi(x,u)=\psi(ubq(x)) \phi(x,u),$$ 
 the multiplier $\psi(ubq(x))=1$ whenever $(x,u)$ is in the support of $\phi$.
  The second assertion is then equivalent to the invariance of $\phi$ under the subgroup $\smalltwomat 1{}{}{\OO_{F}^{\times}}\subset \GL_{2}(\OO_{F})$, which is clear. 
\end{proof}

\subsubsection{Fock model and reduced Fock model} 
Assume that $F=\R$ and $V$ is positive definite. Then we will prefer to consider a modified version of the previous setting.   Let the \emph{Fock model} $\calS(V\times \R^{\times}, \C)$ be the space of functions spanned by those of the form 
$$H(u)P(x) e^{-2\pi |u| q(x)},$$ where $H$ is a compactly supported smooth function on $\R^{\times}$ and $P$ is a complex  polynomial function on $V$.  This space is not stable under the action of $\GL_{2}(\R)$, but it is so under the restriction of the induced $(\mathfrak{gl}_{2, \R},{\bf O}_{2}(\R))$-action on the usual Schwartz space (see \cite[\S 2.1.2]{yzz}). 

We will also need to consider the \emph{reduced Fock space} $\bcalS(V\times \R^{\times})$ spanned by functions of the form
$$\phi(x,u)=(P_{1}(uq(x))+\mathrm{sgn}(u)P_{2}(uq(x))) e^{- 2\pi |u| q(x)}$$
where $P_{1}$, $P_{2}$ are polynomial functions with rational coefficients.
It contains the \emph{standard Schwartz function}
$$\phi(x,u)=\one_{\R_{+}}(u)e^{-2\pi|u| q(x)},$$
which for $x\neq 0$ satisfies 
\begin{align}\label{weilwhitt}
r(g)\phi(x,u)=W^{(d)}_{uq(x)}(g)
\end{align}
if $V$ has dimension $2d $ and $W^{(d)}$ is the standard holomorphic  Whittaker function \eqref{holwhitt} (see \cite[\S 4.1.1]{yzz}).

By \cite[\S 4.4.1, 3.4.1]{yzz}, there is a surjective quotient map 
\begin{equation}\label{sql}
\begin{aligned}
\calS(V\times\R^{\times} , \C)&\to\bcalS(V\times\R^{\times})\otimes_{\Q} \C\\
 \Phi& \mapsto\phi(x,u)=\baar{\Phi}(x,u)= \int_{\R^{\times}}\dashint_{{\bf O}(V)} r(ch) {\Phi}(x,u)\, dh \, dc.
 \end{aligned}
\end{equation}
We let $\calS(V\times\R^{\times} ) \subset \calS(V\times\R^{\times} , \C)$ be the preimage of $\bcalS(V\times\R^{\times})$.
For the sake of uniformity, when $F$ is non-archimedean we set $\bcalS(V\times F^{\times}):=\calS(V\times F^{\times})$.

\subsubsection{Global case} Let $({\bf V},q)$ be an even-dimensional quadratic space over the ad\`eles $\A=\A_{F}$ of a totally real number field $F$, and suppose that ${\bf V}_{\infty}$ is positive definite; we say that ${\bf V}$ is \emph{coherent} if it has a model over $F$ and \emph{incoherent} otherwise. Given an $\widehat{\OO}_{F}$-lattice ${\mathscr{V}}\subset {\bf V}$, we define the space $\calS({\bf V}\times \A^{\times})$ as the restricted tensor product of the corresponding local spaces, with respect to the spherical  elements
$$\phi_{v}(x,u)=\one_{\mathscr{V}_{v}}(x)\one_{\vpi_{v}^{n_{v}}}(u),$$ 
if $\psi_{v}$ has level $n_{v}$. We call such $\phi_{v}$ the \emph{standard Schwartz function} at a non-archimedean place $v$.
We define similarly the reduced  space $\bcalS({\bf V}\times \A^{\times})$, which admits a quotient map
\begin{gather}\label{schwartz quotient}
\calS({\bf V}\times \A^{\times})\to\bcalS({\bf V}\times \A^{\times})
\end{gather}
defined by the product of the maps \eqref{sql} at the infinite places and of the identity at the finite places. The Weil representation of $\GL_{2}(\A^{\infty})\times \GO({\bf V}^{\infty})\times (\mathfrak{gl}_{2,F_{\infty}},{\bf O}({\bf V}_{\infty}))$ is  the restricted tensor product of the local representations.

\subsection{Eisenstein series}\label{sec: eis series}
Let ${\bf V}_{2}$ be a two-dimensional  quadratic space over $\A_{F}$, totally definite at the archimedean places. Consider the Eisenstein series
$$E_{r}(g, u , \phi_{2},\chi_{F})=\sum_{\gamma\in P^{1}(F)\bks SL_{2}(F)} \delta_{\chi_{F}, r} (\gamma gw_{r}) r(\gamma g) \phi_{2}(0,u)$$
where 
$$\delta_{\chi_{F}, r}(g)=
\begin{cases} 
\chi_{F}(d)^{-1} &\textrm{if\ } g= \smalltwomat ab{}d   k \textrm{\ with\ } k\in K_{1}^{1}(p^{r})
\\ 0 & \textrm{if\ } g\notin PK_{0}(p^{r}).
\end{cases}$$
and $\phi_{2}\in \bcalS({\bf V}_{2}\times \A^{\times})$. 
(The defining sum is in fact not absolutely convergent, so it must be interpreted in the sense of analytic continuation at $s=0$ from the series obtained by replacing $\delta_{\chi_{F,r}}$ with $\delta_{\chi_{F, r}}\delta_{s}$, where $\delta_{s}(g)=|a/d|^{s}$ if $g=\smalltwomat ab{}dk$, $k\in K_{0}(1)$.) It belongs to the space $M_{1}^{\rm tw}( \eta\chi_{F}^{-1},\C)$ of twisted modular form of parallel weight~$1$ and central character $\eta\chi_{F}^{-1}$.

After  a suitable modification, we study its Fourier--Whittaker expansion and show that it interpolates to a $\Y_{F}$-family of $q$-expansions of twisted modular forms.

\begin{prop}\label{eis-exp}
We have
$$L^{(p)}(1, \eta\chi_{F}) E_{r}(\smallmat yx{}1,u, \phi_{2}, \chi_{F})= \sum_{a\in F} W_{a,r}( \smallmat y{}{}1, u, \phi_{2}, \chi_{F})\psi(ax)$$
where
$$W_{a,r}(g,u, \phi_{2}, \chi_{F})= \prod_{v} W_{a,r,v}(g,u, \phi_{2,v}, \chi_{F,v})$$
with, for each $v$ and $a\in F_{v}$,
\begin{multline*}
W_{a,r,v}(g,u, \phi_{2,v}, \chi_{F,v})= 
 L^{(p)}(1, \eta_{v}\chi_{F,v}) \\
\times \int_{F_{v}}\delta_{\chi_{F},v,r}(wn(b)gw_{r}) r(wn(b)g ) \phi_{2,v }(0,u)\psi_{v}(-ab) \, db.
\end{multline*}
Here $L^{(p)}(s, \xi_{v}):=L(s, \xi_{v})$ if $v\nmid p$ and $L^{(p)}(s,\xi_{v}):=1$ if $v\vert p$, and we use the  convention that $r_{v}=0$ if $v\nmid p$.
\end{prop}
\begin{proof} The standard expansion of Eisenstein series reads
$$E_{r}(gw_{r},u, \phi_{2}, \chi_{F})= \delta_{\chi_{F},r}(gw_{r})r(g)\phi_{2}(0,u)+ \sum_{a\in F} W_{a,r}^{*}( g, u, \phi_{2}, \chi_{F})\psi(ax)$$
where $W_{a,r}^{*}=L^{(p)}(1, \eta\chi_{F})^{-1}W_{a,r}$; 
but it is easy to check that $\delta_{\chi_{F,r}}(\smallmat yx{}1 w_{r})=0$.
\end{proof}

We choose convenient normalisations for the local Whittaker functions: let $\gamma_{u, v}=\gamma({\bf V}_{2,v}, uq)$  be the Weil index, and for $a\in F_{v}^{\times}$ set
$$W^{\circ}_{a,r,v}(g,u, \phi_{2,v}, \chi_{F,v}):=\gamma_{u,v}^{-1}W_{a,r,v}(g,u, \phi_{2,v}, \chi_{F,v}).$$
For the constant term,  set
$$
W^{\circ}_{0,r,v}(g,u, \phi_{2,v}, \chi_{F,v}):={\gamma_{u,v}^{-1} \over L^{(p)}(0, \eta_{v}\chi_{F,v})} W_{0,r,v}(g,u, \phi_{2,v}, \chi_{F,v}).
$$

Then for the global Whittaker functions we have 
\begin{align}\label{signW}
W_{a,r}(g,u, \phi_{2}, \chi_{F})=-\eps({\bf V}_{2})\prod_{v}W^{\circ}_{a,r,v}(g,u, \phi_{2,v}, \chi_{F,v})
\end{align}
if $a\in F^{\times}$, where $\eps({\bf V}_{2})=\prod_{v}\gamma_{u, v}$ equals $-1$ if ${\bf V}_{2}$ is coherent or $+1$ if ${\bf V}_{2}$ is incoherent; and 
\begin{align}\label{whitt0}
W_{0,r}(g,u, \phi_{2}, \chi_{F})=-\eps({\bf V}_{2}) L^{(p)}(0, \eta\chi_{F})\prod_{v}W^{\circ}_{0,r,v}(g,u, \phi_{2,v}, \chi_{F,v}).
\end{align}

\medskip

We sometimes drop $\phi_{2}$ from the notation in what follows.

\begin{lemm}\label{change y} For each finite  place $v$ and $y\in F_{v}^{\times}$, $x\in F_{v}$, $u\in F_{v}^{\times}$, we have 
$$W_{a,v}(\smallmat y{x}{}1, u)=\psi_{v}(ax)\chi_{F}(y)^{-1}|y|^{1/2}W_{ay,v}(1, y^{-1}u).$$
\end{lemm}
The proof is an easy calculation.

\begin{prop}\label{whitt-eis} The local Whittaker functions satisfy the following.
\begin{enumerate}
\item\label{vnotp} If $v\nmid p\infty$, then $W_{a,v,r}^{\circ}=W_{a,v}^{\circ}$ does not depend on $r$, and for all $a\in F_{v}$
 $$W_{a,v}^{\circ}(1,u, \chi_{F})=|d_{v}|^{1/2}L(1, \eta_{v}\chi_{F,v}) (1-\chi_{F,v}(\vpi_{v}))
\sum_{n=0}^{\infty} \chi_{F,v}(\vpi_{v})^{n} q_{F,v}^{n}\int_{D_{n}(a)}\phi_{2,v}(x_{2},u) \, d_{u}x_{2},$$
where $d_{u}x_{2}$ is the self-dual measure on $({\bf V}_{2,v}, uq)$ and 
$$D_{n}(a)=\{ x_{2}\in {\bf V}_{2,v}\, |\, uq(x_{2})\in a+p_{v}^{n}d_{v}^{-1}\}.$$
(When the sum is infinite, it is to be understood in the sense of analytic continuation from characters $\chi_{F}|\cdot|^{s}$ with $s>0$, cf. the proof of Lemma \ref{whitt-an} below.) 

\item\label{whitt-p} If $v\vert p$,  and $\phi_{2,v}$ is the standard Schwartz function, then
$$ W_{a,r, v}^{\circ}(1, u, \chi_{F})=
\begin{cases} |d_{v}|^{3/2}|D_{v}|^{1/2}  \chi_{F,v}(-1)  & \textrm{\ if \ } v(a)\geq - v(d_{v})\textrm{\ and \ } v(u)=-v(d_{v}) \\  
 0 & \textrm{\ otherwise.} 
\end{cases}
$$
\item\label{whitt-inf} If $v\vert \infty$ and $\phi_{2,v}$ is the standard Schwartz function, 
then
$$ W_{a,v}(1, u)=
 \begin{cases} 2 e^{-2\pi a} &\textrm{\ if \ } ua>0\\ 
1  &\textrm{\ if \ } a=0\\
0  &\textrm{\ if \ } ua<0.
\end{cases}
 $$
\end{enumerate}
\end{prop}
\begin{proof}
Part \ref{vnotp} is proved  similarly  to  \cite[Proposition 6.10 (1)]{yzz}, whose Whittaker function $W_{a, v}^{\circ}(s, 1, u)$ equals our $L(1, \eta_{v}|\cdot|_{v}^{s})^{-1}W^{\circ}_{a, v}(1, u, |\cdot|_{v}^{s})$. 
The proof of Part  2 is similar to that of \cite[Proposition 3.2.1, places $\vert M/\delta$]{dd}.
 Part \ref{whitt-inf} is also well-known, see e.g. \cite[Proposition 2.11]{yzz} whose normalisation differs from  ours by a factor of $\gamma_{v}L(1,\eta_{v})^{-1}=\pi i$. 
\end{proof}

\begin{lemm}\label{almost all 1} Let $a\in F$. For all finite places $v$, $|d|_{v}^{-3/2}|D_{v}|^{-1/2}W_{a,v}^{\circ}(1, u, \chi_{F})\in \Q[\chi_{F}, \phi_{v}]$, and for almost all $v$ we have
$$|d|_{v}^{-3/2}|D_{v}|^{-1/2}W^{\circ}_{a,v}(1,u, \chi_{F})=
\begin{cases} 1 &  \textrm{\ if \ } v(a)\geq - v(d_{v})\textrm{\ and \ } v(u)=-v(d_{v}) \\
0 & \textrm{\ otherwise.}
\end{cases}$$
\end{lemm}
\begin{proof} This follows from Proposition \ref{whitt-eis}.\ref{vnotp} by an explicit computation which is neither difficult nor unpleasant: we leave it to the reader.
\end{proof}

\subsection{Eisenstein family}\label{sec: eis fam}  Recall from \S\ref{intro-plF} the profinite groups $\Gamma$  and $\Gamma_{F}$  and the associated rigid spaces $\Y'$, $\Y$, $\Y_{F}$ (only the latter is relevant for this subsection).  
For each finite place $v\nmid p$ of $F$, there  are local versions  
\begin{align}\label{local Ys}
\Y'_{v},\qquad \Y_{v}, \qquad \Y_{F,v}
\end{align}
which are schemes over $M$  representing the corresponding spaces of ${\bf G}_{{ m},M}$-valued homomorphisms with domain  $E_{v}^{\times}/(V^{p}\cap E_{v}^{\times})$ (for $\Y'_{v}$, $\Y_{v}$, where $V^{p}\subset E_{\A^{p\infty}}^{\times}$ is the subgroup  fixed in the introduction) or  $F_{v}^{\times}$ (for $\Y_{F,v}$).\footnote{Concretely they are closed subschemes of split tori over $M$, cf. the proof of Proposition \ref{interp rnat}.}
 Letting $\bigotimes'$ denote the restricted tensor product with respect to the constant function~$1$, and the symbol $\Y^{?}$ stand for any of the symbols $\Y'$, $\Y$, $\Y_{F}$, we let
 $$\OO_{\Y^{?}}(\Y^{?})^{\rm f}\subset \OO_{\Y^{?}}(\Y^{?})^{\rm b}$$  denote the image of 
  $\bigotimes'_{v\nmid p} \OO(\Y^{?}_{v})\otimes_{M} L\to \OO_{\Y^{?}}(\Y^{?})$.

\begin{lemm}\label{whitt-an} For each $a\in F$, $y\in \A^{\infty,\times}$  and rational Schwartz function $\phi_{2}^{p\infty}$, 
there are:
\begin{enumerate}
\item\label{whitt-an loc} for each $v\nmid p\infty$:
	\begin{enumerate} 
	\item a Schwartz function $\phi_{2,v}(\cdot)\in \mathscr{S}(V_{2,v}, \OO(\Y_{F,v}))$ such that $\phi_{2,v}(\one)=\phi_{2,v}$ and $\phi_{2,v}(\cdot)$ is identically equal to $\phi_{2,v}$ if $\phi_{2,v}$ is standard;
	\item a function 
	$$W_{a,v}^{\circ}(y_{v},u, \phi_{2,v})\in \OO_{\Y_{F,v}}(\Y_{F,v})$$ satisfying
	$$\mathscr{W}_{a,v}^{\circ}(y_{v}, u, \phi_{2,v};\chi_{F}) = |d_{v}|^{-3/2}|D_{v}|^{-1/2} W_{a,r,v}^{\circ}(\smalltwomat {y_{v}}{}{}1, u,\phi_{2,v}(\chi_{F,v}),\chi_{F,v})$$
	for all $\chi_{F, v}\in \Y_{F,v}(\C)$;
\end{enumerate}
\item\label{whitt-an glo}  a global function
$$\mathscr{W}_{a}(y, u, \phi_{2}^{p\infty})\in \OO_{\Y_{F}}(\Y_{F})^{\rm b}, $$
which is algebraic on   $\Y_{F}^{\rm l.c.}$ and satisfies
$$\mathscr{W}_{a}(y, u, \phi_{2}^{p\infty};\chi_{F}) = |D_{F}|^{1/2}|D_{E}|^{1/2} W_{a,r}^{\infty}(\smalltwomat y{}{}1, u,\phi_{2}(\chi_{F}),\chi_{F})
$$
for each $\chi_{F}\in \Y_{F}^{\rm l.c.}(\C)$;  here $\phi_{2}(\chi_{F})=\prod_{v\nmid p\infty}\phi_{2,v}(\chi_{F,v})_{2}\phi_{2,p\infty}$ with $\phi_{2,v}$ the standard Schwartz function for each $v\vert p\infty$.  The function $|y|^{-1/2}\mathscr{W}_{a}(y,u, \phi_{2}^{p\infty})$ is bounded solely in terms of $\max  |\phi_{2}^{p\infty}|$, and if $a\neq 0$, then $\mathscr{W}_{a}\in \OO(\Y_{F})^{\rm f}$.
\end{enumerate}
\end{lemm}
\begin{proof} If $a\neq 0$,  by Lemma \ref{almost all 1} and Proposition \ref{whitt-eis}.\ref{whitt-p}, we can deduce the existence of the global function in part \ref{whitt-an glo} from the local result of part \ref{whitt-an loc}. If $a=0$, then by \eqref{whitt0} the same is true thanks to the well-known existence \cite{DR} of a bounded analytic function on $\Y_{F}$ interpolating $\chi_{F}\mapsto L^{(p)}(0, \eta\chi_{F})$.

It thus suffices to prove part  \ref{whitt-an loc}, and moreover we may restrict to  $y=1$  in view of Lemma \ref{change y}. We can uniquely  write $\phi_{2,v}=c\phi_{2,v}^{\circ}+\phi_{2,v}'$, where $\phi_{2,v}^{\circ}$ is the standard Schwartz function and $c=\phi_{2,v}(0)$. Then we set
\begin{align}\label{phi2chi}
\phi_{2,v}(\chi_{F,v}):=c\phi_{2,v}^{\circ}+{L(1, \eta_{v})\over L(1, \eta_{v}\chi_{F,v })} \phi_{2,v}'.
\end{align}
We need to show that, upon  substituting it in the expression for the local Whittaker functions given in Proposition \ref{whitt-eis}.\ref{vnotp}, we obtain a Laurent polynomial in $\chi_{F,v}(\vpi_{v})$ (which gives the canonical coordinate on $\Y_{F,v}\cong{\bf G}_{m, M}$). By  linearity and Lemma \ref{almost all 1}, it suffices to show this for the summand ${L(1, \eta_{v})\over L(1, \eta_{v}\chi_{F,v })} \phi_{2,v}'$, whose  coefficient is designed to cancel the factor $L(1, \eta_{v}, \chi_{v})$ appearing in that expression. The only  source of possible poles is the infinite sum.  For $n$ sufficiently large, if $a$ is not in the image of $uq$ then $D_{n}(a)$ is empty and therefore the sum is actually finite. On the other hand if  $a=uq(x_{a})$, then for $n$ large the function $\phi_{2,v}$ is constant and equal to $\phi_{2,v}(x_{a})$ on $D_{n}(a)$; it follows that $$\int_{D_{n}(a)}\phi_{2,v}(x_{2},u)d_{u}x_{2}= c'q_{F,v}^{-n}$$ for some constant $c'$ independent of $n$ and $\chi_{F,v}$. Then the tail of the sum is 
$$\sum_{n\geq n_{0}}c'\chi_{F,v}(\vpi_{v})^{n}=c'{\chi_{F,v}(\vpi_{v})^{n_{0}}\over 1-\chi_{F, v}(\vpi)};$$
its product with the factor $1-\chi_{F,v}(\vpi_{v})$ appearing in front of it  is then also a polynomial in $\chi_{F,v}(\vpi_{v})$.

Finally, the last two statements of part \ref{whitt-an glo}  follow by the construction and Lemma \ref{change y}.
\end{proof}

\begin{prop} There is a bounded $\Y_{F}$-family of $q$-expansions of twisted modular forms of parallel weight~$1$  
$$\mathscr{E}(u,\phi_{2}^{p\infty})
$$
such that for any $\chi_{F}\in \Y_{F}^{\rm l.c.}(\C)$  and  any $r=(r_{v})_{v|p}$ satisfying $c(\chi_{F})\vert p^{r}$, 
 we have 
$$\mathscr{E}(u,\phi_{2}^{p\infty};\chi_{F})=  
|D_{F}| 
{L^{(p)}(1, \eta\chi_{F})\over L^{(p)}(1, \eta)}
\, {}^{\qqq} E_{r}(u, \phi_{2} ,\chi_{F}),$$
where $\phi_{2}=\phi_{2}^{p\infty}(\chi_{F})\phi_{2,p\infty}$ with $\phi_{2,v}$ the standard Schwartz function for $v\vert p\infty$.
\end{prop}
\begin{proof}
This follows from  Lemma \ref{whitt-an} and
 Proposition \ref{whitt-eis}.\ref{whitt-inf}: we take the $q$-expansion  with coefficients ${2^{[F:\Q]} |D_{F}|^{1/2}  \over |D_{E}|^{1/2} | L^{(p)}(1, \eta)}    \mathscr{W}_{a}(y, u, \phi_{2}^{p\infty}) $.
 
\end{proof}

\subsection{Analytic kernel}\label{sec: an ker}

We first construct certain bounded $\Y^{?}$-families  of $q$-expansions of modular forms, for $\Y^{?}=\Y_{F}$ or $\Y'$. In general, if $\Y^{?}$ is the space of $p$-adic  characters  of a profinite group $\Gamma^{?}$, then it is equivalent to give a compatible system, for each extension $L'$ of $L$, of bounded functionals $\mathscr{C}(\Gamma^{?}, L')\to {\bf M}(K^{p}, L') $, where the source is the space of $L'$-valued continuous functions on $\Gamma^{?}$. This can be applied to the case of $\Y_{F}$ (with $\Gamma_{F}$), and to the case of $\Y'$ with the variation that $\Y'$-families correspond to  bounded functionals  on the space $\mathscr{C}(\Gamma,\omega, L')$ of functions $f$ on $\Gamma$ satisfying $f(zt)=\omega^{-1}(z)f(t)$ for all $z \in \A^{\infty,\times}$.

Let $\B$ be a (coherent or incoherent) totally definite quaternion algebra over  $\A=\A_{F}$ and let $E$ be a totally imaginary quadratic extension of $F$ with an embedding $E_{\A}\into \B$ which we fix.   Let $ {\bf  V}$  be the orthogonal space $\B$ with reduced norm $q$.  We have an orthogonal decomposition 
$${\bf V}={\bf V}_{1}\oplus {\bf V}_{2}$$
where 
$${\bf V}_{1}=E_{\A}, \quad {\bf V}_{2}=E_{\A}\mathfrak{j},\qquad \mathfrak{j}\notin E_{\A},\quad \mathfrak{j}^{2}\in \A^{\times}.$$
The restriction of $q$ to ${\bf V}_{1}$ is the adelisation of the norm of $E/F$. 

We have an embedding (cf. \cite[p. 36]{yzz}) 
$$\A^{\times}\bks \B^{\times}\times \B^{\times}\into {\bf GO}({\bf V})$$
where $\B^{\times}\times \B^{\times}$ acts on ${\bf V}$ by $(h_{1}, h_{2})x=h_{1} x h_{2}^{-1}$. 

\medskip

 Let $\phi^{p\infty}\in \calS({\bf V}^{p\infty}\times \A^{p\infty, \times})$ be a Schwartz function  and let  $U^{p}\subset \B^{\infty\times}$ be a compact open subgroup fixing $\phi^{p\infty}$. For $\phi_{1}\in \calS({\bf V}_{1}\times\A^{\times})$ a Schwartz function such that $\phi_{1, \infty}$ is standard, let $\theta(u, \phi_{1})$ be the twisted modular form
$$\theta(g, u, \phi_{1}):= \sum_{x_{1}\in E}r(g)\phi_{1}(x_{1}, u).$$
We define the modular form
\begin{align}\label{I_{F}}
I_{F,r}(\phi_{1}\otimes\phi_{2}, \chi_{F})=
{c_{U^{p}}\over |D_{F}|^{1/2}} \cdot
{L^{(p)}(1, \eta\chi_{F}^{})\over L^{(p)}(1, \eta)}  \sum_{u\in\mu_{U^{p}}^{2}\bks F^{\times}} \theta(u, \phi_{1}) E_{r}(u, \phi_{2} ,\chi_{F}^{})
\end{align}
for sufficiently large $r=(r_{v})_{v\vert p}$,
and the $\Y_{F}$-family of $q$-expansions of weight $2$ modular forms
\begin{align}\label{calI_{F}}
\calI_{F}(\phi_{1}^{\infty}\otimes \phi_{2}^{p\infty}; \chi_{F})=c_{U^{p}}\sum_{u\in\mu_{U^{p}}^{2}\bks F^{\times}} {}^{\qqq}\theta(u, \phi_{1})\mathscr{E}(u,\phi_{2}^{p\infty};\chi_{F}),
\end{align}
where letting $\mu_{U^{p}}=F^{\times}\cap U^{p}\OO_{\B,p}^{\times}$ we set
\begin{align}\label{c_{U}}
c_{U^{p}}:=  {2^{[F:\Q]-1}h_{F}\over [\OO_{F}^{\times}:\mu_{U^{p}}^{2}] }   
\end{align}
and $\phi(x_{1},x_{2}, u)=\phi_{1}(x_{1},u)\phi_{2}(x_{2},u)$
with $\phi_{i}=\phi_{i}^{p\infty}\phi_{i,p\infty}$ for $\phi_{i,v}$ the standard Schwartz function if $v\vert \infty$ or $i=2$ and $v|p$. The definition is independent of the choice of $U^{p}$ (cf.  \cite[(5.1.3)]{yzz}).

The   action of the subgroup $T(\A)\times T(\A)\subset \B^{\times}\times\B^{\times}$ on $\calS({\bf V}\times \A^{\times})=\calS({\bf V}_{1}\times \A^{\times})\otimes \calS({\bf V}_{2}\times \A^{\times})$ preserves this tensor product decomposition, thus it can be written as $r=r_{1}\otimes r_{2}$ for the actions $r_{1}$, $r_{2}$ on each of the two factors. We obtain an action of $T(\A^{\infty})\times T(\A^{\infty})$ on the forms $I_{F,r}$ and the families $\calI_{F}$ with orbits
\begin{align*}
I_{F,r}((t_{1}, t_{2}),\phi_{1}\otimes\phi_{2}, \chi_{F})
&:=
{c_{U^{p}} \over |D_{F}|^{1/2}}
{L^{(p)}(1, \eta\chi_{F})\over L^{(p)}(1, \eta)}  \sum_{u\in\mu_{U^{p}}^{2}\bks F^{\times}} \theta(u, r_{1}(t_{1}, t_{2}),\phi_{1}) E_{r}(u, \phi_{2} ,\chi_{F}^{\iota})
\\
\calI_{F}((t_{1}, t_{2}),\phi_{1}^{\infty}\otimes\phi_{2}^{p\infty}; \chi_{F})
&:=
c_{U^{p}}\sum_{u\in\mu_{U^{p}}^{2}\bks F^{\times}} {}^{\qqq}\theta(q(t)u,r_{1}(t_{1},t_{2}) \phi_{1})\mathscr{E}(q(t)u,\phi_{2}^{p\infty};\chi_{F}).
\end{align*}

It is a bounded action in the sense that  the orbit $\{\calI_{F}((t_{1}, t_{2}), \phi_{1}^{\infty}\otimes\phi_{2}^{p\infty}) \, | \, t_{1}, t_{2}\in T(\A^{\infty})\}$ is a bounded subset of the space of $\Y_{F}$-families of $q$-expansions, as both $\mathscr{E} $  and ${}^{\qqq}\theta $ are bounded in terms of $\max |\phi^{p\infty}|$.

Define,  for the fixed finite order character $\omega\colon F^{\times}\bks \A^{\times}\to M^{\times}$,
\begin{align}
\calI_{F, \omega^{-1}}((t_{1}, t_{2}),\phi_{1}^{\infty}\otimes\phi_{2}^{p\infty}; \chi_{F})
:=
\dashint_{\A^{\times}} \omega^{-1}(z)\chi_{F}(z) \calI_{F}((zt_{1},t_{2}),\phi_{1}^{\infty}\otimes\phi_{2}^{p\infty}; \chi_{F}) \, dz,
\end{align}
a bounded $\Y_{F}$-family of $q$-expansions of forms of central character $\omega^{-1}$, corresponding to a bounded functional on $\mathscr{C}(\Gamma_{F}, L)$ valued in ${\bf M}(K^{p},\omega^{-1}, L)$ for a suitable $K^{p}$.

We  further obtain a bounded functional $\calI$ on $\mathscr{C}(\Gamma, \omega, L)$,  valued in  ${\bf M}(K^{p},\omega^{-1}, L)$, which is defined on the set (generating a dense subalgebra) of finite order characters $\chi' \in \mathscr{C}(\Gamma, \omega, L)$ by
$$\calI(\phi^{p\infty}; \chi'):=\int_{[T]} \chi'(t) \calI_{F, \omega^{-1}}((t,1),\phi_{1}^{p}\phi_{1,p}\otimes \phi_{2}^{p}; \omega\cdot\chi'|_{\A^{\times}})\, {d^{\circ}t    }     $$
if   $\phi^{p\infty}=\phi_{1}^{p\infty}\otimes\phi_{2}^{p\infty}$. Here  $\phi_{1}^{\infty }=\phi_{1}^{p\infty}\phi_{1,p}$ with 
\begin{align}\label{phi1}
\phi_{1,v}(x_{1},u)=\delta_{1, U_{T,v}}(x_{1})\one_{\OO_{F,v}^{\times}}(u)\
\end{align}
 if $v\vert p$, where  $U_{T,p}\subset \OO_{E, p}^{\times}$ is a compact open subgroup small enough that  $\chi'_{p}|_{U_{T,p}}=1$ and 
$$\delta_{1, U_{T,v}} (x_{1})={\vol(\OO_{E,v},dx)\over \vol(U_{T,v}, dx) }   \one_{U_{v}\cap \OO_{E}}(x_{1})  .$$ 
(The notation is meant to suggest a Dirac delta at $1$ in the variable $x_{1}$, to which this is the finest $U_{v}\times U_{v}$-invariant approximation. Both volumes are taken with respect to a Haar measure on $E_{v}$.)

By construction, the induced rigid analytic function on $\Y'=\Y_{\omega}'$, still denoted by $\calI$, satisfies the following.

\begin{prop} There is a bounded $\Y'$-family of $q$-expansions of modular forms  $\calI(\phi^{p\infty})$ such that for each $\chi'\in \Y'^{\rm l.c.}(\C)$
we have
$$\calI(\phi^{p\infty}; \chi')=
|D_{E}|^{1/2}|D_{F}|\, 
{}^{\qqq} I_{r}(\phi, \chi'{}^{\iota}),$$
where
$$I_{r}(\phi, \chi'):=\int^{*}_{[T]} \chi'(t)I_{F,r}((t,1),\phi, \chi_{F})\, dt$$
with
  $I_{F,r}(\phi)$ as in \eqref{I_{F}}, with $\phi_{p\infty}$  chosen as above.
\end{prop}

\subsection{Waldspurger's Rankin--Selberg integral}\label{wrs}
Let $\chi'\in \Y'^{\rm l.c.}_{M(\alpha)}(\C)$  be a character, $\iota\colon M(\alpha)\into \C$ be the induced embedding.
 Let $\psi \colon \A/F\to \C^{\times}$ be an additive character, and let $r=r_{\psi}$ be the associated Weil representation.
\begin{prop}\label{prop2.6} Let  $\vphi\in \sigma^{\iota}$ be a form with factorisable Whittaker function, and let $\phi=\otimes_{v}\phi_{v}\in\bcalS({\bf V}\times \A^{\times})$. For sufficiently large $r=(r_{v})_{v\vert p}$, we have
\begin{equation}\label{yzz2.6}
\prod_{v\vert p}{\iota\alpha(\vpi_{v})^{-r_{v}} }\cdot  (\vphi, w_{r}^{-1} I_{r}(\phi,\chi'{}))= \prod_{v} R_{r,v}^{\circ}(W_{v}, \phi_{v}, \chi_{v}',\psi_{v} )\end{equation}
where 
$$
R_{r,v}^{\circ}(W_{v}, \phi_{v}, \chi_{v}', \psi_{v})={\iota\alpha_{v}(\vpi_{v})^{-r_{v}} } {L^{(p)}(1, \eta_{v} \chi_{F,v})\over L^{(p)}(1, \eta_{v})} R_{r,v} $$
with
$$R_{r,v}= \int_{Z(F_{v})N(F_{v})\bks \GL_{2}(F_{v})} W_{-1,v}(g) \delta_{\chi_{F,r,v}}(g) \int_{T(F_{v})} \chi_{v}'(t)r(gw_{r}^{-1})\Phi_{v}(t^{-1}, q(t))\,   dt\, dg.
$$
Here $\Phi_{v}=\phi_{v}$ if $v$ is non-archimedean and $\Phi_{v}$ is a preimage of $\phi_{v}$ under \eqref{schwartz quotient} if $v$ is archimedean,  $W_{-1,v}$ is the local Whittaker function of $\vphi$ for the character $\baar{\psi}_{v}$, and we convene that $r_{v}=0$, $w_{r,v}=1$, $\iota\alpha_{v}(\vpi_{v})^{-r_{v}}=1$   if $v\nmid p$. 
\end{prop}
 Note that the integral $R_{r,v}$ does not depend on  $r\geq \underline{1}$ unless $v\vert p$ and it does not depend on $\chi'$ if $v\vert \infty$; we will accordingly simplify the notation  in these cases.
\begin{proof} This is shown similarly to \cite[Proposition 2.5]{yzz}; see \cite[(5.1.3)]{yzz}  for the equality between  the kernel functions denoted there by $I(s, \chi, \phi)$ (similar to our $c_{U^{p}}^{-1}I_{r}(\phi, \chi')$) and $I(s,\chi,\Phi)$ (which intervenes in the analogue in \emph{loc. cit} of the left-hand side of \eqref{yzz2.6}).
\end{proof}

We will sometimes lighten a bit the notation for $R^{\circ}_{v}$ by omitting $\psi_{v}$ from it.
\begin{lemm}\label{3.6.2}
When everything is unramified, we have
$$R_{v}^{\circ}(W_{v},\phi_{v}, \chi_{v}') = { L(1/2, \sigma_{E,v}\otimes\chi'_{v})\over \zeta_{F, v}(2) L(1, \eta_{v})}.$$
\end{lemm}
\begin{proof}
With a slightly different setup,\footnote{Notably, the local measures in \cite{wald} are normalised by $\vol(\GL_2(\OO_{F,v}))=1$  for almost all finite  place $v$, whereas we have $\vol(\GL_2(\OO_{F,v}))=\zeta_{F, v}(2)^{-1}|d|_{v}^{2}$ (cf. \cite[p. 23]{yzz}; the second displayed formula of \cite[p. 42]{yzz} neglects this discrepancy).}
 Waldspurger \cite[Lemme 2, Lemme 3]{wald} showed that
$$R_{v}( W_{v},\phi_{v}, \chi_{v}') = { L(1/2, \sigma_{E,v}\otimes \chi_{v}' )\over \zeta_{F, v}(2) L(1, \eta_{v}\chi_{F,v})}$$
when $\chi_{F,v}=|\cdot |^{s}$, but his calculation goes through for any unramified character~$\chi_{F,v}$.
\end{proof}

Define 
\begin{align}\label{Rnat}
R_{r,v}^{\natural}(W_{v},\phi_{v} \chi'_{v}, \psi_{v}): =
{ |d_{v}|^{-2} |D_{v}|^{-1/2}}
 {  \zeta_{F, v}(2)L(1, \eta_{v})\over L(1/2, \sigma_{E,v}\otimes \chi'_{v})}
R_{r,v}^{\circ}( W_{v},\phi_{v} \chi_{v}', \psi_{v}).
\end{align}

Then the previous lemma combined with Proposition \ref{prop2.6} gives:
\begin{prop}\label{RSp} We have
\begin{align*}
\iota\alpha^{-r} ( \vphi,w_{r}^{-1} {I}(\phi, \chi'))& =
{|D_{F}|^{-1} |D_{E}|^{-1/2}}
 {L^{(\infty)}(1/2, \sigma_{E}\otimes \chi')\over \zeta_{F}^{(\infty)}(2) L^{(\infty)}(1,\eta)}
 \prod_{v\nmid \infty} R_{r,v}^{\natural}( W_{v}, \phi_{v},\chi_{v}')
 \prod_{v|\infty} R_{ v}^{\circ} (W_{v}, \Phi_{v},\chi_{v}'),
\end{align*}
where all but finitely many of the factors in the infinite product are equal to~$1$.
\end{prop}

\subsubsection{Archimedean zeta integral}
We compute the local integral $R_{v}$ when $v\vert \infty$. 
\begin{lemm}\label{arch-int}
If $v\vert \infty$, $\phi_{v}$ is standard, and $W_{-1,v}$ is the standard antiholomorphic Whittaker funciton of weight $2$ of  \eqref{antiholwhitt},  then
$$R^{\circ}_{v}(W_{v}, \phi_{v}, \chi_{v}')=R_{v}(W_{v}, \phi_{v}, \chi_{v}')=1/2.$$
\end{lemm}
\begin{proof} 
By the Iwasawa decomposition we can uniquely write any $g\in\GL_{2}(\R)$ as
$$g=\twomat 1 x{}1\twomat z{}{}z  \twomat y{}{}1 \twomat {\cos \theta} {\sin \theta}{-\sin\theta}{\cos\theta}$$
with $x\in\R$, $z\in \R^{\times}$, $y\in \R^{\times}$, $\theta\in [0, 2\pi)$; the local Tamagawa measure is then $dg=dxd^{\times}z{d^{\times}y\over{|y|}}{d\theta\over 2}$.  The integral in $Z(\R)\subset T(\R)$ realizes the map $\Phi\to \phi$; and it is easy to verify that $r(g)\phi(1,1)$ is the standard holomorphic Whittaker functon of weight $2$.

We then have, dropping subscripts $v$:
\begin{align*}
R_{v}(\vphi, \phi)&=\int_{T(\R)/Z(\R)}\int_{0}^{2\pi} \int_{\R^{\times}} (|y|e^{-2\pi y})^{2} {d^{\times}y\over |y|}\, {d\theta\over 2}\, dt\\
&=2 \cdot ( 4\pi)^{-1} \pi=1/2,
\end{align*}
where $(4\pi)^{-1}$ comes from a change of variable, 
$2=\vol(T(\R)/Z(\R))$, and $\pi$ comes from the integration in $d\theta$.
\end{proof}

\subsection{Interpolation of  local zeta integrals}
When $v\nmid p$, the normalised local zeta integrals admit an interpolation as well. Recall from \S\ref{intro-plF} that $\Psi_{v}$ denotes the scheme of all local additive characters of level $0$.

\begin{prop}\label{interp rnat} Let $v\nmid p$ be a finite place, 
and let $\mathscr{K}(\sigma_{v}, \psi_{{\rm univ},v})$ be  the universal Kirillov model of $\sigma_{v}$. Then for any $\phi_{v}\in \mathscr{S}(V_{v}\times F_{v}^{\times})$, $W_{v}\in \mathscr{K}(\sigma_{v}, \psi_{{\rm univ},{v}})$ there exists a  function 
$$\calR^{\natural}(W_{v},\phi_{v})\in L(1, \eta_{v}\chi_{F,v}) \OO_{\Y'_{v}\times\Psi_{v}}(\Y'_{v}, \omega_{v}\chi_{F, v, {\rm univ}}^{-1})$$
 such that for all $\chi'_{v}\in \Y'_{v}(\C)$, $\psi_{v}\in \Psi_{v}(\C)$, we have 
$$\calR_{v}^{\natural}(W_{v}, \phi_{v};\chi'_{v}, \psi_{v})=R_{v}^{\natural}(W^{\iota}_{v}, \phi_{v}(\chi'_{v}), \chi'_{v}, \psi_{v}),$$
where $\phi_{v}(\chi_{v}')=\phi_{1,v}\phi_{2,v}(\chi_{F,v})$ with $\phi_{2,v}(\chi_{F,v})$ is as in \eqref{phi2chi}.
\end{prop}
In the statement, we consider $L(1,\eta_{v}\chi_{F,v})^{-1}$ as an element of $\OO(\Y'_{v})$ (coming by pullback from $\Y_{F,v}$). Note that it equals the nonzero constant $L(1, \eta_{v})^{-1}$ along $\Y_{v}\subset \Y_{v}'$.

\begin{proof}
Note that the assertion on the subsheaf of $\OO_{\Y_{v}\times\Psi_{v}}$ of which $\calR^{\natural}_{v}$ is a section simply encodes the dependence of $R^{\natural}_{v}$ on the additive character, which is easy to ascertain by a change of variables. 
By the definitions, it suffices to show that 
\begin{align}\label{that extends}
L(1/2, \sigma_{E,v}\otimes \chi_{v}')^{-1} R_{v}(W_{v}, \phi_{v}, \chi_{v}')
\end{align}
 extends to a regular function on $\Y_{v}'$. We will more precisely show that $L(1/2, \sigma_{E,v}\otimes\chi_{v}')^{-1}$ is a product of various factors all of which extend to regular functions on $\Y_{v}'$, and that the product of some of those factors and $R_{v}(W_{v}, \phi_{v}, \chi_{v}')$ also extends to a regular function on $\Y_{v}'$.
Concretely, if   $A\subset E_{v}^{\times}/(E_{v}^{\times}\cap V^{p})$ is any finite set, then the evaluations $\chi'_{v}\mapsto (\chi_{v}'(a))_{a\in A}$ define a morphism ${\rm ev}_{A}\colon\Y_{v}'\to{\bf G}_{m, M}^{A}$,\footnote{Moreover if $A$ is sufficiently large the morphism ${\rm ev}_{A}$ is a closed embedding.} so that finite sums of evaluations of characters are regular functions on $\Y_{v}'$ obtained by pullback along ${\rm ev}_{A}$.

\paragraph{Interpolation of   $R_{v}$} Within the expression for $R_{v}$, we can use the Iwasawa decomposition and note that integration over $K=\GL_{2}(\OO_{F,v})$ yields a finite sum of integrals of the form (dropping subscripts $v$)\footnote{See Proposition \ref{interpolation factor}, Lemma \ref{basic local integral} for some more detailed calculations similar to the ones of the present proof.}
\begin{align*}
\int_{F^{\times}} f'(y)\int_{T(F)}\chi'(t)\phi'(yt^{-1}, y^{-1}q(t))\, dt \,{dy}
\end{align*}
for some Schwartz functions $\phi'$ and elements $f'$ of the Kirillov model of $\sigma$  -- namely, the translates of $W_{-1}$ and of $\phi_{v}$ by the action of $K$. (More precisely, taking into account the dependence on $\chi'$ of $\phi$, also products of the above integrals and of $L(1, \eta\chi_{F,v})^{-1}$ can occur; the factor $L(1, \eta\chi_{F,v})^{-1}$ clearly interpolates to a regular function on $\Y_{v}'$.)

 It is easy to see that the integral reduces to a finite sum if either $W$ is compactly supported or $\phi'_{1}(\cdot, u)$ is supported away from $0\in E$. It thus suffices to study the case where  $\phi'_{1}(x_{1}, u)=\one_{\OO_{E}}(x_{1})\phi'_{F}(u)$, and $f'$ belongs to the basis of the quotient space $\baar{\mathscr{K}}$ introduced in the proof of Lemma \ref{lem univ kir}. 
Moreover up to simple manipulations we may assume that $\phi_{F}(u)$ is  is close to a delta function supported at $u=1$. We distinguish three different cases.
\paragraph{$\sigma_{v}$ is supercuspidal} In this case $\baar{\mathscr{K}} =0$ and there is nothing to prove.
\paragraph{$\sigma_{v}$ is a special representation  ${\rm St}(\mu|\cdot|^{-1})$} In this case $\baar{\mathscr{K}}$ is spanned by $f_{\mu}=\mu\cdot\one_{\OO_{F}-\{0\}}$. We find that the integral is essentially\footnote{Here we use this adverb with the precise meaning: up to addition of and multiplication by finite combination of evaluations of $\chi'$.}  $0$ if there is a place $w$ of $E$ above $v$ such that, for $\chi'_{w}:=\chi'_{v}|_{E_{w}^{\times}}$, the character $\chi'_{w}\cdot \mu\circ q$ of $E_{w}^{\times}$ is ramified; and it essentially equals
\begin{align}\label{if unr}
\prod_{w\vert v} (1-\chi'_{w}(\vpi_{w})\mu(q(\vpi_{w})) q_{E,w}^{-1})^{-1}
\end{align}
otherwise.\footnote{In the last expression, $q$ is the norm of $E_{w}/F_{v}$, whereas $q_{E,w}$ is the cardinality of the residue field of $E_{w}$. We apologise for the near-clash of notation.} In the latter case, $L(1/2, \sigma_{E,v}\otimes \chi'_{w})$ is also equal to \eqref{if unr}. We conclude that \eqref{that extends} extends to a regular function on $\Y_{v}'$.
\paragraph{$\sigma_{v}$ is an irreducible principal series ${\rm Ind}(\mu, \mu'|\cdot|^{-1})$}\footnote{Here ${\rm Ind}$ is plain  (un-normalised) induction.} The space $\baar{\mathscr{K}}$ has dimension $2$ and $f_{\mu}$ as above provides a nonzero element. Again the corresponding integral yields either $0$ or \eqref{if unr}, the latter happening precisely when \eqref{if unr} is a factor of $L(1/2, \sigma_{E}\otimes \chi)$. If $\mu'\neq \mu$, then a second basis element is $f_{\mu'}$, for which the same discussion applies. If $\mu'=\mu$, then a second basis element is $f_{\mu}'(y):=v(y)\mu(y)\one_{\OO_{F}-\{0\}}(y)$. The integral is essentially $0$ if some $\chi'_{w}\cdot \mu\circ q$ is ramified, and 
\begin{align}\label{if unr2}
\prod_{w\vert v} (1-\chi'_{w}(\vpi_{w})\mu(q(\vpi_{w})) q_{E,w}^{-1})^{-2}
\end{align}otherwise. In the latter case, $ L(1/2, \sigma_{E,v}\otimes \chi'_{w})$ equals \eqref{if unr2} as well.

\paragraph{Interpolation  of $L(1/2, \sigma_{E,v}\otimes\chi_{v}')^{-1}$}  Depending only on $\sigma_{v}$, as recalled above for each place $w\vert v$ of $E$ there exist at most two characters $\nu_{w, i_{w}}$ of $E_{w}^{\times}$ such that  for all $\chi_{v}'\in \Y'_{v}(\C)$, we can write   $L(1/2, \sigma_{E,v}\otimes\chi_{v}')^{-1}=\prod_{w, i_{w}}'(1-\nu_{w}\chi'_{w}(\vpi_{w}))$, where the product $\prod'$ extends over those pairs $(w,i_{w}) $ such that $\nu_{w, i_{w}}\chi'_{w}$ is unramified. We can replace the partial product by a genuine product and each  of the factors by 
$$1-\left(\sum^{{}'}_{x\in \OO_{E,w}^{\times}/(V^{p}\cap E_{w}^{\times})}\nu_{w, i_{w}}(x)\chi_{w}'(x)\right)\cdot \nu_{w, i_{w}}\chi'_{w}(\vpi_{w}) $$
where $\sum'$ denotes average. This expression is the value at $\chi'_{v}$ of an element of $\OO(\Y_{v}')$, as desired.

\end{proof}

\subsection{Definition and interpolation property}

Let let $\mathscr{M}_{\Y'-\Y}$
 be the multiplicative part of $\OO(\Y')^{\rm f}$ consisting of functions whose restriction to $\Y$ is invertible. (Recall that $\OO(\Y')^{\rm f}\subset \OO(\Y' )^{\rm b}$ is the image of $\otimes_{v\nmid p\infty}\OO(\Y_{v}')$.) 
\begin{theo}\label{theo A text} There exists a unique function
$$L_{p,\alpha}(\sigma_{E})\ \in  \OO_{\Y'\times \Psi_{p}}(\Y',\omega_{p}\chi_{F, {\rm univ}, p}^{-1})^{\rm b} [\mathscr{M}_{\Y'-\Y}^{-1}]$$
which is algebraic on  $\Y_{M(\alpha)}'^{\rm l.c.}\times \Psi_{p}$ and satisfies
$$ L_{p,\alpha}(\sigma_{E})(\chi', \psi_{p}) =
 {\pi^{2[F:\Q]} |D_{F}|^{1/2}   L^{(\infty)}(1/2, \sigma_{E}^{\iota}, \chi'{}^{\iota})\over
  2 L^{(\infty)}(1, \eta) L^{(\infty)}(1, \sigma^{\iota}, \ad)}
  \prod_{v\vert p}  Z_{v}^{\circ}(\chi_{v}', \psi_{v})$$
for every $\chi'\in \Y_{ M(\alpha)}'(\C)$ inducing  an  embedding $\iota\colon M(\alpha)\into \C$. Here $Z_{v}^{\circ}$ is as in Theorem \ref{A}.

Let $\Y'^{\circ }\subset \Y'$ be any connected component, $\Y^{\circ }:=\Y\cap \Y'^{\circ } $ the corresponding connected component of $\Y$, and let $\B$ be the quaternion algebra over $\A^{\infty}$ determined by \eqref{local cond} for any (equivalently, all) points $\chi\in \Y^{\circ}$. 
For any $\vphi^{p\infty}\in \sigma^{p\infty}$ and $\phi^{p\infty}\in \mathscr{S}({\bf V}^{p\infty}\times{\A^{p\infty,\times}})$, we have 
\begin{align}\label{def plf}
\lf(\calI(\phi^{p\infty}))|_{\Y'^{\circ}}=
L_{p, \alpha}(\sigma_{E})|_{\Y'^{\circ}\times\Psi_{p}}
 \prod_{v\nmid p \infty} \mathscr{R}_{v}^{\natural}(W_{v}, \phi_{v})|_{\Y'^{\circ}\times \Psi_{p}}
\end{align}
in $\OO_{\Y'}(\Y'^{\circ})^{\rm b}$, where both $\mathscr{I}$ and $\calR^{\natural}$ are constructed using ${\bf V}$.  On the right-hand side, the product  $\prod_{v\nmid p}\mathscr{R}_{v}^{\natural}$ makes sense over $\Y'^{\circ}\times \Psi_{p}$ by the decomposition $\sigma\cong \mathscr{K}(\sigma^{p}, \Psi_{p})\otimes \mathscr{K}(\sigma_{p}, \Psi_{p})$ induced by the Whittaker functional fixed in the definition of $\lf$.\footnote{The $p$-adic $L$-function $L_{p, \alpha}(\sigma_{E})$ does not depend on this choice. Here, letting $\Psi_{v}'$ denote the space of all nontrivial additive characters of $F_{v}^{\times}$, the space  $\mathscr{K}(\sigma^{p}, \Psi_{p}) $ is the restriction of $  \prod_{v\nmid p}'\mathscr{K}(\sigma_{v}, \Psi_{v}')$
via an embedding $\Psi_{p}\into \prod_{v\nmid p}\Psi_{v}'$ obtained as follows: fix any nontrivial  character $\psi_{0} $ of $\A/F$ into $\bmu_{\Q}$, then $\psi_{p}\mapsto  (\psi_{0}/\psi_{p}|_{F_{v}})_{v}$.  }
\end{theo}

\begin{proof}  The definition can be given locally by taking quotients in \eqref{def plf} for any given $(W^{p\infty},\phi^{p\infty})$.  Note that on the right-hand side of \eqref{def plf}, the product is finite since by   Lemma \ref{3.6.2} we have $\mathscr{R}_{v}^{\natural}(W_{v}, \phi_{v})=1$ identically on $\Y'$ if all the data are unramified. The analytic properties of $L_{p, \alpha}(\sigma_{E})$  are then a consequence of the following  claim. Let  $S$ be any finite set  of places $v\nmid  p\infty$, containing all the ones such  that either  $\sigma_{v}$ is ramified or  the subgroup $V_{v}\subset \OO_{E,v}^{\times}$ fixed in the Introduction is not maximal. Let $\Y'^{\circ}$ be a connected component and ${\bf B}$ be the associated quaternion algebra. Then   \emph{ for each $v\in S$, there exists a finite set of pairs $(W_{v}, \phi_{v})$ such that the locus of common vanishing of the corresponding functions $ \mathscr{R}_{v}^{\natural}(W_{v}, \phi_{v})|_{\Y^{\circ}}$ is empty}. 

We prove the claim.  Let $\chi_{v}\in \Y_{v}^{\circ}$ be any closed point, where $v\in S$ and $\Y_{v}^{\circ}\subset \Y_{v}$ is   the union of connected    components corresponding to $\Y^{\circ}$. By Lemma \ref{Qtheta} below, we have $\mathscr{R}^{\natural}(W_{v}, \phi_{v}, \chi_{v}) =   Q_{v}(\theta_{\psi,v}(W_{v},\phi_{v}), \chi_{v})$, where $\theta_{\psi,v}$ is a Shimizu lift sending $\sigma_{v}\times \mathscr{S}({\bf V}_{v}\times F^{\times})$ onto $\pi_{v}\otimes \pi_{v}^{\vee}$, with $\pi_{v}$  the Jacquet--Langlands transfer of $\sigma_{v}$ to $\B_{v}^{\times}$. By construction of $\B_{v}$ and the result mentioned in  \S\ref{hm1}, the functional $Q_{v}(\cdot, \chi_{v})$ is non-vanishing. Therefore given $\chi_{v}\in \Y_{v}^{\circ}$, we can find $(W_{v}, \phi_{v})$ such that $\mathscr{R}_{v}^{\natural}(W_{v}, \phi_{v}; \chi_{v})\neq 0$. 

Consider the set of all functions $\mathscr{R}_{v}^{\natural}(W_{v}, \phi_{v})|_{\Y_{v}^{\circ}}$ for varying $(W_{v}, \phi_{v})$. As the locus of their common vanishing is empty, it follows by the Nullstellensatz that finitely many of them generate the unit ideal of $\OO(\Y_{v}^{\circ})$.\footnote{Recall that $\Y_{v}$ is an affine scheme of finite type over $M$ (more precisely it is a closed subscheme of a split torus).} This completes the proof of the claim.

 We now move to the interpolation property. The algebraicity on $\Y_{M(\alpha)}'^{\, \rm l.c.}$ is clear from the definition just given.
 By $\iota$, which we will omit from the notation below, we can  identify $\vphi$ with an antiholomorphic automorphic form $\vphi^{\iota}$.  By the definitions and Proposition \ref{RSp}, we have
\begin{equation*}\begin{split}
L_{p,\alpha}(\sigma_{E})(\chi', \psi_{p})&=
{\lf(\mathscr{I}(\phi^{p},\chi' ))\over  \prod_{v\nmid p\infty}\mathscr{R}_{v}^{\natural}(W_{v},\phi_{v};\chi'_{v}, \psi_{v})}\\
&={|D_{F}|^{1/2}  \zeta_{F}(2)\over 2 L(1, \sigma, \ad) }  
	\cdot { | D_{E}|^{1/2} |D_{F}| \alpha^{-r}(\vphi,  w_{r}^{-1}I(\phi, \chi')) \over \prod_{v\nmid p\infty}  {R}^{\natural}_{v}(W_{v},\phi_{v}, \chi', \psi_{v})}\\
&=  { |D_{F}|^{1/2}  \zeta_{F}(2) \over 2L(1, \sigma, \ad)}
 	\cdot {L^{(\infty)}(1/2, \sigma_{E}, \chi')\over \zeta_{F}^{(\infty)}(2) L^{(\infty)}(1, \eta)}   
	  \prod_{v\vert \infty}{R}^{\circ}_{v}(\phi_{v}, W_{v}, \chi'_{v}, \psi_{v}) 
	    \prod_{v\vert p}{R}^{\natural}_{r,v}(\phi_{v}, W_{v}, \chi'_{v}, \psi_{v})\\\
&=  {\zeta_{F, \infty}(2)\over 2^{[F:\Q]} L(1, \sigma_{\infty}, \ad ) }
	\cdot { |D_{F}|^{1/2}  L^{(\infty)}(1/2, \sigma_{E}, \chi')\over 2 L^{(\infty)}(1, \eta)   L^{(\infty)}(1, \sigma, \ad)} \prod_{v\vert p}{R}^{\natural}_{r,v}(W_{v}, \phi_{v}, \chi'_{v}, \psi_{v}).
\end{split}\end{equation*}
Here $\psi^{p}$ is any additive character such that $\psi=\psi^{p}\psi_{p}\psi_{\infty}$ vanishes on $F$.
For $v\vert \infty$, we have $\zeta_{F,v}(2)/L(1, \sigma_{v}, \ad)=\pi^{-1}/(\pi^{-3}/2)=2\pi^{2}$, so that the first fraction in the last line equals $\pi^{2[F:\Q]}$. The result follows.

The proof is completed by the  identification  $R_{r,v}^{\natural}=Z^{\circ}_{v}$ for $v\vert p$ carried out in Proposition \ref{interpolation factor}.
\end{proof}

\section{$p$-adic heights}

We recall the definition and properties of $p$-adic heights and prove two integrality criteria for them. The material of \S\S\ref{hc}-\ref{intH} will not be used until \S\S\ref{dec gk}-\ref{lsp}.

For background in $p$-adic Hodge theory see the summary in \cite[\S 1]{nekheights} and references therein. The notation we use is completely standard; it coincides with that of  \emph{loc. cit.} except that we shall prefer to write ${\bf D}_{\rm dR}$ instead of  $DR$ for the functor of de Rham periods.

\subsection{Local and global height pairings}\label{lgh}
Let $F$ be a number field and $\calG_{F}:=\Gal(\baar{F}/F)$.  Let $L$ be a finite extension of $\Q_{p}$, and let $V$ be a finite-dimensional $L$-vector space with a continuous action of $\calG_{F}$. For each place $v$ of $F$, we denote by $V_{v}$ the space $V$ considered as a representation of $\calG_{F,v}:= \Gal(\baar{F}_{v}/F_{v})$. 

  Recall that  the Bloch--Kato Selmer group of $V$
  $$H^{1}_{f}(F, V)$$
  is the subset of $H^{1}(F, V)={\rm Ext}^{1}_{F}(L,V)$ (extensions in the category of continuous $\calG_{F}$-representations over $L$) consisting of the classes of those extensions $0\to V\to E_{1}\to L\to 0$ which are unramified at all $v\nmid p$ and crystalline at all $v\vert p$ (that is such that $E_{1}$ is).    
  
   Suppose that: 
 \begin{itemize}
 \item
$V$ is unramified outside of a finite set of primes of $F$;
\item $V_{v}$  is de Rham, hence potentially semistable,  for all $v\vert p$; 
\item $H^{0}(F_{v}, V)=H^{0}(F_{v},V^{*}(1))=0$ for all $v\nmid p$; 
\item ${\bf D}_{\rm crys}(V_{v})^{\vphi=1} ={\bf D}_{\rm crys}(V_{v})^{\vphi=1} =0$  for all $v\vert p$ (where $\vphi$ denotes the crystalline Frobenius).
 \end{itemize}

  Under those conditions, Nekov{\'a}{\v{r}} \cite{nekheights} (to which we refer for more details; see also \cite{nek-selmer})
 constructed a bilinear  pairing on the Bloch--Kato Selmer groups
 \begin{align}\label{htpa}
 \langle\, , \,\rangle\colon H^{1}_{f}(F, V) \times H^{1}_{f}(F, V^{*}(1)) \to \Gamma_{F}\hat{\otimes} L
 \end{align}
   depending on choices of $L$-linear splittings 
of the Hodge filtration 
\begin{align}\label{hdg}
{\rm Fil}^{0}{\bf D}_{\rm dR}(V_{v})\subset {\bf D}_{\rm dR}(V_{v})
\end{align}
for the primes $v\vert p$. In fact in \cite{nekheights} it is assumed that $V_{v}$ is semistable; we will recall the definitions under this assumption, and at the same time see that they can be made compatible with extending the ground field (in particular, to reduce the potentially semistable case to the semistable case). Cf. also \cite{benois-hts}  for a very general treatment.

 Post-composing $\langle\, , \,\rangle$ with a continuous  homomorphism $\ell\colon \Gamma_{F}\to L'$, for some $L$-vector space $L'$, yields an $L'$-valued pairing $\langle\, , \,\rangle_{\ell}$  (the cases of interest to us are $L'=L$ with any $\ell$, or $L'=\Gamma_{F}\hat{\otimes}L $ with the tautological $\ell$). For such an $\ell$ we write $\ell_{v}:=\ell|_{F_{v}^{\times}}$ and we say that $\ell_{v}$ is \emph{unramified} if it is trivial on $\OO_{F,v}^{\times}$ (note that this is automatic if $v\nmid p$).
 
Let $x_{1}\in H^{1}_{f}(F, V) $, $x_{2}\in H^{1}_{f}(F, V^{*}(1)) $, and   view them  as  classes $x_{1}=e_{1}=[E_{1}]$,  $x_{2}=e_{2}=[E_{2}]$ of extensions of Galois representations
\begin{align*}
&e_{1}\colon \qquad 0\to V\to E_{1}\to L\to 0\\
&e_{2}\colon\qquad 0\to V^{*}(1)\to E_{2}\to L\to 0.
\end{align*}
For any $e_{1}$, $e_{2}$ as above,
  the set of  Galois representations  $E$ fitting into a commutative diagram
\begin{equation*}
\xymatrix{
   &   &0\ar[d]  &0\ar[d]  &  \\
0\ar[r] & L(1)\ar[r]\ar@{=}[d]&E_{2}^{*}(1)  \ar[d]\ar[r] &V\ar[r]\ar[d]& 0\\
0\ar[r] & L(1)\ar[r] &E\ar[d]\ar[r] &E_{1}\ar[r]\ar[d] &0\\
&&L\ar@{=}[r]\ar[d]&L\ar[d]&\\
&&0&0&
}
\end{equation*}
is an $H^{1}(F, L)$-torsor 
\cite[Proposition 4.4]{nekheights}; any such $E$  is called a \emph{mixed extension} of $e_{1}$, $e_{2}^{*}(1)$. Depending on the choice of (extensions $e_{1}$ and $e_{2}$ and) a mixed extension $E$, there is a decomposition 
\begin{align}\label{loc-gl-hts}
 \langle x_{1} , x_{2}\rangle_{\ell}=\sum_{v\in S_{F}} \langle x_{1,v}, x_{2,v}\rangle_{\ell_{v},E,v}
\end{align}
of the height pairing into a (convergent) sum of local symbols indexed by the \emph{non-archimedean} places of $F$. We recall the definition of the latter \cite[\S 7.4]{nekheights}. The representation $E$ can be shown to be automatically semistable at any $v\vert p$; for each $v$ it then yields a class $[E_{v}]\in H^{1}_{\rm *}(F_{v}, E_{2})$ with  $*=\emptyset$ if $v\nmid p$, $*={\rm st}$ if $v\vert p$. This group sits in a diagram of exact sequences
\begin{equation}\label{ex seq me}
\xymatrix{
0\ar[r] &H^{1}(F_{v}, L(1))\ar[r]  &H^{1}_{*}(F_{v}, E_{2})\ar[r] &H^{1}_{f}(F_{v}, V)\ar[r] &0\\
0\ar[r] &H^{1}_{f}(F_{v}, L(1))\ar[r]\ar[u]  &H^{1}_{f}(F_{v}, E_{2})\ar[r]\ar[u] &H^{1}_{f}(F_{v}, V)\ar[r]\ar@{=}[u] &0.
}
\end{equation}
If $v\vert p$, the chosen splitting of \eqref{hdg} uniquely determines  a splitting $s_{v}\colon  H^{1}_{*}(F_{v}, E_{2})\to H^{1}(F_{v}, L(1))$; 
 if $v\nmid p$, there is a canonical splitting  independent of choices, also denoted by $s_{v}$. In both cases, the local symbol is  
$$  \langle x_{1,v}, x_{2,v}\rangle_{\ell_{v}, E,v}:=-\ell_{v}(s_{v}([E_{v}])),$$
where we still denote by  $\ell_{v}$ the composition $ H^{1}(F_{v}, L(1))\cong F_{v}^{\times}\hat{\otimes}L\to \Gamma_{F}\hat{\otimes} L\to L'$.  When $v\vert p$, we say that $[E_{v}]$ is \emph{essentially crystalline} if $[E_{v}]\in H^{1}_{f}(F_{v}, E_{2})\subset H^{1}_{\rm st}(F_{v}, E_{2})$; equivalently, $s_{v}([E_{v}])\in H^{1}_{f}(F_{v}, L(1))$.

\subsubsection{Behaviour under field extensions} If $F'_{w}/F_{v}$ is a finite extension of local non-archimedean fields, the pairing
\begin{align}\label{F'}
\langle\, , \, \rangle_{\ell_{v}\circ N_{F_{w}/F_{v}}}\colon H^{1}_{f}(F'_{w},V_{w})\times H^{1}_{f}(F'_{w},V_{w}^{*}(1))\to L'
\end{align}
 defined using  the induced  Hodge splittings and the map $\ell_{{w}}:=\ell_{v}\circ N_{F_{w}/F_{v}}$
satisfies
 \begin{align}\label{local ht norm}
  \langle {\rm cores}^{F'_{w}}_{F_{v} } x_{1},   x_{2}\rangle_{\ell_{v}} = \langle  x_{1},  {\rm res}^{F'_{w}}_{F_{v}} x_{2}\rangle_{\ell_{v}\circ N_{F_{w}/F_{v}}} 
  \end{align}
   for all $x_{1}\in H^{1}_{f}(F_{v},V_{v})$, $x_{2}\in H^{1}_{f}(F'_{w},V|_{\calG_{F'_{w}}}^{*}(1))$. 

Back to the global situation, it follows that extending any  $\ell\colon \Gamma_{F}\to L'$  to the direct system $(\Gamma_{F'})_{F'/F \text{ finite}}$ by 
\begin{align}\label{ellw}
\ell_{w}=\ell|_{F_{w}'^{\times}}:={1\over [F':F]} \ell_{v}\circ N_{F_{w}'/F_{v}}
\end{align}
 we can extend $\langle\, , \, \rangle_{\ell}$ to a  pairing
 \begin{align}\label{F'gl}
 \langle\, , \, \rangle_{\ell} \colon \varinjlim_{F'}  H^{1}_{f}(F',V|_{\calG_{F'}})\times H^{1}_{f}(F',V|_{\calG_{F'}}^{*}(1))\to L',
 \end{align}
 where the limit is taken with respect to restriction maps. This allows to define the pairing in the potentially semistable case as well.

\subsubsection{Ordinariness} Let $v\vert p$ be a place of $F$.
\begin{defi}\label{pot-ord} 
 We say that a de Rham representation $V_{v}$ of $\calG_{F_{v}}$ satisfies the  \emph{Panchishkin condition} or that is \emph{potentially ordinary} 
 if there is a (necessarily unique) exact sequence of de Rham $\calG_{F_{v}}$-representations
$$0\to V_{v}^{+}\to V_{v}\to V_{v}^{-}\to 0$$
with ${\rm Fil}^{0}{\bf D}_{\rm dR}(V_{v}^{+})= {\bf D}_{\rm dR}(V_{v}^{-})/{\rm Fil}^{0}=0$. 
\end{defi}
If $V_{v}$ is potentially ordinary, there   is a canonical splitting of \eqref{hdg}  given by
\begin{align}\label{panc}
{\bf D}_{\rm dR}(V_{v})\to {\bf D}_{\rm dR}(V_{v}^{-})={\rm Fil}^{0}{\bf D}_{\rm dR}(V_{v}).
\end{align}

\subsubsection{Abelian varieties} If $A/F$ is an abelian variety with potentially semistable reduction at all $v\vert p$, then the rational Tate module $V=V_{p}A$  satisfies the required assumptions, and there is a canonical isomorphism $V^{*}(1)\cong V_{p}A^{\vee}$. Suppose that there is an embedding of a number field $M\into \End^{0}(A)$; its action on $V$ induces a decomposition $V=\oplus_{\frakp \vert p}V_{\frakp}$ indexed by the primes of $\OO_{M}$ above $p$. Given such  a prime $\frakp$, a finite extension $L $ of $M_{\frakp}$, and splittings of the Hodge filtration on ${\bf D}_{\rm dR}(V_{\frakp}|_{\calG_{F_{v}}})\otimes_{M_{\frakp}}L$ for $v\vert p$,
  we obtain from the compatible pairings \eqref{F'gl} a  height pairing
\begin{gather}\label{ht on A}
\langle\, , \, \rangle \colon A(\baar{F})\times A^{\vee}(\baar{F})\to\Gamma_{F}\hat{\otimes} L
\end{gather}
via the Kummer maps $\varkappa_{A, F', \frakp}\colon A(F')\to H^{1}_{f}(F', V)\to H^{1}_{f}(F', V_{\frakp})$ and  $\varkappa_{A^{\vee}, F', \frakp}\colon A(F')\to H^{1}_{f}(F', V^{*}_{\frakp}(1))$ for any $F\subset F'\subset \baar{F}$.

If $\frakp$ is a prime of $M$ above $p$ and  $V_{\frakp}A\otimes L$ is potentially ordinary  for all $v\vert p$, the height pairing \eqref{ht on A} is then canonical (cf. \cite[\S11.3]{nek-selmer}). Such  is the situation of Theorem \ref{B}. In that  case
we consider the restriction of \eqref{ht on A} to $A(\chi)$, coming from the Kummer maps 
$$\varkappa_{A(\chi)}\colon A(\chi)\to H^{1}_{f}(E, V_{\frakp}A(\chi)),\qquad
\varkappa_{A^{\vee}(\chi^{-1})}\colon A^{\vee}(\chi^{-1})\to H^{1}_{f}(E, (V_{\frakp}A(\chi)^{*})(1)),$$ where
$$V_{\frakp}A(\chi):= V_{\frakp}A|_{\calG_{E}}\otimes L(\chi)_{\chi}.$$
Note that by the condition $\chi|_{\A^{\infty,\times}}=\omega_{A}^{-1}$, we have $(V_{\frakp }A(\chi))^{*}(1)\cong V_{\frakp }A(\chi)$.

\subsection{Heights and intersections on curves}\label{hc} If $X/F$ is a (connected, smooth, proper) curve with semistable reduction at all $v\vert p$, let $V:=H^{1}_{\text{\'et}}(X_{\baar{F}},\Q_{p}(1))$. Then $V$ satisfies the relevant assumptions; moreover it carries a non-degenerate symplectic form by Poincar\'e duality,
inducing an isomorphism $V\cong V^{*}(1)$. For any finite extension $L$ of $\Q_{p}$, any Hodge splittings on $({\bf D}_{\rm dR}(V_{v}\otimes L))_{v\vert p}$,  and any continuous homomorphism $\ell\colon \Gamma_{F}\to L'$, we  obtain a pairing 
\begin{align}\label{ht-curve}
\langle\, , \rangle_{X, \ell} \colon {\rm Div}^{0}(X_{\baar{F}})\times {\rm Div}^{0}(X_{\baar{F}})\to L'\end{align}
via the Kummer maps similarly to the above. The pairing factors through ${\rm Div}^{0}(X_{\baar{F}})\to J_{X}(\baar{F})$ where $J_{X}$ is the Albanese variety; it corresponds to the height pairing on $J_{X}(\baar{F})\times J_{X}^{\vee}(\baar{F})$ via the canonical autoduality of $J_{X}$.

 The restriction of \eqref{ht-curve} to the set $ ({\rm Div}^{0}(X_{\baar{F}})\times {\rm Div}^{0}(X_{\baar{F}}))^{*}$ of pairs of divisors with disjoint supports admits a canonical decomposition 
$$\langle\, , \rangle_{X, \ell} =\sum_{w\in S_{F'}}   \langle\, , \rangle_{X, \ell_{w}, w}.$$ 
Namely, the local symbols are continuous symmetric bi-additive maps
 given by 
\begin{align}\label{loc gl hts curve}\langle D_{1}, D_{2} \rangle_{X, \ell_{w}} := \langle x_{1}, x_{2}\rangle_{\ell_{w}, E, w}\end{align}
where $x_{i}$ is the class of $D_{i}$ in $H^{1}_{f}(F', V)$, and if $Z_{1}$, $Z_{2}\subset X_{\baar{F}}$ are disjoint proper closed subsets of $X_{\baar{F}}$ such that the support of $D_{i}$ is contained in $Z_{i}$, then $E$, $E_{1}$, $E_{2}$ are the extensions obtained from  the diagram of \'etale cohomology groups
\begin{equation*}
\xymatrix{
   &   &0\ar[d]  &0\ar[d]  &  \\
0\ar[r] & H^{0}(Z_{2}, L(1)) \ar[r]\ar@{=}[d]&   H^{1}((X_{\baar{F}},  Z_{2}), L(1))    \ar[d]\ar[r] &  H^{1}(X_{\baar{F}}, L(1))   \ar[r]\ar[d]& 0\\
0\ar[r] &  H^{0}(Z_{2}, L(1)) \ar[r] & H^{1}((X_{\baar{F}}-Z_{1},  Z_{2}), L(1))   \ar[d]\ar[r] &  H^{1}(X_{\baar{F}}-Z_{1}, L(1))   \ar[r]\ar[d] &0\\
&&    H^{2}_{Z_{1}}(X_{\baar{F}}, L(1) )\ar@{=}[r]\ar[d]&   H^{2}_{Z_{1}}(X_{\baar{F}}, L(1)) \ar[d]&\\
&&0&0&
}
\end{equation*}
 by pull-back along  ${\rm cl}_{{D}_{1}} \colon L\to H^{2}_{Z_{1}}(X_{\baar{F}}, L(1))$ and push-out along $-{\rm Tr}_{{D}_{2}}\colon H^{0}(Z_{2}, L(1))\to L(1)$.

If $X$ does not have semistable reduction at the primes above $p$, we can still find a finite extension $F'/F$ such that $X_{F'}$ does, and define the pairing on $X_{F'}$.  If $X=\coprod_{i} X_{i}$ is a disjoint union of finitely many connected curves, then $\Div^{0}(X_{\baar{F}})$ will denote the group of divisors having degree zero on each connected component; it affords local and global pairings by direct sum.

\subsubsection{A uniqueness principle} 
Suppose that $D_{1}={\rm div}\, (h)$ is a principal divisor with support disjoint from the support of $D_{2}$, and let $h(D_{2}):=\prod_{P} h(P)^{n_{P}}$. Then the mixed extension $[E_{w}]=[E_{D_{1}, D_{2},w}]$ is the image of  $h(D_{2})\otimes 1  \in F_{w}'^{\times}\hat{\otimes} L \cong H^{1}(F_{w}', L(1))$ in $H^{1}_{*}(F_{w}', E_{2})$ under \eqref{ex seq me}; it follows that 
\begin{align}\label{ht inter princ}
\langle D_{1}, D_{2}\rangle_{X, \ell_{w}}=\ell_{w} (h(D_{2}))
\end{align}
independently of the choice of Hodge splittings. When $\ell_{w}$ is unramified, this property in fact suffices to characterise the pairing.

\begin{lemm}\label{ur ht} Let $X/F_{v}$ be a smooth proper curve over a local field $F_{v}$ and suppose that $\ell_{v}\colon F_{v}^{\times}\to L$ is unramified. Then there exists a unique locally constant  symmetric biadditive pairing  
$$\langle\, , \rangle_{X, \ell_{v}} \colon  ({\rm Div}^{0}(X_{{F}_{v}})\times {\rm Div}^{0}(X_{{F}_{v}}))^{*}\to L$$
such  that 
$$\langle{\rm div}\,(h),D_{2}\rangle_{X, \ell_{v}}=
\ell_{v}(h(D_{2}))$$
whenever the two arguments have disjoint supports.
\end{lemm}
\begin{proof}  
The result is well-known, see e.g. \cite[Proposition 1.2]{CG}, but  for the reader's convenience we recall the proof. A construction of such a pairing has just been recalled, and a second one will be given below.  For  the uniqueness, note that the difference of any two such pairings is a locally constant homomorphism $J({F}_{v})\times J({F}_{v})\to L$. As the source is a compact group and the target is torsion-free,  such a homomorphism must be trivial.
\end{proof}

\subsubsection{Arithmetic intersections} 
Let $F'\subset \baar{F}$ be a finite extension of $F$ and $\X/\OO_{F'}$ be  a regular   integral model of $X$. 
For a divisor $D\in \Div^{0}(X_{F'})$, we define its \emph{flat extension} to the model $\X$ to be the unique extension of $D$ which has intersection zero with any vertical divisor; it can be uniquely written as $\baar{D}+V$, where $\baar{D}$ is the Zariski closure of $D$ in $\X$ and $V$ is a vertical divisor.

Let $D_{1}$, $D_{2}\in \Div^{0}(X_{\baar{F}})$ be divisors with disjoint supports, with each $D_{i}$ defined over a finite extension $F_{i}$; assume that $F\subset F_{2}\subset F_{1}\subset \baar{F}$. Let $\X/\OO_{F_{2}}$ be a regular and semistable model. Then for each finite place $w\in S_{F_{1}}$, we can define partial  local   intersection multiplicities $i_{w}$, $j_w$ of the flat extensions $\baar{D}_{1}+V_{1} $ of $D_{1}$, $\baar{D}_{2}+V_{2}$ of   $D_{2}$ to $\X_{\OO_{F_{1,w}}}$. If the latter model is still regular, they are defined by 
\begin{equation}\label{i and j}\begin{split}
i_{w}(D_{1}, D_{2})=\frac{1}{ [F_{1}:F]}(\baar{D}_{1}\cdot\baar{D}_{2})_{w},
\\
j_{w}(D_{1}, D_{2}) = \frac{1}{ [F_{1}:F]}(\baar{D}_{1}\cdot V_{2})_{w},
\end{split}
\end{equation}
where on the right-hand sides, $(\, \cdot \,)_{w}$ are the usual $\Z$-valued intersection multiplicities in $\X_{\OO_{F_{1,w}}}$; see \cite[\S 7.1.7]{yzz} for the generalisation of the definition to the case when $\X_{\OO_{F_{1,w}}}$ is not regular. The total intersection
$$m_{w}(D_{1}, D_{2})= i_{w}(D_{1}, D_{2})+j_{w}(D_{1}, D_{2})$$ 
is of course independent of the choice of models.

Fix an extension $\baar{v}$ to $\baar{F}$ of the valuation  $v$ on $F$. Then we have pairings $i_{\baar{v}}$, $j_{\baar{v}}$ on divisors on $X_{\baar{F}_{\baar{v}}}$ with disjoint supports by the above formulas. We can group together the contributions of $i$ and $j$ according to the places of $F$ by
$$\lambda_{v}(D_{1}, D_{2})=\dashint_{\Gal(\baar{F}/F)} \lambda_{\baar{v}}(D_{1}^{\sigma}, D_{2}^{\sigma})\, d\sigma$$
for $v$ any finite place of $F$ and  $\lambda=i$, $ j$, or $\lambda_{v}(D_{1}, D_{2})=\langle D_{1}, D_{2}\rangle_{v}$. Here the integral uses the Haar measure of total volume one, and reduces to a finite weighted average for any fixed $D_{1}$, $D_{2}$.

\begin{prop}\label{compatibility} Suppose that $D_{1}$ and $D_{2}$ are divisors of degree zero on $X$, defined over  an extension $F'$ of $F$. Then for all finite places $w\nmid p$ of $F'$,
$$ \langle D_{1}, D_{2}\rangle_{X, \ell_{w}} =m_{w}( D_{1}, D_{2})\cdot\ell_{w}(\vpi_{w})=( i_{w}( D_{1}, D_{2}) + j_{w}(D_{1}, D_{2}))\cdot\ell_{w}(\vpi_{w}).$$
\end{prop}
\begin{proof} This follows from Lemma \ref{ur ht} and \eqref{ht inter princ}; the verification that the arithmetic intersection pairing $m_{w}$ also satisfies the required properties can be found in \cite{gross}.
\end{proof}

\subsection{Integrality crieria}\label{intH}
The result of Proposition \ref{compatibility} applies with the same proof if $w\vert p$ when $\ell_{w}$ is an unramified logarithm such as the valuation. In this case we will view it as a first integrality criterion for local heights.
 \begin{prop}\label{integrality and intersect} Let $\ell_{v}\colon F_{v}^{\times}\hat{\otimes} {L}\to \Gamma_{F}\hat{\otimes}L$ be the tautological logarithm and let $\ell_{w}$ be  as in \eqref{ellw}. Let $v\colon F_{v}^{\times}\hat{\otimes} L\to L$ be the valuation. Then for all  $D_{1}$, $D_{2}\in \Div^{0}(X_{F'})$ we have
$$ v (\langle D_{1}, D_{2}\rangle_{X, \ell_{w}} )=  [F_{w}': F_{v}]\cdot m_{w}(D_{1}, D_{2}).$$
In particular, if $m_{w}(D_{1},D_{2})=0$ then 
$$\langle D_{1}, D_{2}\rangle_{X, \ell_{w}} =\ell_{w}(s_{w}([E_{D_{1}, D_{2},w}])) \in \OO_{F,v}^{\times}\hat{\otimes} L =\ell_{w} (H^{1}_{f}(F_{w}, L(1));$$
equivalently,  the mixed extension $[E_{D_{1}, D_{2},w}]$ is essentially crystalline.
\end{prop}

We need a finer integrality property for local heights, slightly generalising \cite[Proposition 1.11]{nekovar}. Let $F_{v}$ and $L$  be  finite extensions of $\Q_{p}$, let $V$ be a $\calG_{F_{v}}$-representation on an $L$-vector space  equipped with a splitting of \eqref{hdg}, and $\ell_{v}\colon  F_{v}^{\times}\to L$ be a logarithm. Suppose the following conditions are satisfied:
\begin{enumerate}
\item[(a)] $\ell_{v}\colon F_{v}^{\times}\to L$ is ramified;
\item[(b)] the space $V$ admits a direct sum decomposition $V=V'\oplus V''$ as $\calG_{F_{v}}$-representation, such that $V'$ satisfies the  {Panchishkin condition}, with a decomposition $0\to V'^{+}\to V'\to V'^{-}\to 0$,  and the restriction of the Hodge splitting of ${\bf D}_{\rm dR}(V)$ to  ${\bf D}_{\rm dR}(V')$ coincides with the canonical one of \eqref{panc}. 
\item[(c)]  $H^{0}(F_{v},V'^{-})=H^{0}(F_{v}, V'^{+*}(1))=0$.
\end{enumerate}
By \cite[Proposition 1.28 (3)]{nekheights}, the last condition is equivalent to ${\bf D}_{\rm pst}(V')^{\vphi=1} = {\bf D}_{\rm pst}(V'^{*}(1))^{\vphi=1} = 0$, where ${\bf D}_{\rm pst}(V'):=\lim_{F\subset F'} {\bf D}_{\rm st}(V'|_{\calG_{F'}})^{\Gal(F'/F)}$ (the limit ranging over all sufficiently large finite Galois extensions $F'/F$).

Let $T$ be a $\calG_{F,v}$-stable $\OO_{L}$-lattice in $V$, $T':=T\cap V'$, $T''=T\cap V''$; let   $d_{0}\geq 0$  be an integer such that $\frakp_{L}^{d_{0}}T\subset T'\oplus T''\subset T$, where $\frakp_{L}\in L$ is the maximal ideal of $\OO_{L}$. 

Let $F_{v}\subset F_{v,\infty}\subset F_{v}^{\rm ab}$ be the intermediate extension determined by $\Gal(F_{v}^{\rm ab}/F_{v, \infty}) \cong \ker({\ell_{v}})\subset F_{v}^{\times}$ under the reciprocity isomorphism. Let 
$${N}_{\infty,\ell_{v}}H^{1}_{f}(F_{v}, T'):=\bigcap_{F'_{w}} {\rm cores}^{F'_{w}}_{F_{v}}(H^{1}_{f}(F'_{w}, T') )$$
be the subgroup of \emph{universal norms}, where the intersection ranges over all finite extensions $F_{v}\subset F_{v'}'$ contained  in $ F_{v, \infty}$.

\begin{prop}\label{int panc}  Let $x_{1}\in H^{1}_{f}(F_{v}, T)$, $x_{2}\in H^{1}_{f}(F_{v}, T^{*}(1))$ and suppose that the image of $x_{2}$ in  $H^{1}_{f}(F_{v}, T''^{*}(1))$ vanishes.
Let ${d_{1}}$ be the $\OO_{L}$-length  of $H^{1}(F_{v}, T''^{*}(1))_{\rm tors}$, ${d_{2}}$ the length of $H^{1}_{f}(F_{v}, T')/{N}_{\infty,\ell_{v}}H^{1}_{f}(F_{v}, T')$. Then 
$$\frakp_{L}^{d_{0}+d_{1}+d_{2}}\langle x_{1}, x_{2}\rangle_{\ell_{v}, E, v}\subset \ell_{v}(F_{v}^{\times}\hat{\otimes}\OO_{L})$$
for any mixed extension $E$. If moreover $[E_{v}]$ is essentially crystalline, then 
$$\frakp_{L}^{d_{0}+d_{1}+d_{2}}\langle x_{1}, x_{2}\rangle_{\ell_{v}, E, v}\subset \ell_{v}(\OO_{F,v}^{\times}\hat{\otimes}\OO_{L}).$$
\end{prop}
\begin{proof}  The first assertion  (which implicitly contains the assertion that  $H^{1}_{f}(F_{v}, T')/{ N}_{\infty,\ell_{v}}H^{1}_{f}(F_{v}, T')$ is finite) is identical to  \cite[Proposition 1.11]{nekovar}, whose assumptions however are slightly more stringent. First, $L$ is assumed to be $\Q_{p}$; this requires only cosmetic changes in the proof. Secondly, in \emph{loc. cit.} the representation $V$ (hence $V'$) is further assumed to be crystalline. This assumption is  used  via \cite[\S 6.6]{nekheights} to apply various consequences of the existence of the exact sequence 
\begin{gather}\label{exxx}
0\to H^{1}_{f}(F_{v},V'^{+})\to 
H^{1}_{f}(F_{v},V')\to 
H^{1}_{f}(F_{v},V'^{-})\to 0\end{gather}
which is established  in \cite[Proposition 1.25]{nekheights} under the assumption that $V'$ is crystalline. However \eqref{exxx} still exists under our assumption that $H^0(F_{v}, V'^{-})=0$ by   
 \cite[Corollary 1.4.6]{benois-greenberg}.\footnote{The finiteness of $H^{1}_{f}(F_{v}, T')/{ N}_{\infty,\ell_{v}}H^{1}_{f}(F_{v}, T')$ under our assumption is  also in \cite[Corollary 8.11.8]{nek-selmer}.}

 The second assertion follows from the proof of the first one: the height is the image under $\ell_{v}$  of an element of $H^{1}(F_{v}, L(1))=F_{v}^{\times}\hat{\otimes} L$, which belongs to $H^{1}_{f}(F_{v}, L(1))=\OO_{F,v}^{\times}\hat{\otimes} L$ if $[E_{v}]$ is essentially crystalline.
\end{proof}

\section{Generating series and strategy of proof}\label{sec GS}
We introduce the  constructions that will serve to prove the main theorem. We have expressed the $p$-adic $L$-function as the $p$-adic Petersson product of a form $\vphi$ in $\sigma_{A}$ and a certain kernel function depending on a Schwartz function $\phi$. The connection between the data of $(\vphi, \phi)$ and the data of $(f_{1},f_{2})$ appearing in the main theorem is given by the Shimizu lifting introduced in \S\ref{shim}. 
An arithmetic-geoemetric analogue of the latter allows to express also the left-hand side of Theorem \ref{B} as the image under the $p$-adic Petersson product of another  kernel function, introduced in \S\ref{sec gk}. The main result of this section is thus the reduction of Theorem \ref{B} to an identity between the two kernel functions (\S\ref{sec:KI}).

\subsection{Shimizu's theta lifting}\label{shim}
Let $B$ be a quaternion algebra over a local or global field $F$, $V=(B,q)$ with the reduced norm $q$. The action $(h_{1}, h_{2})\cdot x  := h_{1}xh_{2}^{-1}$ embeds $(B^{\times}\times B^{\times})/F^{\times}$ inside ${\bf GO}(V)$. If $F$ is a local field,    $\sigma$ is a representation of $\GL_{2}(F)$, and $\pi$ is a representation of $B^{\times}$, then the space of liftings
$$ \Hom_{\GL_{2}(F)\times B^{\times}\times B^{\times}}(\sigma\otimes\calS(V\times F^{\times}), \pi \otimes \pi^{\vee})$$
has dimension zero unless either $B=M_{2}(F)$ and $\pi=\sigma$ or $\sigma $ is discrete series and $\pi$ is its image under the Jacquet--Langlands correspondence; in the latter cases the dimension is one. An explicit generator was constructed by Shimizu in the global coherent  case, and we can use it to normalise a generator in the local case and construct a generator in the global incoherent case.

\subsubsection{Global lifting} Let $B$ be a quaternion  algebra over a  number field $F$, $V=(B,q)$ with the reduced norm. Let $\sigma$ be a cuspidal automorphic representation of $\GL_{2}(\A_{F})$ which is discrete series at all places where $B$ is ramified,  $\pi$ the automorphic representation of $B^{\times}_{\A}$ attached to $\sigma$ by the Jacquet--Langlands correspondence. Fix a nontrivial additive character $\psi\colon \A/F\to \C^{\times}$. Consider  the theta series
$$\theta(g,h,\Phi)=\sum_{u\in F^{\times}}\sum_{x\in V}r_{\psi}(g,h)\Phi(x,u), \quad g\in \GL_{2}(\A), h\in (B^{\times}_{\A}\times B^{\times}_{\A})/\A^{\times}\subset {\bf GO}(V_{\A})$$
for $\Phi\in \calS(V_{\A}\times \A^{\times})$. Then  the Shimizu theta lift of any $\vphi\in \sigma$ is defined to be
$$\theta(\vphi, \Phi)(h):= { \zeta_{F}(2)\over 2L(1, \sigma, \ad)}\int_{\GL_{2}(F)\bks \GL_{2}(\A)} \vphi(g)\theta(g, h, \Phi)\,dg\quad \in \pi\times \pi^{\vee},$$
and it is independent of the choice of $\psi$.

If $F$ is totally real, $B$ is totally definite, and $\phi\in \bcalS(V_{\A}\times \A^{\times})$, we denote $\theta(\vphi, \phi)
:=\theta(\vphi, \Phi)$ for any ${\bf O}(V_{\infty})$-invariant preimage $\Phi$ of $\phi$ under \eqref{schwartz quotient}.
Let 
$\mathscr{F}\colon \pi\otimes \pi^{\vee} \to \C$ be the  duality defined by the Petersson bilinear pairing on $B_{\A}^{\times}$ (for the Tamagawa measure). By  \cite[Proposition 5]{wald}, we have
\begin{multline}\label{FShimizu}
\mathscr{F}\theta(\vphi, \Phi)=
 {   (\pi^{2}/2)^{[F:\Q]  }  \over  |D|_{F}^{3/2}   \zeta_{F}^{\infty}(2)}
	\prod_{v\nmid \infty}   {|d_{v}|^{-3/2}   \zeta_{F,v}(2)^{2}\over L(1,\sigma_{v}, \ad)} 
	\int_{N(F)_{v}\bks \GL_{2}(F_{v})} W_{\vphi, -1,v}(g) r(g)\Phi_{v}(1,1)\, dg\\
\times \prod_{v \vert\infty}{2\zeta_{F,v}(2) \over \pi^{2} L(1,\sigma_{v}, \ad) }\int_{N(F)_{v}\bks \GL_{2}(F_{v})} W_{\vphi, -1,v}(g) r(g)\Phi_{v}(1,1)\, dg.
\end{multline}
The terms in the first line are all rational if $W_{\phi, -1,v}$ and $\Phi_{v}$ are, and almost all of the factors equal~$1$.\footnote{The analogous  assertion in \cite[Proposition 2.3]{yzz} is incorrect.}
For $v$ a real place, if $W_{\vphi, -1,v}$ is the standard antiholomorphic discrete series of weight $2$ and $\Phi_{v}$ is a preimage of the standard Schwartz function $\phi_{v}\in \bcalS(V_{v}\times F_{v}^{\times})$ under \eqref{sql}, the terms in the last line are rational too,\footnote{The analogous statement holds for discrete series of arbitrary weight.} and in fact equal to~$1$ by the calculation of Lemma \ref{arch-int}.

\subsubsection{Local lifting} In the local case, depending on the choice of $\psi_{v}$,  we can then normalise a generator
$$\theta_{v}=\theta_{\psi_{v}}\in  \Hom_{\GL_{2}(F_{v})\times B_{v}^{\times}\times B_{v}^{\times}}(\W(\sigma_{v}, \baar{\psi}_{v})\otimes\calS(V_{v}\times F_{v}^{\times}), \pi_{v} \otimes \pi_{v}^{\vee})$$
(where $\W(\sigma_{v}, \baar{\psi}_{v})$ is the Whittaker model for $\sigma_{v}$ for the conjugate character $\baar{\psi}_{v}$)
by 
\begin{align}\label{shimizu local norm}
\mathscr{F}_{v}\theta_{v}(W, \Phi)=
{ {c_{v} \zeta_{F,v}(2)\over L(1,\sigma_{v}, \ad)} \int_{N(F)_{v}\bks \GL_{2}(F_{v})}} W(g) r(g)\Phi(1,1)\, dg
\end{align}
with $c_{v}=|d_{v}|^{-3/2}  \zeta_{F,v}(2)$ if $v$ is finite and $c_{v}=2\pi^{-2}$ if $v$ is archimedean. Here  the decompositions $\pi=\otimes_{v}\pi_{v}$,  $\pi^{\vee}=\otimes_{v}\pi_{v}^{\vee}$ are taken to satisfy  $\mathscr{F}=\prod_{v}\mathscr{F}_{v}$ for the natural dualities   $\mathscr{F}_{v}\colon \pi_{v}\otimes \pi_{v}^{\vee} \to \C$.

Then by \eqref{FShimizu} we have  a decomposition
\begin{align}\label{Shimizu incoh}
\theta= {(\pi^{2}/2)^{[F:\Q]} \over  |D|_{F}^{3/2}  \zeta_{F}^{\infty}(2)}\otimes_{v}\theta_{v}.
\end{align}
As in the global case, we  define $\theta_{v}(W, \phi):=\theta_{v}(W, \Phi)$ for $\phi=\baar{\Phi}\in\bcalS(V_{v}\times F_{v}^{\times})$.
\subsubsection{Incoherent lifting} Finally, suppose that $F$ is totally real and  $\B$ is a totally definite  \emph{incoherent} quaternion algebra over $\A=\A_{F}$, and let ${\bf V}=(\B, q)$.  Let $\sigma$ be a cuspidal automorphic representation of $\GL_{2}(\A)$ which is discrete series at all places of ramification of $\B$ and let $\pi=\otimes_{v} \pi_{v}$ be the representation of $\B^{\times}$ associated to $\sigma$ by the local Jacquet--Langlands correspondence. 
Then \eqref{Shimizu incoh}  defines a lifting
$$\theta\in \Hom_{\GL_{2}(\A)\times \B^{\times}\times \B^{\times}}(\sigma \otimes \calS({\bf V}\times \A^{\times}), \pi\otimes \pi^{\vee}).$$
It coincides  with the lifting denoted by the same name in \cite{yzz}.

\subsubsection{Rational liftings}
If $M$ is a number field,  $\sigma^{\infty}$ is an $M$-rational cuspidal automorphic representation of $\GL_{2}$ of weight $2$  as in the Introduction  and $\pi$ is its transfer to an $M$-rational representation of $\B^{\infty\times}$ under the rational Jacquet--Langlands correspondence of \cite[Theorem 3.4]{yzz}, let $\sigma^{\iota}$ be the associated complex automorphic representation, and $\pi^{\iota}=\pi\otimes_{M, \iota}\C$, for any $\iota\colon M\into \C$.  Let $\theta^{\iota}=\otimes_{v}\theta_{v}^{\iota}$ be the liftings just constructed.
 Then, if $\B$ is coherent, using the algebraic Petersson product of Lemma \ref{alg-pet},  there is a lifting $\theta\colon \sigma^{\infty} \otimes \bcalS({\bf V}^{\infty}\times \A^{\infty,\times})\to \pi\otimes\pi^{\vee}$, which is
defined over $M$
and satisfies
\begin{gather}\label{alg theta lift}
\iota\theta(\vphi, \phi^{\infty})=\theta^{\iota}(\vphi^{\iota}, \phi^{\infty}\phi_{\infty}^{\iota})
\end{gather}
if $\vphi^{\iota}$ is as described before Lemma \ref{alg-pet} and $\phi_{\infty}^{\iota}$ is standard. On the left hand side, we view $\pi$ and $ \pi^{\vee}$  indifferently as a representation of $\B^{\infty\times}$ or $\B^{\times}$ by tensoring on each with generators of the trivial representation at infinite places which pair to~$1$ under the duality (this ensures compatibility with the decomposition). After base-change to $M \otimes \OO_{\Psi_{v}}(\Psi_{v})$, there are local liftings at finite places
$\theta_{\psi_{{\rm univ}},v}\colon \W(\sigma_{v}, \baar{\psi}_{{\rm univ},v}) \otimes \bcalS(V_{v}\times F_{v}^{\times})\to \pi_{v}\otimes\pi_{v}^{\vee}\otimes \OO_{\Psi_{v}}(\Psi_{v})$, 
satisfying $\iota\theta_{\psi_{\rm univ},v}(W, \phi)(\psi_{v})=\theta_{\psi,v}^{\iota}(W^{\iota}, \phi^{\iota})$. They induce an incoherent global lifting on $\sigma^{\infty }\otimes \bcalS({\bf V}^{\infty}\times \A^{\infty,\times})$,  which is defined over $M$   independently of the choice of an additive character  of $\A$  \emph{trivial on $F$},\footnote{In the following sense. Let $\Psi:=\prod_{v}\Psi_{v}$ and let $\Psi^{\circ}\subset \Psi$ be defined by $\prod_{v}\psi_{v}|_{F}=1$. Then the global lifting is first defined on  $ \sigma^{\infty}\otimes \bcalS({\bf V}^{\infty}\times \A^{\infty,\times})\otimes \OO_{\Psi}(\Psi)$. Its restriction to $ \sigma^{\infty}\otimes \bcalS({\bf V}^{\infty}\times \A^{\infty,\times})\otimes \OO_{\Psi}(\Psi^{\circ})$ is invariant under the homogeneous action of $\OO_{F}^{\times}$ on $\Psi^{\circ}$ and hence is the base-change of an $M$-linear map   $ \sigma^{\infty}\otimes \bcalS({\bf V}^{\infty}\times \A^{\infty,\times})\to \pi\times \pi^{\vee} $.}
 and satisfies \eqref{alg theta lift}.

Finally, for an embedding $\iota'\colon M\into L$ with $L$ a $p$-adic field, we let 
$$\theta_{\iota'}\colon( \sigma^{\infty}\otimes \bcalS({\bf V}^{\infty}\times \A^{\infty,\times}))\otimes_{M}L\to (\pi\times \pi^{\vee})\otimes_{M, \iota'} L$$
be the base-change.

\subsubsection{Local toric periods and zeta integrals}
Recall from the Introduction that for any dual pair of representations $\pi_{v}\otimes \pi_{v}^{\vee}$ isomorphic (possibly after an extension of scalars) to the local component of $\pi\otimes \pi^{\vee}$,
 the  normalised toric integrals of matrix coefficients of \eqref{Qvdef} are defined, for any  $\chi\in \Y_{v}(\C)$, by
$$
  Q_{v}(f_{1,v}, f_{2,v}, \chi)= |D|^{-1/2}_{v}|d|^{-1/2}_{v}
 {L(1, \eta_{v})L(1,\pi_{v}, \ad) \over \zeta_{F,v}(2) L(1/2, \pi_{E,v}^{\iota}\otimes\chi_{v}) }
 Q_{v}^{\sharp}(f_{1,v}, f_{2,v}, \chi^{\iota}_{v}),
 $$
 \begin{align}\label{Qsharp}
 Q_{v}^{\sharp}(f_{1,v}, f_{2,v}, \chi_{v})=
\int_{E_{v}^{\times}/F_{v}^{\times}} \chi_{v}(t) (\pi(t)f_{1,v}, f_{2,v})\,  {dt}.
\end{align}

The following lemma follows from the  normalisation \eqref{shimizu local norm} and the definitions of the  local toric zeta integrals $R_{r,v}^{\natural}$ in \eqref{Rnat}.

\begin{lemm}\label{Qtheta} Let $\chi\in \Y^{\rm l.c.}_{ M(\alpha)}(\C)$, and $\psi=\otimes_{v} \psi_{v}\colon \A/F\to \C$ be  a nontrivial additive character. Let $Q_{v}$ and $Q_{v}^{\sharp}$ be the pairings defined above for the representation $\theta_{\psi,v}(\W(\sigma_{v}, \baar{\psi}_{v})\otimes \bcalS(V_{v}\times F_{v}^{\times}))$. 
Then for all $v\nmid \infty$, we have
$$|d|_{v}^{-3/2} R_{r,v}^{\circ}(W_{v}, \phi_{v},\chi_{v}, \psi_{v}) = Q_{v}^{\sharp}(\theta_{\psi,v}(W_{v},\iota\alpha_{v}(\vpi_{v})^{-r_{v}} w_{r,v}^{-1}\phi_{v}), \chi_{v}),$$
where as usual if $v\nmid p$ we set $\alpha_{v}=1$, $r_{v}=0$, $w_{r,v}=1$.

If $v\nmid p$, we have
$$\mathscr{R}^{\natural}(W_{v}, \phi_{v}, \chi_{v}, \psi_{v}) =   Q_{v}(\theta_{\psi,v}(W_{v},\phi_{v}), \chi_{v}),$$
 and for the product $Q^{p}=\prod_{v\nmid p\infty}Q_{v}$ we have 
$$  \prod_{v \nmid p\infty} \mathscr{R}_{r,v}^{\natural}(W_{v}, \phi_{v},\chi_{v}, \psi_{v})= 
{D_{F}^{3/2} \zeta_{F}^{\infty}(2)  \over (\pi^{2}/2)^{[F:\Q]} }
 Q^{p}(\theta_{\psi}(\vphi,\alpha(\vpi)^{-r} w_{r}^{-1}\phi), \chi).$$
\end{lemm}

\subsection{Hecke correspondences and generating series} 
Referring to \cite[\S 3.1]{yzz} for more details, let us recall some basic notions on the Shimura curves $X_{U}$. The set of geometrically connected components is $\pi_{0}(X_{U, \baar{F}})\cong F_{+}^{\times}\bks \A^{\infty,\times}/q(U)$. The curve $X_{U}$ admits a canonical divisor class (the \emph{Hodge class}) $\xi_{U}={1\over \deg L_{U}} L_{U}$ of degree~$1$ on each geometrically connected component; here 
$$L_{U}=\omega_{X_{U}/F}+\sum_{x\in X_{U}(\baar F)}(1-e_{x}^{-1}) x,$$
a line bundle defined over $F$; here $\omega_{X_{U}/F}$ is the canonical bundle and $e_{x}$ is the ramification index (see \cite[\S 3.1.3]{yzz} for the precise definition) of the point $x$. 

\subsubsection*{Hecke correspondences} For  $x\in \B^{\infty\times}$, let 
  ${\rm T}_{x}\colon X_{xUx^{-1}}\to X_{U}$ be the  translation, given 
 in the complex uniformisation by  ${\rm T}_{x}([z,y])=[z, yx]$.
  Let ${\rm p}\colon X_{U\cap xUx^{-1}}\to X_{U}$ be the projection and let $Z(x)_{U}$ be the image of 
 $$({\rm p}, {\rm p} \circ {\rm T}_{x})\colon   X_{U\cap xUx^{-1}}\to X_{U}\times X_{U}.$$
  We view $Z(x)_{U}$ as a correspondence on $X_{U}$, and we will sometimes use the same notation for the image of $Z(x)_{U}$ in $\Pic(X_{U}\times X_{U})$ (such abuses will be made clear in what follows).

 We obtain 
  an action of the Hecke algebra  $\mathscr{H}_{\B^{\infty\times}, U}:=C^{\infty}_{c}(\B^{\infty\times})^{U\times U}$  
 of $U$-biinvariant functions on $\B^{\infty\times}$ by
  $${\rm T}(h)_{U}= \sum_{x\in U\bks \B^{\infty\times}/U} h(x) Z(x)_{U}.$$
Note the obvious relation  
 $Z(x)_{U}={\rm T}(\one_{UxU})_{U}$.  The transpose ${\rm T}(h)^{\rm t}$ equals ${\rm T}(h^{\rm t})$ with $h^{\rm t}(x):=h(x^{-1})$. It is then easy to verify that if $x$ has trivial components away from the set of places where $U$ is maximal, we have 
 \begin{align}\label{transpose}
Z(x)_{U}^{\rm t}=Z(q(x)^{-1})_{U}Z(x)_{U}.
 \end{align}

Finally, for any simple quotient $A'/F$ of $J$ with $M'=\End^{0}(A')$, we have a  $\Q$-linear map
\begin{align*}
{\rm T}_{\rm alg}\colon  \pi_{A'}\otimes_{M'} \pi_{A'^{\vee}}&\to  \Hom^{0}(J, J^{\vee})\\
f_{1}\otimes f_{2} \qquad &\mapsto f_{2}^{\vee}\circ f_{1}.
\end{align*}
If $\iota\colon M'\into L'$ is any embedding into a $p$-adic field $L'$, we denote 
$${\rm T}_{\rm alg, \iota}\colon  \pi_{A'}\otimes_{M'} \pi_{A'^{\vee}}\otimes_{M'} L'  \stackrel{\iota}{\to} \pi_{A'}\otimes_{M'} \pi_{A'^{\vee}}\otimes_{\Q} L' \stackrel{{\rm T}_{\rm alg}\otimes 1}{\longrightarrow} \Hom^{0}(J, J^{\vee})\otimes_{\Q} L' $$
the composition in which the first arrow is deduced from the unique $L'$-linear embedding $L'\into M'\otimes_{\Q}L'$  whose composition with   $M'\otimes_{\Q} L'\to L'$, $x\otimes y\mapsto \iota(x)y$, is ${\rm id}_{L'}$.

\subsubsection{Generating series}

For any $\phi\in \bcalS({\bf V}\times \A^{\times})$ invariant under $K=U\times U$, define a generating series
$$Z(\phi):=Z_{0}(\phi)_{U}+Z_{*}(\phi)_{U}$$
where
\begin{align*}
Z_{0}(\phi)_{U}&:= -\sum_{\beta\in F_{+}^{\times}\bks \A^{\infty,\times} /q(U)}\sum_{u\in \mu_{U^{p}}^{2}\bks F^{\times}} E_{0}(\beta^{-1}u, \phi)L_{K, \beta},\\
Z_{a}(\phi)_{U}&:= w_{U}\sum_{x\in K\bks \B^{\infty\times}}\phi(x, aq(x)^{-1}) Z(x)_{U} \quad \text{for \ } a\in F^{\times},\\
Z_{*}(\phi)_{U}&:=\sum_{a\in F^{\times}}Z_{a}(\phi)_{U}
\end{align*}
with $w_{U}=|\{\pm 1\}\cap U|$. Here  $L_{K, \beta}$ denotes the component of a Hodge class in $ {\rm Pic}(X_{U}\times X_{U})_{\Q}$  obtained from the classes $L_{U}$ (see \cite[\S 3.4.4]{yzz}),  and 
$$E_{0}(u, \phi)=\phi(0,u)+W_{0}(u, \phi)$$
is the constant term of the standard Eisenstein series: its intertwining part $W_{0}(u, \phi)$ is the value at $s=0$ of 
$$W_{0}(s,u, \phi)=\int_{\A}\delta(wn(b))^{s}r(wn(b))\phi(0,u)\, db,$$
where $\delta(g)=|a/d|^{1/2}$ if $g=\smallmat a*{}d k$ with $k\in \GL_{2}(\widehat{\OO}_{F})\mathbf{SO}(2, F_{\infty})$.

For $g\in \GL_{2}(\A)$,  define
$${Z}(g,\phi)={Z}(r(g)\phi),$$
and similarly $Z_{0}(g,\phi)_{U}$, $Z_{a}(g, \phi)_{U}$, $Z_{*}(g, \phi)_{U}$.

Let $U=U^{p}U_{p}$ and  $c_{U^{p}}$ be as in \eqref{c_{U}}. By  \cite[\S 3.4.6]{yzz},   the normalised versions 
$$\wtil{Z}(g,\phi):= {c_{U^{p} }}  {Z}(g, \phi)_{U}, 
	\quad \wtil{Z}_{a}(g,\phi):={c_{U^{p}} }  {Z}_{a}(g, \phi)_{U}, \quad \ldots$$
are  independent of $U^{p}$.  
A key result, which is essentially a special case of the main theorem of \cite{yzzgkz},  is that the series $\wtil{Z}(g, \phi)$   defines an automorphic form valued in $\Pic(X\times X)_{\Q}$.
\begin{theo}\label{modularity} The map
$$(\phi, g)\mapsto \wtil{Z}(g, \phi)$$
defines an element
$$\wtil{Z}\in \Hom_{\B^{\times }\times \B^{\times}} (\bcalS({\bf V}\times \A^{\times}), C^{\infty}(\GL_{2}(F)\bks \GL_{2}(\A))\otimes \Pic(X\times X)_{\Q}).$$
\end{theo}
Here the target denotes the set of $\Pic(X\times X)$-valued series $\wtil{Z}$ such that for any linear functional $\lambda\colon \Pic(X\times X)_{\Q}\to \Q$, the series $\lambda(\wtil{Z})$ is absolutely convergent and defines an automorphic form. Its constant term is non-holomorphic in general (in fact only when $F=\Q$ and $\Sigma={\infty}$). (However the geometric kernel that we introduce next will always be a holomorphic cusp form of weight~$2$.) See   \cite[Theorem 3.17, Lemma  3.18]{yzz} for the proof of the theorem.

\medskip

Assume from now on that $\phi_{\infty}$ is standard; we accordingly only write $Z(\phi^{\infty})$, $\tZ(\phi^{\infty})$, \ldots\ Define, for each $a \in \A^{\times}$,
\begin{equation}\label{za}
{\wtil Z}_{a}(\phi^{\infty}):= c_{U^{p}}w_{U} |a| \sum_{x\in K\bks \B^{\infty\times}}\phi^{\infty}(x, aq(x)^{-1}) Z(x)_{U}
\end{equation}
for any sufficiently small $U$. This extends the previous definition for $a\in F^{\times}$, and it is easy to check that for every $y\in \A^{\infty,\times}$, $a\in F^{\times}$, we have
$${\wtil Z}_{a}( \smalltwomat y {}{}1,\phi^{\infty})
=|ay|_{\infty}{\wtil Z}_{ay}(\phi^{\infty})\psi(iay_{\infty}).$$

In other words, the images in $\Pic(X_{U}\times X_{U})$ of the $\tZ_{a}(\phi^{\infty})$  are the reduced $q$-expansion coefficients of $\tZ(\phi^{\infty})$, in the following sense:  for any functional $\lambda$ as above, $(\lambda(\tZ_{0}(y, \phi)), (\lambda(\tZ_{a}(\phi^{\infty})))_{a})  $ are  the reduced $q$-expansion coefficients of the modular form $ \lambda({\wtil Z})( \phi)$. 

\subsubsection{Hecke operators and Hecke correspondences}
For the following lemma, let $S$ be a set of finite places of $F$ such that for all $v\notin S$, $\B_{v}$ is split, $U_{v}$ is maximal,  and $\phi_{v}$ is standard.  Fix any  isomorphism $\gamma\colon \B^{S}\to M_{2}(\A^{S}) $ of $\A^{S}$-algebras carrying the reduced norm to the determinant and $\OO_{\B^{S}}$ to $M_{2}(\widehat{\OO_{F}}{}^{S})$; such  an isomorphism is unique up to conjugation by $\OO_{\B^{S}}^{\times}$. 
\begin{lemm}\label{various Ts}
Let  $U'{}^{S}= \GL_{2}(\widehat{\OO_{F}}{}^{S})$, and identify the commutative algebras $\mathscr{H}^{S}_{U'{}^{S}}=\mathscr{H}^{S}_{\GL_{2}(\A^{\infty}), U'{}^{S}}$ with $\mathscr{H}_{\B^{\infty\times}, U}^{S}$ via the isomorphism $\gamma^{*}$ induced by $\gamma$ above. Then for each $h\in \mathscr{H}^{S}_{U'{}^{S}}$ we have
$$T(h)Z_{*}(\phi^{\infty})_{U}= {\rm T}(\gamma^{*}h)_{U}\circ Z_{*}(\phi^{\infty})_{U}.$$
In the left-hand side,  we view $Z_{*}(\phi^{\infty})$ as a reduced $q$-expansion of central character $z\mapsto Z(z)$, and
 $T(h)$ is the usual Hecke operator acting  by \eqref{hecke on q}; in the right-hand side, the symbol $\circ$ denotes composition of correspondences on $X_{U}$.
 
In particular, the right-hand side is independent of the choice of $\gamma$.
\end{lemm}
\begin{proof} 
It suffices to check the statements for the set of generators of the algebra $\mathscr{H}^{S}_{U'{}^{S}}$ consisting of elements $h=h_{v}h^{vS}$, with $h^{vS}$  the unit of $\mathscr{H}^{vS}_{U'{}^{S}}$ and  $h_{v}=\one_{U_{v}x_{v}U_{v}}$ for $x_{v}=\smalltwomat {\vpi_{v}}{}{}1$ or $x_{v}=\smalltwomat {\vpi_{v}^{\pm1 }} {}{}{\vpi_{v}^{\pm 1}}$ and $v\notin S$. In the second case the statement is clear.

Suppose then that  $x_{v}=\smalltwomat {\vpi_{v}}{}{}1$. Decomposing  $Z_{a}(\phi)_{U}=Z_{a^{v}}(\phi^{v})_{U}Z_{a_{v}}(\phi_{v})_{U}$,
the $a^{\rm th}$ coefficient of the left-hand side equals 
$$Z_{a\vpi_{v}}(\phi^{\infty})_{U}+Z(\vpi_{v})_{U}Z_{a/\vpi_{v}}(\vphi^{\infty})_{U}=Z_{a^{v}}(\phi^{v})_{U}\circ (Z_{a_{v}\vpi_{v}}(\phi_{v})_{U}+Z(\vpi_{v})_{U}Z_{a_{v}/\vpi_{v}}(\phi_{v})).$$
It is not difficult to identify this with the $a^{\rm th}$ coefficient of the  right-hand side using   the Cartan decomposition 
\begin{align}\label{cartan}
Z_{a_{v}}(\phi_{v})_{U}=\sum_{0\leq j\leq i\leq v(a), i+j=v(a)} Z\left({\smalltwomat  {\vpi_{v}^{i} }{} {}{\vpi_{v}^{j}}}\right)_{U}
\end{align}
and the relation 
$$    Z\left({\smalltwomat  {\vpi_{v} }{} {}{1}}\right)_{U}\circ
 Z\left({\smalltwomat  {\vpi_{v}^{i} }{} {}{\vpi_{v}^{j}}}\right)_{U}
 = Z\left({\smalltwomat  {\vpi_{v}^{i+1} }{} {}{\vpi_{v}^{j}}}\right)_{U} +
Z(\vpi_{v})_{U}  Z\left({\smalltwomat  {\vpi_{v}^{i-1} }{} {}{\vpi_{v}^{j}}}\right)_{U}.$$
valid whenever $i>0$.\footnote{Note that a term with $i=0$ only appears in \eqref{cartan}  in the case $v(a)=0$, which is easily dealt with separately.}
\end{proof}

\subsection{Geometric kernel function}\label{sec gk}
Fix a point $P\in X^{E^{\times}}(E^{\rm ab})$ as in the Introduction, and for any $h\in\B^{\infty\times}$ denote 
$$[h]:= {\rm T}_{h}P, \qquad [h]^{\circ}= h-\xi_{q(h)},$$ where we identify $\pi_{0}(X_{U, \baar{F}})\cong F_{+}^{\times}\bks {\A}^{\times }/q(U)$ so that $P_{U}$ is in the component indexed by $1$ (the point $T(h)P$ is then in the component indexed by $q(h)$, see \cite[\S 3.1.2]{yzz}).

\begin{lemm}\label{ht on J} Fix $L$-linear Hodge splittings on all the abelian varieties $A'/F$ parametrised by $J$ and let $\langle \, , \rangle_{A', *}$ be the associated local (for $*=v$) or global ($*=\emptyset$) height pairings. There are  unique local and global height pairings 
$$\langle\, ,\, \rangle_{J,*}\colon J^{\vee}(\baar{F})\times J(\baar{F})\to \Gamma_{F}\hat{\otimes} L$$
such that for any $A'$ and $f'_{1}\in \pi_{A'}$, $f'_{2}\in \pi_{A'}^{\vee}$, and any $P_{1}\in J^{\vee}(\baar{F})$, $P_{2}\in J(\baar{F})$ 
$$\langle P_{1}, P_{2}\rangle_{*} =\langle f_{2}'^{\vee}\circ f_{1}'( P_{1}), P_{2}\rangle_{ A', *}.$$
\end{lemm}
\begin{proof} For each fixed level $U$, there is a decomposition $J_{U}^{\vee}\sim \oplus_{A'} A'^{\vee}\otimes \pi_{A'^{\vee}}^{\vee,U}$ in the isogeny category of abelian varieties, induced by $P_{A'^{\vee}}\otimes f'^{\vee}\mapsto f'^{\vee}(P_{A'})$. Then the Hodge splittings on each $A'$ induce Hodge splittings on $J_{U}^{\vee}$. The associated pairing on $J_{U}^{\vee}\times J_{U}$ is then the unique one satisfying the required property by the projection formula for heights (see \cite{MT}). The same formula implies the compatibility with respect to changing $U$.
\end{proof}

We consider the pairing given by the Lemma associated with arbitrary Hodge splittings on $V_{p}A'$ for $A'\neq A$, and any splittings on $V_{p}A\otimes L=\oplus_{\frakp'}V_{\frakp'}A\otimes_{M_{\frakp'}} L$ which induce the canonical one on $V_{\frakp}A$. The subscript $J$ will be generally omitted when there is no risk of confusion.

Let $\phi$ be a Schwartz function, and $\wtil{Z}(\phi)$ be as above. Each $\wtil{Z}_{a}(\phi)$ gives a map
$$\wtil{Z}_{a}(\phi)\colon J(\baar{F})_{\Q}\to J^{\vee}(\baar{F})_{\Q}$$
by the action of Hecke correspondences.  When $a$ has trivial components at infinity and $\phi^{\infty}$ is standard, we write  $\tZ_{a}(\phi^{\infty}):= \tZ_{a}(\phi)$. 
Then for $g\in \GL_{2}(\A), h_{1}, h_{2}\in \B^{\infty\times}$ we define the height generating series
$$\wtil{Z}(g, h_{1}, h_{2}, \phi^{\infty}):=\langle \wtil{Z}(g, \phi)[h_{1}]^{\circ}, [h_{2}]^{\circ}\rangle.$$

\begin{prop}\label{is mod} The series  $\wtil{Z}(g, (h_{1}, h_{2}, \phi)$ is well-defined independently of the choice of the point $P$.  It is invariant under the left action of $T(F)\times T(F)$ and it belongs to the space of weight~$2$  cuspforms $S_{2}(K',\Gamma_{F}\hat{\otimes} L)$ 
for a suitable open compact subgroup $K'\subset \GL_{2}(\A^{\infty})$.
\end{prop}
\begin{proof}
We  explain the modularity with coefficients in the $p$-adic vector space $\Gamma_{F}\hat{\otimes}L$. Suppose that $U$ is small enough so that $\phi$ is invariant under $U$ and $U$ is invariant under the conjugation action of $h_{1}$, $h_{2}$.    Pick a finite abelian extension $E'$ of $E$ such that $P\in X_{U}(E')$ and a basis $\{z_{i}\}$ of $J^{\vee}(E')_{\Q}$, and let $e_{i}\colon J^{\vee}(E')_{\Q}\to\Q$ be the projection onto the line spanned by $z_{i}$. Then we can write 
$$\wtil{Z}(g, (h_{1}, h_{2}), \phi^{\infty})= \sum_{i}\langle z_{i}, [h_{2}]^{0}\rangle \,  \lambda_{i}(\wtil{Z}(g, \phi^{\infty}))$$
where $\lambda_{i}(T)= e_{i}(T[h_{1}]^{\circ})$. Each summand $\lambda_{i}(\wtil{Z}(g, \phi^{\infty}))$ is automorphic by Theorem \ref{modularity}, and in fact a holomorphic cuspform by \cite[Lemma 3.19]{yzz} (the weight can be easily computed  from the shape of $\phi_{\infty}$). The other statements  are   proved in \emph{loc. cit.} too.
\end{proof}

Define 
\begin{align*}
\wtil{Z}(g, \phi^{\infty}, \chi)&:= \int_{T(F)\bks T(\A)/Z(A)}^{*} \chi(t) \wtil{Z}(g, (t,1), \phi^{\infty})\, {d^{\circ} t   }   \\
&=\int_{T(F)\bks T(\A)/Z(A)}^{*} \chi(t) \wtil{Z}(g, (1,t^{-1}), \phi^{\infty})\, d^{\circ} t    \\
&= \int_{T(F)\bks T(\A)/Z(A)}\chi(t) \wtil{Z}_{\omega^{-1}}(g, (1,t^{-1}), \phi^{\infty})\,   {d^{\circ} t}
\end{align*}
where 
$$\wtil{Z}_{\omega^{-1}}(g,  (1,t^{-1}),\phi^{\infty})=\dashint_{Z(\A)}\omega^{-1}(z)\wtil{Z}(g, (1,z^{-1}t^{-1}), \phi^{\infty})\, dz.$$

Note that  we have 
\begin{align}\label{Znorm}
\tZ(\phi^{\infty}, \chi)=  |D_{E}|^{1/2}\tZ^{\rm [YZZ]}(\phi^{\infty}, \chi)\end{align}
 if $\tZ^{\rm [YZZ]}(\phi^{\infty}, \chi)$
 is the function denoted by $\tZ(\chi, \phi)$ in \cite[\S\S\ 3.6.4, 5.1.2]{yzz}.

\subsection{Arithmetic theta lifting and kernel identity}\label{sec:KI}  Similarly to  \cite{yzz}, we conclude this section  by reducing  our main theorem to the form in which we will prove, namely as an identity between two kernel functions. The fundamental  ingredient is the following theorem of Yuan--Zhang--Zhang.
\begin{theo}[Arithmetic theta lifting]\label{atl} Let $\sigma_{A}^{\infty}$ be the $M$-rational automorphic representation of $\GL_{2}(\A)$ attached to $A$.
For any $\vphi^{}\in \sigma^{\infty}$, we have 
$$(\vphi,\tZ(\phi^{\infty}))_{\sigma^{\infty}}={\rm T}_{\rm alg}(\theta(\vphi, \phi^{\infty})).$$
in $\Hom(J, J^{\vee})\otimes M$.

For any $\vphi^{\infty}\in \sigma^{\infty}$, we have 
$$(\vphi,\tZ(\phi^{\infty}))_{\sigma^{\infty}}=|D_{F}|{\rm T}_{\rm alg}(\theta(\vphi, \phi^{\infty})).$$
in $\Hom(J, J^{\vee})\otimes M$. 

Let $\iota_{\frakp}\colon M\into M_{\frakp}\subset L$. For each $\vphi^{p}\in \sigma_{A}^{p\infty}\otimes L$, completing $\vphi^{p}$ to a normalised $(\Up_{v}^{*})_{v\vert p}$-eigenform $\vphi\in\sigma\otimes L$ as before Proposition \ref{p-pet},  we have 
$$\lf(\tZ(\phi^{\infty}))= |D_{F}|  {\rm T}_{\rm alg}(\theta_{\iota_{\frakp}}(\vphi, \alpha(\vpi)^{-r}w_{r}\phi^{\infty})) $$
for any sufficiently large $r\geq \underline{1}$.

\end{theo}
\begin{proof} In the first identity, both sides in fact belong to $M(\alpha)$, and the result holds if and only if it holds after applying any embedding $\iota\colon M(\alpha)\into \C$. It is then equivalent to  \cite[Theorem 3.22]{yzz} via Proposition \ref{p-pet} and \cite[Proposition 3.16]{yzz}. The second identity follows from the first one and the properties of $\lf$.
\end{proof}

We can now rephrase the main theorem in the form of the following kernel identity.

\begin{theo}[Kernel identity]\label{yzz5} Let  $\vphi^{p}\in \sigma_{A}^{p\infty}$ and let $\phi^{p\infty}\in \bcalS({\bf V}^{p\infty}\times
\A^{p\infty,\times})$. For any  compact open subgroup $U_{T,p}=\prod_{v}U_{T,v}\subset 1+(\prod_{v\vert p}\vpi_{v}) \OO_{E,p}$ such that $\chi_{p}|_{U_{T,p}}=1$, let $\phi^{\infty}=\phi^{p\infty}\phi_{p,U_{T,p}}$ where $\phi_{p , U_{T,p}}=\otimes_{v\vert p}\phi_{v, U_{T,v}}$ with 
$$\phi_{v, U_{T,v}}(x,u)= \delta_{1, U_{T,v}\cap {\bf V}_{1}}(x_{1})\one_{\OO_{{\bf V}_{2}}}(x_{2})\one_{d_{v}^{-1}\OO_{F,v}^{\times}}$$
for $\delta_{1, U_{T,v}}$  as in \eqref{phi1}.

Suppose that all primes $v\vert p$ split in $E$. Then we have
$$\lf({\rm d}_{F}\calI(\phi^{p\infty}; \chi))= {2 |D_{F}|  L_{(p)}(1, \eta)}\cdot   \lf(\tZ(\phi^{\infty}, \chi)).$$
\end{theo}
The proof will occupy  most of the rest of the paper (cf. the very end of \S\ref{dec gk} below).

\begin{prop}\label{sufficient} If Theorem \ref{yzz5} is true for some $(\vphi^{p}, \phi^{p\infty})$ such that for all $v\nmid p\infty$ the local integral $R_{v}(W_{v}, \phi_{v}, \chi_{v})\neq 0$, then it is true for all $(W^{p}, \phi^{p\infty})$, and Theorem \ref{B} is true for all $f_{1}\in \pi$,  $f_{2}\in\pi^{\vee}$.
\end{prop}
\begin{proof} 
Consider the identity
\begin{align}\label{3.15}
\langle {\rm T}_{\rm alg, \iota_{\frakp}}(f_{1}\otimes f_{2}) P_{\chi}, P_{{\chi}^{-1}}\rangle_{J}
={      \zeta_{F}^{\infty}(2)\over 2 (\pi^{2}/2)^{[F:\Q]}  |D_{E}|^{1/2}  L(1, \eta)} 
\prod_{v|p}Z_{v}^{\circ}( \alpha_{v}, \chi_{v})^{-1} 
\cdot{\rm d}_{F} L_{p, \alpha}(\sigma_{A,E})( \chi)\cdot Q(f_{1}, f_{2}, \chi)
\end{align}
where $\iota_{\frakp}\colon M\into L(\chi)$, 
and
we set
$$P_{\chi} =\dashint_{[T]} {\rm T}_{t}(P-\xi_{P})\chi(t)\, dt \in J(\baar{F})_{L(\chi)}.$$
The identity \eqref{3.15} is equivalent to Theorem \ref{B} by Lemma \ref{ht on J}, but it has the advantage of making sense,  by linearity, for  any element of $\pi\otimes\pi^{\vee}$. By the multiplicity one result,  it suffices to prove it for a single element of this space which is not annihilated by the functional $Q(\cdot, \chi)$. Such element will arise as a Shimizu lift. (The similar assertion on the validity of Theorem \ref{yzz5}  for all $(\vphi^{p}, \phi^{p\infty})$ follows from the uniqueness of  the Shimizu lifting.)

By  \eqref{def plf}, we can write
$$\lf({\rm d}_{F}\calI(\phi^{p\infty}; \chi))=
 {\rm d}_{F}L_{p, \alpha}(\sigma_{E})(\chi)
 \prod_{v\nmid p \infty} \mathscr{R}_{v}^{\natural}(W_{v}, \phi_{v}; \chi_{v})$$
(note that as the functional $\lf$ is bounded, we can interchange it with the differentiation;  the fact that the Leibniz rule does not introduce other terms follows from the vanishing of $\calI(\phi^{p\infty}; \chi)$, which will be shown in Proposition \ref{6.1-2-3}.\ref{6.2} below). By Lemma \ref{Qtheta}, this equals
$$
{ |D_{F}|^{3/2}  \zeta_{F}^{\infty}(2)  \over (\pi^{2}/2)^{[F:\Q]} }
\prod_{v\vert p}
 Q_{v}(\theta_{v}(W_{v}, \alpha(\vpi_{v})^{-r_{v}}_{v} w_{r,v}^{-1}\phi_{v}),\chi_{v})^{-1}
\cdot {\rm d}_{F}L_{p, \alpha}(\sigma_{E})(\chi)
\cdot  Q(\theta_{\iota_{\frakp}}(\vphi,\alpha(\vpi)^{-r}w_{r}^{-1}\phi), \chi).$$

For the geometric kernel, by Theorem \ref{atl} and the calculation of \cite[\S 3.6.4]{yzz},
we have 
\begin{align*}
 \lf(\tZ(\phi^{\infty}, \chi))=2 |D_{F}|^{1/2} |D_{E}|^{1/2}    L(1, \eta)\langle {\rm T}_{\rm alg, \iota_{\frakp}}(\theta_{\iota_{\frakp}}(\vphi,\alpha(\vpi)^{-r}w_{r}^{-1}\phi)) P_{\chi}, P_{\chi}^{-1}\rangle.
 \end{align*}
 Then \eqref{3.15} follows from Theorem \ref{yzz5} provided we show that, for all $v\vert p$, 
$$ Q_{v}(\theta_{v}(W_{v}, \alpha_{v}(\vpi_{v})^{-r_{v}} w_{r,v}^{-1}\phi_{v}),\chi_{v}) = L(1, \eta_{v})^{-1} \cdot Z^{\circ}_{v}( \chi_{v}).$$
This is proved by explicit computation in Proposition \ref{Qv vs}.
\end{proof}

\section{Local assumptions}\label{local ass}
We list here the local assumptions which simplify the computations, while implying the desired identity in general. We recall on the other hand the essential assumption, valid until the end of this paper, that all primes $v\vert p$ split in $E$. 

 Let $S_{F}$ be the set of \emph{finite} places of $F$. We partition it as 
$$S_{F}=S_{\text{non-split}}\cup S_{\rm split}$$
with the obvious meaning according to the behaviour in $E$, and further as 
$$S_{F}=S_{p}\cup S_{1}\cup S_{2} \cup (S_{\text{non-split}}- S_{1})\cup (S_{\rm split}-S_{p}-S_{2}), $$
where:
\begin{itemize}
\item $S_{p}\subset S_{\textup{split}}$ is the set of places above $p$;
\item $S_{1}$ is a finite subset of $S_{\text{non-split}}$ containing all places where $E/F$ or $F/\Q$ is ramified, or $\sigma$ is not an unramified principal series,  or $\chi$ is ramified, or $\B$ is ramified; we assume that $|S_{1}|\geq 2$;
\item $S_{2}$ consists of two places in $S_{\rm split}-S_{p}$ at which $\sigma$ and $\chi$ are unramified.
\end{itemize}

We further denote by $S_{\infty}$ the set of archimedean places of $F$.

\subsection{Assumptions away from $p$}\label{local ass away p}
Consider the following assumptions from \cite[\S 5.2]{yzz}.

\begin{enonce}{Assumption}[\emph{cf. }{\cite[Assumption 5.2]{yzz}}]\label{yzz5.2} The Schwartz function $\phi=\otimes \phi_{v}\in \bcalS(\B\times \A^{\times})$ is a pure tensor, $\phi_{v}$ is standard for any $v\in S_{\infty}$, and $\phi_{v}$ has values in $\Q$ for any $v\in S_{F}$.
\end{enonce}

\begin{enonce}{Assumption}[{\cite[Assumption 5.3]{yzz}}]\label{yzz5.3} For all $v\in S_{1}$, $\phi_{v}$ satisfies
$$\phi_{v}(x, u)=0 \quad  \text{if }\ v(uq(x))\geq -v(d_{v})\ \text{or}\ v(uq(x_{2}))\geq -v(d_{v}).$$
\end{enonce}

\begin{enonce}{Assumption}[{\cite[Assumption 5.4]{yzz}}]\label{yzz5.4} For all $v\in S_{2}$, $\phi_{v}$ satisfies
$$r(g)\phi_{v}(0, u)=0 \qquad \text{for all}\ g \in \GL_{2}(F_{v}), u\in F_{v}^{\times}.$$
\end{enonce}
See \cite[Lemma 5.10]{yzz} for an equivalent condition.

\begin{enonce}{Assumption}[{\cite[Assumption 5.5]{yzz}}]\label{yzz5.5} For all $v\in S_{\textup{non-split}}-S_{1}$, $\phi_{v}$ is the standard Schwartz function $\phi_{v}(x, u)=\one_{\OO_{\B_{v}}}(x)\one_{d_{v}^{-1}\OO_{F,v}^{\times}}(u)$.
\end{enonce}

\begin{enonce}{Assumption}[\emph{cf.} {\cite[Assumption 5.6]{yzz}}]\label{yzz5.6} The open compact subgroup $U^{p}=\prod_{v\nmid p}U_{v}\subset \B(\A^{p\infty})$ satisfies the following:
\begin{enumerate}[label={\roman*.}]
\item\label{(i)} $U_{v}$ is of the form $(1+\vpi^{r}_{v}\OO_{\B_{v}})^{\times}$ for some $r\geq 0$;
\item\label{(ii)} $\chi$ is invariant under $U_{T}^{p}:=U^{p}\cap T(\A^{p\infty})$;
\item\label{(iii)} $\phi$ is invariant under $K^{p}=U^{p}\times U^{p}$;
\item\label{(iv)} $U_{v}$ is maximal for all $v\in S_{\textup{non-split}}-S_{1}$ and all $v\in S_{2}$;
\item $U^{p}U_{0,p}$ does not contain $-1$;
\item $U^{p}U_{0,p}$ is sufficiently small so that each connected component of the complex points of the Shimura curve $X_{U}$ is an unramified quotient of $\mathfrak{H}$ under the complex uniformisation.
\end{enumerate}
Here we have denoted by $U_{0,p}\subset \B^{\times}_{p}$ the maximal compact  subgroup.
\end{enonce}

See \cite[\S 5.2.1]{yzz} for an introductory discussion of the effect of those assumptions.

\begin{lemm}\label{sufficient2}
 For each  $v\nmid p\infty$, there exist $W_{v}\in \sigma_{v}$ and a Schwartz function $\phi_{v}$ satisfying all of  Assumptions \ref{yzz5.2}--\ref{yzz5.6} such that
$$R_{v}^{\natural}(W_{v}, \phi_{v})\neq 0.$$
For all but finitely many places $v$, we can take $W_{v}$ to be an unramified vector and $\phi_{v}$ to be the standard Schwartz function.
\end{lemm}
\begin{proof}  The existence of the sought-for pairs $(W_{v}, \phi_{v})$ is proved in \cite[Proposition 5.8]{yzz}. The second assertion follows from the unramified calculation Lemma \ref{3.6.2}.
\end{proof}

\subsection{Assumptions at  $p$}\label{local ass at p}
We make some further  assumptions at the places $v\in S_{p}$. After stating the restrictions on $\phi_{v}$ and $U_{v}$, we will impose at the end of this section some restrictions on $\chi_{v}$ and on choices of a $p$-adic logarithm. 

 Concerning $(\phi_{v}, U_{v})$, we need two conditions. On  the one hand, that the centre of the open compact subgroup $U_{v}$ is sufficiently large so that, roughly speaking, for \emph{all but finitely many} characters $\chi$, no nonzero vector in a $T(F_{v})$-representation can be both  $\chi_{v}$-isotypic and invariant under $U_{T,v}:=U\cap T(F_{v})$; we will apply this in particular for the space generated by the Hodge classes on $X_{U}$. On the other hand, we need $\phi_{v}$ to be sufficiently close (a condition depending on $\chi_{v}$) to a Dirac delta, so as to match the Schwartz functions used in the construction  of the $p$-adic analytic kernel. As stated, this is apparently incompatible with the previous condition.  However, as $\chi_{v}|_{F_{v}^{\times}}=\omega^{-1}_{v}$ (fixed), an agreeable compromise can be found. 
We therefore state two distinct assumptions; while we will eventually work with the second assumption (the ``compromise''), it will be  convenient  to  reduce some proofs to the situation of the first one.

\begin{enonce}{Assumption}\label{ass at p1} For each $v\in S_{p}$, $U_{v}$ and $\phi_{v}$ satisfy the following:
\begin{enumerate}[label={\roman*.}]
\item\label{r11} the subgroup ${U}_{v}=1+\vpi_{v}^{r_{v}}\OO_{\B_{v}}$  for some $r_{v}\geq 1$, 
\item\label{r22} $\chi_{v}$ is invariant under $U_{T,v}$ 
\item $\alpha_{v}$ is invariant under $q(U_{v})$;
\item the Schwartz function is
$$\phi_{v}(x,u)= \delta_{1, U_{T,v}}(x_{1})\one_{\OO_{{\bf V}_{2}}}(x_{2})\delta_{q(U)}(u)$$
where, as in { \eqref{phi1}},
$$\delta_{1, U_{T,v}}(x_{1}):= {\vol(E_{v})\over \vol(U_{T,v}  )} \one_{U_{T,v}}(x_{1}) $$
and  
$\delta_{q(U)}(u) ={ \vol(\OO_{F}^{\times})\over \vol(q(U))}\one_{q(U)}(u).$
\end{enumerate}
\end{enonce}

\begin{enonce}{Assumption}\label{ass at p2} For each $v\in S_{p}$,  $U_{v}$ and $\phi_{v}$ satisfy: 
\begin{enumerate}[label={\roman*.}]
\item  $U_{v}=U_{F,v}^{\circ}\wtil{U}_{v}$ with $U_{F,v}^{\circ}= (1+\vpi_{v}^{n_{v}}\OO_{F,v})^{\times}\subset Z(F_{v})\subset \B_{v}^{\times}$ for some $n_{v}\geq 1$, and
  $\wtil{U}_{v}=1+\vpi_{v}^{r_{v}}\OO_{\B,v}$ satisfies \ref{r11}-\ref{r22}  of Assumption \ref{ass at p1}; 
\item\label{romeg} $\omega_{v}$ is invariant under $U_{F,v}^{\circ}$;
\item  $q(U_{v})\subset (U_{F,v}^{\circ})^{2}$; 
\item $\alpha_{v}$ is invariant under $(U_{F,v}^{\circ})^{2}$;
\item the Schwartz function $\phi_{v}$ is 
\begin{align}\label{phi ass2}
\phi_{v}:= \dashint_{\OO_{F,v}^{\times}} \omega_{v}(z) r((z,1))\wtil{\phi}_{v}\, dz
\end{align}
where $\wtil{\phi}_{v}$ is as in Assumption \ref{ass at p1} for ${\wtil{U}_{v}}$;
\end{enumerate}
\end{enonce}

\begin{rema} By \ref{romeg}, the function $\phi_{v}$  in Assumption \ref{ass at p2} is invariant under $K_{v}=U_{v}\times U_{v}$.  The subgroup $U_{F,v}^{\circ}$ in Assumption \ref{ass at p2} can be chosen \emph{independently of $\chi$}.
\end{rema}

In view of the previous remark, we can introduce the following assumption after fixing  $U_{F,v}^{\circ}$.

\begin{enonce}{Assumption}\label{ass on chi} The character $\chi$ is \emph{not} invariant under $V_{p}^{\circ}:=\prod_{v\vert p}q^{-1}(U_{F,v}^{\circ})\subset \OO_{E,p}^{\times}$. 
\end{enonce}

\begin{lemm} The set of finite order characters $\chi\in\Y$ which do not satisfy Assumption \ref{ass on chi}  is finite.
\end{lemm}
\begin{proof} Recall that by definition $\Y=\Y_{\omega}(V^{p})$ parametrises some $V^{p}$-invariant characters for the  open compact subgroup $V^{p}\subset E_{\A^{p\infty}}^{\times}$ fixed (arbitrarily) in the Introduction. Then a character $\chi$ as in the Lemma factors through 
$$E_{\A^{\infty}}^{\times}/\baar{E^{\times}V^{p}V^{\circ}_{p}  },$$
a finite group.
\end{proof}

\paragraph{$p$-adic logarithms} Recall that a $p$-adic logarithm valued in a finite extension  $L$ of  $\Q_{p}$ is a continuous homomorphism
$$\ell\colon \Gamma_{F}\to L;$$
we call it \emph{ramified} if for all $v\vert p$ the restriction $\ell_{v}:=\ell|_{F_{v}^{\times}}$ is ramified, i.e. nontrivial on $\OO_{F,v}^{\times}$.  
\begin{lemm}\label{asslog} For any finite extension $L$ of $\Q_{p}$, the vector space of continuous homomorphisms $\Hom(\Gamma_{F}, L)$ admits a basis consisting of ramified logarithms.
\end{lemm}
\begin{proof}  If $F=\Q$ then $\Hom(\Gamma_{\Q}, L)$ is one-dimensional with generator the cyclotomic logarithm $\ell_{\Q}$, which is ramified. For general $F$, $\ell_{\Q}\circ N_{F/\Q}\colon \Gamma_{F}\to \Gamma_{\Q}\to \Q_{p}$ is ramified (and it generates $\Hom(\Gamma_{F}, L)$ if the Leopoldt conjecture for $F$ holds). Any other logarithm $\ell$ can be written as $\ell= a\ell_{\Q}\circ N_{F/\Q} +(\ell-a \ell_{\Q}\circ N_{F/\Q})$ for any $a\in L$; for all but finitely many values of $a$, both summands are ramified.
\end{proof}

\section{Derivative  of the analytic kernel }\label{sec: expand an}

For this section, we retain all the notation of \S\S\ref{sec: eis series}-\ref{sec: an ker}, and we keep the assumption that ${\bf V}$ is incoherent.  We assume that all $v\vert p$ split in $E$. 

\subsection{Whittaker functions for the Eisenstein series} We start by studying  the incoherent Eisenstein series $\mathscr{E}$.
\begin{prop}\label{6.1-2-3} 
\begin{enumerate} 
\item\label{6.1} Let $a\in F_{v}^{\times}$.
	\begin{enumerate}
	\item If $a$ is not represented by $({\bf V}_{2,v}, uq)$ then $W_{a,v}^{\circ}(g,u, \one)=0$.
	\item\label{(b)} (Local Siegel--Weil formula.)\quad If there exists  $x_{a}\in {\bf V}_{2,v}$ such that $uq(x_{a})=a$, then 
	$$W_{a,v}^{\circ}(\smallmat y{}{}1,u, \one)= \int_{E_{v}^{1}}r(\smallmat y{}{}1, h)\phi_{2,v}(x_{a},u)\, dh$$
	\end{enumerate}
\item\label{6.3} For any $a$, $u\in F^{\times}$, there is a place $v\nmid p$ of $F$ such that $a$ is not represented by $({\bf V}_{2}, uq)$.
\item\label{6.2} For any $\phi_{2}^{p\infty}\in {\calS}({\bf V}_{2}^{p\infty}\times {\A^{p\infty, \times}})$, $u \in F^{\times}$, we have 
$$\mathscr{E}(u, \phi_{2}^{p\infty};\one)=0$$
 and consequently
 $$\mathscr{I}_{F}(\phi^{p\infty};\one)=0, \quad \mathscr{I}(\phi^{p\infty}; \chi)=0$$
 for any $\phi^{p\infty}\in \calS({\bf V}^{p\infty}\times{\A^{p\infty,\times}})$, $\chi\in \Y_{\omega}$. 
\end{enumerate}
\end{prop}
\begin{proof} Part \ref{6.1} is \cite[Proposition 6.1]{yzz} rewritten in our normalisation -- except for (b) when $v\vert p$, which is verified by explicit computation of both sides (recall that $\phi_{2,v}$ is standard when $v\vert p$).  Part \ref{6.3} is a crucial consequence of the incoherence, proved in \cite[Lemma 6.3]{yzz}. In view of the expansion of Proposition \ref{eis-exp}, the vanishing is a consequence of the vanishing of the nonzero Whittaker functions (which is implied by the previous local results) and of 
$$W_{0}(u, \one)=-L^{(p)}(0, \eta)\prod_{v } W_{0, v}^{\circ}(u, \one):$$
here we have 
$$L^{(p)}(0, \eta)={L(0, \eta)\over \prod_{v\vert p}L(0, \eta_{v})}=0$$
as $L(0, \eta)$ is defined and nonzero whereas $L(s, \eta_{v})$ has a pole at $s=0$ when $v$ splits in $E$.
\end{proof}

\subsection{Decomposition of the derivative}\label{dec der} 
Fix henceforth a tangent vector $\ell\in \Hom (\Gamma_{F}, L(\chi))\cong T_{\one}\Y_{F}\otimes L(\chi)\cong \mathscr{N}_{\Y/\Y'}^{*}{}_{|\chi}$; we assume that $\ell$ is ramified  when viewed as a $p$-adic logarithm (cf. Lemma \ref{asslog}). For any function $f$ on $\Y_{F}$, we denote
$$f'(\one)=D_{\ell}f(\one)$$
the corresponding directional derivative.

Our goal   is to compute, for any locally constant $\chi$, the derivative 
\begin{align*}
\calI'(\phi^{p\infty}; \chi) &=\int_{[T]}^{*}\chi(t) \calI_{F}'((t,1),\phi^{p\infty}; \one)_{U}\, d^{\circ}t   \\
&=\int_{[T]}^{*}\chi(t) \calI_{F}'((1, t^{-1}),\phi^{p\infty}; \one)_{U}\,d^{\circ}t,
\end{align*}
where the first identity (of $q$-expansions) follows from the vanishing of the values $\calI_{F}(\phi^{p\infty}; \one)$.

We can decompose the derivative into a sum of $q$-expansions indexed by the non-split finite places $v$. For each $u\in F^{\times}$ and each place $v$ of $F$, let $F_{u}(v)$ be the set of those $a\in F^{\times}$ represented by $({\bf V}_{2}^{v}, uq)$; by Proposition \ref{6.1-2-3} we have $W_{a, v}^{\circ}(u, \one)=0$ for each $a\in F_{u}(v)$, and moreover $F_{u}(v)$ is always empty if $v$ splits in $E$.

Then 
$$\mathscr{E}'(u;\one)={2^{[F:\Q] }|D_{F}|^{1/2}\over  |D_{E}|^{1/2} L^{(p)}(1, \eta)}\W_{0}'(u; \one)-{2^{[F:\Q]|D|^{1/2}}\over   |D_{E}|^{1/2} L^{(p)}(1, \eta)}\sum_{\substack{v \text{ non-split}\\ a\in F_{u}(v)} } \W_{a,v}^{\circ}{' }(u; \one)\W_{a}^{\circ, v}(u;\one)\,{\bf q}^{a}$$
For a non-split finite place $v$, let
\begin{align*}
{\mathscr E}'(u, \phi_{2}^{p\infty};\one)(v) &: = -{2^{[F:\Q]}  |D_{F}|^{1/2} \over   |D_{E/F}|^{1/2} L^{(p)}(1, \eta)}\sum_{ a\in F_{u}(v)}  \W_{a,v}' (u; \one)\W_{a}^{\circ, v}(u;\one)\,{\bf q}^{a},\\
\calI_{F}'((t_{1}, t_{2}),\phi^{p\infty};\one)(v)&:= c_{U^{p}}\sum_{u\in \mu_{U^{p}}^{2}\bks F^{\times}} \theta(u,r(t_{1}, t_{2}) \phi_{1})\calE'(u, \phi_{2}^{p\infty};\one)(v), \\
\calI'(\phi^{p\infty};\chi)(v)&:=\int^{*}_{[T]}\chi(t)\calI_{F}'((1, t^{-1}),\phi^{p\infty};\one)(v)\, {d^{\circ }t}\end{align*}
if $\phi^{p\infty}=\phi_{1}^{p\infty}\otimes \phi_{2}^{p\infty}$, with $\phi_{1}$ obtained from $\phi^{p\infty}_{1}$ as in \eqref{phi1},  and extended by linearity in general.

\begin{prop}\label{6.7} Under  Assumption \ref{yzz5.3},   
we have
$$\calI_{F}'( \phi^{p\infty};\one)=\sum_{v \text{\emph{ non-split}}}\calI_{F}'( \phi^{p\infty};\one)(v).$$
 \end{prop}
\begin{proof} By the definitions, we only need to show that under our assumptions we have 
$$\W_{0}{'}(u;\one)=0.$$
This is proved similarly to \cite[Proposition 6.7]{yzz}.
\end{proof}

\subsection{Main result on the derivative} We give explicit expressions for the local components at good places, and identify the local components at bad places with certain coherent theta series coming from nearby quaternion algebras $\B(v)$; these theta series will be orthogonal to all forms in $\sigma$ by the Waldspurger formula and the local dichotomy.

\begin{prop}\label{6.77} Let $v$ be a finite place non-split in $E$. Then for any $(t_{1}, t_{2})\in T(\A)$,  we have 
$$\calI_{F}'((t_{1}, t_{2}),\phi^{p\infty};\one)(v)=2|D_{F}| L_{(p)}(1,\eta) \dashint_{[T]} \mathscr{K}_{\phi^{p\infty}}^{(v)}((tt_{1},tt_{2}))\, dt$$
and 
$$\calI'(\phi^{p\infty};\chi)(v)=2 |D_{F}|  L_{(p)}(1,\eta)  \int_{[T]}^{*} \dashint_{[T]} \mathscr{K}_{\phi^{p\infty}}^{(v)}((t,tt_{1}^{-1}))\, dt\,  {d^{\circ}t_{1}}  $$
where 
\begin{align*} 
\mathscr{K}^{(v)}_{\phi}(y,(t_{1}, t_{2}))&= \mathscr{K}^{(v)}_{r(t_{1},t_{2})\phi}(y)\\
&=c_{U^{p}}\sum_{u\in \mu_{U^{p}}^{2}\bks F^{\times}}\sum_{x\in V-V_{1}}k_{r(t_{1},t_{2})\phi_{v}}(y,x,u)r(\smalltwomat y{}{}1,(t_{1},t_{2}))\phi^{v\infty}(x,u)\,{\bf 	q}^{uq(x)}
\end{align*}
with $k_{\phi_{v}}(y, x,u)$ the linear function in $\phi_{v}$ given when $\phi=\phi_{1,v}\otimes\phi_{2,v}$ by
$$k_{\phi_{v}}(y,x,u):=-{|d|_{v}^{1/2}|D|_{v}^{1/2}\over \vol(E_{v}^{1})}r(\smalltwomat {y_{v}}{}{}1)\phi_{1,v}(x_{1},u)\W_{uq(y_{2}), v}^{\circ}{'}(y,u, \phi_{2,v}).$$

\end{prop}
\begin{proof} This follows from the definitions and the Siegel--Weil formula (Proposition \ref{6.1-2-3}.\ref{(b)}). The computation is as in \cite[Proposition 6.5]{yzz}.
\end{proof}
\begin{lemm} Assume that $\phi^{\infty}$ is $\Q$-valued. For each non-split finite place $v$, the values of the function
$$k_{\phi_{v}}^{\natural}(y, x,u):=\ell(\vpi_{v})^{-1}k_{\phi_{v}}(y,x,u)$$
and the  coefficients of the reduced $q$-expansions
\begin{align*}
\mathscr{K}^{(v)\,\natural}_{\phi^{p\infty}}&:=\ell(\vpi_{v})^{-1} \mathscr{K}^{(v)}_{\phi^{p\infty}},\\
\calI_{F}'{}^{\natural}(\phi^{p\infty})(v)&:=\ell(\vpi_{v})^{-1}\calI_{F}'(\phi^{p\infty})(v)
\end{align*}
belong to $\Q$.
\end{lemm} 
\begin{proof} By 
Lemma \ref{whitt-an}, the local Whittaker function
 $\W_{a,v}^{\circ}(y,u,\phi_{2,v};\chi_{F})$ 
 belongs to 
 $\OO(\Y_{F,v})\cong M[X_{v}^{\pm 1}]$ and actually to $\Q[X_{v}^{\pm 1}]$, where
 $X_{v}(\chi_{F,v}):=\chi_{F,v}(\vpi_{v})$ 
 for any uniformiser $\vpi_{v}$. (Recall that the scheme $\Y_{F,v}$ of \eqref{local Ys} parametrises unramified characters of $F_{v}^{\times}$.)
 Therefore its derivative in the direction $\ell$ is a rational multiple of $D_{\ell}X_{v}=\ell(\vpi_{v})$.
\end{proof}

The following is the main result of this section. It is the direct analogue of \cite[Proposition 6.8, Corollary 6.9]{yzz} and it is proved in the same way, using Proposition \ref{whitt-eis}.\ref{vnotp}. To compare signs with \cite{yzz}, note that in Proposition \ref{6.77} we have preferred to place the minus sign  in the definition of $k_{\phi_{v}}$; and  that our $\ell(\vpi_{v})$, which is the derivative at $s=0$ of $\chi_{F}(\vpi_{v})^{s}$, should be compared with $-\log q_{F,v}$ in \cite{yzz} (denoted by $-\log N_{v}$ there), which is the derivative at $s=0$ of $|\vpi_{v}|^{s}$.

\begin{prop}\label{6.8} Let $v$ be a non-split finite place of $F$, and let $B_{v}$ be the quaternion algebra over $F_{v}$ which is not isomorphic to $\B_{v}$.
\begin{enumerate}
\item\label{ns-s1} If $v\in S_{\textup{non-split}}-S_{1}$, then 
$$k_{\phi_{v}}^{\natural}(1,x,u)=\one_{\OO_{B_{v}}\times\OO_{F_{v}}^{\times}}(x,u){v(q(x_{2}))+1\over 2}.$$
\item\label{at-s1} If $v\in S_{1}$ and $\phi_{v}$ satisfies Assumption \ref{yzz5.3}, 
then  $k_{\phi_{v}}^{\natural}(y,x,u)$ extends to a rational  Schwartz function of $(x,u)\in B_{v}\times F_{v}^{\times}$, and we have the identity of $q$-expansions
$$\mathscr{K}_{\phi}^{(v)\,\natural}((t_{1},t_{2}))={}^{\qqq}\theta((t_{1}, t_{2}), k_{\phi_{v}}^{\natural}\otimes \phi^{v}),$$
where  for any $\phi'$,
$$\theta(g, (t_{1},t_{2}), \phi')=c_{U^{p}}\sum_{u \in \mu_{U^{p}}^{2}\bks F^{\times}}\sum_{x\in V}r(g, (t_{1}, t_{2}))\phi'(x,u)$$
is the usual theta series.
\end{enumerate}
\end{prop}

\section{Decomposition of the geometric kernel and comparison}\label{dec gk}

We establish a decomposition of  the geometric kernel according to the places of $F$, and compare its local terms away from $p$ with the corresponding local terms in the expansion of the analytic kernel. Together with a result on the local components of the gemetric kernel at $p$ proved in~\S\ref{lsp}, this  proves the kernel identity 
of Theorem \ref{yzz5} (hence Theorem \ref{B}) when $\chi$ satisfies Assumption \ref{ass on chi}.

\subsection{Vanishing of the contribution of the Hodge classes}
Fix a  level $U$ as in Assumptions \ref{yzz5.6} and \ref{ass at p2}. 

Recall the  height generating series  
$$\tZ((t_{1}, t_{2}), \phi^{\infty})=\langle \tZ_{*}(\phi^{\infty}) (t_{1} -\xi_{q(t_{1})}), t_{2} -\xi_{q(t_{2})} \rangle,$$
and the geometric kernel function
$$\tZ(\phi^{\infty}, \chi)=\int^{*}_{[T]}\chi(t) \tZ(({1}, t^{-1}), \phi^{\infty})\, dt.$$
They are modular cuspforms with coefficients in $\Gamma_{F}\hat{\otimes}L(\chi)$.

\begin{prop}\label{vanish hodge} 
\begin{enumerate}  
\item\label{vh0} If Assumption \ref{yzz5.4} is satisfied, then 
$$\deg\tZ(\phi^{\infty})_{U, \alpha}=0$$
for all $\alpha\in F^{\times}_{+}\bks \A^{\times}/q(U)$.
\item\label{vh1} If Assumption \ref{yzz5.4} is satisfied, then 
$$\tZ(\phi^{\infty})\xi_{\alpha}=0$$
for all $\alpha\in F^{\times}_{+}\bks \A^{\times}/q(U)$.
\item\label{vh2} If  Assumption \ref{ass on chi} is  satisfied, then
$$\int_{[T]}^{*} \chi(t)\xi_{U, q(t)}\, dt = 0 .$$
\item\label{vh3}  If Assumptions \ref{yzz5.4} and \ref{ass on chi} are both satisfied, then 
 $${}^{\qqq}\tZ(\phi^{\infty}, \chi)_{U}=\langle {}^{\qqq}\tZ_{*}(\phi^{\infty})1, t_{\chi}\rangle,$$
where 
$$t_{\chi}=\int^{*}_{[T]}\chi(t) [t^{-1}]_{U}d^{\circ}t   \in \Div^{0}(X_{U})_{L(\chi)}.$$
\end{enumerate}
\end{prop}
\begin{proof} 
Note first that part \ref{vh3} follows from parts \ref{vh1} and \ref{vh2} after expanding and bringing the integration inside. Part \ref{vh1} follows from Part \ref{vh0} as in \cite[\S 7.3.1]{yzz}, and part \ref{vh0}  is  proved in \cite[Lemma 7.6]{yzz}.

For part \ref{vh2}, note first that  the group $V_{p}^{\circ}$ of Assumption \ref{ass on chi} acts trivially on the Hodge classes; in fact we have $r(1,t^{-1})\xi_{U,\alpha}= \xi_{U, \alpha q(t)}$, and by definition $q(V^{\circ}_{p})\subset U$. On the other hand we are assuming that the character $\chi$ is nontrivial on $V_{p}^{\circ}$. It follows that the integration against $\chi$ on $V_{p}^{\circ}\subset T(\A)$ annihilates the Hodge classes.
\end{proof}

\subsection{Decompositon}
 Let 
$$\ell\colon \Gamma_{F}\to L(\chi)$$
be the ramified logarithm fixed in \S\ref{dec der}.
For the rest of this section and in \S\ref{lsp}, 
we  will abuse   notation by writing  $\tZ(\phi^{\infty}, \chi)$ for the image of $\tZ(\phi^{\infty}, \chi)$  under $\ell\colon \Gamma_{F} \hat{\otimes}L(\chi)\to L(\chi)$.

\begin{lemm}\label{no self int} If Assumption \ref{yzz5.3} is satisfied, then for all $a\in {\A^{S_{1}\infty,\times}}$ and for all $t_{1}$, $t_{2}\in T(\A^{\infty})$,  the support of $Z_{a}(\phi^{\infty})t_{1}$ does not contain $[t_{2}]$. 
\end{lemm}
\begin{proof} This is shown in \cite[\S 7.2.2]{yzz}.
\end{proof}

Let  $\baar{\bf S}{}'= \baar{\bf S}{}'_{S_{1}}(L(\chi))$ be the quotient space, relative to the set of primes $S_{1}$, introduced after the Approximation Lemma \ref{approx lemma}. 

\begin{prop}\label{prop dec g} Suppose that Assumptions \ref{yzz5.3}, \ref{yzz5.4} and \ref{ass on chi}  are satisfied.
If $H$ is any sufficiently large finite extension of $E$ and $w$ is a place of $H$, let $\langle\, ,\, \rangle_{\ell, w}$ be the pairing on $\Div^{0}(X_{U,H})$ associated with $\ell_{w}$ of \eqref{ellw}.  Let $ {}^{\qqq}\tZ_{*}(\phi^{\infty})$ be the image of  $\tZ_{*}(\phi^{\infty})$ in $\baar{\bf S}{}'\otimes {\rm Corr}\,(X\times X)_{\Q}$.
Then  in $\baar{\bf S}{}'$ we have the decomposition
$${}^{\qqq} \tZ( \phi^{\infty}, \chi) =\sum_{v}  \tZ( \phi^{\infty}, \chi)(v)$$
where 
$$\tZ( \phi^{\infty}, \chi)(v)=\sum_{w\vert v} \langle {}^{\qqq}\tZ_{*}(\phi^{\infty})1, t_{\chi}\rangle_{\ell,w}.$$
\end{prop}
\begin{proof}
By  Proposition \ref{vanish hodge}.\ref{vh3}, we have $\tZ_{a}( \phi^{\infty}, \chi) = \langle \tZ_{a}(\phi^{\infty})1, t_{\chi}\rangle$ for all $a\in \A^{\infty,\times}$. If $a\in {\A^{S_{1}\infty,\times}}$, the two divisors have disjoint supports by Lemma \ref{no self int}, and we can decompose their local height according to \eqref{loc gl hts curve}.
\end{proof}

For each place $w$ of $E$, fix an extension $\baar{w}$ of $w$ to $\baar{F}\supset E$, and for each finite extension $H\subset\baar{F}$ of $E$, let $\langle\, ,\, \rangle_{\baar{w}}$ be the pairing associated with $\ell_{\baar{w}}:={1\over [H_{\baar{w}}:F_{v}]}\ell_{v}\circ N_{H_{\baar{w}}/F_{v}}$. The absence of the field $H$ from the notation is justified by the compatibility deriving from \eqref{local ht norm}.
 By  the explicit  description of the Galois action on CM points we have
\begin{align*}
\tZ( \phi^{\infty}, \chi)(v) = {1\over |S_{E_{v}}|}
\sum_{w\in S_{E_{v}}} \dashint_{[T]}
\langle {}^{\qqq}\tZ_{*}(\phi^{\infty})t, tt_{\chi}\rangle_{\baar{w}}\, dt
\end{align*}
in $\baar{\bf S}{}'$, and if $v\nmid p$, by Propositon \ref{compatibility} we can further write 
\begin{multline}\label{eq: decompose Zchi}
{\tZ( \phi^{\infty}, \chi)(v)}
= {\ell(\vpi_{v})\over |S_{E_{v}}|}\sum_{w\in S_{E_{v}}} 
\int_{[T]}^{*}  \dashint_{[T]}   i_{\baar{w}}({}^{\qqq}\tZ_{*}( \phi) t,tt_{1}^{-1})\chi(t_{1})\,  dt d^{\circ}t_{1}  \\
 +  \int_{[T]}^{*}  \dashint_{[T]}   j_{\baar{w}}({}^{\qqq}\tZ_{*}( \phi) t,tt_{1}^{-1})\chi(t_{1})\,  dt {d^{\circ}t_{1}}.
\end{multline}

\subsection{Comparison of kernels}\label{sec comparison}
Recall that we want to show the kernel identity
\begin{align}\label{comp kers}
\lf(\mathscr{I}' (\phi^{p\infty}; \chi))= 2  L_{(p)}(1, \eta)  \lf (\tZ(\phi^{\infty}, \chi) )
\end{align}
of Theorem \ref{yzz5} (more precisely we have here projected both sides of that identity to $L(\chi)$ via $\ell$).

Similarly to Proposition \ref{prop dec g},  we have by  Propositions \ref{6.7}, \ref{6.77} a decomposition of reduced $q$-expansions 
$$\mathscr{I}' (\phi^{p\infty}; \chi)=  \sum_{v \textrm{\ non-split}} \mathscr{I}' (\phi^{p\infty}; \chi)(v)$$
with
\begin{align}\label{intK}
\mathscr{I}'(\phi^{p\infty};\chi)(v)=2 |D_{F}| L_{(p)}(1, \eta) \int^{*}_{[T]} \dashint_{[T]}\mathscr{K}^{(v)}_{\phi^{p\infty}}(t, tt_{1}^{-1})\chi(t_{1})\, dt \, d^{\circ} t_{1},
\end{align}
and the  $q$-expansion $\mathscr{K}^{(v)\natural}_{\phi^{p\infty}}= \mathscr{K}^{(v)}_{\phi^{p\infty}}\cdot\ell(\vpi_{v})^{-1}$ has rational coefficients.

We thus state the main theorem on the local components of the kernel function  from which the  identity \eqref{comp kers} will follow, preceded by a result on the components away from~$p$ which facilitates the comparison with  \cite{yzz}.

\begin{prop}\label{prop local comp} Suppose that all of the assumptions of \S\ref{local ass away p}  are satisfied together with Assumptions \ref{ass at p2}. Then for all $t_{1}, t_{2}\in T(\A)$ we have the following identities of reduced $q$-expansions in $\baar{\bf S}{}'$:
\begin{enumerate}
\item\label{part split} If $v\in S_{\rm split} - S_{p}$, then 
$$i_{\baar{v} } ({}^{\qqq}\tZ_{*}(\phi^{\infty})t_{1}, t_{2})=   j_{\baar{v} } ({}^{\qqq} \tZ_{*}(\phi^{\infty})t_{1}, t_{2}) =0.$$
\item\label{part ns} If $v\in S_{\textup{non-split}}- S_{1}$, then 
$$ 
 i_{\baar{v} } ({}^{\qqq}\tZ_{*}(\phi^{\infty})t_{1}, t_{2})=\mathscr{K}^{(v)\natural}_{\phi^{p\infty}}(t_{1}, t_{2}),
   \qquad  j_{\baar{v} } ({}^{\qqq}\tZ_{*}(\phi^{\infty})t_{1}, t_{2})=0.$$
\item\label{part s1} If $v\in S_{1}$, then there exist Schwartz functions $k_{\phi_{v}}$, $m_{\phi_{v}}$, $l_{\phi_{v}}\in \bcalS(B(v)_{v}
\times F_{v}^{\times})$ depending on $\phi_{v}$ and $U_{v}$ such that:
\begin{align*}
 \mathscr{K}^{(v)\natural}_{\phi^{p\infty}}(t_{1}, t_{2}) 	&= {}^{\qqq}\theta((t_{1}, t_{2}), k_{\phi_{v}}\otimes \phi^{v}),\\
  i_{\baar{v} } ( {}^{\qqq} \tZ_{*}(\phi^{\infty})t_{1}, t_{2}) &= {}^{\qqq}\theta((t_{1}, t_{2}), m_{\phi_{v}}\otimes \phi^{v}).\\
    j_{\baar{v} } ({}^{\qqq} \tZ_{*}(\phi^{\infty})t_{1}, t_{2}) &= {}^{\qqq}\theta((t_{1}, t_{2}), l_{\phi_{v}}\otimes \phi^{v}).
\end{align*}
\end{enumerate}
\end{prop}

Here $B(v)$ is the coherent  nearby quaternion algebra to $\B$ obtained by changing invariants at $v$, and for $\phi'\in \bcalS(B(v)_{\A}\times \A^{\times})$, we have the automorphic theta series
$$\theta(g,(t_{1}, t_{2}),  \phi')=  c_{U^{p}} \sum_{u\in \mu_{U^{p}}^{2}\bks F^{\times}} \sum_{x\in B(v)} r(g, (t_{1}, t_{2}))\phi'(x, u).$$
We  denote by 
\begin{align}\label{cohker}
I( \phi', \chi)(g):=\int_{[T]}^{*}\dashint_{[T]}\chi(t_{1})\theta(g,(t,t_{1}^{-1}t), \phi')\, dt \,  d^{\circ} t_{1}
\end{align}
the associated coherent theta function.

\begin{proof}
Part \ref{part split} is \cite[Theorem 7.8 (1)]{yzz}.  
In part \ref{part ns}, the vanishing of $j_{\baar{v}}$ follows by the definitions; the other identity   is obtained by explicit computation of both sides as in \cite[Proposition 8.8]{yzz}, which gives the expression for  the geometric side;\footnote{Recall  that, on the geometric side, $i_{\baar{v}}$ is the same $\Q$-valued intersection multiplicity both  in \cite{yzz} and in our case.}
 on the analytic side we use the result of Proposition \ref{6.8}.\ref{ns-s1}.
Part \ref{part s1} for $\mathscr{K}^{(v)\natural}_{\phi^{p\infty}}$ is Proposition \ref{6.8}.\ref{at-s1}, whereas for $i_{\baar{v}}$ and $j_{\baar{v}}$ it is  \cite[Theorem 7.8 (4)]{yzz}.
\end{proof}

\begin{theo}\label{theo local comp} Suppose that all of the assumptions of \S\ref{local ass away p}  are satisfied together with Assumptions \ref{ass at p2} and \ref{ass on chi}. Then  we have the following identities of reduced $q$-expansions in $\baar{\bf S}{}'$:
\begin{enumerate}
\item\label{part split b} If $v\in S_{\rm split} - S_{p}$, then 
$$\tZ(\phi^{\infty}, \chi)(v)=0.$$
\item\label{part ns b} If $v\in S_{\textup{non-split}}- S_{1}$, then 
$$\calI'(\phi^{p\infty}; \chi)(v)=2 |D_{F}| L_{(p)}(1,\eta) \tZ(\phi^{\infty}, \chi)(v).$$
\item\label{part s1 b} If $v\in S_{1}$, then there exist Schwartz functions $k_{\phi_{v}}$, $n_{\phi_{v}}\in \bcalS(B(v))_{v}
\times F_{v}^{\times})$ depending on $\phi_{v}$ and $U_{v}$ such that, with the notation \eqref{cohker}:
\begin{align*}
 \calI'(\phi^{p\infty}; \chi)(v)&= {}^{\qqq} I( k_{\phi_{v}}\otimes \phi^{v},\chi),		\\
  \tZ(\phi^{\infty}, \chi)(v)&={}^{\qqq} I(  n_{\phi_{v}}\otimes \phi^{v},\chi)
\end{align*}
\item\label{part sp} The sum 
$$\tZ(\phi^{\infty}, \chi)(p):=\sum_{v\in S_{p}} \tZ(\phi^{\infty}, \chi)(v)$$
belongs to the isomorphic image $\baar{\bf S}\subset \baar{\bf S}{}'$ of the space of $p$-adic modular forms ${\bf S}$, and 
we have
$$\lf(\tZ(\phi^{\infty}, \chi)(p))=0.$$
\end{enumerate}
\end{theo}
\begin{proof}[Proof (to be completed in \S\ref{lsp})]
Parts \ref{part split b}--\ref{part s1 b} follow from Proposition \ref{prop local comp} by integration via \eqref{eq: decompose Zchi}, \eqref{intK}. They imply the identity in  $\baar{\bf S}{}'$
\begin{multline}\label{identity in s'}
2 |D_{F}| L_{(p)}(1, \eta) \tZ(\phi^{\infty}, \chi)(p)= 2 |D_{F}| L_{(p)}(1, \eta){}^{\qqq}\tZ(\phi^{\infty}, \chi)- \sum_{v\nmid p} \tZ(\phi^{\infty}, \chi)(v)\\
=   2 |D_{F}| L_{(p)}(1, \eta){}^{\qqq}\tZ(\phi^{\infty}, \chi)-\calI'(\phi^{\infty}; \chi) - \sum_{v\in S_{1}} {}^{\qqq} I(\phi^{v}\otimes d_{\phi_{v}})
\end{multline}
where $d_{\phi_{v}}=2 |D_{F}| L_{(p)}(1, \eta) n_{\phi_{v}}-   k_{\phi_{v}}$ for $v\in S_{1}$.
As all terms in the right-hand side belong to $\baar{\bf S}$, so does $ \tZ(\phi^{\infty}, \chi)(p)$. The proof of the vanishing statement 
of part \ref{part sp} will be given in \S\ref{lsp}. 
\end{proof}

\medskip

\begin{proof}[Proof of Theorem \ref{yzz5} under Assumption \ref{ass on chi}]
We show that Theorem \ref{yzz5} follows from Theorem \ref{theo local comp}, under the same assumptions. By Proposition \ref{sufficient} and Lemma \ref{sufficient2}, only the Assumption \ref{ass on chi} on the character  $\chi$ is restrictive. Moreover by Lemma \ref{asslog} it is equivalent to show that the desired kernel  identity holds after applying to both sides a ramified logarithm $\ell\colon \Gamma_{F}\to L(\chi)$. 

By \eqref{identity in s'},
$$\lf \left(2|D_{F}|L(1,\eta)\cdot {}^{\qqq}\tZ(\phi^{\infty}, \chi )-  \calI'(\phi^{\infty}; \chi)\right)= 
\sum_{v\in S_{1}}\lf ({}^{\qqq} I(\phi^{v}\otimes d_{\phi_{v}}))
+ \lf(  \tZ(\phi^{\infty}, \chi)(p)).$$ 
The vanishing of the  terms  indexed by $S_{1}$ can be shown as in \cite[\S 7.4.3]{yzz} to follow from the local result of Tunnell and Saito together with Waldspurger's formula. 
The term $\lf(  \tZ(\phi^{\infty}, \chi)(p))=0$ by part \ref{part sp} of Theorem \ref{theo local comp}.
\end{proof}

\section{Local heights at $p$}\label{lsp}
After some preparation in \S\ref{hecke galois}, in \S\ref{annih} we prove the vanishing statement of part 4 of Theorem \ref{theo local comp}. We follow a strategy of    Nekov{\'a}{\v{r}} \cite{nekovar} and Shnidman \cite{ari}.

For each $v\vert p$, fix isomorphisms 
\begin{align}\label{f+f}
E_{v}:=E\otimes_{F} F_{v}\cong F_{v}{\oplus }F_{v}
\end{align}
 and $B_{v}\cong M_{2}(F_{v})$ such that the embedding of quadratic spaces $E_{v}\hookrightarrow B_{v}$ is identified with $(a,d)\mapsto \smalltwomat a{}{}d$; then for the decomposition $\B_{v}={\bf V}_{1,v}\oplus {\bf V}_{2,v}=E_{v}\oplus E_{v}^{\perp}$, the first (resp. second) factor consists of  the  diagonal (resp. anti-diagonal) matrices. Let $w$, ${w}^{*}$ be the places of $E$ above $v$ such that $E_{w}$ (resp. $E_{{w}^{*}}$) corresponds to the projection onto the  first (resp. second) factor under \eqref{f+f}. We fix the extension $\baar{v}=\baar{w}$ of $v$ to $\baar{F}\supset E$ to be any one inducing $w$ on $E$, and we will accordingly view the local reciprocity maps ${\rm rec}_{w}\colon E_{w}^{\times}= F_{v}^{\times}\to \Gal(\baar{F}_{\baar{v}}/F_{v})^{\rm ab}  = \Gal(\baar{F}_{\baar{v}}/E_{w})^{\rm ab}$.

\subsection{Local Hecke and Galois actions on CM points}\label{hecke galois}
Let ${U}=U_{F}^{\circ}\wtil{U}=\prod_{v} {U}_{v}\subset \B^{\infty\times}$ be an open compact subgroup and $\phi\in \bcalS(\B\times \A^{\times})$  satisfying 
 Assumption \ref{ass at p2} for integers $r=(r_{v})_{v\vert p}$. Fix  throughout this subsection a prime $v\vert p$.

By Lemma \ref{uppertriang},  the generating series  ${\wtil Z}(\phi^{\infty})$ is invariant under $K^{1}(\vpi^{r'})_{v}$ for some $r'>0$. We compute the action of  the operator $\Up_{v, *}$  on it.
\begin{lemm}\label{Up on Z}
For each $a\in \A^{\times}$, $v\vert p$,  the $a^{\text th}$ reduced
 $q$-expansion coefficient of $\Up_{v,*}{\wtil Z}(\phi^{\infty})$ equals 
$$Z(\vpi_{v}^{-1}) {\wtil Z}_{a\vpi_{v}^{}}(\phi^{\infty}),$$
where the $(a\vpi_{v}^{})^{\text th}$ $q$-expansion coefficient of  ${\wtil Z}(\phi^{\infty})$ is given in \eqref{za}.
\end{lemm}
\begin{proof}
We have
\begin{align*}
\Up_{v,*}^{}\phi_{v}(x,u)&=|\vpi_{v}|^{}\sum_{j\in\OO_{F,v}/\vpi_{v}^{}}r(\smalltwomat {\vpi_{v}^{}} j{}1) \phi_{v}(x,u)\\
&= |\vpi_{v}|^{}\phi_{v}(\vpi_{v}^{}x, \vpi_{v}^{-1}u)
\end{align*}
under our assumptions on $\phi_{v}$. 

Plugging this in the definition of ${\wtil Z}$ and performing the change of variables $x'=\vpi_{v} x$ we obtain the result.
\end{proof}

We wish to give a more  explicit expression for 
\begin{gather}\label{the sum}{\tZ}_{a\vpi_{v}^{s}}(\phi^{\infty})[1]_{{U}}= c_{{U}^{p}} |a\vpi^{s}| 
\sum_{x\in  U\bks \B^{\infty\times}/U}\phi^{\infty}(x, a\vpi_{v}^{s}q(x)^{-1}) [x]_{{U}}
\end{gather}
as $s$ varies.  Let
$$\Xi(\vpi_{v}^{r_{v}})=\left\{ \twomat abcd \in M_{2}(\OO_{F,v})\, |\, a,d\in 1+\vpi_{v}^{r_{v}}\OO_{F,v}\right\}U_{F,v}^{\circ}.$$
Then  $\Xi(\vpi_{v}^{r_{v}})\subset \B_{v}$ is the image of the support of $\phi_{v}$ under the natural projection $\B_{v}\times F_{v}^{\times}\to \B_{v}$; and 
for each $x_{v}\in {U}_{v}\bks \B_{v}^{\times}/{U}_{v}$, some $x^{v}\in \B^{v\infty,\times}$ such that $x^{v}x_{v}$ contributes to the sum \eqref{the sum} exists if and only if $x_{v}$ belongs to 
$$ \Xi(\vpi_{v}^{r_{v}})_{a\vpi_{v}^{s_{v}}}:=\{x_{v }\in \Xi(\vpi_{v}^{r_{v}})\, ,\quad q(x_{v})\in a\vpi_{v}^{s_{v}}(1+\vpi^{r}\OO_{F,v})\}.$$

\begin{lemm}\label{coset reps}
Let ${U}_{v}$, $\phi_{v}$ be as in Assumption \ref{ass at p2}, and let $a\in F_{v}^{\times}$ with $s_{v}=v(a)\geq  r_{v}$. 
 Then the quotient sets  
${U}_{v}\bks  \Xi(\vpi_{v}^{r_{v}})_{a}/{U}_{v}$, $ \Xi(\vpi_{v}^{r_{v}})_{a}/{U}_{v}$ and ${U}_{v} \bks \Xi(\vpi_{v}^{r_{v}})_{a}$ are in bijection, and for each of them the set  of elements 
\begin{align*}
x_{v} (b_{v},a) := {\twomat 1 {b_{v}} {b_{v}^{-1}(1-a)} 1}\in M_{2}(F_{v})=\B_{v},  \qquad b_{v}\in (\OO_{F,v}/\vpi_{v}^{r_{v}+s_{v}})^{\times}
 \end{align*}
is a complete set of representatives.
\end{lemm}
\begin{proof}
We drop the subscripts $v$.  By acting on the right with diagonal elements belonging to $U$, we can bring any element $x\in\Xi(\vpi^{r})_{a}$ to  one of the form $x(b, a')$ with $b\in \OO_{F}^{\times}$, $a'\in a(1+{\vpi^{r}}\OO_{F})$. The right action of an  element $\gamma\in U$ sends an element $x(b,a)$ to one of the same form $x(b', a')$ if and only if 
$$\gamma=\twomat {1+\lambda \vpi^{r}} {-b\mu\vpi^{r}\over 1-a} {-\lambda \vpi^{r}\over b} {1+\vpi^{r}\mu}$$
for some $\lambda, \mu\in\OO_{F}$; in this case we have 
$$b'= b{1-a\mu\vpi^{r}\over 1-a}, \qquad a'= a\left(1-\vpi^{r}(\lambda +\mu)-\vpi^{2r}{\lambda \mu a\over 1-a\vpi^{r}}\right).$$
The situation when considering the left action of $U$ is analogous (as can be seen by   the symmetry  $b\leftrightarrow b^{-1}(1-u\vpi^{s})$). The lemma follows.
\end{proof}

Henceforth we will just write $x_{v}(b_{v})$ for $x_{v}(b_{v},a)$ unless there is risk of confusion.

\begin{lemm}\label{galois-class} 
Fix $a\in F_{v}^{\times}$ with $v(a)\geq r_{v}$.
\begin{enumerate}
\item\label{oneone}
Let $x=x^{v}x_{v}(b_{v})$, $c_{v}\in \OO_{F,v}=\OO_{E_{w}}$. The action of the Galois element ${\rm rec}_{E}(1+c_{v}\vpi_{v}^{r_{v}})$
is
\begin{align}\label{gal act}
{\rm rec}(1+c_{v}\vpi_{v}^{r_{v}})[x]_{{U}}
=[x^{v}x_{v}(b_{v}(1+c_{v}\vpi_{v}^{r_{v}}))]_{{U}}.\end{align}
\item\label{twotwo}
We have 
$$ \Xi(\vpi_{v}^{r_{v}})_{a} / {U}_{v}=\coprod_{\baar{b}\in(\OO_{F}/\vpi_{v}^{r})^{\times}} 
{\rm rec}_{E}((1+\vpi^{r_{v}}\OO_{F,v}/1+\vpi^{s_{v}}\OO_{F,v}))\,
 x_{v}(b){U}_{v},$$
where $b$ is any lift of $\baar{b}$ to $\OO_{F,v}/\vpi^{r_{v}+s_{v}}$.
\end{enumerate}
\end{lemm} 
\begin{proof}
 Both assertions follow  from the explicit   description of the Galois action on CM points: we have
$${\rm rec}_{E}(1+c_{v}\vpi_{v}^{r_{v}})[x]_{{U}}=\left[\twomat{1+c_{v}\vpi_{v}^{r}}{}{}1 x\right]_{{U}},$$
and a calculation establishes part \ref{oneone}. In view of Lemma \ref{coset reps}, part \ref{twotwo} is then a restatement of the obvious  identity $(\OO_{F_{v}}/\vpi_{v}^{r+s})^{\times}= {1+\vpi^{r}\OO_{F,v}\over 1+\vpi^{s}\OO_{F, v}} (\OO_{F_{v}}/\vpi_{v}^{r})^{\times} $.
\end{proof}

\subsubsection{Norm relation for the generating series}
Let
$$ \tZ_{a}^{v}(\phi^{v}):=c_{U^{p}}\sum_{x^{v}\in \B^{v\infty\times}/U}\phi^{v\infty}(x^{v}, aq(x^{v})^{-1})Z(x^{v})_{{U}}.$$
Then we have
$$\tZ_{a\vpi^{s}}(\phi)[1]_{{U}} = |\vpi_{v}|_{v}^{-2r_{v}}  \sum_{x_{v}}
 \tZ_{a}^{v}(\phi^{v})
[x_{v}]_{{U}}$$
where 
the sum runs over $x_{v}\in  \Xi  (\vpi_{v}^{r_{v}})_{v(a)+v(d)+s_{v}} / {U}_{v}$.

For $s\geq r_{v}$, let  ${H}_{s}\subset E^{\rm ab}$ be the extension of $E$ with norm group 
$$U_{F}^{\circ}U^{v}_{T}(1+\vpi_{w}^{s}\OO_{E,w}),$$
where $U_{T}=U\cap E_{\A^{\infty}}^{\times}$. Let $H_{\infty}=\cup_{s\geq r_{v}}H_{s}$.
If $r_{v}$ is sufficiently large,  for all $s\geq r_{v}$ the extension $H_{s}/H_{r_{v}}$ is totally ramified   at $\baar{w}$, and   we have
\begin{align}\label{lgall}
\Gal({H}_{s}/{H}_{r})\cong \Gal(H_{s_{v},\baar{w}}/H_{r_{v},\baar{w}})\cong (1+\vpi_{v}^{r_{v}}\OO_{F,v}) /  (1+\vpi_{v}^{s}\OO_{F,v}).
\end{align}
For convenience we set
$${H}'_{s}:=H_{r_{v}+s}$$
for any $s\in\{ 0, 1, \ldots, \infty\}$, and  
when $s<\infty$
we denote by $\Tr_{s}$, and similarly later $N_{s}$,  the trace (respectively  norm) with respect to the field extension 
$H'_{s}/H_{0}'$.

\begin{prop}\label{isnorm}
Fix any $a\in \A^{\infty,\times}$
with $v(a) = r_{v}$.
 With the notation just defined, we have 
$$\tZ_{a\vpi^{s}}(\phi)[1]_{{U}}=\sum_{i\in I} \sum_{\baar{ b}\in (\OO_{F,v}/\vpi^{r})^{\times}} \Tr_{s} c_{i} [x^{v}_{i}x_{v}(b,a\vpi^{s})]_{{U}} $$
where the finite  indexing set $I$, the constants $c_{i}\in \Q$ and  the cosets $x_{i}^{v}U^{v}$ are independent of~$s$. Moreover there exists an integer $d\neq 0$ independent of $a$ such that $c_{i}\in d^{-1}\Z$ for all $i$.
\end{prop}
\begin{proof} By Lemma \ref{galois-class}  we can write 
\begin{align*}
\tZ_{a\vpi^{s}}(\phi)[1]_{ {U}} &= |\vpi_{v}|_{v}^{-2r_{v}}  \sum_{\baar{b}\in( \OO_{F}/\vpi_{v}^{r_{v}} )^{\times}}
 \tZ_{a}^{v}(\phi^{v})\,
[\Gal(H'_{s,\baar{w}}/H'_{0, \baar{w}})\cdot x_{v}(b,a\vpi^{s})]_{{U}}\\
&\begin{multlined}
=|\vpi_{v}|_{v}^{-2r_{v}}  \sum_{\baar{b}\in( \OO_{F}/\vpi_{v}^{r_{v}} )^{\times}}
 \Tr_{s} 
(\tZ_{a}^{v}(\phi^{v})  [x_{v}(b,a\vpi^{s})]_{{U}})
 \end{multlined}
 \end{align*}
using \eqref{lgall}, as by construction the correspondence $\tZ_{a}^{v}(\phi^{v})_{ {U}}$ is defined over $H'_{0}$. We obtain the result  by writing $\tZ_{a}^{v}(\phi^{v})_{ {U}} =\sum_{i\in I}c_{i}Z(x^{v}_{i})_{ {U}}$.

 Finally, the existence of $d$ follows from the fact that $\phi^{\infty}$ is a Schwartz function.
\end{proof}

\subsubsection{The extension $H_{\infty, \baar{w}}/E_{{w}}$} 
After deShalit \cite{deShalit},   given a non-archimedean local field $K$, we say that an extension $K'\subset K^{\rm ab}$ of $K$ is a  \emph{relative Lubin--Tate extension} if there is a (necessarily unique) finite unramified extension $K\subset K^{\circ}\subset K'$ such that $K'\subset K^{\rm ab}$ is maximal for the property of being  totally ramified above $K^{\circ}$. By local class field theory, for any relative Lubin--Tate extension $K'$ there exists a unique element   $\vpi\in K^{\times}$ with $v_{K}(\vpi)=[K^{\circ}:K]$ (where $v_{K}$ is the valuation of $K$) such that $K'\subset K^{\rm ab}$ is the subfield
  cut out by $\langle\vpi\rangle\subset K^{\times}$ via the reciprocity map of $K$. We call $\vpi$ the \emph{pseudo-uniformiser} associated with $K'$.

\begin{lemm}\label{LT}
The field $H_{\infty, \baar{w}}$ is the relative Lubin--Tate extension of $E_{w}$ associated with a pseudo-uniformiser $\vpi_{\rm LT}\in E^{\times}_{w}$ which is algebraic over $E$ and  satisfies $q_{w}(\vpi_{\rm LT})=1$. 
\end{lemm}
\begin{proof} 
It is easy to verify that $H_{\infty, \baar{w}}$ is a relative Lubin--Tate extension. We only need to identify the corresponding pseudo-uniformiser $\vpi_{{\rm LT}}$. It suffices in fact to find an element $\theta\in E^{\times}$ satisfying $q(\theta)=1$ and lying in the kernel of ${\rm rec}_{E_{w}}\colon E_{w}^{\times}\to \Gal(H_{\infty, \baar{w}}/E_{w})$, as then $\vpi_{\rm LT}$ must be a root of $\theta$ hence also satisfies the required property.   

 Let $\vpi_{w}\in E_{w}$ be a uniformiser at $w$, and let 
$d=[E_{\A^{\infty}}^{\times}: E^{\times}U_{T}]$, where $U_{T}^{v}$ is as before,  $U_{T,w}\subset \OO_{E,w}^{\times}$ is arbitrary, and $U_{T,w^{*}}$ is identified with $U_{F,v}^{\circ}$ under $F_{v}\cong E_{w^{*}}$. 
Then we can find $t\in E^{\times}$, $u\in U_{T}$ such that $\vpi_{w}^{d}=tu$. 
We show 
that  the image of $\theta:= t/\baar{t}$ in $E_{w}^{\times}$ lies in the kernel of ${\rm rec}_{E_{w}}\colon E_{w}^{\times}\to \Gal(H_{\infty, \baar{w}}/E_{w})$.
Letting $\iota_{w}\colon E_{w}^{\times}\into E_{{\A}^{\infty}}^{\times}$ be the inclusion, we   show equivalently that $\iota_{w}(\theta)$ is in the kernel of ${\rm rec}_{E}\colon E_{\A^{\infty}}^{\times}\to \Gal(H_{\infty, \baar{w}}/E)$ or concretely  that $i_{w}({t/ \baar{t}})\in E^{\times}  U^{v}_{T}U_{F}^{\circ}$. Now we have
$$\iota_{w}({t/ \baar{t})}= t\cdot  u^{v} \iota_{w}(u_{w^{*}})\iota_{w^{*}}(u_{w^{*}}).$$
By construction $u^{v}\in U_{T}^{v}$, and
    $\iota_{w}(u_{w^{*}})\iota_{w^{*}}(u_{w^{*}})$
belongs to $U_{F,v}^{\circ}$. This completes the proof.
\end{proof}

\subsection{Annihilation of local heights}\label{annih}

Suppose still that the open compact  $U$ and the Schwartz function $\phi^{\infty}$
satisfy all of the assumptions of \S\ref{local ass away p} together with 
 Assumption \ref{ass at p2}.
In this subsection, we complete the proof of Theorem \ref{theo local comp} by  showing  that the element $\tZ(\phi^{\infty}, \chi)(p) \in \baar{\bf S}$ is annihilated by $\lf$. 
Let $S$ be a finite set of non-archimedean places of $F$ such that, for all $v\notin S$, all the data are unramified, $U_{v}$ is maximal, and $\phi_{v}$ is standard. Let $K=K^{p}K_{p}$
 be the level of the modular form $\tZ(\phi^{\infty})$, and let 
 $T_{\iota_{\frakp}}(\sigma^{\vee})\in \mathscr{H}^{S}({L})=\mathscr{H}^{S}({M})\otimes_{M, {\iota_{\frakp}}} {L}$ be any element as in Proposition \ref{p-pet}.\ref{idempot}.  By that result, it suffices to prove that \begin{gather}\label{bbb}
 \lf(T_{\iota_{\frakp}}(\sigma^{\vee}) \tZ(\phi^{\infty}, \chi)(p))=0.\end{gather}
  We will in fact prove the following.
 
 \begin{prop}\label{is crit} For all $v\vert p$, the element $T_{\iota_{\frakp}}(\sigma^{\vee})\tZ(\phi^{\infty}, \chi)(v)\in \baar{\bf S}{}'$ is $v$-critical in the sense of Definition \ref{crit}.
 \end{prop}
 
Recall   from \S\ref{up etc} that the commutative ring $\mathscr{H}^{S}(M)$ acts on the space of reduced $q$-expansions ${\bf S}'(K^{p})$ and its quotient $\baar{\bf S}{}'_{S_{1}}(K^{p})$, so that the expression $T_{\iota_{\frakp}}(\sigma^{\vee})\tZ(\phi^{\infty}, \chi)(v)$ makes sense.
Proposition \ref{is crit} implies that  $T_{\iota_{\frakp}}(\sigma^{\vee})\tZ(\phi^{\infty}, \chi)(p)$ is  a $p$-critical element of $\baar{\bf S}$, hence it is annihilated by $\lf$ by Proposition \ref{p-pet}.\ref{Lkillscrit}, establishing \eqref{bbb}.
 
 \medskip

By Lemma \ref{various Ts}, there is a Hecke correspondence ${\rm T}(\sigma^{\vee})_{U}$ on $X_{U}$ (with coefficients in $M$) such that  
$$T_{\iota_{\frakp}}(\sigma^{\vee}){}^{\qqq}Z_{*}(\phi^{\infty})_{U}={\rm T}_{\iota_{\frakp}}(\sigma^{\vee})_{U} \circ Z_{*}(\phi^{\infty})_{U}$$
as correspondences on $X_{U}$. Then 
$T_{\iota_{\frakp}}(\sigma^{\vee})\tZ(\phi^{\infty}, \chi)_{U}(v)$ is an average of 
\begin{align}\label{sagan}
\langle {\rm T}_{\iota_{\frakp}}(\sigma^{\vee})_{U}  {}^{\qqq}\tZ_{*}(\phi^{\infty})_{U}[1], t_{\chi}\rangle_{\baar{w}} 
= \langle {}^{\qqq}\tZ_{*}(\phi^{\infty})_{U}[1], {\rm T}_{\iota_{\frakp}}(\sigma^{\vee})_{U}^{\rm t} t_{\chi}\rangle_{\baar{w}}\end{align}
for $\baar{w}\vert v$. 
 
We study the class of  ${\rm T}_{\iota_{\frakp}}(\sigma^{\vee})_{U}^{\rm t} t_{\chi}$ in $H^{1}_{f}(E, {\rm Ind}^{H}_{E}V_{p}J_{U}|_{\calG_{E}})$, where $H\subset E^{\rm ab}$ is any sufficiently large finite extension. 
Let $L'$ denote $\baar{\Q}_{p}$ or   any sufficiently large finite extension of $\Q_{p}$.  As $\calG_{F}$-representations, we have 
$$V_{p}J_{U}\otimes_{{\Q}_{p}} L' =\bigoplus_{A', \iota'\colon M_{A'}\into L'} \pi_{A'}^{U}\otimes_{M_{A'}}V_{p} A'^{\iota}$$
 for some   pairwise non-isogenous simple abelian varieties $A'/F$ with $\End^{0}A'=M_{A'}$; here $V_{p} A'^{\iota}:= V_{p}A'\otimes_{M_{A'}, \iota} L'$. 
 More generally, let 
\begin{gather}\label{ind H}
V:=  {\rm Ind}^{H}_{E}V_{p}J_{U}|_{\calG_{H}} 
\end{gather}
for a finite extension $H$ of $E$ as above; then we have 
\begin{align}\label{dir sum}
V\otimes_{{\Q}_{p}}L' =\bigoplus_{A', \iota'\colon M_{A'}\into L', \chi'\colon \Gal(H/E)\to L'^{\times}} \pi_{A'}^{U}\otimes_{M_{A'}}V_{p}A'^{\iota}({\chi'}),
\end{align}
where   $V_{p}A'^{\iota}({\chi'}):=V_{p}A'\otimes_{M_{A'}, \iota} L'_{\chi}.$

Let $V':= \pi_{A^{\vee}}^{U}\otimes_{M} V_{p}A^{\vee \iota_{\frakp}}(\chi^{-1})\subset V$, where $\iota_{\frakp}\colon M\into L\subset L(\chi)$ is the usual embedding. Let $V''\subset V$ be its complement in the direct sum \eqref{dir sum}, $V=V'\oplus V''$. 

\begin{prop}\label{is in V'}  The  class of  ${\rm T}_{\iota_{\frakp}}(\sigma^{\vee})_{U}^{\rm t} t_{\chi}$ in $H^{1}_{f}(E, V)$ belongs to the subspace $H^{1}_{f}(E, V')$.
\end{prop}
\begin{proof} 
We may and do replace $L(\chi)$ by a sufficiently large finite extension $L'$.   It is clear that the class  of $t_{\chi}$ belongs to    $H^{1}_{f}(E, V_{p}J_{U}(\chi^{-1}))$. Then it is enough to verify that ${\rm T}_{\iota_{\frakp}}(\sigma^{\vee})_{U}^{\rm t}$ annihilates $\bigoplus_{(A', \iota)\neq (A^{\vee}, \iota_{\frakp})} (V_{p}A'^{\iota})^{m_{A'}}$. Equivalently,
 we show that for  any $(A', \iota)\neq (A, \iota_{\frakp})$ and any $f_{1}\in \pi_{A'}$, $f_{2}\in \pi_{A'}$, 
 $${\rm T}_{{\rm alg}, \iota} (f_{1}, f_{2})\circ {\rm T}_{\iota_{\frakp}}(\sigma^{\vee})^{\rm t}_{U}=0$$
in $\Hom(J_{U}, J_{U}^{\vee})$.

 We have 
 $${\rm T}_{{\rm alg}, \iota} (f_{1}, f_{2})\circ {\rm T}_{\iota_{\frakp}}(\sigma^{\vee})^{\rm t}_{U}
 =  {\rm T}_{\iota_{\frakp}}(\sigma^{\vee})^{\rm t}_{U}\circ {\rm T}_{{\rm alg}, \iota} (f_{1}, f_{2})
 = {\rm T}_{\iota_{\frakp}}(\sigma^{\vee})^{\rm t}_{U}\circ {\rm T}_{{\rm alg}, \iota} (\theta_{\iota}(\vphi', \phi^{\infty}{}'))$$
 for some $\vphi'\in\sigma_{A'}^{\infty}$ and some rational Schwartz function $\phi^{\infty}{}'$. By Theorem \ref{atl}, this can be rewritten as 
$$  {\rm T}_{\iota_{\frakp}}(\sigma^{\vee})^{\rm t}_{U}\circ (\iota\vphi',  \tZ_{*}({\phi^{\infty}}{}'))_{ \sigma_{A'}^{\infty}} 
=     (\iota\vphi', {\rm T}_{\iota_{\frakp}}(\sigma^{\vee})^{\rm t}_{U}\circ \tZ_{*}({\phi^{\infty}}{}'))_{\sigma_{A'}^{\infty}}
$$
using an obvious commutativity.
Applying  Lemma \ref{various Ts} again, we have
$$ {\rm T}_{\iota_{\frakp}}(\sigma^{\vee})^{\rm t}_{U}\circ \tZ_{*}({\phi^{\infty}}{}')=\iota_{\frakp} T(\sigma)\tZ_{*}({\phi^{\infty}}{}'), $$
where in view of \eqref{transpose}, it is easy to see that we are justified in calling $T(\sigma)$ the Hecke operator corresponding to $T(\sigma^{\vee})^{\rm t}_{U}$; i.e., this Hecke operator acts as the idempotent projection onto $\sigma^{K}\subset M_{2}(K, M)$ for the appropriate level $K$. It is then clear that modular forms in the image of $T(\sigma)$ are in the right kernel of $(\, ,\, )_{\sigma_{A'}^{\infty}}$ if $\sigma_{A'}\neq \sigma_{A^{\vee}}$ or equivalently (as $A'\mapsto \sigma_{A'}$ is injective) if $A'\neq A^{\vee}$.

 If $A'=A^{\vee}$, then the expression of interest is the image of  $\vphi'\otimes T(\sigma)\tZ_{*}(\phi^{\infty}{}')\in \sigma_{A^{\vee}}\otimes (\Hom(J_{U}, J_{U}^{\vee})\otimes S_{2}( M))$ under the $M$-linear algebraic Petersson product and two projections applied to  the two factors, induced respectively from $\iota\colon M\otimes L'\to L'$  and $\iota_{\frakp}\colon M\otimes L
'\to L'$. If $\iota\neq \iota_{\frakp}$, their combination is zero.
\end{proof}

\begin{prop}\label{is ess crys} For each $\baar{w}\vert p$ and each $a\in {\A^{S_{1}\infty,\times}}$, the mixed extension $E$ associated with the divisors $\tZ_{a}(\phi^{\infty})_{U}[1]$ and ${\rm T}_{\iota_{\frakp}}(\sigma^{\vee}) t_{\chi}$ is essentially crystalline at $\baar{w}$.
\end{prop}
 \begin{proof} By Proposition \ref{integrality and intersect}, it is equivalent to show that 
 $$m_{\baar{w}}\left(\tZ_{a}(\phi^{\infty})_{U}[1],{\rm T}_{\iota_{\frakp}}(\sigma^{\vee}) t_{\chi}\right)=0.$$
  Under  Assumptions \ref{yzz5.3} and \ref{yzz5.4} (which are local at $S_{1}\cup S_{2}$, hence  unaffected by  the action of ${\rm T}_{\iota_{\frakp}}(\sigma^{\vee})$), this is proved in \cite[Proposition 8.15, \S 8.5.1]{yzz}. 
 \end{proof}

 \begin{proof}[Proof of Proposition \ref{is crit}] For each $\baar{w}\vert v$, by the discussion preceding \eqref{sagan} it suffices to show that  
 $$
 \langle {}^{\qqq}\tZ_{*}(\phi^{\infty})_{U}[1], {\rm T}_{\iota_{\frakp}}(\sigma^{\vee})_{U}^{\rm t} t_{\chi}\rangle_{\baar{w}}$$ 
 is a $v$-critical element of $\baar{\bf S}{}'$, that is (by the definition and Lemma \ref{Up on Z}) that
$$ \langle \tZ_{a\vpi^{s}}(\phi^{\infty})_{U}[1], {\rm T}_{\iota_{\frakp}}(\sigma^{\vee})_{U}^{\rm t} t_{\chi}\rangle_{\baar{w}}\in {q_{F,v}^{s-c}}\OO_{L(\chi)}$$
for all $a$ with $v(a)= r_{v}$ and
some constant
 $c$ independent of $a$ and $s$. 

 We may assume that $\baar{w}$ extends the place $w$ of $E$ fixed above.

By Proposition \ref{isnorm}, 
when $v(a)= r_{v}$
  the divisor $\tZ_{a\vpi^{s}}(\phi^{\infty})_{U}[1]$ is a finite sum, with $p$-adically bounded coefficients, of elements ${\rm Tr}_{s} [x_{j,s}]_{U}$, where ${\rm Tr}_{s}$ denotes the trace map on divisors for the field extension ${H}'_{s}/{H}'_{0}$, and $[x_{j,s}]_{U}\in\Div^{0}(X_{U, {H}'_{s}})$. 

Recall that $\langle\, , \, \rangle_{\baar{w}}=\langle\, , \, \rangle_{\ell_{\baar{w}}}$ for a fixed ${\ell_{\baar{w}}}\colon\baar{F}_{\baar{w}}^{\times}\to L(\chi)$.\footnote{We only need to consider $\ell_{\baar{w}}$ on the field $H'_{0}$ recalled just below.} By \eqref{local ht norm}, we have 
\begin{align}\label{esso}
\langle {\rm Tr}_{s} [x_{j,s}]_{U},  {\rm T}_{\iota_{\frakp}}(\sigma^{\vee})_{U}^{\rm t} t_{\chi}\rangle_{\ell_{\baar{w}}}
= \langle  [x_{j,s}]_{U},  {\rm T}_{\iota_{\frakp}}(\sigma^{\vee})_{U}^{\rm t} t_{\chi}\rangle_{\ell_{\baar{w}} \circ N_{s, \baar{w}}}
\end{align}
where we recall that $N_{s}=N_{s, \baar{w}}$ is   the norm  for ${H}'_{s, \baar{w}}/{H}'_{0, \baar{w}}$. 

Now take the field $H$ of \eqref{ind H} to be  $H={H}'_{s}$, and consider the $\calG_{E_{w}}$-representations $V'\subset V=V_{s}$ over $L(\chi) $ defined above Proposition \ref{is in V'}. By our assumptions,  $V'$ satisfies the condition ${\bf D}_{\rm pst}(V')^{\vphi=1}=0$ and the Panchishkin condition, with an exact sequence of $\calG_{E_{w}}$-representations
$$0\to V'^{+}\to V'\to V'^{-}\to 0.$$
 The representation $V_{s}$ has a natural $\calG_{E}$-stable lattice, namely $T_{s}:={\rm Ind}^{{H}'_{s}}_{E} T_{p}J_{U}|_{\calG_{H'_{s}}}$. Let $T':=T\cap V'$, $T''_{s}:=T_{s}\cap V''$, $T'^{+}:=T'\cap V'^{+}$, $T'^{-}=T'/T'^{+}$; note that $V'$, $T'$, $T'^{\pm}$ are independent of $s$ (hence the notation).

Let $\wtil{E}_{\baar{w}}/E_{w}$ be a finite extension over which $V_{\frakp}A $ (hence $V'$)  becomes semistable.
We may assume that the extension $\wtil{E}_{\baar{w}}\subset E^{\rm ab}$ is abelian and totally ramified (see \cite[Proposition  12.11.5 (iv), Proposition 12.5.10 and its proof]{nek-selmer}). For $s\in\{0, 1, \ldots, \infty\}$, let $\wtil{H}_{s, \baar{w}}':=\wtil{E}_{\baar{w}}{H}_{s, \baar{w}}'$ be the compositum.

  We can then apply Proposition \ref{int panc}, with the fields $\wtil{H}_{s}'$ playing the role of the field denoted there by `$F_{v}$'; 
    together with Proposition \ref{is ess crys}, it implies that   \eqref{esso} belongs to 
$$\frakp_{L(\chi)}^{-(  d_{0}+d_{1,s}+d_{2})} \ell_{\baar{w}} \circ \wtil{N}_{s}  \left(\OO^{\times}_{{H}'_{s}, \baar{w}}\hat{\otimes }\OO_{L(\chi)}\right)
\subset \frakp_{L(\chi)}^{-(d_{0}+d_{1,s}+d_{2}+d')} q_{F,v}^{s-c_{0}}   \OO_{L(\chi)}    $$
where $d_{0}$, $d_{1,s}$, $d_{2}$ are the integers of Proposition \ref{int panc},\footnote{Note that by construction there is  an obvious direct sum decomposition $T_{s}=T_{0}\oplus T^{s-1}$ for a complementary subspace $T^{s-1}$; so that the integer $d_{0}$ is independent of $s$.}, $\wtil{N}_{s}$ is the norm of $\wtil{H}_{s}'/\wtil{H}_{0}'$, 
 and $d'$ accounts for the denominators coming from \eqref{ellw}.
The containment follows from the fact that the extension ${H}'_{s}/H'_{0}$ is  totally ramified at $\baar{w}$ of degree $q_{F,v}^{s}$ so that $\wtil{H}'_{s}/\wtil{H}'_{0}$ has ramification degree at least $q_{F,v}^{s-c_{0}}$ for some constant $c_{0}$.

To complete the proof, we  need to establish  the boundedness of the integer sequence 
$$d_{1,s}={\rm length}_{\OO_{L(\chi)}}  H^{1}(\wtil{H}'_{0,\baar{w}}, T_{s}''^{*}(1)\otimes L(\chi)/\OO_{L(\chi)})_{\rm tors}.$$
As 
\begin{multline*}
H^{1}(\wtil{H}'_{0,\baar{w}}, T_{s}''^{*}(1)\otimes L(\chi)/\OO_{L(\chi)})_{\rm tors}\cong
H^{0}(\wtil{H}'_{0,\baar{w}}, T_{s}''^{*}(1))\\
\subset H^{0}\left(E_{w},  T_{p}J_{U}^{*}(1) \otimes {\OO_{L(\chi)}}[\Gal(\wtil{H}'_{s,\baar{w}}/E_{w})] \right),
\end{multline*}
the boundedness follows from the next Lemma.
\end{proof}

\begin{lemm} Let $\Gamma_{\rm LT}':= \Gal(\wtil{H}'_{\infty, \baar{w}}/E_{w})$. Then 
$$H^{0}(E_{w}, V_{p}J_{U}^{*}(1)\otimes \OO_{L(\chi)}\llb \Gamma_{\rm LT}'\rrb)=0.$$
\end{lemm}

 \begin{proof}
We  use the results and notation of Lemma \ref{LT}.
Note first that we may safely replace $L(\chi)$ by a finite extension $L'$ splitting $E_{w}/\Q_{p}$.  As $V_{p}J_{U}=V_{p}J_{U}^{*}(1)$ is Hodge--Tate,
 we have 
$$H^{0}\left(E_{w},   V_{p}J_{U} \otimes_{\OO_{L(\chi)}}\OO_{L(\chi)}\llb\Gamma_{\rm LT}'\rrb \right)
\subset \bigoplus_{\psi} H^{0}\left(E_{{w}},  V_{p}J_{U}( \psi)\right)( \psi^{-1}),$$
where $\psi$ runs through the Hodge--Tate characters of $\calG_{E}$ factoring through $\Gamma_{\rm LT}'$. Since the latter is a quotient of $E_{w}^{\times}$ dominating $\Gamma_{\rm LT}=\Gal(H_{\infty, \baar{w}}/E_{w})\cong E^{\times}_{w}/\langle \vpi_{\rm LT}\rangle$, we have $\Gamma_{\rm LT}'\cong E^{\times}_{w}/\langle\vpi_{\rm LT}^{e}\rangle$ for some $e\geq 1 $. Then the condition that $\psi$ factor through  $\Gamma_{\rm LT}'$ is equivalent to 
 $\psi\circ{\rm rec}_{E,w}(\vpi_{\rm LT}^{e})=1$ for the 
 psuedo-uniformiser $\omega_{\rm LT}\in E_{w}^{\times}$.
 
 By \cite[Appendix III.A]{serre}, a character $\psi$ of $\calG_{E,w}$ is Hodge--Tate of some weight $\underline{n}\in \Z[\Hom(E_{w}, L')]$  if and only if  the maps $E_{w}^{\times}\to L'^{\times}$ given by $\psi\circ {\rm rec}_{E,w}$ and $x\mapsto x^{-\underline{n}}:=\prod_{\tau\in \Hom (E_{w}, L')} \tau(x)^{-\underline{n}(\tau)}$ coincide near $1\in E_{w}^{\times}$.
By \cite[Proposition B.4 (i)]{conrad}, $\psi$ is crystalline if and only if those maps coincide on $\OO_{E,w}^{\times}$; therefore we may write any Hodge--Tate character $\psi$ as $\psi=\psi_{0}\psi_{1}$ with $\psi_{0}$ of finite order and $\psi_{1}$ crystalline. 

Then, letting $E_{w,0}$ the maximal unramified extension of $\Q_{p}$ contained in $E_{w}$ and $d=[E_{w}: E_{w,0}]$, we can first write
$$H^{0}(E_{w}, V_{p}J_{U}(\psi))= {\bf D}_{\rm crys}(V_{p}J_{U}(\psi))^{\vphi=1}\subset {\bf D}_{\rm crys}(V_{p}J_{U}(\psi))^{\vphi^{d}=1}$$
for the  $E_{w,0}$-linear endomorphism $\vphi^{d}$ (where $\vphi$ is the crystalline Frobenius), and then
$${\bf D}_{\rm crys}(V_{p}J_{U}(\psi))^{\vphi^{d}=1}={\bf D}_{\rm crys}(V_{p}J_{U}(\psi_{0}))^{\vphi^{d}=\lambda_{1}^{-1}}.$$
where $\lambda_{1}\in L'$ is the scalar giving the action of $\vphi^{d}$ on ${\bf D}_{\rm crys}(\psi_{1})$.

By \cite[Theorem 5.3]{mokrane},
all eigenvalues of $\vphi^{d}$ on ${\bf D}_{\rm crys}(V_{p}J_{U}(\psi_{0}))$ are Weil $q_{E,w}$-numbers  of strictly negative weight. 
To conclude that $ {\bf D}_{\rm crys}(V_{p}J_{U}(\psi))^{\vphi^{d}=1}=0$ it thus suffices to 
 show that $\lambda_{1}$ is an algebraic number of weight $0$.

By  \cite[Proposition B.4 (ii)]{conrad}, we have 
$$\lambda_{1}=\psi_{1}\circ{\rm rec}_{E,w}(\vpi_{w})^{-1}\cdot \vpi_{w}^{-\underline{n}}$$
where   $\vpi_{w} \in E_{w}^{\times}$ is any uniformiser and $\underline{n}$ are the Hodge--Tate weights. Writing $\vpi^{m}=u\vpi_{\rm LT}$ for  $m=w(\vpi_{\rm LT})$ and some $u\in \OO_{E,w}^{\times}$, we have
$$\lambda_{1}^{m}=\psi_{1}\circ{\rm rec}_{E,w}(u\vpi_{\rm LT})^{-1}\cdot u^{-\underline{n}}\vpi_{\rm LT}^{-\underline{n}}.$$
 Now $\psi_{1}\circ{\rm rec}_{E,w}(u)= u^{-\underline{n}}$ by the crystalline condition , and $\psi_{1}^{e}\circ{\rm rec}_{E,w}(\vpi_{\rm LT})=1$ as $\psi_{1}$ factors through $\Gamma_{\rm LT}'$.  Hence $\lambda_{1}^{em}=\vpi_{\rm LT}^{-e\underline{n}}$. By Lemma \ref{LT}, $\vpi_{\rm LT}$ is an algebraic number of weight $0$, hence so is $\lambda_{1}$.
\end{proof}

\section{Formulas in anticyclotomic families}

In \S\S\ref{sec: theta elts}-\ref{10.2}, we   prove Theorem \ref{C} after filling in some details in its setup. It is largely a corollary of Theorem \ref{B} (which we have proved for all but finitely many characters), once a construction of the \emph{Heegner--theta elements} interpolating automorphic toric periods and Heegner points is carried out. Finally, Theorem \ref{B} for the missing characters will be recovered as a corollary of Theorem \ref{C}.

\medskip

In \S\ref{sec iwbsd},  we  prove Theorem \ref{iwbsd}.

\medskip

We invite the reader to go back to \S\ref{intro anti} for the setup and notation that we are going to use  in this section (except for the preliminary Lemma \ref{f limit}).

\subsection{A local construction}\label{sec: theta elts}
 Let $L$ and $F$ be finite extensions of $\Q_{p}$, and denote by $v$ the valuation of $F$ and by $\vpi\in F$ a fixed uniformiser. Let $\pi$ be a  smooth representation of $\GL_{2}(F)$ on an $L$-vector space with central character $\omega$ and a stable $\OO_{L}$-lattice $\pi_{\OO_{L}}$. Let $E^{\times}\subset \GL_{2}(F)$ be the diagonal torus. Assume that $\pi $ is nearly ordinary in the sense of Definition  \ref{def-n-ord} with unit  character $\alpha\colon F^{\times}\to L^{\times}$. Let $f_{\alpha}^{\circ}\in \pi-\{0\}$ be any nonzero element satisfying $\Up_{v}^{*}f_{\alpha}^{\circ}=\alpha(\vpi)f_{\alpha}^{\circ}$, which 
   is unique up to multiplication by $L^{\times}$.  For $r\geq 1$, let $s_{r}=\twomat {\vpi^{r}}1{}1$ and 
$$f_{\alpha, r}:=|\vpi|^{-r}\alpha(\vpi)^{-r} s_{r}f_{\alpha}^{\circ}.$$
It is easy to check that $f_{\alpha,r}$ is independent of the choice of $\vpi$, and it is  invariant under $V_{r}=\smalltwomat {1+\vpi^{r}\OO_{F,v}}{}{}1 V_{F}$, where $V_{F}:=\Ker(\omega)\subset Z(F)$.

\begin{lemm}\label{f limit}
\begin{enumerate}
\item
The collection $f_{\alpha, V_{r}}:=f_{\alpha,r}$, for $r\geq 0$, 
defines an element
$$f_{\alpha}=(f_{\alpha,V})\in \varprojlim_{V} \pi^{V},$$
where the inverse system runs over compact open subgroups $V_{F}\subset V\subset E^{\times}$, and the transition maps $\pi^{V'}\to \pi^{V}$
are given by 
\begin{align*}
f&\mapsto \dashint_{V/V'}\pi(t)f\, dt.
\end{align*}
\item \label{bounded-til}
Let $\pi_{\OO_{L}}\subset \pi$ be a $\GL_{2}(F)$-stable $\OO_{L}$-lattice containing $f_{\alpha}^{\circ}$. The collection of elements $\tilde{f}_{\alpha,V_{r}}:=|\vpi|^{r}f_{\alpha,r}=\alpha(\vpi)^{-r}s_{r}f_{\alpha}^{\circ}$ defines an element  
$$\tilde{f}_{\alpha}\in\varprojlim_{V}\pi_{\OO_{L}}^{V}$$ where the transition maps $\pi^{V'}\to \pi^{V}$
are given by 
\begin{align*}
f&\mapsto \sum_{t\in V/V'}\pi(t)f.
\end{align*}
\end{enumerate}
\end{lemm}
\begin{proof}
We need to prove that \begin{align}\label{fff}
\dashint_{V_{r}/V_{r+1}} \pi(t)f_{\alpha,r+1}\, dt=f_{\alpha, r}.
\end{align}
 A set of representatives for $V_{r}/V_{r+1}$ is $\left\{\smalltwomat {1+j\vpi^{r}}{}{}1\right\}_{j\in \OO_{F}/\vpi}$; 
on the other hand  recall that $\Up_{v}^{*}= \sum_{j\in \OO_{F}/\vpi_{v}} \smalltwomat {\vpi} j {}1$. From  the identity
$$  \twomat {1+j\vpi^{r}}{}{}1   \twomat {\vpi^{r+1}}1{}1 =  \twomat {\vpi^{r}}1{}1 \twomat {\vpi} j {}1 \twomat {1+j\vpi^{r}}{}{}1$$
we obtain
$$\dashint_{V_{r}/V_{r+1}}\pi(t) s_{r+1} f_{\alpha}^{\circ}= q_{F,v}^{-1} s_{r}\Up_{v}^{*}f_{\alpha}^{\circ}=|\vpi|^{-1}\alpha(\vpi) s_{r}f_{\alpha}^{\circ},$$
as desired. The integrality statement of part 2 is clear as $\alpha(\vpi)\in \OO_{L}^{\times}$.
\end{proof}

Let us  restore the notation $\pi^{+}=\pi$, $\pi^{-}=\pi^{\vee}$, employing it in the current local setting. Then if $\pi^{+}$ is as in the previous Lemma, it is easy to check that $\pi^{-}$ is also nearly $\frakp$ ordinary.
 Explicitly, 
  the element 
\begin{align}\label{f1}
f^{+, \circ}_{\alpha}(y)=\one_{\OO_{F}- \{0\}}(y)|y|_{v}\alpha_{v}(y)
\end{align}
 in the  $L$-rational subspace\footnote{Cf. \S\ref{univ kir}.} of any Kirillov model of $\pi^{+}_{v}$ satisfies $\Up_{v}^{*}f^{+, \circ}_{\alpha}=\alpha_{v}(\vpi_{v})f^{+, \circ}_{\alpha}$, and the element 
  \begin{align}\label{f2}
 f^{-, \circ}_{\alpha}(y)=\one_{\OO_{F}- \{0\}}(y)|y|_{v}\omega^{-1}(y)\alpha_{v}(y)
 \end{align} in the $L$-rational subspace of any Kirillov model of $\pi^{-}_{v}$
 satisfies 
 $\Up_{v, *}f^{-, \circ}_{\alpha}=\alpha_{v}(\vpi_{v})f^{-, \circ}_{\alpha}$.

We  can then construct an element $f_{\alpha}^{-}=(f^{-}_{\alpha,V})_{V}=(f_{\alpha,r}^{-})_{r}$ with the property of  the previous Lemma
as   $f_{\alpha,r}^{-}:=|\vpi|^{-r}\alpha(\vpi)^{-r}s_{r}^{*}f_{\alpha}^{-,\circ,}$ with
  $s_{r}^{*}= \twomat 1{\vpi_{}^{-r}}{}{\vpi^{-r}}$.

\medskip

\subsubsection{Local toric periods}  
Let us restore the subscripts $v$. Recall the universal Kirillov models $\mathscr{K}(\pi_{v}^{\pm}, \psi_{{\rm univ},v})$  of \S\ref{univ kir}. Then the elements $f_{\alpha,v}^{\pm, \circ}$ of \eqref{f1}, \eqref{f2} yield, by the proof of the Lemma, explicit elements 
\begin{gather}\label{univ local kir elt}
f_{\alpha,v}^{\pm}\in \varprojlim_{V}  \mathscr{K}(\pi^{\pm}_{v},\psi_{{\rm univ}, v})^{ V},
\end{gather}
where the transition maps are given by averages.

 Recall 
the  local toric period $Q_{v}(f^{+}_{v}, f^{-}_{v}, \chi_{v})$ of \eqref{Qvdef}, for a character $\chi_{v}\in \Y_{v}^{\rm l.c.}$, which we  define on 
$ \mathscr{K}(\pi^{+}_{v},\psi_{{\rm univ}, v}) \otimes  \mathscr{K}(\pi^{-}_{v},\psi_{{\rm univ}, v})$
using the canonical pairing of Lemma \ref{kir pair} on 
the universal Kirlillov models.
 
By the previous discussion and the defining property of $f_{\alpha}^{\pm}$, for any character $\chi_{v}\in \Y_{v}^{\rm l.c.}$,
 the element 
$$Q_{v}(f_{\alpha,v}^{+}, f_{\alpha,v}^{-}, \chi_{v}):=\lim_{V}  {L(1, \eta_{v})L(1,\pi_{v}, \ad) \over \zeta_{F,v}(2) L(1/2, \pi_{E,v}\otimes\chi_{v})_{v} }
 \int_{E_{v}^{\times}/F_{v}^{\times}} \chi_{v}(t) (\pi(t)f^{+}_{\alpha,v, V}, f^{-}_{\alpha,v, V})\, d^{\circ}t $$
 is well-defined and it belongs to $M(\chi_{v})\otimes \OO_{\Psi_{v}}( \omega_{v})$. In fact, if $M(\alpha_{v}, \chi_{v})\subset L$ is a subfield containing the values of $\alpha_{v}$, $\omega_{v}$, and $\chi_{v}$, then $Q_{v}(f_{\alpha,v}^{+}, f_{\alpha,v}^{-}, \chi_{v})$ belongs to o $ \OO_{\Psi_{v, M(\alpha_{v},\chi_{v})}}( \omega_{v})$.

\begin{lemm}\label{toric at p} With notation as above, we have
$$Q_{v}(f_{\alpha,v}^{+}, f_{\alpha,v}^{-})= \zeta_{F,v}(2)^{-1} \cdot Z_{v}^{\circ}$$
as sections of $\OO_{\Y_{v}\times \Psi_{v}}(\omega_{v})$; 
here $Z_{v}^{\circ}$ is   as in Theorem \ref{A}.
\end{lemm}
\begin{proof}
It suffices to show that the result holds at any complex geometric point $(\chi_{v},\psi_{v})\in \Y_{v}^{\rm l.c.}\times\Psi_{v}(\C)$. 
Drop all subscripts $v$, and fix a sufficiently large integer $r$ (depending on $\chi$).
Recalling the pairing \eqref{kir pair}, we have  by definition
\begin{align*}
Q(f_{\alpha}^{+}, f_{\alpha}^{-}, \chi) &=   {L(1,\eta)L(1, \pi, {\rm ad})\over \zeta_{F}(2) L(1/2, \pi_{E}, \chi)}  \int_{E^{\times}/F^{\times}} \chi(t) (\pi(t)f_{\alpha}^{+}, f_{\alpha}^{-})\,{dt\over |d|^{1/2}}.
\end{align*}

We denote by $E_{w}$ (respectively $E_{w^{*}}$) the image of $F$ under the map $F\to M_{2}(F)$ sending $t\mapsto \smalltwomat t{}{}1$ (respectively $t\mapsto \smalltwomat 1{}{}t$), and by $\chi_{w}$ (respectively $\chi_{w^{*}}$) the restriction of $\chi$
 to $E_{w}^{\times}$ (respectively $E_{w^{*}}^{\times}$).

We can then compute that 
$${ \zeta_{F}(1) L(1/2, \pi_{E}, \chi) \over L(1,\eta)}
Q(f_{\alpha}^{+}, f_{\alpha}^{-}, \chi)$$
equals
\begin{align*}
&|d|^{-1}\int_{F^{\times}} \int_{F^{\times}}|\vpi|^{-r} \alpha(\vpi)^{-r} s_{r}f_{\alpha}^{+, \circ} (ty)
\cdot|\vpi|^{-r} \alpha(\vpi)^{-r} s_{r}^{*}f_{\alpha}^{{-, \circ}}(y) \chi_{w}(t) \,d^{\times}y\, d^{\times}t \\
=& |d|^{-1}\int_{F^{\times}} \int_{F^{\times}}|\vpi|^{-r} \alpha(\vpi)^{-r}\psi(-ty)|ty\vpi^{r}|\alpha(ty\vpi^{r})\one_{\OO_{F}-\{0\}}(ty\vpi^{r}) \\
&\cdot |\vpi|^{-r}\alpha(\vpi)^{-r}\psi(y)\omega(\vpi)^{r}
|y\vpi^{r} | 
\alpha(y\vpi^{r})\omega^{-1}(y\vpi^{r}) \one_{\OO_{F}-\{0\}}(y\vpi^{r})   \chi_{w}(t)   \,d^{\times}y\, d^{\times}t.
\end{align*}
We now perform the change of variables $t'=ty$ and observe that $\chi_{w}(t)=\chi_{w}(t')\chi_{w}^{-1}(y)=\chi_{w}(t')\omega(y)\chi_{{w}^{*}}(y)$; we conclude after simplification that the above expression equals
\begin{align*}
|d|^{-1}\int_{v(t')\geq -r}  \psi(-t')|t'|\alpha(t')\chi_{w}(t') \, d^{\times}t' \int_{v(y)\geq -r} \psi(-y)|y|\alpha(y)\chi_{{w}^{*}}(y)\,d^{\times}y.
\end{align*}
If $r$ is sufficiently large, the domains of integration can be replaced by $F^{\times}$. The  computation of the integrals  is carried out in Lemma \ref{basic local integral}. We obtain
$$Q_{v}(f_{\alpha,v}^{+}, f_{\alpha,v}^{-}, \chi_{v}) = {L(1, \eta_{v})\over   \zeta_{F,v}(1) L(1/2, \pi_{E,v}, \chi_{v})} \zeta_{E,v}(1) \prod_{w\vert v} Z_{w}( \chi_{w}) =\zeta_{F,v}(2)^{-1}   \cdot Z_{v}^{\circ}(\chi_{v}).$$
\end{proof}

\subsection{Gross--Zagier and Waldspurger formulas}\label{10.2} Here we prove Theorem \ref{C}. We continue with the notation of the previous subsection, and we  suppose that  $\pi_{v}^{\pm} $ is isomorphic to the local component at $v\vert p$ of the representation $\pi^{\pm}$ of the Introduction. Let $w$, ${w}^{*}$ be the two places of
 $E$ above $v$, and  fix an isomorphism $\B_{v}\cong M_{2}(F_{v})$ such that the map  $E_{v}\cong E_{w}\oplus E_{{w}^{*}}\to\B_{v}$ is identified with  the map $F_{v}\oplus F_{v} \to M_{2}(F_{v})$  given by $(t_{1}, t_{2})\mapsto \smalltwomat {t_{1}}{}{}{t_{2}}$.

We go back to the global situation with the notation and assumption of \S\ref{intro anti}.  Choose a universal Whittaker (or Kirillov) functional for $\pi^{+}$ at $p$, that is, a $\B_{p}^{\times}$-equivariant map $\mathscr{K}_{p}^{+} \colon \pi^{+}\otimes\OO_{\Psi_{p}}(\Psi_{p})\to \bigotimes_{v\vert p}\mathscr{K}(\pi^{+}, \psi_{{\rm univ}, v})$.  By the natural dualities of $\pi^{\pm}$ and the Kirillov models,
it induces a $\B^{\times}_{p}$-equivariant  map $\mathscr{K}^{+\vee}_{p}\colon \bigotimes_{v\vert p}\mathscr{K}(\pi^{-}, \psi_{{\rm univ},v})\to\pi^{-}\otimes\OO_{\Psi_{p}}(\Psi_{p})$, whose inverse  $\mathscr{K}^{-}_{p}$ is  a  universal Kirillov functional  for $\pi^{-}_{p}$. Letting $\pi^{\pm}_{p, \OO_{\Psi_{p}}(\Psi_{p}) }:=\bigotimes_{v\vert p}\mathscr{K}(\pi^{\pm}, \psi_{{\rm univ}, v})$, we obtain a unique decomposition $\pi^{\pm}\otimes \OO_{\Psi_{p}}(\Psi_{p})\cong \pi^{\pm, p }_{ \OO_{\Psi_{p}}(\Psi_{p})}\otimes \pi^{\pm}_{p, \OO_{\Psi_{p}}(\Psi_{p})}$. The decomposition arises from a decomposition of the natural $M$-rational subspaces $\pi^{\pm}\cong \pi^{\pm, p}\otimes\pi_{p}^{\pm}$.\footnote{However the $M$-rational subspaces are not stable under the $\B^{\infty\times}$-action.}

 \subsubsection{Heegner--theta elements}
 For each  $f^{\pm, p }\in \pi^{\pm p}\otimes M(\alpha)$,  let
$$f_{\alpha}^{ \pm, p}:=f^{\pm p}\otimes f_{\alpha, p}^{\pm}\in \pi^{\pm, p}_{M(\alpha)}\otimes\varprojlim_{V\subset\OO_{E,p}^{\times}} \pi_{p, \OO_{M(\alpha)}}^{\pm, V},$$
where $f_{\alpha,p}^{\pm }=\otimes_{v\vert p}f_{\alpha,v}^{\pm}$ with $f_{\alpha,v}^{\pm}$ the elements \eqref{univ local kir elt}. 

Fix a component $\Y_{\pm}^{\circ}\subset \Y_{\pm }$ of type $\eps\in \{+1, -1\}$ as in \S\ref{intro anti}. Then the elements 
\begin{align}\label{def theta}
\Theta^{\pm}_{\alpha}(f^{\pm, p}):=&\dashint_{E^{\times}\bks E_{\A^{\infty}}^{\times}} f_{\alpha}^{p ,\pm}(i(t))
 \chi_{\rm univ}(t)\, dt
 &\in \OO_{\Y_{\pm}}(\Y_{\pm})^{\rm b},\\
 \label{def Htheta}
 \mathscr{P}^{\pm}_{\alpha}(f^{\pm, p}):=
 \lim_{U_{T,p}\to \{1\}} 
& \dashint_{E^{\times}\bks E_{\A^{\infty}}^{\times}/U_{T}} \kappa(f_{\alpha}^{p ,\pm}({\rm rec}_{E}(t)\iota_{\xi}(P))\otimes \chi_{{\rm univ,} U_{T}}^{\pm}(t))\, dt
 &\in  {\bf S}_{\frakp}(A_{E}, \chi_{\rm univ}^{\pm}, \Y_{\pm})^{\rm b},
 \end{align}  
 or rather their restriction to $\Y_{\pm}^{\circ}\subset \Y_{\pm}$,
 satisfy the property of Theorem \ref{C}.\ref{C1}. Here $\chi_{{\rm univ}, U_{T}}^{\pm}\colon \Gamma\to \OO^{\times}(\Y_{\pm})^{U_{T}}$ is the convolution of $\chi_{\rm univ}^{\pm}$ with the finest $U_{T}$-invariant approximation to a delta function at $\one\in \Gamma$.
 
  We explain the boundedness in the case of $ \mathscr{P}^{\pm}_{\alpha}(f^{\pm, p})$.  
The rigid space $\Y_{\pm}$ is the generic fibre of an $\OO_{L}$-formal scheme $\mathfrak{Y}_{\pm}:={\rm Spf}\, \OO_{L}\llb \Gamma\rrb/((\omega^{\pm 1}(\gamma)[\gamma]-[1])_{\gamma \in \A^{\infty,\times}})$.\footnote{The quotienting ideal is  finitely generated as the image of $\A^{\infty,\times}$ in $\Gamma $ is a finitely generated $\Z_{p}$-submodule.} (Similarly each geometric connected component $\Y^{\circ}_{\pm}$ has a formal model $\mathfrak{Y}^{\circ}_{\pm}\subset \mathfrak{Y}$).
 This identifies $\OO_{\Y_{\pm}}(\Y_{\pm})^{\rm  b}= \OO_{\mathfrak{Y_{\pm}}}(\mathfrak{Y_{\pm}})\otimes_{\OO_{L}}{L}$.
 By Lemma \ref{f limit}.\ref{bounded-til} applied to the  natural lattice $\pi^{\pm}_{\OO_{L}}\subset \pi^{\pm}\otimes L$ given by $\Hom(J,A^{\pm})\otimes_{\End(A)}\OO_{L}\subset \Hom^{0}(J, A^{\pm})\otimes L$,   after possibly replacing $f^{ \pm, p} $ by a fixed  multiple
the elements $\tilde{f}_{\alpha}^{p ,\pm}({\rm T}_{t} \iota_{\xi}(P))$ belong to $A^{\pm}({E}^{\rm ab})$.
  Then 
 each term in the sequence at the right-hand side of \eqref{def Htheta} is a fixed multiple of 
\begin{align}\label{is limit of}
\sum_{t\in E^{\times}\bks E_{\A^{\infty}}^{\times}/U_{T}} \kappa(\tilde{f}_{\alpha}^{p ,\pm}({\rm rec}_{E}(t) \iota_{\xi}(P))\otimes \chi_{{\rm univ}, U_{T}}^{\pm}(t)),
\end{align}
which belongs to the $\OO_{L}$-module $H^{1}_{f}(E, T_{\frakp }A^{\pm} \otimes \OO_{\mathfrak{Y_{\pm}}}(\mathfrak{Y_{\pm}})^{U_{T}}(\chi_{{\rm univ}, U_{T}}^{\pm}))$.
 Hence some nonzero  multiple of $ \mathscr{P}^{\pm}_{\alpha}(f^{\pm, p})$ belongs to the limit $\varprojlim_{U_{T}}H^{1}_{f}(E, T_{\frakp }A^{\pm} \otimes \OO_{\mathfrak{Y_{\pm}}}(\mathfrak{Y_{\pm}})^{U_{T}}(\chi_{{\rm univ}, U_{T}}^{\pm}))$, whose tensor product with $L$ is indeed ${\bf S}_{\frakp}(A_{E}, \chi_{\rm univ}^{\pm}, \Y_{\pm})^{\rm b}$. 

\subsubsection{Local toric periods away from $p$}  Given the chosen decomposition $\gamma^{\pm}\colon\pi^{\pm}\otimes \OO_{\Psi_{p}}(\Psi_{p})\cong \pi^{\pm, p}_{\OO_{\Psi_{p}}(\Psi_{p})} \otimes  \pi^{\pm }_{p,\OO_{\Psi_{p}}(\Psi_{p})}$, let $(\, , \, )^{p}$ be the unique pairing on $\pi^{+,p}_{\OO_{\Psi_{p}}}\otimes \pi^{-,p}_{\OO_{\Psi_{p}}}$ which makes $\gamma^{+}\otimes \gamma^{-}$ into an isometry for the natural pairings on $\pi^{\pm }$ and $\pi_{p,  \OO_{\Psi_{p}}(\Psi_{p})}^{\pm}$.\footnote{Recall that the natural pairing on the factor at $p$ comes from its description as a Kirillov model.} Then for each $\chi=\chi^{p}\chi_{p}\in \Y^{\circ\,  {\rm l.c.}}$, the toric period $Q^{p}$ of \eqref{toric away p} is defined. By Lemma \ref{Qtheta},  Theorem \ref{C}.\ref{C2} then follows from Proposition \ref{interp rnat}. It is also proved in slightly different language in \cite[Lemma 4.6 (ii)]{lzz}.

\subsubsection{Formulas}
We prove the anticyclotomic formulas of  Theorem \ref{C}, and at the same time complete the proof of Theorem \ref{B} for the characters who do not satisfy Assumption \ref{ass on chi}.

\begin{lemm}\label{zerozerozero} Let $L$ be a non-archimedean local field with ring of integers $\OO_{L}$, $n\geq 1$, and let $\mathscr{D}_{n}$ be the rigid analytic polydisc over $L$ in $n$ variables, that is the generic fibre of ${\rm Spf} \,\OO_{L}\llb X_{1}, \ldots , X_{n}\rrb$. Let $\Sigma_{n}\subset {\mathscr{D}}_{n}(\baar{L})$ be the set of points of the form $x=(\zeta_{1}-1, \ldots , \zeta_{n}-1)$ with each $\zeta_{i}$ a root of unity of $p$-power order. Let 
 $$f\in \OO({\mathscr{D}}_{n})^{\rm b}=\OO_{L}\llb X_{1}, \ldots , X_{n}\rrb\otimes L$$ 
 be such that $f(x)=0$ for all but finitely may $x\in \Sigma_{n}$. Then $f=0$.
\end{lemm}
\begin{proof}  By induction of $n$, the case $n=1$ being well-known \cite{AV}; we will abbreviate  $X=(X_{1}, \ldots, X_{n})$ and $X'=(X_{1}, \ldots, X_{n-1})$.  Up to multiplying $f$ by a suitable nonzero polynomial  we may assume that $f\in \OO_{L}\llb X\rrb$ and that it vanishes on all of $\Sigma_{n}$. We may write, with multiindex notation,
$f=\sum_{J\subset \N^{n}}a_{J}X^{J}$
 (where $\N=\{0, 1,2, \ldots, \}$) with each $a_{J}\in \OO_{L}$. Let $f_{j}(X'):=\sum_{J'\subset \N^{n-1}} a_{J'j}(X')^{J'}\in \OO_{L}\llb X'\rrb$, then $$f(X)=\sum_{j=0}^{\infty }f_{j}(X') X_{n}^{j}.$$
 By assumption, $f_{0}(X')=f(X', 0)$ vanishes on all of $\Sigma_{n-1}$, hence by the induction hypothesis $f_{0}=0$ and $X_{n}|f$. By induction on $j$, repeatedly replacing  $f $ by $X_{n}^{-1}f$, we find that each $f_{j}=0$, hence $f=0$. 
\end{proof}

\begin{lemm}\label{isinI} Let $\Y^{\circ}\subset \Y$ be a connected component of type $\eps=-1$, and let $\Y^{\circ}{}'\subset \Y'$ be the connected component containing $\Y^{\circ}$. The $p$-adic $L$-function $L_{p ,\alpha}(\sigma_{E})|_{\Y^{\circ}{}'}$ is a section of $\calI_{\Y}\subset\OO_{\Y'}$.
\end{lemm}
\begin{proof} By the interpolation property and the functional equation, $L_{p, \alpha}(\sigma_{E})$ vanishes on $\Y^{\circ}\cap \Y^{\rm l.c., \, an}(\baar{L})$. We conclude by applying Lemma \ref{zerozerozero}, noting that after base-change to a finite extension of $L$, there is an isomorphism $\Y\to \coprod_{i\in I} {\mathscr{D}}_{[F:\Q]}^{(i)}$ to a finite disjoint union of rigid polydiscs, taking $\Y^{\rm l.c., \, an}(\baar{L})$ to $\coprod_{i\in I}\Sigma_{[F:\Q]}^{(i)}$.

\end{proof}

\begin{prop}\label{BandC} The following are equivalent:
\begin{enumerate}
\item\label{fo1}   Theorem \ref{B} is true for all $f_{1}$, $f_{2}$ and all locally constant characters $\chi \in \Y^{\rm l.c.}_{L}$;
\item\label{fo2} Theorem \ref{B} is true for all $f_{1}$, $f_{2}$ and all but finitely many locally constant characters $\chi \in \Y^{\rm l.c.}_{L}$;
\item\label{fo3} Theorem \ref{C}.\ref{C4} is true for all $f^{+, p}$ and $f^{-, p}$.
\end{enumerate}
\end{prop}
\begin{proof} 
It is clear that \ref{fo1} implies \ref{fo2}. 
That \ref{fo2} implies \ref{fo3} follows from  Lemma \ref{zerozerozero} applied to the difference of the two sides of the desired equality, together with the interpolation properties and the evaluation of the local toric integrals in Lemma \ref{toric at p}.
 Finally,   the multiplicity one result together with Lemma \ref{toric at p} shows that \ref{fo3} implies~\ref{fo1}.
\end{proof}
Since we have already shown at the end of \S\ref{sec comparison} that Theorem \ref{B} is true for all but finitely many finite order characters, this completes the proof of Theorem \ref{B} in general and proves Theorem \ref{C}.\ref{C4}.
Finally, the anticyclotomic Waldspurger formula of Theorem \ref{C}.\ref{C3} follows from the Waldspurger formula at finite order characters \eqref{WF} by the  argument in the proof of Proposition \ref{BandC}.

\subsection{Birch and Swinnerton-Dyer formula}\label{sec iwbsd}
Theorem \ref{iwbsd} follows immediately from combining the first and second parts of the   following Proposition. We abbreviate
${\bf S}_{\frakp}^{\pm}:={\bf S}_{\frakp}(A_{E}^{\pm},\chi_{\rm univ}^{\pm}, \Y^{\circ})^{\rm b}$, and remark that, under the assumption $\omega=\one $ of Theorem \ref{iwbsd},  we have $A=A^{+}=A^{-}$ and $\pi=\pi^{+}= \pi^{-}$.

\begin{prop}Under the assumptions and notation  of Theorem \ref{iwbsd}, the following hold.
\begin{enumerate} \item Let $$\underline{\mathscr{H}}\subset {\bf S}_{\frakp}^{+}\otimes {\bf S}_{\frakp}^{-, \iota}$$ be the saturated $\Lambda$-submodule generated by the Heegner points $\mathscr{P}_{\alpha}^{+}(f^{p})\otimes \mathscr{P}_{\alpha}^{-}(f^{ p})$ for $f^{p}\in \pi^{p}$.
The $\Lambda$-modules ${\bf S}_{\frakp}^{\pm}$ are  generically of rank one, and moreover $\underline{\mathscr{H}}$ is free of rank~$1$  over $\Lambda$, generated  by an explicit  element $\mathscr{P}^{+}_{\alpha}\otimes \mathscr{P}_{\alpha}^{-, \iota}$. 
\item We have the divisibility of $\Lambda$-ideals 
\begin{gather*}  
{\rm char}_{\Lambda} \wtil{H}^{2}_{f}(E, V_{\frakp}A\otimes \Lambda(\chi_{\rm univ}) )_{\rm tors} \quad
{\Huge|}
\quad
  {\rm char}_{\Lambda}
  \left(
 {\bf S}_{\frakp}^{+}\otimes_{\Lambda} {\bf S}_{\frakp}^{-, \iota}\,
   /
\underline{\mathscr{H}} \right).\end{gather*}
\item Letting $\big\langle\, \, \big\rangle $ denote the height pairing \eqref{bight}, we have
$$\big\langle \mathscr{P}^{+}_{\alpha}\otimes \mathscr{P}_{\alpha}^{-, \iota} \big\rangle={c_{E}\over 2}\cdot {\rm d}_{F}L_{p, \alpha}(\sigma_{E})|_{\Y^{\circ}}.$$

\end{enumerate}

\end{prop}
\begin{proof} 
We define $\wtil{Q}_{v}(f_{v}, \chi_{v}):= Q_{v}(f_{v}, f_{v}, \chi_{v})$ for $f_{v}\in \pi_{v}$, and similarly  $\wtil{Q}(f, \chi):=\prod_{v}\wtil{Q}(f_{v}, \chi_{v})$ if $f=\otimes_{v}f_{v}\in \pi$.  By \cite[Lemme 13]{wald} (possibly applied to  twists $(\pi_{v}\otimes \mu_{v}^{-`}, \chi_{v}\cdot (\mu_{v}\circ q_{w}))$ for some character $\mu_{v}$ of $F_{v}^{\times}$), the spaces ${\rm H}(\pi_{v}, \chi_{v})={\rm Hom}_{E_{v}^{\times}}(\pi_{v}\otimes \chi_{v}, L(\chi_{v}))$ are nonzero if and only if $\wtil{Q}_{v}(\cdot, \chi_{v})$ is nonzero on $\pi_{v}$.  
We also define  $\wtil{\mathscr{Q}}_{v}(f_{v}):={\mathscr{Q}}_{v}(f_{v}, f_{v})$ and $\wtil{\mathscr{Q}}(f^{p}):=\zeta_{F, p}(2)^{-1}\prod_{v\nmid p}\wtil{\mathscr{Q}}(f_{v})\in\Lambda$ for $f^{p}=\otimes f_{v}\in \pi^{p}$. If the local conditions \eqref{local cond} are satisfied, as we assume, then the spaces   ${\rm H}(\pi_{v}, \chi_{v})$ are nonzero for all $\chi\in \Y^{\circ}$, and $\wtil{\mathscr{Q}}$ is not identically zero on  $\pi^{p}$.

We will invoke the results of Fouquet in \cite{fouquet} after  comparing our setup with his.  Let $U_{r}=U^{p}\prod_{v\vert p}U_{r,v}\subset \B^{\infty, \times}$ be such that $U_{r, v}=K_{0}(\vpi_{v}^{r_{v}})$ for $v\vert p$, $r_{v}\geq 1$. Let $eH^{1}_{\text{ \'et}}(X_{U_{r}, \baar{F}}, \OO_{L}(1))$ be the image of $H^{1}_{\text{ \'et}}(X_{U_{r},  \baar{F}}, \OO_{L}(1))$ under the product $e$ of the projectors $e_{v}:=\lim_{n}\Up_{v}^{n!}$. Let ${\mathscr{J}}_{\pi}\subset \mathscr{H}_{\B^{\infty\times}}^{\rm sph}$ be the annihilator of  $\pi$ viewed as a module over the spherical Hecke algebra $  \mathscr{H}_{\B^{\infty\times}}^{\rm sph}$ for $\B^{\infty\times}$, and let  $$eH^{1}_{\text{ \'et}}(X_{U^{p}, \baar{F}}, \OO_{L}(1))[\pi]:= eH^{1}_{\text{ \'et}}(X_{U_{r}, \baar{F}}, \OO_{L}(1))/{\mathscr{J}}_{\pi},$$
 a Galois-module which  is independent of $r\geq\underline{1}$. The operators $\Up_{v}$ act invertibly on  $eH^{1}(X_{U_{r}, \baar{F}}, \OO_{L}(1))$, and in fact by $\alpha_{v}(\vpi_{v})$ on  $eH^{1}_{\text{ \'et}}(X_{U^{p}, \baar{F}}, \OO_{L}(1))[\pi]$.  
Let $f^{p}\in \pi^{p}$  be such that $\wtil{\mathscr{Q}}(f^{p})\neq 0$.
 Denote by $\kappa $ the Abel--Jacobi functor and by $f_{\alpha}^{\circ}:=f^{p}\otimes f_{\alpha, p}^{\circ}$, with $f_{\alpha,p}^{\circ}$  the product of the elements $ f_{\alpha,v}^{\circ}$  of \eqref{f1} for $v\vert p $. 
  Then, up to a fixed nonzero multiple, the class ${\mathscr{P}}^{+}_{\alpha}(f^{p})$ is the image under $\kappa(f_{\alpha}^{\circ})$ of the limit of the compatible sequence 
$${\mathscr{P}}_{\alpha, r}:=\Up_{p}^{-r}\sum_{t\in E^{\times}\bks E_{\A^{\infty}}^{\times}/U_{T}} \kappa({\rm rec}_{E}(t) {\rm T}_{s_{r}}\iota_{\xi}(P))\otimes \chi_{{\rm univ}, U_{T}}^{\pm}(t))$$
of integral elements of $ H^{1}_{f}(E, eH^{1}_{\text{ \'et}}(X_{U^{p},  \baar{F}}, \OO_{L}(1))\otimes \Lambda^{U_{T}}(\chi_{{\rm univ}, U_{T}}))$. Fouquet takes as input a certain compatible sequence $(z({\mathfrak{c}}_{p}, S))_{S}$ (\cite[Definitions 4.11, 4.14]{fouquet}) of classes in  the latter space to construct, via suitable local modifications according to the method of Kolyvagin, an Euler  system (\emph{ibid.} \S5). Noting that the local modifications occur at \emph{well-chosen, good} primes $\ell$ of $E$, his construction can equally well be applied to the sequence $({\mathscr{P}}_{\alpha, r})_{r}$ in place of $(z({\mathfrak{c}}_{p}, S))_{S}$. This Euler system can then be projected via $\kappa(f_{\alpha})$ to yield an Euler system for $V_{\frakp}A\otimes \Lambda$. Under  the condition that the first element ${\mathscr{P}}_{\alpha}^{+}(f^{p})\in {\mathbf S}_{\frakp}^{+}$ (corresponding to $z_{f, \infty}$ in \cite{fouquet})
of the projected Euler system   is non-torsion,
 it is proved in \cite[Theorem B (ii)]{fouquet} that ${\bf S}_{\frakp}^{\pm}$ have generic rank~$1$ over $\Lambda$.  By the main result of \cite{AN}, generalising \cite{CV05}, the family of points $\mathscr{P}_{\alpha}^{+}(f^{p})\in {\bf S}_{\frakp}^{+}$ is indeed not $\Lambda$-torsion provided $\wtil{\mathscr{Q}}(f^{p})\neq 0$ (i.e. $f_{v}$ is a ``local test vector'' for all $v\nmid p$). We conclude as desired that ${\bf S}_{\frakp}^{\pm}$ have generic rank~$1$ over $\Lambda$ and that the same is true 
 of the submodule ${\underline{\mathscr{H}}}$.
  
  We now proceed to complete the proof of part 1 by showing that the `Heegner submodule' ${\underline{\mathscr{H}}}$ is in fact \emph{free} of rank~$1$ and constructing a `canonical' generator.  First we note that by   \cite[Theorem 6.1]{fouquet},    each special fibre ${ \underline{\mathscr{H}}}_{|\chi}$ (for arbitrary $\chi\in \Y^{\circ}$) has dimension either $0$ (we will soon exclude this case) or $1$ over $L(\chi)$. Let $\{{\mathscr{P}}_{\alpha}^{+}(f^{p}_{i})\otimes {\mathscr{P}}_{\alpha}^{-}(f^{p}_{i})^{\iota}\, :\, i\in I\}$ be finitely many  sections of ${\underline{\mathscr{H}}}$.  By \cite{yzz}, for each  $\chi\in \Y^{\circ, {\rm l.c.}}:= \Y^{\circ}\cap \Y^{\rm l.c., \, an}$,   the specialisation ${\mathscr{P}}_{\alpha}^{+}(f^{p}_{i})\otimes {\mathscr{P}}_{\alpha}^{-}(f^{p}_{i})^{\iota}(\chi)$ is nonzero  if and only if $\wtil{\mathscr{Q}}(f_{i}^{p})(\chi)\neq 0$, and moreover  the images of the specialisations at $\chi$ of the  global sections of ${\underline{\mathscr{H}}}$
  \begin{align}\label{scrPi}
  \prod_{j\in I, j\neq i}\wtil{\mathscr{Q}}(f_{j}^{p})\cdot {\mathscr{P}}_{\alpha}^{+}(f^{p}_{i})\otimes {\mathscr{P}}_{\alpha}^{-}(f^{p}_{i})^{\iota}
  \end{align} 
  under the N\'eron--Tate height pairing (after choosing any embedding $L(\chi)\into \C$) coincide. 
As ${\underline{\mathscr{H}}}_{|\chi}$ has dimension at most $1$, the N\'eron--Tate height pairing on ${\underline{\mathscr{H}}}_{\chi}\otimes\C$ is  an isomorphism onto its image in  $\C$, and we deduce that that the elements \eqref{scrPi} coincide over $\Y^{\circ, {\rm l.c.}}$. Since the latter set is dense in $\Spec \Lambda$ by Lemma \ref{zerozerozero}, they coincide everywhere and glue to a global section  $(\mathscr{P}^{+}_{\alpha}\otimes \mathscr{P}_{\alpha}^{-, \iota})'$ of ${\underline{\mathscr{H}}}$  over $\Y^{\circ}$.\footnote{We remark  that a similar argument, in conjunction with the previous observation that $\wtil{\mathscr{Q}}\neq 0$ on $\pi^{p}$ if and only if ${\mathscr{Q}}\neq 0$ on $\pi^{p}\otimes \pi^{p}$, shows that ${\underline{\mathscr{H}}}$ equals  the \emph{a priori} larger saturated submodule ${\underline{\mathscr{H}}}'\subset  {\bf S}_{\frakp}^{+}\otimes_{\Lambda} {\bf S}_{\frakp}^{-, \iota}$ generated by the ${\mathscr{P}}_{\alpha}^{+}(f^{+,p})\otimes {\mathscr{P}}_{\alpha}^{-}(f^{-,p})^{\iota}$ for possibly different $f^{\pm, p}\in \pi$.}

     Similarly to what claimed in the proof of Theorem \ref{theo A text}, there exists a \emph{finite} set $\{f_{i}^{p}\, :\, i\in I\}\subset \pi^{p}$ such that the open sets ${\mathscr{U}}_{f_{i}^{p}}:=\{\wtil{\mathscr{Q}}(f^{p}_{i})\neq 0\}\subset \Y^{\circ}$ cover $\Y^{\circ}$. Then the section 
$$ \mathscr{P}^{+}_{\alpha}\otimes \mathscr{P}_{\alpha}^{-, \iota}:=   \prod_{i\in I}\wtil{\mathscr{Q}}(f_{i}^{p})^{-1}\cdot(\mathscr{P}^{+}_{\alpha}\otimes \mathscr{P}_{\alpha}^{-, \iota})'$$
is nowhere vanishing and a generator of ${\underline{\mathscr{H}}}$.  It is independent of choices since for any $f^{p}\in \pi^{p}$
 it coincides  with  $\wtil{\mathscr{Q}}(f^{p})^{-1}\cdot {\mathscr{P}}_{\alpha}^{+}(f^{p})\otimes {\mathscr{P}}_{\alpha}^{-}(f^{p})^{\iota}$ over ${\mathscr{U}}_{f^{p}}$.  This completes the proof of Part 1.
  
 Part 2 is \cite[Theorem B (ii)]{fouquet} with ${\underline{\mathscr{H}}}$ replaced by  its submodule generated by a non-torsion   element $z_{f, \infty}\otimes z_{f, \infty}\in {\underline{\mathscr{H}}}$; as noted above we can replace this element by any of the elements  ${\mathscr{P}}_{\alpha}^{+}(f^{p})\otimes {\mathscr{P}}_{\alpha}^{-}(f^{p})^{\iota}$, and by a glueing argument with  $\mathscr{P}^{+}_{\alpha}\otimes \mathscr{P}_{\alpha}^{-, \iota}$. 
 
 Part 3 is an immediate consequence of Theorem \ref{C}.\ref{C4}.
\end{proof}

\appendix
\section{Local integrals}

\subsection{Basic integral} All the integrals computed in the appendix will ultimately reduce to the following.
\begin{lemm}\label{basic local integral} Let $F_{v}$ be a non-archimedean local field, and let $E_{w}/F_{v}$ be an extension of degree $fe \leq 2$, with $f$ the inertia degree and $e$ the ramification degree. Let $q_{w}$ be the relative norm and   $D\in \OO_{F,v}$ be a generator of the relative discriminant.

 Let $L$ be a field of characteristic zero,  let $\alpha_{v}\colon F_{v}^{\times}\to L^{\times}$ and $\chi'\colon E_{w}^{\times}\to L^{\times}$ be  multiplicative characters, $\psi_{v}\colon F_{v}\to L^{\times}$ be an additive character of level $0$, and $\psi_{E,w}=\psi_{v}\circ\Tr_{E_{w}/F_{v}}$. Define
$$ Z_{w}( \chi', \psi_{v}):=    \int_{E_{w}^{\times}}    \alpha\circ q (t)  \chi'(t) \psi_{E_{w}}(t)\, {dt\over  |D|^{1/2}|d|_{v}^{f/2}}. $$
where $dt$ is the restriction of the standard measure on $E_{w}$.

Then we have:
\begin{equation*}
Z_{w}( \chi', \psi_{v})=
\begin{cases}\displaystyle
\alpha_{v}(\vpi_{v})^{-v(D)} \chi_{w}'(\vpi_{w})^{-v(D)} 
{1-\alpha_{v}(\vpi_{v})^{-f}\chi_{w}'(\vpi_{w})^{-1}
\over
 1-\alpha_{v}(\vpi_{v})^{f}\chi'_{w}(\vpi_{w})q_{F,v}^{-f}}
  &\text{\ if $\chi_{w}'\cdot \alpha_{v}\circ q$ is unramified,}\\
 \tau(\chi'_{w}\cdot \alpha_{v}\circ q, \psi_{E_{w}}) &\text{\ if $\chi_{w}'\cdot \alpha_{v}\circ q$ is ramified.}
\end{cases}
\end{equation*}
Here for any character $\wtil{\chi}_{w}'$ of conductor $\mathfrak{f}$,
$$ \tau(\wtil{\chi}', \psi_{E_{w}})=\int_{w(t)=-w({\frak f})} \wtil{\chi}_{w}(t)\psi_{E,w}(t) \, {dt\over  |D|^{1/2}|d|_{v}^{f/2}}. $$
with
 $n=-w({\frak f}(\chi'_{w}))-w(d_{E,w})$.
\end{lemm}
\begin{proof}  If $\chi'$ is unramified, only the subset $\{w(t)\geq -1-w(d)-v(D)\}\subset E_{w}^{\times}$ contributes to the integral,  and we have 
\begin{align*} 
Z_{w}( \chi', \psi)&= 
\alpha^{-v(D)} \chi_{w}'(\vpi_{w})^{-ev(d)-v(D)} 
	\left({\zeta_{E}(1)^{-1} \over 1-\alpha^{f}\chi'(\vpi)}
	-{ 1 \over q_{F,v}^{f}}\alpha^{-f}\chi'(\vpi)^{-1}\right)\\
&=
\alpha^{-v(D)} \chi_{w}'(\vpi_{w})^{-ev(d)-v(D)} 
{ 1-\alpha^{-f}\chi'(\vpi)^{-1}q_{F,v}^{-f}  \over 1-\alpha^{f}\chi'(\vpi)}.
\end{align*}
If $\chi'$ is ramified of conductor ${\frak f}={\frak f}(\chi')$ then only the annulus  $w(t)=-w({\frak f})-w(d)-v(D)$ contributes, and we get
$$Z_{w}(\chi', \psi) =  
\alpha_{v}^{-fw({\frak f})-fv(D)} 
\tau(\chi', \psi_{E}).\qedhere$$
 \end{proof}

\subsection{Interpolation factor}
We compute the integral giving the interpolation factor for the $p$-adic $L$-function.

The following  Iwahori decomposition can be proved similarly to  \cite[Lemma A.1]{Hu}.
 
 \begin{lemm}\label{iwasawish}
For a local field $F$ with uniformiser $\vpi$, and for any $r\geq 1$, the double quotient 
$$N(F)A(F)Z(F)\bks \GL_{2}(F)/K_{1}^{1}(\vpi^{r})$$
admits the set of representatives  
$$\twomat 1{}{}1, \quad  \twomat 1{} {c^{(i)}\vpi^{r-i}} 1, \quad 1\leq i\leq r,\  c^{(i)}\in   (\OO_{F}/\vpi^{i})^{\times}.$$
 \end{lemm}
Note that we may also replace the first representative by $\smalltwomat 1{}{\vpi^{r}} 1\in K_{1}^{1}(\vpi^{r})$.

\begin{prop}\label{interpolation factor}
 Let $\chi'\in \Y'^{\rm l.c.}_{M(\alpha)}(\C)$ and let $\iota\colon M(\alpha)\into \C$ be the induced embedding.  Let $v\vert p$, and let $\phi_{v}$ be either as in Assumption \ref{ass at p1}  for some sufficiently small $U_{T,v}\subset \OO_{E,v}^{\times}$, or as in Assumption \ref{ass at p2}.  Then for any sufficiently large integer $r$, the normalised integral $R_{r,v}^{\natural}(W_{v}, \phi_{v}, \chi_{v}')$ of \eqref{Rnat} equals 
$$R_{r,v}^{\natural}(W_{v}, \phi_{v}, \chi_{v}'{}^{\iota}) =Z_{v}^{\circ}( \chi'_{v}):=   
 {\zeta_{F,v}(2)L(1, \eta_{v})^{2}   \over L(1/2, \sigma_{E,v}^{\iota}, \chi'_{v}{})} 
  {\prod_{w\vert v}Z_{w}(  \chi_{w}' )},$$ 
with $Z_{w}( \chi'_{w})$ as in Lemma \ref{basic local integral}.

\end{prop}
\begin{proof}
We omit the subscripts $v$ and the embedding $\iota$ in the calculations which follow. 
By definition, we need to show that  the integral 
$R_{r,v}^{\circ }$ of Proposition \ref{prop2.6} equals 
$$R_{r,v}^{\circ }=R_{r}^{\circ}(W, \phi, \chi')   =|D|^{1/2}|d|^{2} L(1, \eta)\prod_{w\vert v} Z_{w}( \chi'_{w}).$$
Note that the assertion in the case of Assumption \ref{ass at p2} is implied by the assertion in the case of Assumption  \ref{ass at p1} by \eqref{phi ass2}, so we will place ourselves in the latter situation.

By the  decomposition of Lemma \ref{iwasawish},  and observing that $\delta_{\chi_{F},r}$ vanishes on $K- K_{0}(\vpi^{r})$, we have
\begin{multline*}
R_{r,v}^{\circ}=\alpha(\vpi)^{-r}\int_{F^{\times}} W_{-1}(\smallmat y{}{}1)\delta_{\chi_{F},r}(\smallmat y {}{}1)\\
 \int_{T(F)}\chi'(t)  \int_{P(\vpi^{r})\bks K_{1}^{1}(\vpi^{r})} |y| r(kw_{r}^{-1})\phi(yt^{-1}, y^{-1}q(t))
dk\, d^{\times}t \, {d^{\times }y\over |y|}
\end{multline*}
where $P(\vpi^{r})=P\cap K_{1}^{1}(\vpi^{r})$ (recall that $P=NZA$). Here we have preferred to  denote by $d^{\times}t$ the standard Haar measure on $T(F_{v})=E_{v}^{\times}$; later $dt$ will denote the additive measure on $E_{v}$.

Changing variables $k'=w_{r}kw_{r}^{-1}$, we observe that by Lemma \ref{uppertriang} the group $K_{1}^{1}(\vpi^{r})$ acts trivially for  sufficiently large $r$. 
Then we can  plug in  
$$W_{-1}(\smallmat y {}{}1)= \one_{\OO_{F}-\{0\}}(y)|y|\alpha(y)
$$
and
$$r( w_{r}^{-1})\phi(x,u)= |\vpi^{-r}|\psi_{E,U}(u x_{1})\one_{\OO_{{\bf V}_{2}}}(x_{2}) \delta_{q(U)}(\vpi^{r}u),$$
where $\psi_{E,U}=r(\vol(U)^{-1}\one_{U})\psi_{E}$ for the extension of $r$ to functions on $K$ (so that  $\psi_{E,U}$ is the finest $U$-invariant approximation to $\psi_{E}$). We obtain
\begin{multline*}
R_{r,v}^{\circ}=\alpha(\vpi)^{-r} |d| ^{1/2} \zeta_{F,v}(1)^{-1}\int_{\OO_{F}-\{0\}} |y|\alpha(y) 
 \int_{T(F)}\chi'( t) 
 \psi_{E,U}(t)\delta_{q(U)}(\vpi^{r}y^{-1}q(t))\, dt \, d^{\times}y,
\end{multline*}
where $|d|^{1/2}\zeta_{F,v}(1)^{-1}$ appears as $\vol(P(\vpi^{r})\bks K_{1}^{1}(\vpi^{r}))|\vpi|^{-r}$.
 We get
 \begin{align*}
R_{r,v}^{\circ}= |d|\zeta_{F,v}(1)^{-1}
\int_{v(q(t))\geq -r} |q(t)|\alpha(q(t))
\chi'(t) \psi_{E}(t)\, d^{\times}t. 
 \end{align*}
If $r$ is sufficiently large, the domain of integration can be replaced with all of $T(F)$. Switching to the additive measure, and using the isomorphism $E_{v}^{\times}=E_{w}^{\times}\times E_{w^{*}}^{\times}$ in the split case, the integral equals
 \begin{align*}
R_{r,v}^{\circ}= |d|L(1, \eta)
\int_{E_{v}^{\times}} \alpha(q(t))
\chi'(t) \psi_{E}(t)\, dt = |D|^{1/2}|d|^{2} L(1, \eta)\prod_{w\vert v} Z_{w}( \chi'_{w})
 \end{align*}
 as desired.
\end{proof}

\subsection{Toric period}
We compare the normalised toric periods with the interpolation factor. 
\begin{prop}\label{Qv vs} Suppose that $v\vert p$ splits in $E$. Then for any finite order character $\chi\in \Y^{\rm l.c.}$  we have 
$$ Q_{v}(\theta_{v}(W_{v}, \alpha^{-r_{v}}_{v} w_{r,v}^{-1}\phi_{v}),\chi_{v}) = L(1, \eta_{v})^{-1}\cdot Z^{\circ}_{v}(\alpha_{v}, \chi_{v}),$$
for any $\phi_{v}$ is as in Proposition \ref{interpolation factor} and any sufficiently large $r_{v}\geq 1$.
\end{prop}
For consistency with  the proof of Proposition \ref{interpolation factor}, in the proof we will denote by $d^{\times}t$ the Haar measure on $T(F)$ denoted by $dt$ in the rest of the paper.
\begin{proof}
By the definitions and Proposition \ref{interpolation factor},   it  suffices to show that for any $\chi\in \Y^{\rm l.c.}(\C)$ we have
$$ |d|_{v}^{3/2}Q_{v}^{\sharp}(\theta_{v}(W_{v}, \alpha^{-r_{v}}_{v} w_{r,v}^{-1}\phi_{v}),\chi_{v}) = L(1, \eta_{v})^{-1}\cdot R_{r,v}^{\circ}(W_{v}, \phi_{v},\chi_{v})$$
where $Q_{v}^{\sharp}$ is the toric integral of \eqref{Qsharp}.

By  the  Shimizu lifting (Lemma \ref{Qtheta}) and Lemma \ref{iwasawish}  we can write 
$$Q^{\sharp}_{v}:=|d|_{v}^{3/2}Q_{v}(\theta_{v}(W_{v}, \alpha^{-r_{v}}_{v} w_{r,v}^{-1}\phi_{v}),\chi_{v})= Q^{\sharp \, (0, 1)}_{v} +\sum_{i=1}^{r}\sum_{c\in (\OO_{F}/\vpi^{i})^{\times}} Q^{\sharp\, (i,c)}_{v}$$
where for each $(i, c)$, omitting the subscripts $v$,
\begin{multline*} 
Q^{\sharp\,  (i, c)}_{v}:=\alpha(\vpi)^{-r}\int_{F^{\times}} W_{-1}(\smallmat y{}{}1  n^{-}(c\vpi^{r-i}))\\
 \int_{T(F)}\chi'(t) \int_{P(\vpi^{r})\bks K_{1}^{1}(\vpi^{r})} |y| r( n^{-}(c\vpi^{r-i})   kw_{r}^{-1})\phi(yt^{-1}, y^{-1}q(t))
dk\, d^{\times}t \, {d^{\times }y\over |y|}.
\end{multline*}

Note that $Q^{\sharp\, (0,1)}_{v}=R_{v}^{\circ}$, where $R_{v}^{\circ }$ is as in the previous Proposition, since $n^{-}(\vpi^{r})\in K_{1}^{1}(\vpi^{r})$. We will compute the other terms. 

We have $\twomat 1{}{c\vpi^{r-i}} 1 w_{r}^{-1 } = w_{r}^{-1} \twomat 1 {-c\vpi^{-i}} {}1$,
and when $x=(x_{1}, x_{2})$ with $x_{2}=0$:
\begin{align*}
r(n^{-}(c\vpi^{r-i} )w_{r}^{-1})\phi(x,u)&=
 |\vpi^{-r}| \int_{E} 
  \psi_{E}(u x_{1}\xi_{1})      \psi(-uc\vpi^{r-i} q(\xi_{1}))   \delta_{U}(\xi_{1}) \delta_{q(U)} (\vpi^{r}u) d\, \xi_{1}\\
&\phantom{=|\vpi^{-r}|} \int_{\OO_{{\bf V}_{2}}}     \psi(-uc\vpi^{r-i} q(\xi_{2}))     \delta_{q(U)}  (\vpi^{r}u)d\,\xi_{2}  \\
&=|\vpi|^{i-r} \psi_{E,U}(ux_{1})\psi_{q(U)}  (-c\vpi^{-i})   \delta_{q(U)} (\vpi^{r}u) .
\end{align*}
Plugging this in, we obtain
\begin{multline*}
Q_{v}^{\sharp\, (i,c)}= |\vpi|^{i}  \alpha(\vpi)^{-r} |d| ^{1/2}\zeta_{F,v}(1)^{-1}\int_{F^{\times}} W(\smallmat y{}{}1 n^{-}(c\vpi^{r-i}))
 \int_{T(F)}\chi( t) \\
\cdot
 \psi_{E,U}(t)   \psi_{q(U)}(-c\vpi^{-i} )  \delta_{q(U)}(\vpi^{r}y^{-1}q(t))\, d^{\times}t \, d^{\times}y,
\end{multline*}
We have already noted that $Q_{v}^{\sharp (0,1)}=R_{v}^{\circ}$. For $i=1$, if $r$ is sufficiently large then $W$ is still invariant under $K_{1}^{1}(\vpi^{r-1})$; then $\sum_{c}Q_{v}^{\sharp(1,c)}$ equals $C\cdot R_{v}^{\circ}$ with $C=|\vpi|\sum_{c\in (\OO/\vpi)^{\times}} \psi_{q(U)}(-c\vpi^{-1})=-|\vpi|$.

Finally, we claim that for each $i\geq 2$, $\sum_{c}Q_{v}^{\sharp(i, c)}=0$.  Indeed let $q(U)=1+\vpi^{n}\OO_{F,v}$. If  $i\geq n+1$, then $\psi_{q(U)}(-c\vpi^{-i})=0$; if $i\leq n$, then if $r$ is sufficiently large $W$ is still invariant under $K^{1}_{1}(\vpi^{r-i})\subset K^{1}_{1}(\vpi^{r-n})$, and summing the only terms depending on $c$ produces a factor $\sum_{c\in (\OO/\vpi^{i})^{\times}} \psi(c\vpi^{-i})=0$.

 Summing up, we have
$$Q_{v}^{\sharp} = Q_{v}^{\sharp (0,1)}+\sum_{c\in (\OO_{F}/\vpi_{v})^{\times}}Q_{v}^{\sharp(1,c)}=(1-  |\vpi|) R_{r,v}^{\circ}=L(1, \eta_{v})^{-1}\cdot R_{r,v}^{\circ}$$
as desired.
\end{proof}

\begin{enonce}{Question} \textup{A comparison between Propositions \ref{toric at p} and \ref{Qv vs} suggests that the identity
$$\lim_{r\to \infty}L(1, \eta_{v})\cdot\theta_{v}(W_{v}, \alpha^{-r}_{v} w_{r}^{-1}\phi_{v}) = \zeta_{F,v}(2)\cdot f_{\alpha,v}^{+}\otimes f_{\alpha,v}^{-}.$$
might hold in $\varprojlim_{V}(\pi^{+}_{v})^{V}\otimes \varprojlim_{V}(\pi_{v}^{-})^{V}$ (with notation as in Lemma \ref{f limit}).  Is this the case?}
\end{enonce}

\bigskip

\backmatter
\addtocontents{toc}{\medskip}

\end{document}